\documentclass[reqno,11pt]{amsart}
\usepackage{amsmath}
\usepackage{amsxtra}
\usepackage{amscd}
\usepackage{amsthm}
\usepackage{amsfonts}
\usepackage{amssymb}
\usepackage{eucal}
\usepackage{scalerel}
\usepackage{verbatim}
\usepackage[matrix,arrow,curve]{xy}

\usepackage{a4wide}

\usepackage[T1]{fontenc}

\usepackage{xr-hyper}
\usepackage[
pdftex,
bookmarks=false,
colorlinks=true,
debug=true,
pdfnewwindow=true]{hyperref}

\makeatletter
\let\@wraptoccontribs\wraptoccontribs
\makeatother

\theoremstyle{plain}
\newtheorem{Thm}[equation]{Theorem}
\newtheorem{Cor}[equation]{Corollary}
\newtheorem{Lem}[equation]{Lemma}
\newtheorem{Prop}[equation]{Proposition}
\newtheorem{Conj}[equation]{Conjecture}

\theoremstyle{definition}
\newtheorem{Def}[equation]{Definition}

\theoremstyle{remark}

\newtheorem{Rem}[equation]{Remark}

\newtheorem{conj}{Conjecture}

\numberwithin{equation}{section}
\renewcommand{\rm}{\normalshape}

\newif\ifShowLabels
\ShowLabelstrue
\newdimen\theight
\def\TeXref#1{%
    \leavevmode\vadjust{\setbox0=\hbox{{\tt
        \quad\quad  {\small \rm #1}}}%
    \theight=\ht0
    \advance\theight by \lineskip
    \kern -\theight \vbox to
    \theight{\rightline{\rlap{\box0}}%
    \vss}%
    }}%

\ShowLabelsfalse

\newenvironment{thm}[1]%
    { \begin{Thm} \label{T:#1}  \ifShowLabels \TeXref{T:#1} \fi }%
    { \end{Thm} }

\renewcommand{\th}[1]{\begin{thm}{#1} \sl }
\renewcommand{\eth}{\end{thm} }

\newenvironment{lemma}[1]%
    { \begin{Lem} \label{L:#1}  \ifShowLabels \TeXref{L:#1} \fi }%
    { \end{Lem} }
\newcommand{\lem}[1]{\begin{lemma}{#1} \sl}
\newcommand{\elem}{\end{lemma}}

\newenvironment{propos}[1]%
    { \begin{Prop} \label{P:#1}  \ifShowLabels \TeXref{P:#1} \fi }%
    { \end{Prop} }
\newcommand{\prop}[1]{\begin{propos}{#1}\sl }
\newcommand{\eprop}{\end{propos}}

\newenvironment{corol}[1]%
    { \begin{Cor} \label{C:#1}  \ifShowLabels \TeXref{C:#1} \fi }%
    { \end{Cor} }
\newcommand{\cor}[1]{\begin{corol}{#1} \sl }
\newcommand{\ecor}{\end{corol}}

\newenvironment{defeni}[1]%
    { \begin{Def} \label{D:#1}  \ifShowLabels \TeXref{D:#1} \fi }%
    { \end{Def} }
\newcommand{\defe}[1]{\begin{defeni}{#1} \sl }
\newcommand{\edefe}{\end{defeni}}

\newenvironment{remark}[1]%
    { \begin{Rem} \label{R:#1}  \ifShowLabels \TeXref{R:#1} \fi }%
    { \end{Rem} }
\newcommand{\rem}[1]{\begin{remark}{#1}}
\newcommand{\erem}{\end{remark}}

\newenvironment{conjec}[1]%
    { \begin{Conj} \label{Co:#1}  \ifShowLabels \TeXref{Co:#1} \fi }%
    { \end{Conj} }
\renewcommand{\conj}[1]{\begin{conjec}{#1} \sl }
\newcommand{\econj}{\end{conjec}}

\newcommand{\eq}[1]%
    { \ifShowLabels \TeXref{E:#1} \fi
       \begin{equation} \label{E:#1} }
\newcommand{\eeq}{ \end{equation} }

\newcommand{\prf}{ \begin{proof} }
\newcommand{\epr}{ \end{proof} }

\usepackage[dvipsnames]{xcolor}

\newcommand\nc{\newcommand}
\nc{\unl}{\underline}
\nc{\ol}{\overline}
\nc{\on}{\operatorname}

\nc{\BA}{{\mathbb{A}}}
\nc{\BC}{{\mathbb{C}}}
\nc{\BD}{{\mathbb{D}}}
\nc{\BF}{{\mathbb{F}}}
\nc{\BG}{{\mathbb{G}}}
\nc{\BM}{{\mathbb{M}}}
\nc{\BN}{{\mathbb{N}}}
\nc{\BO}{{\mathbb{O}}}
\nc{\BQ}{{\mathbb{Q}}}
\nc{\BP}{{\mathbb{P}}}
\nc{\BR}{{\mathbb{R}}}
\nc{\BZ}{{\mathbb{Z}}}
\nc{\BS}{{\mathbb{S}}}

\nc{\CA}{{\mathcal{A}}} \nc{\CB}{{\mathcal{B}}} \nc{\CalC}{{\mathcal
C}} \nc{\CalD}{{\mathcal D}} \nc{\CE}{{\mathcal{E}}}
\nc{\CF}{{\mathcal{F}}} \nc{\CG}{{\mathcal{G}}}
\nc{\CH}{{\mathcal{H}}} \nc{\CI}{{\mathcal{I}}}
\nc{\CK}{{\mathcal{K}}} \nc{\CL}{{\mathcal{L}}}
\nc{\CM}{{\mathcal{M}}} \nc{\CN}{{\mathcal{N}}}
\nc{\CO}{{\mathcal{O}}} \nc{\CP}{{\mathcal{P}}}
\nc{\CQ}{{\mathcal{Q}}} \nc{\CR}{{\mathcal{R}}}
\nc{\CS}{{\mathcal{S}}} \nc{\CT}{{\mathcal{T}}}
\nc{\CU}{{\mathcal{U}}} \nc{\CV}{{\mathcal{V}}}
\nc{\CW}{{\mathcal{W}}} \nc{\CX}{{\mathcal{X}}}
\nc{\CY}{{\mathcal{Y}}} \nc{\CZ}{{\mathcal{Z}}}
\nc{\CJ}{{\mathcal{J}}}

\nc{\fa}{{\mathfrak{a}}}
\nc{\fb}{{\mathfrak{b}}}
\nc{\fg}{{\mathfrak{g}}}
\nc{\fgl}{{\mathfrak{gl}}}
\nc{\fh}{{\mathfrak{h}}}
\nc{\fj}{{\mathfrak{j}}}
\nc{\fl}{{\mathfrak{l}}}
\nc{\fm}{{\mathfrak{m}}}
\nc{\fn}{{\mathfrak{n}}}
\nc{\fu}{{\mathfrak{u}}}
\nc{\fp}{{\mathfrak{p}}}
\nc{\frr}{{\mathfrak{r}}}
\nc{\fs}{{\mathfrak{s}}}
\nc{\ft}{{\mathfrak{t}}}
\nc{\fw}{{\mathfrak{w}}}
\nc{\fz}{{\mathfrak{z}}}

\nc{\fA}{{\mathfrak{A}}}
\nc{\fB}{{\mathfrak{B}}}
\nc{\fD}{{\mathfrak{D}}}
\nc{\fE}{{\mathfrak{E}}}
\nc{\fF}{{\mathfrak{F}}}
\nc{\fG}{{\mathfrak{G}}}
\nc{\fI}{{\mathfrak{I}}}
\nc{\fJ}{{\mathfrak{J}}}
\nc{\fK}{{\mathfrak{K}}}
\nc{\fL}{{\mathfrak{L}}}
\nc{\fM}{{\mathfrak{M}}}
\nc{\fN}{{\mathfrak{N}}}
\nc{\frP}{{\mathfrak{P}}}
\nc{\fQ}{{\mathfrak Q}}
\nc{\fR}{{\mathfrak R}}
\nc{\fS}{{\mathfrak S}}
\nc{\fT}{{\mathfrak{T}}}
\nc{\fU}{{\mathfrak{U}}}
\nc{\fW}{{\mathfrak{W}}}
\nc{\fY}{{\mathfrak{Y}}}
\nc{\fZ}{{\mathfrak{Z}}}

\nc{\ba}{{\mathbf{a}}}
\nc{\bb}{{\mathbf{b}}}
\nc{\bc}{{\mathbf{c}}}
\nc{\bd}{{\mathbf{d}}}
\nc{\be}{{\mathbf{e}}}
\nc{\bi}{{\mathbf{i}}}
\nc{\bj}{{\mathbf{j}}}
\nc{\bn}{{\mathbf{n}}}
\nc{\bp}{{\mathbf{p}}}
\nc{\bq}{{\mathbf{q}}}
\nc{\bu}{{\mathbf{u}}}
\nc{\bv}{{\mathbf{v}}}
\nc{\bw}{{\mathbf{w}}}
\nc{\bx}{{\mathbf{x}}}
\nc{\by}{{\mathbf{y}}}
\nc{\bz}{{\mathbf{z}}}

\nc{\bA}{{\mathbf{A}}}
\nc{\bB}{{\mathbf{B}}}
\nc{\bC}{{\mathbf{C}}}
\nc{\bD}{{\mathbf{D}}}
\nc{\bE}{{\mathbf{E}}}
\nc{\bI}{{\mathbf{I}}}
\nc{\bK}{{\mathbf{K}}}
\nc{\bH}{{\mathbf{H}}}
\nc{\bM}{{\mathbf{M}}}
\nc{\bN}{{\mathbf{N}}}
\nc{\bO}{{\mathbf{O}}}
\nc{\bQ}{{\mathbf Q}}
\nc{\bS}{{\mathbf{S}}}
\nc{\bT}{{\mathbf{T}}}
\nc{\bV}{{\mathbf{V}}}
\nc{\bW}{{\mathbf{W}}}
\nc{\bX}{{\mathbf{X}}}
\nc{\bP}{{\mathbf{P}}}
\nc{\bY}{{\mathbf{Y}}}
\nc{\bZ}{{\mathbf{Z}}}
\nc{\bU}{{\mathbb{U}}}

\nc{\sA}{{\mathsf{A}}}
\nc{\sB}{{\mathsf{B}}}
\nc{\sC}{{\mathsf{C}}}
\nc{\sD}{{\mathsf{D}}}
\nc{\sF}{{\mathsf{F}}}
\nc{\sK}{{\mathsf{K}}}
\nc{\sM}{{\mathsf{M}}}
\nc{\sO}{{\mathsf{O}}}
\nc{\sQ}{{\mathsf{Q}}}
\nc{\sP}{{\mathsf{P}}}
\nc{\sT}{{\mathsf{T}}}
\nc{\sU}{{\mathsf{U}}}
\nc{\sV}{{\mathsf{V}}}
\nc{\sW}{{\mathsf{W}}}
\nc{\sX}{{\mathsf{X}}}
\nc{\sZ}{{\mathsf{Z}}}
\nc{\sS}{{\mathsf{S}}}

\nc{\sfb}{{\mathsf{b}}}
\nc{\sfc}{{\mathsf{c}}}
\nc{\sd}{{\mathsf{d}}}
\nc{\sg}{{\mathsf{g}}}
\nc{\sk}{{\mathsf{k}}}
\nc{\sfl}{{\mathsf{l}}}
\nc{\sfp}{{\mathsf{p}}}
\nc{\sr}{{\mathsf{r}}}
\nc{\st}{{\mathsf{t}}}
\nc{\sfu}{{\mathsf{u}}}
\nc{\sw}{{\mathsf{w}}}
\nc{\sz}{{\mathsf{z}}}
\nc{\sx}{{\mathsf{x}}}

\nc{\bLambda}{{\boldsymbol{\Lambda}}}
\nc{\vv}{{\boldsymbol{v}}}
\nc{\Fl}{{{\mathcal F}\ell}}
\nc{\Gr}{{\on{Gr}}}
\nc{\CHH}{{\CH\!\!\CH}}
\nc{\lambdavee}{{\lambda^{\!\scriptscriptstyle\vee}}}
\nc{\alphavee}{\alpha^{\!\scriptscriptstyle\vee}}
\nc{\gammavee}{\gamma^{\!\scriptscriptstyle\vee}}
\nc{\rhovee}{{\rho^{\!\scriptscriptstyle\vee}}}
\newcommand\iso{\,\vphantom{j^{X^2}}\smash{\overset{\sim}{\vphantom{\rule{0pt}{0.20em}}\smash{\longrightarrow}}}\,}
\nc{\oQM}{\vphantom{j^{X^2}}\smash{\overset{\circ}{\vphantom{\vstretch{0.7}{A}}\smash{\QM}}}}
\nc{\oZ}{{}^\dagger\!\vphantom{j^{X^2}}\smash{\overset{\circ}{\vphantom{\vstretch{0.7}{A}}\smash{Z}}}}
\nc{\odZ}{{}^\dagger\!\vphantom{j^{X^2}}\smash{\overset{\circ}{\vphantom{\vstretch{0.7}{A}}\smash{\mathfrak Z}}}^{c',c}}
\nc{\bdZ}{{}^\dagger\!\vphantom{j^{X^2}}\smash{\overset{\bullet}{\vphantom{\vstretch{0.7}{A}}\smash{\mathfrak Z}}}^{c',c}}
\nc{\oS}{\vphantom{j^{X^2}}\smash{\overset{\circ}{\vphantom{\vstretch{0.7}{A}}\smash{S}}}}
\nc{\buM}{\vphantom{j^{X^2}}\smash{\overset{\bullet}{\vphantom{\vstretch{0.7}{A}}\smash{M}}}}
\nc{\dW}{{}^\dagger\ol\CW{}}
\nc{\hW}{{}^\dagger\hat\CW{}}
\nc{\wW}{{}^\dagger\widetilde\CW{}}
\nc{\dZ}{{}^\dagger\!\fZ^{c',c}}
\nc{\dZc}{{}^\dagger\!\fZ^{c,c}}
\nc{\tZ}{{}^\dagger\!\tilde{Z}{}}
\nc{\hZ}{{}^\dagger\!\hat{Z}{}}

\nc{\ssl}{\mathfrak{sl}} \nc{\gl}{\mathfrak{gl}}
\nc{\wt}{\widetilde} \nc{\Sym}{\mathrm{Sym}} \nc{\Res}{\mathrm{Res}}
\nc{\sE}{{\mathsf{E}}} \nc{\bs}{{\mathbf{s}}}
\nc{\trig}{\mathrm{trig}} \nc{\rat}{\mathrm{rat}}
\nc{\sign}{\mathrm{sign}} \nc{\sL}{{\mathsf{L}}}
\nc{\fv}{{\mathfrak{v}}} \nc{\ad}{\mathrm{ad}}
\nc{\spsi}{{\mathsf{\psi}}} \nc{\sh}{{\mathsf{h}}}
\nc{\rtt}{\mathrm{rtt}} \nc{\qdet}{\mathrm{qdet}} \nc{\pt}{{\operatorname{pt}}}
\nc{\M}{\mathrm{M}} \nc{\Ker}{\mathrm{Ker}} \nc{\ssc}{\mathrm{sc}}
\nc{\loc}{\mathrm{loc}} \nc{\fra}{\mathrm{frac}}
\nc{\ddj}{\mathrm{DJ}} \nc{\End}{\mathrm{End}} \nc{\ev}{\mathrm{ev}}
\nc{\GL}{\mathrm{GL}} \nc{\Rees}{\mathrm{Rees}} \nc{\tr}{\mathrm{tr}}
\nc{\ext}{\mathrm{ext}} \nc{\op}{\mathrm{op}}

\nc{\eps}{\varepsilon}
\nc{\cI}{\mathcal{I}}
\nc{\Yangian}{Y_\hbar}

\begin{document}
\title[Shifted quantum affine algebras: integral forms in type $A$]
{Shifted quantum affine algebras: integral forms in type $A$}

\author{Michael Finkelberg}
 \address{M.F.:
  National Research University Higher School of Economics, Russian Federation,
  Department of Mathematics, 6 Usacheva st., Moscow 119048;
  Skolkovo Institute of Science and Technology;
  Institute for Information Transmission Problems}
 \email{fnklberg@gmail.com}

\author[Alexander Tsymbaliuk]{Alexander Tsymbaliuk}
 \address{A.T.:  Yale University, Department of Mathematics, New Haven, CT 06511, USA}
 \email{sashikts@gmail.com}
 \contrib[with Appendices by]{Alexander Tsymbaliuk and Alex Weekes}

\dedicatory{To Rafail Kalmanovich Gordin on his 70th birthday}
\thanks{M.F.\ is grateful to his high school teacher Rafail Kalmanovich Gordin for
showing him the beauty of geometry. Since the high school years, he was suffering
from an unrequited love for geometry, the present contribution being but a pathetic
illustration of it.}

\begin{abstract}
  We define an integral form of shifted quantum affine algebras of type $A$ and construct
  Poincar\'e-Birkhoff-Witt-Drinfeld bases for them. When the shift is trivial, our integral
  form coincides with the RTT integral form. We prove that these integral forms are closed
  with respect to the coproduct and shift homomorphisms. We prove that the homomorphism from
  our integral form to the corresponding quantized $K$-theoretic Coulomb branch of a quiver gauge
  theory is always surjective. In one particular case we identify this Coulomb branch with the
  extended quantum universal enveloping algebra of type $A$. Finally, we obtain the rational
  (homological) analogues of the above results (proved earlier in~\cite{kmwy,ktwwya} via different
  techniques).
\end{abstract}
\maketitle
\tableofcontents


\section{Introduction}


\subsection{Summary}
\

This paper is a sequel to~\cite{ft}, where we initiated the study of shifted quantum
affine algebras. Recall that the shifted quantum affine algebra $U^{\mu}_\vv$ depends
on a coweight $\mu$ of a semisimple Lie algebra $\fg$, and in case $\mu=0$ it is just a
central extension of the quantum loop algebra $U_\vv(L\fg)$ over the field $\BC(\vv)$.
Let us represent $\mu$ in the form $\mu=\lambda-\alpha$, where $\lambda$ is a dominant
coweight of $\fg$, and $\alpha$ is a sum of positive coroots. Also, let us assume
from now on that $\fg$ is simply-laced. Then $\lambda$ encodes a framing of a Dynkin
quiver of $\fg$, and $\alpha$ encodes the dimension vector of a representation of this
quiver. Let $\CA^\vv$ stand for the quantized $K$-theoretic Coulomb branch of the
corresponding $3d\ \CN=4$ SUSY quiver gauge theory. It is a $\BC[\vv,\vv^{-1}]$-algebra,
and we denote $\CA^\vv_\fra:=\CA^\vv\otimes_{\BC[\vv,\vv^{-1}]}\BC(\vv)$. One of the main
motivations for our study of shifted quantum affine algebras was the existence of a homomorphism
  $\ol{\Phi}{}^{\unl\lambda}_{\mu}\colon
   U^{\mu}_{\vv}[\sz_1^{\pm1},\ldots,\sz_N^{\pm1}]\to\CA^\vv_\fra$,
where $N$ is the total dimension of the framing. We conjectured that this homomorphism is
surjective and also conjectured an explicit description of its kernel. In other words, we
gave a conjectural presentation of $\CA^\vv_\fra$ by generators and relations as a
truncated shifted quantum affine algebra $U^{\unl\lambda}_\mu$.

It is very much desirable to have a similar presentation for the genuine quantized $K$-theoretic
Coulomb branch $\CA^\vv$ (e.g.\ in order to study the non-quantized $K$-theoretic Coulomb branch
at $\vv=1$). To this end, it is necessary to construct an integral form
(a $\BC[\vv,\vv^{-1}]$-subalgebra)
  $\fU_\vv^{\mu}[\sz_1^{\pm1},\ldots,\sz_N^{\pm1}]\subset
   U^{\mu}_\vv[\sz_1^{\pm1},\ldots,\sz_N^{\pm1}]$
such that
  $\ol{\Phi}{}^{\unl\lambda}_{\mu}(\fU^{\mu}_\vv[\sz_1^{\pm1},\ldots,\sz_N^{\pm1}])=\CA^\vv$
and the specialization
  $\fU_{\vv=1}^{\mu}[\sz_1^{\pm1},\ldots,\sz_N^{\pm1}]$
is a commutative $\BC$-algebra.
Then $\CA^\vv$ would be represented as an explicit quotient algebra $\fU^{\unl\lambda}_\mu$.

In the present paper, we restrict ourselves to the case $\fg=\ssl_n$, and propose a definition
of the desired integral form $\fU^{\mu}_\vv[\sz_1^{\pm1},\ldots,\sz_N^{\pm1}]$. It possesses a
PBWD (Poincar\'e-Birkhoff-Witt-Drinfeld) $\BC[\vv,\vv^{-1}]$-base, cf.~\cite{t}. We prove the
surjectivity of
  $\ol{\Phi}{}^{\unl\lambda}_{\mu}\colon
   \fU^{\mu}_\vv[\sz_1^{\pm1},\ldots,\sz_N^{\pm1}]\to\CA^\vv$
in~Theorem~\ref{Surjectivity}. Unfortunately, we are still unable to say much about the kernel
ideal of
  $\ol{\Phi}{}^{\unl\lambda}_{\mu}\colon
   \fU^{\mu}_\vv[\sz_1^{\pm1},\ldots,\sz_N^{\pm1}]\to\CA^\vv$
in the general case. The only case when we were able to determine the kernel ideal explicitly is
$\fg=\ssl_n,\ \mu=0,\ \lambda=n\omega_{n-1}$ (a multiple of the last fundamental coweight).
Then the corresponding truncated shifted quantum affine $\BC[\vv,\vv^{-1}]$-algebra
$\fU^{\unl\lambda}_\mu$ is isomorphic to an integral form $\widetilde\fU_\vv(\ssl_n)$ of an extended
version $\widetilde{U}_\vv(\ssl_n)$ of the quantized universal enveloping algebra of $\ssl_n$.
More precisely, the Harish-Chandra center $Z$ of $U_\vv(\ssl_n)$ is isomorphic to the ring of
symmetric polynomials
  $\left(\BC(\vv)[\sz_1^{\pm1},\ldots,\sz_n^{\pm1}]\right)^{\Sigma_n}/(\sz_1\cdots\sz_n-1)$,
and
  $\widetilde{U}_\vv(\ssl_n):=
   U_\vv(\ssl_n)\otimes_Z\BC(\vv)[\sz_1^{\pm1},\ldots,\sz_n^{\pm1}]/(\sz_1\cdots\sz_n-1)$,
cf.~\cite{bg}. The corresponding integral form
$\fU_\vv(\ssl_n)=\widetilde\fU_\vv(\ssl_n)\cap U_\vv(\ssl_n)$ of the non-extended quantized
universal enveloping algebra $U_\vv(\ssl_n)$ is nothing but the RTT integral form $\fU_\vv^\rtt(\ssl_n)$.
It is free over $\BC[\vv,\vv^{-1}]$ and admits a PBW basis. The truncation homomorphism
  $U^{0}_\vv[\sz_1^{\pm1},\ldots,\sz_n^{\pm1}]\to \widetilde{U}_\vv(\ssl_n)$
factors through Jimbo's evaluation homomorphism
  $U_\vv(L\ssl_n)[\sz_1^{\pm1},\ldots,\sz_n^{\pm1}]\to \widetilde{U}_\vv(\ssl_n)$ of~\cite{jim},
and $\fU^0_\vv[\sz_1^{\pm1},\ldots,\sz_n^{\pm1}]$ is nothing but the pull-back of the RTT
integral form of $U_\vv(L\ssl_n)[\sz_1^{\pm1},\ldots,\sz_n^{\pm1}]$ along the projection
  $U^{0}_\vv[\sz_1^{\pm1},\ldots,\sz_n^{\pm1}]\twoheadrightarrow
   U_\vv(L\ssl_n)[\sz_1^{\pm1},\ldots,\sz_n^{\pm1}]$.
In fact, our definition of the integral
form $\fU^\mu_\vv[\sz_1^{\pm1},\ldots,\sz_N^{\pm1}]$ for general $\mu$ was found as a
straightforward generalization of the RTT integral form expressed in terms of a PBWD basis.

Note that $U_\vv(\ssl_n)$ possesses three different integral forms:\\
(a) Lusztig's $U\otimes_\BZ\BC$ of~\cite[0.4]{l1};\\
(b) the form $\CU_\CA$ of De Concini and Kac~\cite[1.5]{dk};\\
(c) $\fU_\vv^\rtt(\ssl_n)$ (its specialization at
$\vv=1$ is the commutative ring of functions on the big Bruhat cell of $\on{SL}(n)$).
It is dual to (a) with respect to a natural $\BC(\vv)$-valued pairing on $U_\vv(\ssl_n)$.

Note also that there is a duality between $\fU_\vv(L\ssl_n)$ and the integral form of~\cite{cp2}
and~\cite[$\S 7.8$]{gr} of $U_\vv(L\ssl_n)$ with respect to the new Drinfeld pairing,
cf.~\cite[Lemma~9.1]{gr}, as established in~\cite[Theorem 4.1, Remark 4.38]{t2}.

Finally, recall that in~\cite{ft} we have constructed the comultiplication
$\BC(\vv)$-algebra homomorphisms (in case $\fg=\ssl_n$)
  $\Delta_{\mu_1,\mu_2}\colon U^{\mu_1+\mu_2}_\vv \to U^{\mu_1}_\vv\otimes U^{\mu_2}_\vv$
for any coweights $\mu_1,\mu_2$.
We prove in~Theorem~\ref{coproduct on integral form} that this coproduct preserves
our integral forms, and induces the $\BC[\vv,\vv^{-1}]$-algebra homomorphisms
  $\Delta_{\mu_1,\mu_2}\colon \fU^{\mu_1+\mu_2}_\vv\to\fU^{\mu_1}_\vv\otimes\fU^{\mu_2}_\vv$.

To simplify the exposition of the paper, we start by establishing the rational/homological
counterparts of the aforementioned results, proved earlier in~\cite{kmwy,ktwwya} using
different techniques.

In Appendix~\ref{appendix with Alex Weekes 1}, we collect the relevant results on shifted
Yangians and Drinfeld-Gavarini duals, which are used in Section~\ref{sec additive counterpart}.
Our objectives are two-fold. First, we establish the PBW property for the Drinfeld-Gavarini
dual (Proposition~\ref{appendix 1: main prop}) and apply it to the Yangians
(Theorems~\ref{appendix: main thm 1},~\ref{appendix: main thm 2}). Second, we identify two different
approaches (of~\cite{kwwy} and~\cite{bfn,fkprw}) towards dominantly shifted Yangians of semisimple
Lie algebras (Theorem~\ref{identification of two definitions}).

In Appendix~\ref{appendix with Alex Weekes 2}, we provide a short proof of the well-known
PBW property for the Yangian $Y_\hbar(\fg)$, since the original proof of~\cite{le} contains a gap.


\newpage
\subsection{Outline of the paper}
\

$\bullet$
In Section~\ref{ssec RTT Yangian}, we recall the RTT Yangians
$Y^\rtt_\hbar(\gl_n), Y^\rtt_\hbar(\ssl_n)$ and their $\BC[\hbar]$-subalgebras
$\bY^\rtt_\hbar(\gl_n), \bY^\rtt_\hbar(\ssl_n)$. Since the terminology varies in
the literature, we shall stress right away that the former two are quantizations
of the universal enveloping $U(\gl_n[t]), U(\ssl_n[t])$
(see Remark~\ref{limit of rtt yangian}), while the latter two quantize the algebras
of functions on the congruence subgroups ${\GL(n)[[t^{-1}]]}_1, {\mathrm{SL}(n)[[t^{-1}]]}_1$
(see Remark~\ref{limit of integral rtt yangian}) and can be viewed as the
Drinfeld-Gavarini dual~\cite{g} of the former, see
Appendix~\ref{ssec Drinfeld functor},~\ref{ssec RTT version of Drinfeld-Gavarini}.

In Section~\ref{ssec quantum minors}, we recall the standard definition of the quantum minors and
the quantum determinant of $T(z)$, as well as the description of the center $ZY^\rtt_\hbar(\gl_n)$.
All of this is crucially used in Section~\ref{ssec truncation ideal yangian}.

In Section~\ref{ssec RTT evaluation yangian}, we recall the RTT evaluation homomorphism
$\ev^\rtt\colon Y^\rtt_\hbar(\gl_n)\twoheadrightarrow U(\gl_n)$ as well as the induced
homomorphism between their $\BC[\hbar]$-subalgebras
$\ev^\rtt\colon \bY^\rtt_\hbar(\gl_n)\twoheadrightarrow \bU(\gl_n)$.
The main result of this subsection provides a ``minimalistic description''
of the kernels of these homomorphisms, see Theorems~\ref{alternative kernel yangian}
and~\ref{alternative kernel Yangian integral} (the former is essentially due to~\cite{bk}).

In Section~\ref{ssec Yangian sln}, we recall the Drinfeld Yangians $Y_\hbar(\gl_n)$ and
$Y_\hbar(\ssl_n)$. The isomorphism $\Upsilon\colon Y_\hbar(\gl_n)\iso Y^\rtt_\hbar(\gl_n)$
(see Theorem~\ref{Yangian Gauss decomposition}) is due to~\cite{i} and is essentially a Yangian
counterpart of~\cite{df}. Following~\cite{kwwy}, we define their $\BC[\hbar]$-subalgebras
$\bY_\hbar(\gl_n),\bY_\hbar(\ssl_n)$, and the main result identifies the former with
$\bY^\rtt_\hbar(\gl_n)$ via the isomorphism $\Upsilon$, see
Proposition~\ref{yangian integral forms coincide} (a straightforward proof is sketched right after it,
while a more conceptual one is provided in Appendix~\ref{ssec RTT version of Drinfeld-Gavarini}).

In Section~\ref{ssec Drinfeld evaluation}, we recall the evaluation homomorphism
$\ev\colon Y_\hbar(\ssl_n)\to U(\ssl_n)$ of~\cite{d1} and verify its compatibility
with $\ev^\rtt$ via $\Upsilon$, see Theorem~\ref{compatibility of evaluations yangian}.

In Sections~\ref{ssec shifted Yangian} and~\ref{ssec shifted Yangian 2}, we recall
two alternative definitions of the shifted Yangian $\bY_\mu$ for a general shift $\mu$
and for a dominant shift $\mu$, respectively ($\mu$ is an element of the coweight lattice).
The fact that those two approaches are indeed equivalent for dominant shifts is the subject
of Theorem~\ref{Gavarini=Rees}, the proof of which is presented in
Appendix~\ref{appendix with Alex Weekes 1}, see~Theorem~\ref{identification of two definitions}.

In Sections~\ref{ssec hom to diff op-s yangian} and~\ref{ssec Coulomb branch cohomology},
we recall two key constructions of~\cite[Appendix B]{bfn}: the homomorphism
$\Phi^{\unl\lambda}_\mu\colon \bY_\mu[z_1,\ldots,z_N]\to \wt{\CA}_\hbar$ of
Theorem~\ref{Homomorphism yangian}, which factors through the quantized Coulomb
branch $\CA_\hbar$ giving rise to the homomorphism
  $\ol{\Phi}^{\unl\lambda}_\mu\colon \bY_\mu[z_1,\ldots,z_N]\to \CA_\hbar$.
The main result of this subsection, Proposition~\ref{Surjectivity yangian} due to~\cite{ktwwya},
establishes the surjectivity of $\ol{\Phi}^{\unl\lambda}_\mu$ in type $A$.
An alternative proof of this result is outlined in Remark~\ref{Surjectivity for Yangian via Shuffle}
and crucially utilizes the shuffle realizations of $Y_\hbar(\ssl_n),\bY_\hbar(\ssl_n)$
of~\cite[\S6]{t}.

In Section~\ref{ssec truncation ideal yangian}, we prove a reduced version
of the conjectured description~\cite[Remark B.21]{bfn} of $\Ker(\Phi^{\unl\lambda}_\mu)$
as an explicit truncation ideal $\CI^{\unl{\lambda}}_\mu$ in the particular case
$\mu=0,\lambda=n\omega_{n-1}$ (which corresponds to the dimension vector
$(1,2,\ldots,n-1)$ and the framing $(0,\ldots,0,n)$), see Theorem~\ref{Main Theorem 2}.
An alternative proof of this result was given earlier in~\cite{kmwy}.
The key ingredient in our proof, Theorem~\ref{truncation as kernel yangian}, identifies
the reduced truncation ideal $\unl{\CI}^{n\omega_{n-1}}_0$ with the kernel of a certain
version of the evaluation homomorphism $\ev$. This culminates in
Corollary~\ref{explicit Coulomb yangian}, where we identify the corresponding reduced
Coulomb branch $\unl{\CA}_\hbar$ with the integral form of the extended
(in the sense of~\cite{bg}) universal enveloping algebra of $\ssl_n$.

$\bullet$
In Section~\ref{ssec RTT finite}, we recall the RTT integral form $\fU^\rtt_\vv(\gl_n)$
following~\cite{frt, df}. The latter is a $\BC[\vv,\vv^{-1}]$-algebra, which can be thought
of as a quantization of the algebra of functions on the big Bruhat cell in $\GL(n)$
(see~(\ref{integral_finite_gln}) and Remark~\ref{limit of rtt quantum finite}) as $\vv\to 1$.

In Section~\ref{ssec RTT affine}, we recall the RTT integral form $\fU^\rtt_\vv(L\gl_n)$
following~\cite{frt, df}. The latter is a $\BC[\vv,\vv^{-1}]$-algebra, which can be thought
of as a quantization of the algebra of functions on the thick slice $^\dagger\CW_0$
of~\cite[4(viii)]{ft} (see~(\ref{integral_affine_gln}) and Remark~\ref{thick slice}) as $\vv\to 1$.

In Section~\ref{ssec RTT evaluation}, we recall the RTT evaluation homomorphism
$\ev^\rtt\colon \fU^\rtt_\vv(L\gl_n)\twoheadrightarrow \fU^\rtt_\vv(\gl_n)$.
The main result of this subsection provides a ``minimalistic description''
of the kernel of this homomorphism, see Theorem~\ref{alternative kernel quantum}.

In Section~\ref{ssec DJ finite}, we recall the Drinfeld-Jimbo quantum
$U_\vv(\gl_n), U_\vv(\ssl_n)$ defined over $\BC(\vv)$, and an isomorphism
  $\Upsilon\colon U_\vv(\gl_n)\iso \fU^\rtt_\vv(\gl_n)\otimes_{\BC[\vv,\vv^{-1}]} \BC(\vv)$
of~\cite{df} (see Theorem~\ref{Ding-Frenkel finite}). We introduce
$\BC[\vv,\vv^{-1}]$-subalgebras $\fU_\vv(\gl_n), \fU_\vv(\ssl_n)$ in
Definition~\ref{integral finite}, and identify the former with $\fU^\rtt_\vv(\gl_n)$
via $\Upsilon$, see Proposition~\ref{comparison of integral forms quantum finite}.
Finally, linear $\BC[\vv,\vv^{-1}]$-bases of $\fU_\vv(\gl_n), \fU_\vv(\ssl_n)$ are
constructed in Theorem~\ref{PBW basis coordinate finite}.

In Section~\ref{ssec DJ affine}, we recall the Drinfeld-Jimbo quantum loop algebras
$U_\vv(L\gl_n), U_\vv(L\ssl_n)$ defined over $\BC(\vv)$, and an isomorphism
  $\Upsilon\colon U_\vv(L\gl_n)\iso \fU^\rtt_\vv(L\gl_n)\otimes_{\BC[\vv,\vv^{-1}]} \BC(\vv)$
of~\cite{df} (see Theorem~\ref{Ding-Frenkel affine}). Following~\cite{t}, we
introduce $\BC[\vv,\vv^{-1}]$-subalgebras $\fU_\vv(L\gl_n), \fU_\vv(L\ssl_n)$ in
Definition~\ref{integral loop}, and identify the former with $\fU^\rtt_\vv(L\gl_n)$
via $\Upsilon$, see Proposition~\ref{comparison of integral forms quantum}.
Finally, based on Theorem~\ref{PBW for half-integral} (proved in~\cite{t}),
we construct linear $\BC[\vv,\vv^{-1}]$-bases of $\fU_\vv(L\gl_n), \fU_\vv(L\ssl_n)$
in Theorem~\ref{PBW basis coordinate affine}.

In Section~\ref{ssec shuffle algebra}, we recall the shuffle realizations of
$U^>_\vv(L\gl_n)$ and its integral form $\fU^>_\vv(L\gl_n)$ as recently established
in~\cite{t}, see Theorems~\ref{shuffle rational},~\ref{shuffle integral} and
Proposition~\ref{key properties of integral shuffle}.
This is crucially used in Section~\ref{sec K-theoretic Coulomb branch}.

In Section~\ref{ssec Jimbo evaluation}, we recall the evaluation homomorphism
$\ev\colon U_\vv(L\ssl_n)\to U_\vv(\gl_n)$ of~\cite{jim} (see Theorem~\ref{Jimbo's evaluation})
and verify its compatibility with (a $\BC(\vv)$-extension of) $\ev^\rtt$ via $\Upsilon$,
see Theorem~\ref{compatibility of evaluations}.

In Section~\ref{ssec quantum minors quantum}, we recall the standard definition of the quantum
minors and the quantum determinant of $T^\pm(z)$, as well as the description of the center of
$\fU^\rtt_\vv(\gl_n)$. All of this is crucially used in Section~\ref{ssec truncation ideal}.

In Section~\ref{ssec enhanced algebras}, we slightly generalize the algebras of the previous
subsections, which is needed for Section~\ref{ssec truncation ideal}.

$\bullet$
In Section~\ref{ssec hom to diff op-s}, we recall the notion of shifted
quantum affine algebras of~\cite{ft}:
  $U^{\ssc,\mu}_{\vv}$ and $U^{\ad,\mu}_{\vv}[\sz^{\pm 1}_1,\ldots,\sz^{\pm 1}_N]$
(depending on a coweight $\mu$). We introduce their $\BC[\vv,\vv^{-1}]$-subalgebras
  $\fU^{\ssc,\mu}_{\vv},\fU^{\ad,\mu}_{\vv}[\sz^{\pm 1}_1,\ldots,\sz^{\pm 1}_N]$
and construct linear $\BC[\vv,\vv^{-1}]$-bases for those in
Theorem~\ref{PBW for integral shifted}. We also recall the homomorphism
  $\wt{\Phi}^{\unl\lambda}_\mu\colon
   U^{\ad,\mu}_\vv[\sz^{\pm 1}_1,\ldots,\sz^{\pm 1}_N]\to
   \wt{\CA}^\vv_\fra[\sz^{\pm 1}_1,\ldots,\sz^{\pm 1}_N]$
of~\cite{ft} (see Theorem~\ref{Homomorphism}).

In Section~\ref{ssec Coulomb branch}, we recall the notion of the (extended)
quantized $K$-theoretic Coulomb branch $\CA^\vv$ (which is a $\BC[\vv,\vv^{-1}]$-algebra)
and the fact that $\wt{\Phi}^{\unl\lambda}_\mu$ gives rise to a homomorphism
  $\ol{\Phi}^{\unl\lambda}_\mu\colon U^{\ad,\mu}_\vv[\sz^{\pm 1}_1,\ldots,\sz^{\pm 1}_N]
   \to \CA^\vv\otimes_{\BC[\vv,\vv^{-1}]} \BC(\vv)$.
In Proposition~\ref{integrality of all t-modes} we prove that the integral form
$\fU^{\ad,\mu}_\vv[\sz^{\pm 1}_1,\ldots,\sz^{\pm 1}_N]$ is mapped to $\CA^\vv$ under
$\ol{\Phi}^{\unl\lambda}_\mu$, which is based on explicit
formulas~(\ref{image of e-modes},~\ref{image of f-modes}).
In Theorem~\ref{shuffle homomorphism}, we provide a shuffle interpretation of the
homomorphism $\wt{\Phi}^{\unl\lambda}_\mu$ when restricted to either positive or
negative halves of $U^{\ad,\mu}_\vv[\sz^{\pm 1}_1,\ldots,\sz^{\pm 1}_N]$.
In Proposition~\ref{image of important elements}, we combine this result
with the shuffle description of the integral forms $\fU^>_\vv(L\gl_n),\fU^<_\vv(L\gl_n)$
to compute $\wt{\Phi}^{\unl\lambda}_\mu$-images of certain elements in
$\fU^{\ad,\mu}_\vv[\sz^{\pm 1}_1,\ldots,\sz^{\pm 1}_N]$.
Combining this computation with the ideas of~\cite{cw}, we finally prove that
  $\ol{\Phi}^{\unl\lambda}_\mu\colon
   \fU^{\ad,\mu}_\vv[\sz^{\pm 1}_1,\ldots,\sz^{\pm 1}_N]\to \CA^\vv$
is surjective, see Theorem~\ref{Surjectivity}.

In Section~\ref{ssec truncation ideal}, we prove a reduced version of the integral
counterpart of~\cite[Conjecture 8.14]{ft}, see Conjecture~\ref{conjectural description of kernel},
which identifies $\Ker(\Phi^{\unl\lambda}_\mu)$ with an explicit truncation ideal
$\fI^{\unl{\lambda}}_\mu$ in the particular case $\mu=0,\lambda=n\omega_{n-1}$
(which corresponds to the dimension vector $(1,2,\ldots,n-1)$ and the framing $(0,\ldots,0,n)$),
see Theorem~\ref{Main Theorem 1}. The key ingredient in our proof,
Theorem~\ref{truncation as kernel quantum}, identifies the reduced truncation ideal
$\unl{\fI}^{n\omega_{n-1}}_0$ with the kernel of a certain version of the evaluation homomorphism $\ev$.
This culminates in Corollary~\ref{explicit Coulomb quantum}, where we identify the corresponding
reduced quantized Coulomb branch $\unl{\CA}^\vv$ with the extended version
(in the sense of~\cite{bg}) of $\fU_\vv(\ssl_n)$.

In Section~\ref{ssec coproduct on integral shifted}, we prove that the
$\BC(\vv)$-algebra homomorphisms
  $\Delta_{\mu_1,\mu_2}\colon U^{\ssc,\mu_1+\mu_2}_\vv\to
   U^{\ssc,\mu_1}_\vv\otimes U^{\ssc,\mu_2}_\vv$
of~\cite[Theorem 10.26]{ft} generalizing the Drinfeld-Jimbo coproduct on
$U_\vv(L\ssl_n)$ give rise to $\BC[\vv,\vv^{-1}]$-algebra homomorphisms
  $\Delta_{\mu_1,\mu_2}\colon \fU^{\ssc,\mu_1+\mu_2}_\vv\to
   \fU^{\ssc,\mu_1}_\vv\otimes \fU^{\ssc,\mu_2}_\vv$,
see Theorem~\ref{coproduct on integral form}. We also prove that the
integral forms $\fU^{\ssc,\bullet}_\vv$ are intertwined by the shift
homomorphisms of~\cite[Lemma 10.24]{ft}, see Lemma~\ref{integral iota}.

$\bullet$
In Appendix~\ref{ssec Drinfeld functor}, we recall the notion of the Drinfeld-Gavarini dual
$A'$ of a Hopf algebra $A$ defined over $\BC[\hbar]$,
see~(\ref{delta for Drinfeld dual},~\ref{definition of Drinfeld dual}).

In Appendix~\ref{ssec PBW for Drinfeld-Gavarini}, following the ideas of~\cite{g},
we establish a PBW theorem for the Drinfeld-Gavarini dual $A'$ of a Hopf algebra $A$
satisfying assumptions~(\ref{assumption 1})--(\ref{assumption 3}), see
Proposition~\ref{appendix 1: main prop}. This yields an explicit description of $A'$.

In Appendix~\ref{sec: Rees algebras}, assuming that the Hopf algebra $A$ is in addition graded
(see assumption~(\ref{assumption 4})), we identify its Drinfeld-Gavarini dual $A'$ with
the Rees algebra of the specialization $A_{\hbar=1}$ with respect to the
filtration~(\ref{eq: Kazhdan filtration}), see Proposition~\ref{appendix 1: main cor}.

In Appendix~\ref{appendix: yangian section}, we briefly recall the Yangian $\Yangian=\Yangian(\fg)$
of a semisimple Lie algebra $\fg$ (generalizing the case $\fg=\ssl_n$ featuring in
Section~\ref{sec additive counterpart}) and its key relevant properties.

In Appendix~\ref{ssec Drinfeld-Gavarini of Yangian}, we verify that the aforementioned
assumptions~(\ref{assumption 1})--(\ref{assumption 3}) hold for $\Yangian$, hence,
Proposition~\ref{appendix 1: main prop} applies. This culminates in the explicit description
of the Drinfeld-Gavarini dual $\Yangian'$ (thus filling in the gap of the description of
$\Yangian'$ given just before~\cite[Theorem 3.5]{kwwy}) and establishes a PBW theorem for it,
see Theorem~\ref{appendix: main thm 1}. The validity of the assumption~(\ref{assumption 4})
for $\Yangian$ and Proposition~\ref{appendix 1: main cor} yield a Rees algebra description
of $\Yangian'$, see Corollary~\ref{cor: Gavarini = Rees}.

In Appendix~\ref{ssec RTT version of Drinfeld-Gavarini}, we verify that
assumptions~(\ref{assumption 1})--(\ref{assumption 3}) hold for the RTT Yangian $Y^\rtt_\hbar(\gl_n)$.
This gives rise to the identification of its Drinfeld-Gavarini dual ${Y^\rtt_\hbar(\gl_n)}'$ with
the subalgebra $\bY^\rtt_\hbar(\gl_n)$ of Definition~\ref{RTT integral Yangian}, as well as
establishes the PBW theorem (that we referred to in Section~\ref{sec additive counterpart})
for the latter, see Theorem~\ref{appendix: main thm 2}. As an immediate corollary,
we also deduce a new conceptual proof of Proposition~\ref{yangian integral forms coincide}.

In Appendix~\ref{ssec appendix shifted yangians}, we compare two definitions of
dominantly shifted Yangians for any semisimple Lie algebra $\fg$: the Rees algebra construction
of~Section~\ref{ssec shifted Yangian} (following the approach undertaken in~\cite{bfn,fkprw})
and the subalgebra construction of~Section~\ref{ssec shifted Yangian 2} (following the original
approach of~\cite{kwwy}). Our main result, Theorem~\ref{identification of two definitions}
(generalizing Theorem~\ref{Gavarini=Rees} stated for $\fg=\ssl_n$)
provides an identification of these two definitions.

In Appendix~\ref{ssec shifted Yangian 3}, we introduce one more definition of the shifted Yangian
and prove in Theorem~\ref{I=III} that it is equivalent to the Rees algebra construction.

$\bullet$
In Appendix~\ref{ssec useful lemma}, we state a simple but useful general result,
Lemma~\ref{useful lemma}, relating the specializations of the graded $\BC[\hbar]$-algebra
at $\hbar=0$ and $\hbar=1$. This is needed for Theorem~\ref{PBW for h=1 yangian}.

In Appendix~\ref{ssec Setup for PBW yangian}, we recall the basic facts about $Y=Y_{\hbar=1}$.

In Appendix~\ref{ssec PBW yangian}, we establish the PBW theorem for $Y$
(thus filling in the gap of~\cite{le}, though our proof is different), see
Theorem~\ref{PBW for h=1 yangian}, which allows us to immediately deduce
the PBW theorem for the Yangian $\Yangian$, see Theorem~\ref{PBW Theorem Yangian}.


\newpage

\subsection{Acknowledgments}
\

We are deeply grateful to A.~Braverman, P.~Etingof, B.~Feigin, A.~Molev, and A.~Weekes.
Special thanks go to R.~Kodera and C.~Wendlandt for pointing out two inaccuracies in
the earlier version of this paper.

A.T.\ gratefully acknowledges support from Yale University, and is extremely grateful to
MPIM (Bonn, Germany), IPMU (Kashiwa, Japan), and RIMS (Kyoto, Japan) for the hospitality
and wonderful working conditions in the summer $2018$ when this project was performed.
The final version of this paper was prepared during A.T.'s visit to
IHES (Bures-sur-Yvette, France) in the summer $2019$,
sponsored by the European Research Council (ERC) under the European Union's Horizon $2020$
research and innovation program (QUASIFT grant agreement $677368$).
A.T.\ is indebted to T.~Arakawa (RIMS), T.~Milanov (IPMU), and V.~Pestun (IHES) for their invitations.

M.F.\ was partially funded within the framework of the HSE University Basic Research Program
    and the Russian Academic Excellence Project `5-100'.

A.T.\ was partially supported by the NSF Grant DMS-$1821185$.


\section{Shifted Yangian}\label{sec additive counterpart}

This section is a rational/cohomological prototype of
Sections~\ref{sec quantum algebras},~\ref{sec K-theoretic Coulomb branch}.


\subsection{The RTT Yangian of $\gl_n$ and $\ssl_n$}\label{ssec RTT Yangian}
\

Let $\hbar$ be a formal variable. Consider the \emph{rational} $R$-matrix
\begin{equation}\label{rational R matrix}
  R_\rat(z)=R^\hbar_\rat(z)=1-\frac{\hbar}{z}P
\end{equation}
which is an element of $\BC[\hbar]\otimes_{\BC} (\End\ \BC^n)^{\otimes 2}$, where
$P=\sum_{i,j}E_{ij}\otimes E_{ji}\in (\End\ \BC^n)^{\otimes 2}$ is the permutation
operator. It satisfies the famous \emph{Yang-Baxter equation with a spectral parameter}:
\begin{equation}\label{hYB}
  R_{\rat;12}(u)R_{\rat;13}(u+v)R_{\rat;23}(v)=
  R_{\rat;23}(v)R_{\rat;13}(u+v)R_{\rat;12}(u).
\end{equation}

Following~\cite{frt}, define the \emph{RTT Yangian of $\gl_n$}, denoted by
$Y^\rtt_\hbar(\gl_n)$, to be the associative $\BC[\hbar]$-algebra generated by
$\{t^{(r)}_{ij}\}_{1\leq i,j\leq n}^{r\geq 1}$ subject to the following defining relations:
\begin{equation}\label{ratRTT}
  R_{\rat}(z-w)T_1(z)T_2(w)=T_2(w)T_1(z)R_\rat(z-w).
\end{equation}
Here $T(z)$ is the series in $z^{-1}$ with coefficients in the algebra
$Y^\rtt_\hbar(\gl_n)\otimes \End\ \BC^n$, defined by
$T(z)=\sum_{i,j} t_{ij}(z)\otimes E_{ij}$ with
  $t_{ij}(z):=\delta_{ij}+\hbar\sum_{r>0} t^{(r)}_{ij}z^{-r}$.
Multiplying both sides of~(\ref{ratRTT}) by $z-w$, we obtain
an equality of series in $z,w$ with coefficients in
$Y^\rtt_\hbar(\gl_n)\otimes (\End\ \BC^n)^{\otimes 2}$.

Let $ZY^\rtt_\hbar(\gl_n)$ denote the center of $Y^\rtt_\hbar(\gl_n)$.
Explicitly, $ZY^\rtt_\hbar(\gl_n)\simeq \BC[\hbar][d_1,d_2,\ldots]$ with
$d_r$ defined via $\qdet\ T(z)=1+\hbar\sum_{r\geq 1}d_rz^{-r}$, see
Definition~\ref{quantum determinant} and Proposition~\ref{center of Yangian}.

For any formal series $f(z)\in 1+\frac{\hbar}{z}\BC[\hbar][[z^{-1}]]$, the assignment
\begin{equation}\label{yangian automorphims}
  T(z)\mapsto f(z)T(z)
\end{equation}
defines an algebra automorphism of $Y^\rtt_\hbar(\gl_n)$.

\begin{Def}\label{RTT yangian sln}
The $\BC[\hbar]$-subalgebra $Y^\rtt_\hbar(\ssl_n)$ of
$Y^\rtt_\hbar(\gl_n)$ formed by all the elements
fixed under all automorphisms~(\ref{yangian automorphims})
is called the \emph{RTT Yangian of $\ssl_n$}.
\end{Def}

Analogously to~\cite[Theorem 1.8.2]{m}\footnote{We note that the $\BC$-algebras
of~\emph{loc.cit}.\ are the quotients of their $\BC[\hbar]$-counterparts above by $(\hbar-1)$.},
we have a $\BC[\hbar]$-algebra isomorphism
\begin{equation}\label{RTT gln vs sln}
  Y^\rtt_\hbar(\gl_n)\simeq Y^\rtt_\hbar(\ssl_n)\otimes_{\BC[\hbar]} ZY^\rtt_\hbar(\gl_n).
\end{equation}
Hence, there is a natural projection
  $\pi\colon Y^\rtt_\hbar(\gl_n)\twoheadrightarrow Y^\rtt_\hbar(\ssl_n)$
with $\Ker(\pi)=(d_1,d_2,\ldots)$.

\begin{Rem}\label{limit of rtt yangian}
Note that the assignment $t^{(r)}_{ij}\mapsto E_{ij}\cdot t^{r-1}$ gives rise to
a $\BC$-algebra isomorphism $Y^\rtt_\hbar(\gl_n)/(\hbar)\iso U(\gl_n[t])$.
This explains why $Y^\rtt_\hbar(\gl_n)$ is usually treated as a
quantization of the universal enveloping algebra $U(\gl_n[t])$.
\end{Rem}

\begin{Def}\label{RTT integral Yangian}
Let $\bY^\rtt_\hbar(\gl_n)$ be the $\BC[\hbar]$-subalgebra of $Y^\rtt_\hbar(\gl_n)$
generated by $\{\hbar t^{(r)}_{ij}\}_{1\leq i,j\leq n}^{r\geq 1}$.
\end{Def}

Let us note right away that~(\ref{yangian automorphims}) with
$f(z)\in 1+\frac{\hbar}{z}\BC[\hbar][[z^{-1}]]$ defines an algebra
automorphism of $\bY^\rtt_\hbar(\gl_n)$.
As in Definition~\ref{RTT yangian sln}, define $\bY^\rtt_\hbar(\ssl_n)$ to be
the $\BC[\hbar]$-subalgebra of $\bY^\rtt_\hbar(\gl_n)$ formed by all
the elements fixed under these automorphisms. We also note that the center
$Z\bY^\rtt_\hbar(\gl_n)$ of $\bY^\rtt_\hbar(\gl_n)$ is explicitly given by
$Z\bY^\rtt_\hbar(\gl_n)\simeq \BC[\hbar][\hbar d_1,\hbar d_2,\ldots]$
(clearly $\{\hbar d_r\}_{r\geq 1}\subset \bY^\rtt_\hbar(\gl_n)$).
Finally, we also have a $\BC[\hbar]$-algebra isomorphism
  $\bY^\rtt_\hbar(\gl_n)\simeq
   \bY^\rtt_\hbar(\ssl_n)\otimes_{\BC[\hbar]} Z\bY^\rtt_\hbar(\gl_n)$,
cf.~(\ref{RTT gln vs sln}). Hence, there is a natural projection
$\pi\colon \bY^\rtt_\hbar(\gl_n)\twoheadrightarrow \bY^\rtt_\hbar(\ssl_n)$
with $\Ker(\pi)=(\hbar d_1,\hbar d_2,\ldots)$.

\begin{Rem}\label{limit of integral rtt yangian}
In contrast to Remark~\ref{limit of rtt yangian}, we note that the assignment
$\hbar t^{(r)}_{ij}\mapsto \mathsf{t}^{(r)}_{ij}$ gives rise to a $\BC$-algebra
isomorphism
  $\bY^\rtt_\hbar(\gl_n)/(\hbar)\simeq
   \BC[\mathsf{t}^{(r)}_{ij}]_{1\leq i,j\leq n}^{r\geq 1}$.
In other words, $\bY^\rtt_\hbar(\gl_n)$ can be treated as a
quantization of the algebra of functions on the congruence subgroup
  ${\GL(n)[[t^{-1}]]}_{1}:=\mathrm{the\ kernel\ of\ the\ evaluation\ homomorphism}\
   \GL(n)[[t^{-1}]]\to \GL(n)$.
\end{Rem}


\subsection{Quantum minors of $T(z)$}\label{ssec quantum minors}
\

We recall the notion of quantum minors following~\cite[\S1.6]{m}.
This generalizes $\qdet\ T(z)$ featuring in subsection~\ref{ssec RTT Yangian},
and will be used in the proof of Theorem~\ref{truncation as kernel yangian}.
For $1<r\leq n$, define $R(z_1,\ldots,z_r)\in (\End\ \BC^n)^{\otimes r}$ via
\begin{equation*}
  R(z_1,\ldots,z_r):=(R_{r-1,r})(R_{r-2,r}R_{r-2,r-1})\cdots (R_{1r}\cdots R_{12})\
  \mathrm{with}\ R_{ij}:={R_{\rat;ij}(z_i-z_j)}.
\end{equation*}
The following is implied by~(\ref{hYB}) and~(\ref{ratRTT}),
cf.~\cite[Proposition 1.6.1]{m}:

\begin{Lem}\label{molev 1}
  $R(z_1,\ldots,z_r)T_1(z_1)\cdots T_r(z_r)=T_r(z_r)\cdots T_1(z_1)R(z_1,\ldots,z_r)$.
\end{Lem}

Let $A_r\in (\End\ \BC^n)^{\otimes r}$ denote the image of the antisymmetrizer
$\sum_{\sigma\in \Sigma_r}(-1)^\sigma \cdot \sigma\in \BC[\Sigma_r]$
under the natural action of the symmetric group $\Sigma_r$ on $(\BC^n)^{\otimes r}$.
Recall the following classical observation, cf.~\cite[Proposition 1.6.2]{m}:

\begin{Prop}\label{molev 2}
  $R(z,z-\hbar,\ldots,z-(r-1)\hbar)=A_r$.
\end{Prop}

Combining Lemma~\ref{molev 1} and Proposition~\ref{molev 2}, we obtain the following

\begin{Cor}
We have
\begin{equation}\label{defining minors}
  A_rT_1(z)T_2(z-\hbar)\cdots T_r(z-(r-1)\hbar)=
  T_r(z-(r-1)\hbar)\cdots T_2(z-\hbar)T_1(z)A_r.
\end{equation}
\end{Cor}

The operator of~(\ref{defining minors}) can be written as
  $\sum t^{a_1\ldots a_r}_{b_1\ldots b_r}(z)\otimes
   E_{a_1,b_1}\otimes \cdots \otimes E_{a_r,b_r}$
with
  $t^{a_1\ldots a_r}_{b_1\ldots b_r}(z)\in \bY^\rtt_\hbar(\gl_n)[[z^{-1}]]$
and the sum taken over all $a_1,\ldots,a_r,b_1,\ldots,b_r\in \{1,\ldots,n\}$.

\begin{Def}\label{quantum minor}
The coefficients $t^{a_1\ldots a_r}_{b_1\ldots b_r}(z)$
are called the \emph{quantum minors} of $T(z)$.
\end{Def}

In the particular case $r=n$, the image of the operator $A_n$ acting on
$(\BC^n)^{\otimes n}$ is $1$-dimensional. Hence
$A_nT_1(z)\cdots T_n(z-(n-1)\hbar)=A_n\cdot \qdet\ T(z)$ with
$\qdet\ T(z)\in \bY^\rtt_\hbar(\gl_n)[[z^{-1}]]$. We note that
$\qdet\ T(z)=t^{1\ldots n}_{1\ldots n}(z)$ in the above notations.

\begin{Def}\label{quantum determinant}
$\qdet\ T(z)$ is called the \emph{quantum determinant} of $T(z)$.
\end{Def}

Since $t_{ij}(z)\in \delta_{ij}+\hbar Y^\rtt_\hbar(\gl_n)[[z^{-1}]]$, it is
clear that $\qdet\ T(z)\in 1+\hbar Y^\rtt_\hbar(\gl_n)[[z^{-1}]]$.
Hence, it is of the form $\qdet\ T(z)=1+\hbar\sum_{r\geq 1}d_rz^{-r}$ with
$d_r\in Y^\rtt_\hbar(\gl_n)$. The following result is well-known, cf.~\cite[Theorem 1.7.5]{m}:

\begin{Prop}\label{center of Yangian}
The elements $\{d_r\}_{r\geq 1}$ are central, algebraically independent, and generate
the center $ZY^\rtt_\hbar(\gl_n)$ of $Y^\rtt_\hbar(\gl_n)$. In other words, we have
a $\BC[\hbar]$-algebra isomorphism $ZY^\rtt_\hbar(\gl_n)\simeq \BC[\hbar][d_1,d_2,\ldots]$.
\end{Prop}


\subsection{The RTT evaluation homomorphism $\ev^\rtt$}\label{ssec RTT evaluation yangian}
\

\begin{Def}\label{universal enveloping}
Let $U(\gl_n)$ be the universal enveloping algebra of $\gl_n$ over $\BC[\hbar]$.
\end{Def}

Recall the following two standard relations between $Y^\rtt_\hbar(\gl_n)$
and $U(\gl_n)$:

\begin{Lem}\label{embed+ev rtt yangian}
(a) The assignment $E_{ij}\mapsto t^{(1)}_{ij}$ gives rise to a
$\BC[\hbar]$-algebra embedding
\begin{equation*}
  \iota\colon U(\gl_n)\hookrightarrow Y^\rtt_\hbar(\gl_n).
\end{equation*}

\noindent
(b) The assignment $t^{(r)}_{ij}\mapsto \delta_{r,1}E_{ij}$ gives rise to
a $\BC[\hbar]$-algebra epimorphism
\begin{equation*}
  \ev^\rtt\colon Y^\rtt_\hbar(\gl_n)\twoheadrightarrow U(\gl_n).
\end{equation*}
\end{Lem}

The homomorphism $\ev^\rtt$ is called the \emph{RTT evaluation homomorphism}.

\begin{Rem}\label{property T-matrix}
(a) The composition $\ev^\rtt\circ \iota$ is the identity endomorphism of $U(\gl_n)$.

\noindent
(b) Define $T:=\sum_{i,j} E_{ij}\otimes E_{ij}\in U(\gl_n)\otimes \End\ \BC^n$.
Then $\ev^\rtt\colon T(z)\mapsto 1+\frac{\hbar}{z}T$.
\end{Rem}

Let $\bU(\gl_n)$ be the $\BC[\hbar]$-subalgebra of $U(\gl_n)$ generated by
$\{\hbar x\}_{x\in \gl_n}$. It is isomorphic to the $\hbar$-deformed universal enveloping
algebra: $\bU(\gl_n)\simeq T(\gl_n)/(\langle\{xy-yx-\hbar[x,y]\}_{x,y\in \gl_n}\rangle)$,
where $T(\gl_n)$ denotes the tensor algebra of $\gl_n$ over $\BC[\hbar]$.
We note that the homomorphisms $\iota$ and $\ev^\rtt$ of Lemma~\ref{embed+ev rtt yangian}
give rise to $\BC[\hbar]$-algebra homomorphisms
\begin{equation}\label{embed+ev rtt integral yangian}
  \iota\colon \bU(\gl_n)\hookrightarrow \bY^\rtt_\hbar(\gl_n)\ \mathrm{and}\
  \ev^\rtt\colon \bY^\rtt_\hbar(\gl_n)\twoheadrightarrow \bU(\gl_n).
\end{equation}

The PBW theorems for $Y^\rtt_\hbar(\gl_n)$ (see Proposition~\ref{PBW for RTT yangian},
cf.~\cite[Theorem 1.4.1]{m}) and $U(\gl_n)$ imply the following simple result:

\begin{Lem}\label{simple kernel yangian}
$\Ker(\ev^\rtt\colon Y^\rtt_\hbar(\gl_n)\to U(\gl_n))$ is the $2$-sided ideal
generated by $\{t^{(r)}_{ij}\}_{1\leq i,j\leq n}^{r\geq 2}$.
\end{Lem}

However, we will need an alternative description of this kernel $\Ker(\ev^\rtt)$,
essentially due to~\cite[Section 6]{bk} (by taking further Rees algebras).

\begin{Thm}\label{alternative kernel yangian}
Let $I$ denote the $2$-sided ideal of $Y^\rtt_\hbar(\gl_n)$ generated by
$\{t^{(r)}_{11}\}_{r\geq 2}$. Then
$\Ker(\ev^\rtt\colon Y^\rtt_\hbar(\gl_n) \to U(\gl_n))=I$.
\end{Thm}

\begin{proof}
Recall that~(\ref{ratRTT}) is equivalent to
\begin{equation*}
  (z-w)[t_{ij}(z),t_{kl}(w)]=\hbar(t_{kj}(z)t_{il}(w)-t_{kj}(w)t_{il}(z))
\end{equation*}
for any $1\leq i,j,k,l\leq n$, which in turn is equivalent to
(cf.~\cite[Proposition 1.1.2]{m})
\begin{equation}\label{yangian relation}
  [t^{(r)}_{ij},t^{(s)}_{kl}]=
  \hbar\sum_{a=1}^{\min(r,s)} \left(t^{(a-1)}_{kj}t^{(r+s-a)}_{il}-t^{(r+s-a)}_{kj}t^{(a-1)}_{il}\right),
\end{equation}
where we set $t^{(0)}_{ij}:=\hbar^{-1}\delta_{ij}$.

\medskip
\noindent
$\bullet$
Set $i=j=k=1, l>1, s=1$ in~(\ref{yangian relation}) to get
$[t^{(r)}_{11},t^{(1)}_{1l}]=t^{(r)}_{1l}$. Hence
$\{t^{(r)}_{1l}\}_{l>1}^{r\geq 2}\subset I$.

\medskip
\noindent
$\bullet$
Set $i=j=l=1, k>1, s=1$ in~(\ref{yangian relation}) to get
$[t^{(r)}_{11},t^{(1)}_{k1}]=-t^{(r)}_{k1}$. Hence
$\{t^{(r)}_{k1}\}_{k>1}^{r\geq 2}\subset I$.

\medskip
\noindent
$\bullet$
Set $i=l=1, j=k=2, s=1$ in~(\ref{yangian relation}) to get
$[t^{(r)}_{12},t^{(1)}_{21}]=t^{(r)}_{11}-t^{(r)}_{22}$. Hence
$\{t^{(r)}_{22}\}_{r\geq 2}\subset I$.

\medskip
One can now apply the above three verifications with all lower indices
increased by $1$. Proceeding further step by step, we obtain
$\{t^{(r)}_{ij}\}_{1\leq i,j\leq n}^{r\geq 2}\subset I$.

\medskip
This completes our proof of Theorem~\ref{alternative kernel yangian}.
\end{proof}

Likewise, the PBW theorems for $\bU(\gl_n)$ and $\bY^\rtt_\hbar(\gl_n)$
(see Theorem~\ref{appendix: main thm 2}, cf.~\cite[Theorem~1.4.1]{m})
imply the following result:

\begin{Lem}\label{simple kernel yangian integral}
$\Ker(\ev^\rtt\colon \bY^\rtt_\hbar(\gl_n)\to \bU(\gl_n))$ is the $2$-sided ideal
generated by $\{\hbar t^{(r)}_{ij}\}_{1\leq i,j\leq n}^{r\geq 2}$.
\end{Lem}

The following alternative description follows immediately from
Theorem~\ref{alternative kernel yangian}:

\begin{Thm}\label{alternative kernel Yangian integral}
$\Ker(\ev^\rtt\colon \bY^\rtt_\hbar(\gl_n) \to \bU(\gl_n))=\bY^\rtt_\hbar(\gl_n)\cap I$.
\end{Thm}


\subsection{The Drinfeld Yangian of $\gl_n$ and $\ssl_n$}\label{ssec Yangian sln}
\

Following~\cite{d2} (cf.~\cite{i,m}), define the Yangian of $\gl_n$, denoted by
$Y_\hbar(\gl_n)$, to be the associative $\BC[\hbar]$-algebra generated by
  $\{e_i^{(r)},f_i^{(r)},\zeta_j^{(r)}\}_{1\leq i<n, 1\leq j\leq n}^{r\geq 0}$
with the following defining relations:
\begin{equation}\label{Dr Yangian}
\begin{split}
  & [\zeta_j^{(r)},\zeta_{j'}^{(s)}]=0,\\
  & [e_{i}^{(r+1)}, e_{i'}^{(s)}]-[e_{i}^{(r)},e_{i'}^{(s+1)}]=
    \frac{c_{ii'}\hbar}{2}\left(e_{i}^{(r)}e_{i'}^{(s)}+e_{i'}^{(s)}e_{i}^{(r)}\right),\\
  & [f_{i}^{(r+1)}, f_{i'}^{(s)}]-[f_{i}^{(r)},f_{i'}^{(s+1)}]=
    -\frac{c_{ii'}\hbar}{2}\left(f_{i}^{(r)}f_{i'}^{(s)}+f_{i'}^{(s)}f_{i}^{(r)}\right),\\
  & [\zeta_{j}^{(0)},e_{i}^{(r)}]=(-\delta_{ji}+\delta_{j,i+1}) e_{i}^{(r)},\
    [\zeta_{j}^{(0)},f_{i}^{(r)}]=(\delta_{ji}-\delta_{j,i+1}) f_{i}^{(r)},\\
  & [\zeta_{j}^{(s+1)}, e_{i}^{(r)}]-[\zeta_{j}^{(s)},e_{i}^{(r+1)}]=
    \hbar\cdot \left(-\delta_{ji}\zeta_{j}^{(s)}e_{i}^{(r)}+
    \delta_{j,i+1}/2\cdot(\zeta_{j}^{(s)}e_{i}^{(r)}+e_{i}^{(r)}\zeta_{j}^{(s)})\right),\\
  & [\zeta_{j}^{(s+1)}, f_{i}^{(r)}]-[\zeta_{j}^{(s)},f_{i}^{(r+1)}]=
    \hbar\cdot \left(\delta_{ji}\zeta_{j}^{(s)}f_{i}^{(r)}-
    \delta_{j,i+1}/2\cdot (\zeta_{j}^{(s)}f_{i}^{(r)}+f_{i}^{(r)}\zeta_{j}^{(s)})\right),\\
  & [e_{i}^{(r)},f_{i'}^{(s)}]=\delta_{ii'}h_{i}^{(r+s)},\\
  & [e_{i}^{(r)},e_{i'}^{(s)}]=0\ \mathrm{and}\ [f_{i}^{(r)},f_{i'}^{(s)}]=0\  \mathrm{if}\ c_{ii'}=0,\\
  & [e_{i}^{(r_1)},[e_{i}^{(r_2)},e_{i'}^{(s)}]]+
    [e_{i}^{(r_2)},[e_{i}^{(r_1)},e_{i'}^{(s)}]]=0\ \mathrm{if}\ c_{ii'}=-1,\\
  & [f_{i}^{(r_1)},[f_{i}^{(r_2)},f_{i'}^{(s)}]]+
    [f_{i}^{(r_2)},[f_{i}^{(r_1)},f_{i'}^{(s)}]]=0\ \mathrm{if}\ c_{ii'}=-1,
\end{split}
\end{equation}
where $(c_{ii'})_{i,i'=1}^{n-1}$ denotes the Cartan matrix of $\ssl_n$ and
$\{h_i^{(r)}\}_{1\leq i<n}^{r\in \BN}$ are the coefficients of the generating
series $h_i(z)=1+\hbar\sum_{r\geq 0}h_{i}^{(r)}z^{-r-1}$ determined via
$h_i(z):=(\zeta_i(z))^{-1}\zeta_{i+1}(z-\hbar/2)$. Here the generating
series $e_i(z), f_i(z)\ (1\leq i<n)$ and $\zeta_j(z)\ (1\leq j\leq n)$ are defined via
\begin{equation*}
  e_i(z):=\hbar\sum_{r\geq 0}e_{i}^{(r)}z^{-r-1},\
  f_i(z):=\hbar\sum_{r\geq 0}f_{i}^{(r)}z^{-r-1},\
  \zeta_j(z):=1+\hbar\sum_{r\geq 0}\zeta_{j}^{(r)}z^{-r-1}.
\end{equation*}

The $\BC[\hbar]$-subalgebra of $Y_\hbar(\gl_n)$ generated by
  $\{e_{i}^{(r)},f_{i}^{(r)},h_{i}^{(r)}\}_{1\leq i<n}^{r\geq 0}$
is isomorphic to the Yangian of $\ssl_n$, denoted by $Y_\hbar(\ssl_n)$.
To be more precise, this recovers the new Drinfeld realization of $Y_\hbar(\ssl_n)$,
see~\cite{d2}. The latter also admits the original \emph{$J$-presentation}
with generators $\{x,J(x)\}_{x\in \ssl_n}$ and a certain list of the
defining relations which we shall skip, see~\cite{d1}.

To relate $Y^\rtt_\hbar(\gl_n)$ and $Y_\hbar(\gl_n)$, consider
the Gauss decomposition of $T(z)$ of subsection~\ref{ssec RTT Yangian}:
\begin{equation*}
  T(z)=F(z)\cdot G(z)\cdot E(z).
\end{equation*}
Here $F(z),G(z),E(z)$ are the series in $z^{-1}$ with coefficients in the algebra
$\bY^\rtt_\hbar(\gl_n)\otimes \End\ \BC^n$ which are of the form
\begin{equation*}
  F(z)=\sum_{i} E_{ii}+\sum_{i>j} f_{ij}(z)\cdot E_{ij},\
  G(z)=\sum_{i} g_i(z)\cdot E_{ii},\
  E(z)=\sum_{i} E_{ii}+\sum_{i<j} e_{ij}(z)\cdot E_{ij}.
\end{equation*}

\begin{Thm}[\cite{i}, cf.~\cite{df}]\label{Yangian Gauss decomposition}
There is a unique $\BC[\hbar]$-algebra isomorphism
\begin{equation*}
  \Upsilon\colon Y_\hbar(\gl_n)\iso Y^\rtt_\hbar(\gl_n)
\end{equation*}
defined by
\begin{equation}\label{matching yangian}
  e_i(z)\mapsto e_{i,i+1}(z+i\hbar/2),\
  f_i(z)\mapsto f_{i+1,i}(z+i\hbar/2),\
  \zeta_j(z)\mapsto  g_j(z+j\hbar/2).
\end{equation}
\end{Thm}

As an immediate corollary, $\bY^\rtt_\hbar(\gl_n)$ is realized as a
$\BC[\hbar]$-subalgebra of $Y_\hbar(\gl_n)$. To describe this subalgebra
explicitly, define the elements
  $\{E^{(r)}_{\alphavee},F^{(r)}_{\alphavee}\}_{\alphavee\in \Delta^+}^{r\geq 0}$
of $Y_\hbar(\gl_n)$ via
\begin{equation}\label{PBW basis elements yangian}
\begin{split}
  & E^{(r)}_{\alphavee_j+\alphavee_{j+1}+\ldots+\alphavee_i}:=[\cdots[e_{j,}^{(r)},e_{j+1}^{(0)}],\cdots,e_{i}^{(0)}],\\
  & F^{(r)}_{\alphavee_j+\alphavee_{j+1}+\ldots+\alphavee_i}:=[f_{i}^{(0)},\cdots,[f_{j+1}^{(0)},f_{j}^{(r)}]\cdots].
\end{split}
\end{equation}
Here $\{\alphavee_i\}_{i=1}^{n-1}$ are the standard simple roots of $\ssl_n$,
and $\Delta^+$ denotes the set of positive roots,
$\Delta^+=\{\alphavee_j+\alphavee_{j+1}+\ldots+\alphavee_i\}_{1\leq j\leq i\leq n-1}$.

\begin{Def}\label{integral Drinfeld yangian}
(a) Let $\bY_\hbar(\gl_n)$ be the $\BC[\hbar]$-subalgebra of $Y_\hbar(\gl_n)$
generated by
\begin{equation}\label{PBW basis yangian}
   \{\hbar E^{(r)}_{\alphavee}, \hbar F^{(r)}_{\alphavee}\}_{\alphavee\in \Delta^+}^{r\geq 0}\cup
   \{\hbar\zeta_{j}^{(r)}\}_{1\leq j\leq n}^{r\geq 0}.
\end{equation}

\noindent
(b) Let $\bY_\hbar(\ssl_n)$ be the $\BC[\hbar]$-subalgebra of $Y_\hbar(\ssl_n)$
generated by
\begin{equation}\label{PBW basis yangian sln}
   \{\hbar E^{(r)}_{\alphavee}, \hbar F^{(r)}_{\alphavee}\}_{\alphavee\in \Delta^+}^{r\geq 0}\cup
   \{\hbar h_{i}^{(r)}\}_{1\leq i< n}^{r\geq 0}.
\end{equation}
\end{Def}

\begin{Rem}\label{pbw for integral yangian}
The subalgebra $\bY_\hbar(\gl_n)$ is free over $\BC[\hbar]$ and the ordered
PBW monomials in the generators~(\ref{PBW basis yangian}) form its basis.
This can be derived similarly to~Theorem~\ref{PBW basis coordinate affine},
cf.~\cite[Theorem 6.8]{t}. An alternative proof (valid for all Yangians) is provided
in Appendix~\ref{appendix with Alex Weekes 1}, see~Theorem~\ref{appendix: main thm 1}.
\end{Rem}

\begin{Prop}\label{yangian integral forms coincide}
  $\bY_\hbar(\gl_n)=\Upsilon^{-1}(\bY^\rtt_\hbar(\gl_n))$.
\end{Prop}

The proof of Proposition~\ref{yangian integral forms coincide} follows immediately
from Proposition~\ref{higher ef-modes yangian} and Corollary~\ref{explicit modes yangian}
below. To state those, let us express the matrix coefficients of $F(z),G(z),E(z)$
as series in $z^{-1}$ with coefficients in $Y^\rtt_\hbar(\gl_n)$:
\begin{equation}\label{t-modes yangian}
  e_{ij}(z)=\hbar\sum_{r\geq 1} e^{(r)}_{ij}z^{-r},\
  f_{ij}(z)=\hbar\sum_{r\geq 1} f^{(r)}_{ij}z^{-r},\
  g_i(z)=1+\hbar\sum_{r\geq 1} g^{(r)}_i z^{-r}.
\end{equation}

The proof of the following result is analogous to that of
Proposition~\ref{higher ef-modes} (actually it is much simpler),
and we leave details to the interested reader:

\begin{Prop}\label{higher ef-modes yangian}
For any $1\leq j< i<n$, the following equalities hold in $Y^\rtt_\hbar(\gl_n)$:
\begin{equation}\label{Gauss ef-modes yangian}
  e_{j,i+1}(z)=[e_{ji}(z),e^{(1)}_{i,i+1}],\
  f_{i+1,j}(z)=[f^{(1)}_{i+1,i},f_{ij}(z)].
\end{equation}
\end{Prop}

\begin{Cor}\label{explicit modes yangian}
For any $1\leq j\leq i<n$ and $r\geq 1$, the following equalities hold:
\begin{equation}\label{Gauss Matrix Entries yangian}
\begin{split}
  & e^{(r)}_{j,i+1}=[\cdots[e^{(r)}_{j,j+1},e^{(1)}_{j+1,j+2}],\cdots,e^{(1)}_{i,i+1}],\\
  & f^{(r)}_{i+1,j}=[f^{(1)}_{i+1,i},\cdots,[f^{(1)}_{j+2,j+1},f^{(r)}_{j+1,j}]\cdots].
\end{split}
\end{equation}
\end{Cor}

\begin{Rem}
A more conceptual and computation-free proof of Proposition~\ref{yangian integral forms coincide}
is provided in the end of Appendix~\ref{ssec RTT version of Drinfeld-Gavarini}.
\end{Rem}


\subsection{The Drinfeld evaluation homomorphism $\ev$}\label{ssec Drinfeld evaluation}
\

While the universal enveloping algebra (over $\BC[\hbar]$) $U(\fg)$ is
always embedded into the Yangian $Y_\hbar(\fg)$, in type $A$ there also
exists a $\BC[\hbar]$-algebra epimorphism
\begin{equation*}
  \ev\colon Y_\hbar(\ssl_n)\twoheadrightarrow U(\ssl_n)
\end{equation*}
discovered in~\cite[Theorem 9]{d1}.
This homomorphism is given in the $J$-presentation of $Y_\hbar(\ssl_n)$.
We shall skip explicit formulas, referring the reader to~\cite{d1}
and~\cite[Proposition 12.1.15]{cp}.

Define $s_i\in Y_\hbar(\ssl_n)$ via
\begin{equation}\label{shift elements}
  s_i:=h_i^{(1)}-\frac{\hbar}{2}(h_i^{(0)})^2,
\end{equation}
so that
\begin{equation*}
  [s_i,e_{i'}^{(r)}]=c_{ii'}e_{i'}^{(r+1)},\
  [s_i,f_{i'}^{(r)}]=-c_{ii'}f_{i'}^{(r+1)}.
\end{equation*}
As a result, $Y_\hbar(\ssl_n)$ is generated by
$\{e_i^{(0)}, f_i^{(0)}, s_1\}_{i=1}^{n-1}$. We will need
the following explicit formulas:
\begin{equation}\label{evaluation explicit ag}
  \ev(e_i^{(0)})=E_{i,i+1},\ \ev(f_i^{(0)})=E_{i+1,i},\
  \ev(s_1)=\frac{\hbar}{2}(\omega_2 h_1-E_{12}E_{21}-E_{21}E_{12}),
\end{equation}
where
  $h_1=E_{11}-E_{22},\ \omega_2=E_{11}+E_{22}-\frac{2}{n} I_n,\ I_n=E_{11}+\ldots+E_{nn}$.
The last equality of~(\ref{evaluation explicit ag}) is verified by
a straightforward computation (sketched in~\cite[$\S 5.7$]{ag}).

Let $\wt{\gamma}\colon U(\gl_n)\twoheadrightarrow U(\ssl_n)$ be the $\BC[\hbar]$-algebra
epimorphism defined by $\wt{\gamma}(X)=X-\frac{\tr(X)}{n}\cdot I_n$ for $X\in \gl_n$.
We also define a $\BC[\hbar]$-algebra embedding
  $\wt{\Upsilon}\colon Y_\hbar(\ssl_n)\hookrightarrow Y^\rtt_\hbar(\gl_n)$
as a composition of an automorphism of $Y_\hbar(\ssl_n)$ defined by
  $e_i(z)\mapsto e_i(z-\hbar), f_i(z)\mapsto f_i(z-\hbar), h_i(z)\mapsto h_i(z-\hbar)$,
a natural embedding $Y_\hbar(\ssl_n)\hookrightarrow Y_\hbar(\gl_n)$,
and the isomorphism $\Upsilon\colon Y_\hbar(\gl_n)\iso Y^\rtt_\hbar(\gl_n)$
of Theorem~\ref{Yangian Gauss decomposition}.

The key result of this subsection establishes the relation between the
evaluation homomorphism $\ev$ and the RTT evaluation homomorphism $\ev^\rtt$
of Lemma~\ref{embed+ev rtt yangian}(b):

\begin{Thm}\label{compatibility of evaluations yangian}
The following diagram is commutative:
\begin{equation}\label{diagram yangian}
  \begin{CD}
    Y_\hbar(\ssl_n) @>\wt{\Upsilon}>> Y^\rtt_\hbar(\gl_n)\\
    @VV{\on{\ev}}V @V{\ev^\rtt}VV\\
    U(\ssl_n) @<\wt{\gamma}<< U(\gl_n)
    \end{CD}
\end{equation}
\end{Thm}

\begin{proof}
It suffices to verify $\wt{\gamma}(\ev^\rtt(\wt{\Upsilon}(X)))=\ev(X)$
for all $X\in \{e_i^{(0)}, f_i^{(0)}, s_1\}_{i=1}^{n-1}$. This
equality is obvious for $e_i^{(0)}, f_i^{(0)}$, hence, it remains to
verify it for $X=s_1$.

Note that
  $\wt{\Upsilon}(h_1(z))=g_1(z-\hbar/2)^{-1}g_2(z-\hbar/2)$.
Using the notations of~(\ref{t-modes yangian}), this implies
\begin{equation*}
  \wt{\Upsilon}(h_1^{(0)})=g_2^{(1)}-g_1^{(1)},\
  \wt{\Upsilon}(h_1^{(1)})=
  \hbar\left((g_1^{(1)})^2-g_1^{(1)}g_2^{(1)}+\frac{g_2^{(1)}-g_1^{(1)}}{2}\right)
  +(g_2^{(2)}-g_1^{(2)}),
\end{equation*}
so that
\begin{equation*}
  \wt{\Upsilon}(s_1)=
  \frac{\hbar}{2}\left((g_1^{(1)})^2-(g_2^{(1)})^2+g_2^{(1)}-g_1^{(1)}\right)+(g_2^{(2)}-g_1^{(2)}).
\end{equation*}

On the other hand, considering the Gauss decomposition of the matrix
$1+\hbar Tz^{-1}=\ev^\rtt(T(z))$ of~Remark~\ref{property T-matrix}(b), we find
  $\ev^\rtt\colon g_1^{(1)}\mapsto E_{11},\ g_1^{(2)}\mapsto 0,\
   g_2^{(1)}\mapsto E_{22},\ g_2^{(2)}\mapsto -\hbar E_{21}E_{12}$.
Therefore, we obtain
\begin{equation*}
  \ev^\rtt(\wt{\Upsilon}(s_1))=\frac{\hbar}{2}(E_{11}^2-E_{22}^2+E_{22}-E_{11}-2E_{21}E_{12})=
  \frac{\hbar}{2}(E_{11}^2-E_{22}^2-E_{12}E_{21}-E_{21}E_{12}).
\end{equation*}
Applying $\wt{\gamma}$, we finally get
\begin{equation*}
  \wt{\gamma}(\ev^\rtt(\wt{\Upsilon}(s_1)))=
  \frac{\hbar}{2}(\omega_2 h_1-E_{12}E_{21}-E_{21}E_{12})=\ev(s_1),
\end{equation*}
due to the last formula of~(\ref{evaluation explicit ag}).

\medskip
This completes our proof of Theorem~\ref{compatibility of evaluations yangian}.
\end{proof}


\subsection{The shifted Yangian, construction I}\label{ssec shifted Yangian}
\

In this subsection, we recall the notion of shifted Yangians following~\cite[Appendix B]{bfn}.

First, recall that given a $\BC$-algebra $A$ with an algebra filtration
$F^\bullet A=\cdots \subseteq F^{-1}A\subseteq F^0A \subseteq F^1A\subseteq\cdots$
which is separated and exhaustive (that is, $\cap_kF^k A=0$ and $\cup_k F^k A=A$),
we define the \emph{Rees algebra} of $A$ to be the graded $\BC[\hbar]$-algebra
$\Rees^{F^\bullet} A:=\bigoplus_k \hbar^k F^k A$, viewed as a subalgebra of $A[\hbar,\hbar^{-1}]$.

Following~\cite[Definition B.1]{bfn}, define the \emph{Cartan doubled Yangian}
$Y_\infty=Y_{\infty}(\ssl_n)$ to be the $\BC$-algebra generated by
$\{E_{i}^{(r)},F_{i}^{(r)},H_{i}^{(s)}\}_{1\leq i\leq n-1}^{r\geq 1, s\in \BZ}$
with the following defining relations:
\begin{equation}\label{doubled yangian}
\begin{split}
  & [H_i^{(s)},H_j^{(s')}]=0,\\
  & [E_i^{(r)},F_j^{(r')}]=\delta_{ij}H_i^{(r+r'-1)},\\
  & [H_i^{(s+1)},E_j^{(r)}]-[H_i^{(s)},E_j^{(r+1)}]=
    \frac{c_{ij}}{2}(H_i^{(s)}E_j^{(r)}+E_j^{(r)}H_i^{(s)}),\\
  & [H_i^{(s+1)},F_j^{(r)}]-[H_i^{(s)},F_j^{(r+1)}]=
    -\frac{c_{ij}}{2}(H_i^{(s)}F_j^{(r)}+F_j^{(r)}H_i^{(s)}),\\
  & [E_i^{(r+1)},E_j^{(r')}]-[E_i^{(r)},E_j^{(r'+1)}]=
    \frac{c_{ij}}{2}(E_i^{(r)}E_j^{(r')}+E_j^{(r')}E_i^{(r)}),\\
  & [F_i^{(r+1)},F_j^{(r')}]-[F_i^{(r)},F_j^{(r'+1)}]=
    -\frac{c_{ij}}{2}(F_i^{(r)}F_j^{(r')}+F_j^{(r')}F_i^{(r)}),\\
  & [E_i^{(r)},E_j^{(r')}]=0\ \mathrm{and}\ [F_i^{(r)},F_j^{(r')}]=0\  \mathrm{if}\ c_{ij}=0,\\
  & [E_i^{(r_1)},[E_i^{(r_2)},E_j^{(r')}]]+
    [E_i^{(r_2)},[E_i^{(r_1)},E_j^{(r')}]]=0\ \mathrm{if}\ c_{ij}=-1,\\
  & [F_i^{(r_1)},[F_i^{(r_2)},F_j^{(r')}]]+
    [F_i^{(r_2)},[F_i^{(r_1)},F_j^{(r')}]]=0\ \mathrm{if}\ c_{ij}=-1.
\end{split}
\end{equation}

Fix a coweight $\mu$ of $\ssl_n$ and set $b_i:=\alphavee_i(\mu)$.
Following~\cite[Definition B.2]{bfn}, define
$Y_\mu=Y_\mu(\ssl_n)$ as the quotient of $Y_\infty$ by the relations
$H_i^{(r)}=0$ for $r<-b_i$ and $H_i^{(-b_i)}=1$.

Analogously to~(\ref{PBW basis elements yangian}), define the elements
$\{E_{\alphavee}^{(r)},F_{\alphavee}^{(r)}\}_{\alphavee\in \Delta^+}^{r\geq 1}$
of $Y_\mu$ via
\begin{equation}\label{PBW basis elements yangian non-h}
\begin{split}
  & E_{\alphavee_j+\alphavee_{j+1}+\ldots+\alphavee_i}^{(r)}:=
    [\cdots[E_j^{(r)},E_{j+1}^{(1)}],\cdots,E_i^{(1)}],\\
  & F_{\alphavee_j+\alphavee_{j+1}+\ldots+\alphavee_i}^{(r)}:=
    [F_i^{(1)},\cdots,[F_{j+1}^{(1)},F_j^{(r)}]\cdots].
\end{split}
\end{equation}
Choose any total ordering on the following set of \emph{PBW generators}:
\begin{equation}\label{pbw bases}
  \{E_{\alphavee}^{(r)}\}_{\alphavee\in \Delta^+}^{r\geq 1}\cup
  \{F_{\alphavee}^{(r)}\}_{\alphavee\in \Delta^+}^{r\geq 1}\cup
  \{H_i^{(r)}\}_{1\leq i\leq n-1}^{r>-b_i}.
\end{equation}
The following PBW property of $Y_\mu$ was established in~\cite[Corollary 3.15]{fkprw}:

\begin{Thm}[\cite{fkprw}]\label{PBW fkprw}
For an arbitrary coweight $\mu$, the ordered PBW monomials in the
generators~(\ref{pbw bases}) form a $\BC$-basis of $Y_\mu$.
\end{Thm}

Fix a pair of coweights $\mu_1,\mu_2$ such that $\mu_1+\mu_2=\mu$.
Following~\cite[$\S 5.4$]{fkprw}, consider the filtration $F^\bullet_{\mu_1,\mu_2}Y_\mu$
of $Y_\mu$ by defining degrees of the PBW generators as follows:
\begin{equation}\label{yangian filtration}
  \deg E_{\alphavee}^{(r)}=\alphavee(\mu_1)+r,\
  \deg F_{\alphavee}^{(r)}=\alphavee(\mu_2)+r,\
  \deg H_i^{(r)}=\alphavee_i(\mu)+r.
\end{equation}
More precisely, $F^k_{\mu_1,\mu_2}Y_\mu$ is defined as the span of all
ordered PBW monomials whose total degree is at most $k$.

According to~\cite{fkprw}, this defines an algebra filtration and the
Rees algebras $\Rees^{F^\bullet_{\mu_1,\mu_2}} Y_\mu$ are canonically
isomorphic for any choice of $\mu_1,\mu_2$ as above.

\begin{Def}\label{defn 1 of shifted y}
Define the \emph{shifted Yangian} $\bY_\mu=\bY_\mu(\ssl_n)$ via
$\bY_\mu:=\Rees^{F^\bullet_{\mu_1,\mu_2}} Y_\mu$.
\end{Def}


\subsection{The shifted Yangian with a dominant shift, construction II}\label{ssec shifted Yangian 2}
\

Let us now recall an alternative (historically the first) definition of
the dominantly shifted Yangians proposed in~\cite{kwwy}. Fix a dominant coweight $\mu$
of $\ssl_n$ and set $b_i:=\alphavee_i(\mu)$ (the dominance condition on $\mu$ is
equivalent to $b_i \geq 0$ for all $i$).
Let $Y_{\mu,\hbar}$ be the associative $\BC[\hbar]$-algebra generated by
  $\{e_i^{(r)},f_i^{(r)},h_i^{(s_i)}\}_{1\leq i\leq n-1}^{r\geq 0, s_i\geq -b_i}$
with the following defining relations:
\begin{equation}\label{kwwy yangian}
\begin{split}
  & [h_i^{(s)},h_j^{(s')}]=0,\\
  & [e_i^{(r)},f_j^{(r')}]=
    \begin{cases}
      h_i^{(r+r')}, & \mbox{if } i=j\ \mathrm{and}\ r+r'\geq -b_i  \\
      0, & \mbox{otherwise}
    \end{cases},\\
  & [h_i^{(-b_i)},e_j^{(r)}]=c_{ij} e_j^{(r)},\\
  & [h_i^{(s+1)},e_j^{(r)}]-[h_i^{(s)},e_j^{(r+1)}]=
    \frac{c_{ij}\hbar}{2}(h_i^{(s)}e_j^{(r)}+e_j^{(r)}h_i^{(s)}),\\
  & [h_i^{(-b_i)},f_j^{(r)}]=-c_{ij} f_j^{(r)},\\
  & [h_i^{(s+1)},f_j^{(r)}]-[h_i^{(s)},f_j^{(r+1)}]=
    -\frac{c_{ij}\hbar}{2}(h_i^{(s)}f_j^{(r)}+f_j^{(r)}h_i^{(s)}),\\
  & [e_i^{(r+1)},e_j^{(r')}]-[e_i^{(r)},e_j^{(r'+1)}]=
    \frac{c_{ij}\hbar}{2}(e_i^{(r)}e_j^{(r')}+e_j^{(r')}e_i^{(r)}),\\
  & [f_i^{(r+1)},f_j^{(r')}]-[f_i^{(r)},f_j^{(r'+1)}]=
    -\frac{c_{ij}\hbar}{2}(f_i^{(r)}f_j^{(r')}+f_j^{(r')}f_i^{(r)}),\\
  & [e_i^{(r)},e_j^{(r')}]=0\ \mathrm{and}\ [e_i^{(r)},e_j^{(r')}]=0\  \mathrm{if}\ c_{ij}=0,\\
  & [e_i^{(r_1)},[e_i^{(r_2)},e_j^{(r')}]]+
    [e_i^{(r_2)},[e_i^{(r_1)},e_j^{(r')}]]=0\ \mathrm{if}\ c_{ij}=-1,\\
  & [f_i^{(r_1)},[f_i^{(r_2)},f_j^{(r')}]]+
    [f_i^{(r_2)},[f_i^{(r_1)},f_j^{(r')}]]=0\ \mathrm{if}\ c_{ij}=-1.
\end{split}
\end{equation}

\begin{Rem}
The main differences between~(\ref{kwwy yangian}) and~(\ref{doubled yangian}) are:
(1) all indices $r,s$ are shifted by $-1$,
(2) $\hbar$ appears in the right-hand sides to make the equations look homogeneous.
\end{Rem}

Analogously to~(\ref{PBW basis elements yangian},~\ref{PBW basis elements yangian non-h}),
define the elements
  $\{e_{\alphavee}^{(r)},f_{\alphavee}^{(r)}\}_{\alphavee\in \Delta^+}^{r\geq 0}$
of $Y_{\mu,\hbar}$ via
\begin{equation}\label{PBW basis elements yangian non-h II}
\begin{split}
  & e_{\alphavee_j+\alphavee_{j+1}+\ldots+\alphavee_i}^{(r)}:=
    [\cdots[e_j^{(r)},e_{j+1}^{(0)}],\cdots,e_i^{(0)}],\\
  & f_{\alphavee_j+\alphavee_{j+1}+\ldots+\alphavee_i}^{(r)}:=
    [f_i^{(0)},\cdots,[f_{j+1,}^{(0)},f_j^{(r)}]\cdots].
\end{split}
\end{equation}
Choose any total ordering on the following set of \emph{PBW generators}:
\begin{equation}\label{pbw bases II}
  \{e_{\alphavee}^{(r)}\}_{\alphavee\in \Delta^+}^{r\geq 0}\cup
  \{f_{\alphavee}^{(r)}\}_{\alphavee\in \Delta^+}^{r\geq 0}\cup
  \{h_i^{(s_i)}\}_{1\leq i\leq n-1}^{s_i\geq -b_i}.
\end{equation}
The following is analogous to Theorem~\ref{PBW fkprw}:

\begin{Thm}\label{PBW fkprw II}
For an arbitrary dominant coweight $\mu$, the ordered PBW monomials in the
generators~(\ref{pbw bases II}) form a basis of a free $\BC[\hbar]$-module $Y_{\mu,\hbar}$.
\end{Thm}

\begin{proof}
Arguing as in~\cite[Proposition 3.13]{fkprw}, it is easy to check that $Y_{\mu,\hbar}$
is spanned by the ordered PBW monomials. To prove the linear independence of the
ordered PBW monomials, it suffices to verify that their images are linearly independent
when we specialize $\hbar$ to any nonzero complex number
(cf.~our proof of Theorem~\ref{PBW for RTT yangian}). The latter holds for $\hbar=1$
(and thus for any $\hbar\neq 0$, since all such specializations are isomorphic), due to
Theorem~\ref{PBW fkprw} and the isomorphism $Y_{\mu,\hbar}/(\hbar-1)\simeq Y_\mu$.
\end{proof}

Following~\cite[$\S $3D,3F]{kwwy}\footnote{Let us emphasize that~\cite[Theorem 3.5]{kwwy}
is wrong, as pointed out in~\cite{bfn}. That is, it does not include a complete set of relations,
except when $\fg=\ssl_2$.}, we introduce the following:

\begin{Def}\label{dominant shifted Yangian}
Let $\bY'_{\mu}$ be the $\BC[\hbar]$-subalgebra of $Y_{\mu,\hbar}$ generated by
\begin{equation*}
  \{\hbar e_{\alphavee}^{(r)}\}_{\alphavee\in \Delta^+}^{r\geq 0}\cup
  \{\hbar f_{\alphavee}^{(r)}\}_{\alphavee\in \Delta^+}^{r\geq 0}\cup
  \{\hbar h_i^{(s_i)}\}_{1\leq i\leq n-1}^{s_i\geq -b_i}.
\end{equation*}
\end{Def}

The following is the main result of this subsection:

\begin{Thm}\label{Gavarini=Rees}
For any dominant coweight $\mu$, there is a canonical $\BC[\hbar]$-algebra isomorphism
\begin{equation*}
  \bY_\mu\simeq \bY'_\mu.
\end{equation*}
\end{Thm}

This provides an identification of two different approaches towards the dominantly shifted Yangians
(which was missing in the literature, to our surprise). A proof of this result, generalized to any
semisimple Lie algebra $\fg$, is presented in Appendix~\ref{ssec appendix shifted yangians},
see Theorem~\ref{identification of two definitions}.


\subsection{Homomorphism $\Phi^\lambda_\mu$}\label{ssec hom to diff op-s yangian}
\

Let us recall the construction of~\cite[Appendix B]{bfn} for the type $A_{n-1}$
Dynkin diagram with arrows pointing $i\to i+1$ for $1\leq i\leq n-2$.
We fix a dominant coweight $\lambda$ and a coweight $\mu$ of $\ssl_n$, such that
$\lambda-\mu=\sum_{i=1}^{n-1} a_i\alpha_i$ with $a_i\in \BN$, where
$\{\alpha_i\}_{i=1}^{n-1}$ are the simple coroots of $\ssl_n$.
We set $a_0:=0, a_n:=0$.
We also fix a sequence $\unl{\lambda}=(\omega_{i_1},\ldots,\omega_{i_N})$
of fundamental coweights, such that $\sum_{s=1}^N\omega_{i_s}=\lambda$.

Consider the $\BC$-algebra
  $\wt{\CA}=\BC[z_1,\ldots,z_N]
   \langle w_{i,r}, \sfu^{\pm 1}_{i,r}, (w_{i,r}-w_{i,s}+m)^{-1}
   \rangle_{1\leq i\leq n-1, m\in \BZ}^{1\leq r\ne s\leq a_i}$
with the defining relations
  $[\sfu^{\pm 1}_{i,r},w_{j,s}]=\pm \delta_{ij}\delta_{rs}\sfu_{i,r}^{\pm 1}$.
Define $W_0(z):=1, W_n(z):=1$, and
\begin{equation}\label{ZW-series}
  Z_i(z):=\prod_{1\leq s\leq N}^{i_s=i} (z-z_s-1/2),\
  W_i(z):=\prod_{r=1}^{a_i} (z-w_{i,r}),\
  W_{i,r}(z):=\prod_{1\leq s\leq a_i}^{s\ne r} (z-w_{i,s}).
\end{equation}

We define a filtration on $\wt{\CA}$ by setting
  $\deg(z_s)=1, \deg(w_{i,r})=1,
   \deg((w_{i,r}-w_{i,s}+m)^{-1})=-1, \deg(\sfu^{\pm 1}_{i,r})=0$,
and set $\wt{\CA}_\hbar:=\Rees\ \wt{\CA}$. Explicitly, we have
\begin{equation*}
  \wt{\CA}_\hbar\simeq
  \BC[\hbar][z_1,\ldots,z_N]\langle w_{i,r}, \sfu^{\pm 1}_{i,r}, \hbar^{-1},
  (w_{i,r}-w_{i,s}+m\hbar)^{-1}\rangle_{1\leq i\leq n-1, m\in \BZ}^{1\leq r\ne s\leq a_i}
\end{equation*}
with the defining relations
  $[\sfu^{\pm 1}_{i,r},w_{j,s}]=\pm \hbar\delta_{ij}\delta_{rs}\sfu_{i,r}^{\pm 1}$.

\begin{Rem}\label{Warning}
By abuse of notation, for a generator $x$ which lives in a filtered
degree $k$ (but not in a filtered degree $k-1$) we write $x$ for
the element $\hbar^k x$ in the corresponding Rees algebra.
\end{Rem}

We also need the larger algebra
  $Y_\mu[z_1,\ldots,z_N]:=Y_\mu\otimes_\BC \BC[z_1,\ldots,z_N]$.
Define new \emph{Cartan} generators $\{A_i^{(r)}\}_{1\leq i<n}^{r\geq 1}$ via
\begin{equation}\label{A-generators yangian}
  H_i(z)=Z_i(z)\cdot \frac{\prod_{j - i}(z-1/2)^{a_j}}{z^{a_i}(z-1)^{a_i}}\cdot
  \frac{\prod_{j - i}A_j(z-1/2)}{A_i(z)A_i(z-1)},
\end{equation}
where $H_i(z):=z^{b_i}+\sum_{r>-b_i}H_i^{(r)}z^{-r}$ and
$A_i(z):=1+\sum_{r\geq 1}A_i^{(r)}z^{-r}$. The generating series $E_i(z), F_i(z)$
are defined via $E_i(z):=\sum_{r\geq 1}E_i^{(r)}z^{-r}$ and
$F_i(z):=\sum_{r\geq 1}F_i^{(r)}z^{-r}$.

The following result is due to~\cite[Theorem B.15]{bfn}
(for earlier results in this direction see~\cite{gklo,kwwy}):

\begin{Thm}[\cite{bfn}]\label{Homomorphism yangian h=1}
There exists a unique homomorphism
\begin{equation*}
  \Phi^{\unl\lambda}_\mu\colon Y_\mu[z_1,\ldots,z_N]\longrightarrow \wt{\CA}
\end{equation*}
of filtered $\BC$-algebras, such that
\begin{equation*}
\begin{split}
  & A_i(z)\mapsto z^{-a_i}W_i(z),\\
  & E_i(z)\mapsto
    -\sum_{r=1}^{a_i}\frac{Z_i(w_{i,r})W_{i-1}(w_{i,r}-1/2)}{(z-w_{i,r})W_{i,r}(w_{i,r})}\sfu^{-1}_{i,r},\\
  & F_i(z)\mapsto
    \sum_{r=1}^{a_i}\frac{W_{i+1}(w_{i,r}+1/2)}{(z-w_{i,r}-1)W_{i,r}(w_{i,r})}\sfu_{i,r}.
\end{split}
\end{equation*}
\end{Thm}

We extend the filtration $F^\bullet_{\mu_1,\mu_2}$ on $Y_\mu$ to $Y_\mu[z_1,\ldots,z_N]$
by setting $\deg(z_s)=1$, and define
  $\bY_\mu[z_1,\ldots,z_N]:=\Rees^{F^\bullet_{\mu_1,\mu_2}} Y_\mu[z_1,\ldots,z_N]$
(which is independent of the choice of $\mu_1,\mu_2$ up to a canonical isomorphism).
Applying the Rees functor to Theorem~\ref{Homomorphism yangian h=1}, we obtain

\begin{Thm}[\cite{bfn}]\label{Homomorphism yangian}
There exists a unique graded $\BC[\hbar][z_1,\ldots,z_N]$-algebra homomorphism
\begin{equation*}
  \Phi^{\unl\lambda}_\mu\colon \bY_\mu[z_1,\ldots,z_N]\longrightarrow \wt{\CA}_\hbar,
\end{equation*}
such that
\begin{equation*}
\begin{split}
  & A_i(z)\mapsto z^{-a_i}W_i(z),\\
  & E_i(z)\mapsto
    -\sum_{r=1}^{a_i}\frac{Z_i(w_{i,r})W_{i-1}(w_{i,r}-\hbar/2)}{(z-w_{i,r})W_{i,r}(w_{i,r})}\sfu^{-1}_{i,r},\\
  & F_i(z)\mapsto
    \sum_{r=1}^{a_i}\frac{W_{i+1}(w_{i,r}+\hbar/2)}{(z-w_{i,r}-\hbar)W_{i,r}(w_{i,r})}\sfu_{i,r}.
\end{split}
\end{equation*}
\end{Thm}

\begin{Rem}
Following Remark~\ref{Warning}, we note that the defining formulas of
$W_i(z), W_{i,r}(z)$ in $\wt{\CA}_\hbar$ are given again by~(\ref{ZW-series}).
In contrast, $Z_i(z)=\prod_{1\leq s\leq N}^{i_s=i} (z-z_s-\hbar/2)$,
cf.~(\ref{ZW-series}).
\end{Rem}


\subsection{Coulomb branch}\label{ssec Coulomb branch cohomology}
\

Following~\cite{bfna,bfn}, let $\CA_\hbar$ denote the quantized Coulomb branch.
We choose a basis $w_1,\ldots,w_N$ in $W=\bigoplus_{i=1}^{n-1} W_i$ such that
$w_s\in W_{i_s}$, where $i_s$ are chosen as in
subsection~\ref{ssec hom to diff op-s yangian}. Then $\CA_\hbar$ is defined as
  $\CA_\hbar:=H^{(\GL(V)\times T_W)_\CO\rtimes \BC^\times}_{\bullet}(\CR_{\GL(V),\bN})$,
where $\CR_{\GL(V),\bN}$ is the variety of triples, $T_W$ is the maximal torus of
$\GL(W)=\prod_{i=1}^{n-1}\GL(W_i)$, and $\GL(V)=\prod_{i=1}^{n-1} \GL(V_i)$.
We identify $H_{T_W}^{\bullet}(\on{pt})=\BC[z_1,\ldots,z_N]$ and
$H_{\BC^\times}^{\bullet}(\on{pt})=\BC[\hbar]$.
Recall a $\BC[\hbar][z_1,\ldots,z_N]$-algebra embedding
  $\bz^*(\iota_*)^{-1}\colon \CA_\hbar\hookrightarrow \wt{\CA}_\hbar$,
which takes the homological grading on $\CA_\hbar$ to
the above grading on $\wt{\CA}_\hbar$.

According to~\cite[Theorem B.18]{bfn}, the homomorphism
  $\Phi^{\unl\lambda}_\mu\colon \bY_\mu[z_1,\ldots,z_N]\to \wt{\CA}_\hbar$
factors through $\CA_\hbar$. In other words,
there is a unique graded $\BC[\hbar][z_1,\ldots,z_N]$-algebra homomorphism
  $\ol{\Phi}^{\unl\lambda}_\mu\colon \bY_\mu[z_1,\ldots,z_N]\to \CA_\hbar$,
such that the composition
  $\bY_\mu[z_1,\ldots,z_N]\xrightarrow{\ol{\Phi}^{\unl\lambda}_\mu}
   \CA_\hbar \xrightarrow{\bz^*(\iota_*)^{-1}} \wt{\CA}_\hbar$
coincides with $\Phi^{\unl\lambda}_\mu$.

The following result is due to~\cite[Corollary 4.10]{ktwwya}
(see Remark~\ref{Surjectivity for Yangian via Shuffle} for an alternative proof,
based on the shuffle realizations of $Y_\hbar(\ssl_n),\bY_\hbar(\ssl_n)$
of~\cite[\S6]{t}):

\begin{Prop}[\cite{ktwwya}]\label{Surjectivity yangian}
  $\ol{\Phi}^{\unl\lambda}_\mu\colon \bY_\mu[z_1,\ldots,z_N]\to \CA_\hbar$
is surjective.
\end{Prop}

\begin{Lem}\label{t-modes in diff op-s yangian}
For any $1\leq j\leq i<n$ and $r\geq 1$, the following equalities hold:
\begin{equation}\label{image of e-modes yangian}
\begin{split}
   & \Phi^{\unl\lambda}_\mu \left(E_{\alphavee_j+\alphavee_{j+1}+\ldots+\alphavee_i}^{(r)}\right)=(-1)^{i-j+1}\times\\
   & \sum_{\substack{1\leq r_j\leq a_j\\\cdots\\ 1\leq r_i\leq a_i}}
     \frac{W_{j-1}(w_{j,r_j}-\frac{\hbar}{2})\prod_{k=j}^{i-1}W_{k,r_k}(w_{k+1,r_{k+1}}-\frac{\hbar}{2})}
          {\prod_{k=j}^i W_{k,r_k}(w_{k,r_k})}\cdot
     \prod_{k=j}^i Z_k(w_{k,r_k})\cdot w_{j,r_j}^{r-1}\cdot \prod_{k=j}^i \sfu_{k,r_k}^{-1},
\end{split}
\end{equation}
\begin{equation}\label{image of f-modes yangian}
\begin{split}
   & \Phi^{\unl\lambda}_\mu \left(F_{\alphavee_j+\alphavee_{j+1}+\ldots+\alphavee_i}^{(r)}\right)=(-1)^{i-j}\times\\
   & \sum_{\substack{1\leq r_j\leq a_j\\\cdots\\ 1\leq r_i\leq a_i}}
     \frac{\prod_{k=j+1}^{i}W_{k,r_k}(w_{k-1,r_{k-1}}+\frac{\hbar}{2})W_{i+1}(w_{i,r_i}+\frac{\hbar}{2})}
          {\prod_{k=j}^i W_{k,r_k}(w_{k,r_k})}\cdot
     (w_{j,r_j}+\hbar)^{r-1}\cdot \prod_{k=j}^i \sfu_{k,r_k}.
\end{split}
\end{equation}
\end{Lem}

\begin{proof}
Straightforward computation.
\end{proof}

\begin{Rem}\label{t-modes in Coulomb yangian}
For $1\leq j\leq i<n$, we consider a coweight
  $\lambda_{ji}=(0,\ldots,0,\varpi_{j,1},\ldots,\varpi_{i,1},0,\ldots,0)$
(resp.\ $\lambda_{ji}^*=(0,\ldots,0,\varpi_{j,1}^*,\ldots,\varpi_{i,1}^*,0,\ldots,0)$)
of $\GL(V)=\GL(V_1)\times\cdots\times\GL(V_{n-1})$. The corresponding orbits
  $\Gr_{\GL(V)}^{\lambda_{ji}},\Gr_{\GL(V)}^{\lambda_{ji}^*}\subset\Gr_{\GL(V)}$
are closed, and let $\CR_{\lambda_{ji}},\CR_{\lambda_{ji}^*}$ denote their preimages
in the variety of triples $\CR_{\GL(V),\bN}$.
Then, Lemma~\ref{t-modes in diff op-s yangian} implies
\begin{equation*}
  \ol{\Phi}^{\unl\lambda}_\mu \left(E_{\alphavee_j+\alphavee_{j+1}+\ldots+\alphavee_i}^{(r)}\right)=
  (-1)^{\sum_{k=j}^{i} a_k}(c_1(\CS_j)+\hbar)^{r-1}\cap [\CR_{\lambda_{ji}^*}],
\end{equation*}
\begin{equation*}
  \ol{\Phi}^{\unl\lambda}_\mu \left(F_{\alphavee_j+\alphavee_{j+1}+\ldots+\alphavee_i}^{(r)}\right)=
  (-1)^{\sum_{k=j+1}^{i+1} a_k}(c_1(\CQ_j)+\hbar)^{r-1}\cap [\CR_{\lambda_{ji}}].
\end{equation*}
\end{Rem}


\subsection{Explicit description for $\mu=0,\lambda=n\omega_{n-1}$}\label{ssec truncation ideal yangian}
\

Following~\cite{bfn}, define the \emph{truncation ideal} $\CI^{\unl{\lambda}}_\mu$
as the $2$-sided ideal of $\bY_\mu[z_1,\ldots,z_N]$ generated over
$\BC[\hbar][z_1,\ldots,z_N]$ by $\{A_i^{(r)}\}_{1\leq i\leq n-1}^{r>a_i}$.
This ideal is discussed extensively in~\cite{kwwy}. The inclusion
$\CI^{\unl{\lambda}}_\mu\subset \Ker(\Phi^{\unl{\lambda}}_\mu)$ is clear, while
the opposite inclusion was conjectured in~\cite[Remark B.21]{bfn}.
This conjecture is proved for dominant $\mu$ in~\cite{kmwy}.

The goal of this subsection is to provide an alternative proof of a \emph{reduced version}
of that equality in the particular case $\mu=0,\lambda=n\omega_{n-1}$
(so that $N=n$ and $a_i=i$ for $1\leq i<n$; recall that $a_0=0, a_n=0$). Here, a reduced version means that
we impose an extra relation $\sum_{i=1}^n z_i=0$ in all algebras. We use
$\unl{\CI}^{n\omega_{n-1}}_0$ to denote the reduced version of the corresponding
truncation ideal, while $\unl{\Phi}{}^{n\omega_{n-1}}_0$ denotes the resulting
homomorphism between the reduced algebras.

The forthcoming discussion is very close to~\cite{bk} and~\cite{wwy}, while we choose
to present it in full details as it will be generalized along the same lines to the
trigonometric counterpart in subsection~\ref{ssec truncation ideal}.

\begin{Thm}\label{Main Theorem 2}
  $\unl{\CI}^{n\omega_{n-1}}_0=\Ker(\unl{\Phi}{}^{n\omega_{n-1}}_0)$.
\end{Thm}

Our proof of this result is based on the identification of the reduced truncation ideal
$\unl{\CI}^{n\omega_{n-1}}_0$ with the kernel of a certain version of the evaluation
homomorphism $\ev$, which is of independent interest.

Recall the commutative diagram~(\ref{diagram yangian}) of
Theorem~\ref{compatibility of evaluations yangian}. Adjoining extra variables
$\{z_i\}_{i=1}^n$ subject to $\sum_{i=1}^n z_i=0$,
we obtain the following commutative diagram:
\begin{equation}\label{diagram yangian with z}
  \begin{CD}
    Y_\hbar(\ssl_n)[z_1,\ldots,z_n]/(\sum z_i) @>\ev>> U(\ssl_n)[z_1,\ldots,z_n]/(\sum z_i)\\
    @VV{\wt{\Upsilon}}V @A{\wt{\gamma}}AA\\
   Y^\rtt_\hbar(\gl_n)[z_1,\ldots,z_n]/(\sum z_i) @>\ev^\rtt>> U(\gl_n)[z_1,\ldots,z_n]/(\sum z_i)
    \end{CD}
\end{equation}
where
  $U(\gl_n)[z_1,\ldots,z_n]/(\sum z_i):=
   U(\gl_n)\otimes_{\BC[\hbar]}\BC[\hbar][z_1,\ldots,z_n]/(\sum z_i)$
and the other three algebras are defined likewise.

Recall the isomorphism
  $Y^\rtt_\hbar(\gl_n)\simeq Y^\rtt_\hbar(\ssl_n)\otimes_{\BC[\hbar]} ZY^\rtt_\hbar(\gl_n)$
of~(\ref{RTT gln vs sln}), which after adjoining extra variables $\{z_i\}_{i=1}^n$
subject to $\sum_{i=1}^n z_i=0$ gives rise to an algebra isomorphism
  $Y^\rtt_\hbar(\gl_n)[z_1,\ldots,z_n]/(\sum z_i)\simeq
   Y^\rtt_\hbar(\ssl_n)\otimes_{\BC[\hbar]} ZY^\rtt_\hbar(\gl_n)
   \otimes_{\BC[\hbar]}\BC[\hbar][z_1,\ldots,z_n]/(\sum z_i)$.
Let $\wt{\Delta}_n(z)$ denote the quantum determinant of the matrix $zT(z)$,
which is explicitly given by
  $\wt{\Delta}_n(z)=z(z-\hbar)(z-2\hbar)\cdots(z-(n-1)\hbar)\cdot \qdet\ T(z)$.
According to Proposition~\ref{center of Yangian}, the center $ZY^\rtt_\hbar(\gl_n)$
is a polynomial algebra in $\{\wt{d}_r\}_{r=1}^\infty$, where $\wt{d}_r$ are
defined via
  $z^{-n}\wt{\Delta}_n(z+\frac{n-1}{2}\hbar)=1+\hbar\sum_{r\geq 1}\wt{d}_rz^{-r}$.
Let $\CJ$ be the $2$-sided ideal of $Y^\rtt_\hbar(\gl_n)[z_1,\ldots,z_n]/(\sum z_i)$
generated by
  $\{\wt{d}_r\}_{r>n}\cup \{\wt{d}_r-\hbar^{-1}e_r(-\hbar z_1,\ldots,-\hbar z_n)\}_{r=1}^n$,
where $e_r(\bullet)$ denotes the $r$-th elementary symmetric polynomial.
The ideal $\CJ$ is chosen so that
  $z^{-n}\wt{\Delta}_n(z+\frac{n-1}{2}\hbar)-
   \prod_{s=1}^n (1-\frac{\hbar z_s}{z})\in \CJ[[z^{-1}]]$.
Let
  $\pi\colon Y^\rtt_\hbar(\gl_n)[z_1,\ldots,z_n]/(\sum z_i)\twoheadrightarrow
   Y^\rtt_\hbar(\ssl_n)[z_1,\ldots,z_n]/(\sum z_i)$
be the natural projection along $\CJ$. Set $X_r:=\ev^\rtt(\wt{d}_r)$
(note that $X_r=0$ for $r>n$). Then, the center of $U(\gl_n)[z_1,\ldots,z_n]/(\sum z_i)$
is isomorphic to $\BC[\hbar][z_1\ldots,z_n,X_1,\ldots,X_n]/(\sum z_i)$.

Recall the \emph{extended enveloping algebra} $\wt{U}(\ssl_n)$ of~\cite{bg}, defined as the central
reduction of $U(\gl_n)[z_1,\ldots,z_n]/(\sum z_i)$ by the $2$-sided ideal generated by
  $\{X_r-\hbar^{-1}e_r(-\hbar z_1,\ldots,-\hbar z_n)\}_{r=1}^n$
(the appearance of $\ssl_n$ is due to the fact that $X_1=0$).
By abuse of notation, we denote the corresponding projection
$U(\gl_n)[z_1,\ldots,z_n]/(\sum z_i)\twoheadrightarrow \wt{U}(\ssl_n)$
by $\pi$ again. We denote the composition
  $Y^\rtt_\hbar(\gl_n)[z_1,\ldots,z_n]/(\sum z_i)\xrightarrow{\ev^\rtt}
   U(\gl_n)[z_1,\ldots,z_n]/(\sum z_i) \xrightarrow{\pi} \wt{U}(\ssl_n)$
by $\ol{\ev}^\rtt$. It factors through
  $\pi\colon Y^\rtt_\hbar(\gl_n)[z_1,\ldots,z_n]/(\sum z_i)\to
   Y^\rtt_\hbar(\ssl_n)[z_1,\ldots,z_n]/(\sum z_i)$,
and we denote the corresponding homomorphism
$Y^\rtt_\hbar(\ssl_n)[z_1,\ldots,z_n]/(\sum z_i)\to \wt{U}(\ssl_n)$ by
$\ol{\ev}^\rtt$ again. The algebra $\wt{U}(\ssl_n)$ can be also realized as
the central reduction of $U(\ssl_n)[z_1,\ldots,z_n]/(\sum z_i)$ by the $2$-sided ideal
generated by $\{\bar{X}_r-\hbar^{-1}e_r(-\hbar z_1,\ldots,-\hbar z_n)\}_{r=2}^n$,
where $\bar{X}_r=\wt{\gamma}(X_r)$, see subsection~\ref{ssec Drinfeld evaluation}.
We denote the corresponding projection
  $U(\ssl_n)[z_1,\ldots,z_n]/(\sum z_i)\twoheadrightarrow \wt{U}(\ssl_n)$
by $\pi$ again. Finally, we denote the composition
  $Y_\hbar(\ssl_n)[z_1,\ldots,z_n]/(\sum z_i)\xrightarrow{\ev}
   U(\ssl_n)[z_1,\ldots,z_n]/(\sum z_i) \xrightarrow{\pi} \wt{U}(\ssl_n)$
by $\ol{\ev}$.

Summarizing all the above, we obtain the following commutative diagram:
\begin{equation}\label{diagram yangian full}
  \begin{CD}
    Y_\hbar(\ssl_n)[z_1,\ldots,z_n]/(\sum z_i) @>\ol{\ev}>> \wt{U}(\ssl_n)\\
      @VV{\wt{\Upsilon}}V @A{\wt{\gamma}}A\wr A\\
    Y^\rtt_\hbar(\gl_n)[z_1,\ldots,z_n]/(\sum z_i) @>\ol{\ev}^\rtt>> \wt{U}(\ssl_n)\\
    @VV{\pi}V @| \\
    Y^\rtt_\hbar(\ssl_n)[z_1,\ldots,z_n]/(\sum z_i) @>\ol{\ev}^\rtt>> \wt{U}(\ssl_n)\\
    \end{CD}
\end{equation}
We note that the vertical arrows on the right are isomorphisms,
as well as the composition
  $\pi\circ \wt{\Upsilon}\colon Y_\hbar(\ssl_n)[z_1,\ldots,z_n]/(\sum z_i)
   \iso Y^\rtt_\hbar(\ssl_n)[z_1,\ldots,z_n]/(\sum z_i)$ on the left.

The commutative diagram~(\ref{diagram yangian full}) in turn gives rise to
the following commutative diagram:
\begin{equation}\label{diagram yangian full Gavarini}
  \begin{CD}
    \bY_\hbar(\ssl_n)[z_1,\ldots,z_n]/(\sum z_i) @>\ol{\ev}>> \wt{\bU}(\ssl_n)\\
      @VV{\wt{\Upsilon}}V @A{\wt{\gamma}}A\wr A\\
    \bY^\rtt_\hbar(\gl_n)[z_1,\ldots,z_n]/(\sum z_i) @>\ol{\ev}^\rtt>> \wt{\bU}(\ssl_n)\\
    @VV{\pi}V @| \\
    \bY^\rtt_\hbar(\ssl_n)[z_1,\ldots,z_n]/(\sum z_i) @>\ol{\ev}^\rtt>> \wt{\bU}(\ssl_n)\\
    \end{CD}
\end{equation}

Here we use the following notations:

\noindent
$\bullet$
$\wt{\bU}(\ssl_n)$ denotes the reduced extended version of $\bU(\ssl_n)$, or
alternatively it can be viewed as a $\BC[\hbar]$-subalgebra of $\wt{U}(\ssl_n)$
generated by $\{\hbar x\}_{x\in \ssl_n}\cup \{\hbar z_i\}_{i=1}^n$.

\noindent
$\bullet$
  $\bY_\hbar(\ssl_n)[z_1,\ldots,z_n]/(\sum z_i):=\bY_\hbar(\ssl_n)
   \otimes_{\BC[\hbar]}\BC[\hbar][z_1,\ldots,z_n]/(\sum z_i)$,
or alternatively it can be viewed as a $\BC[\hbar]$-subalgebra of
$Y_\hbar(\ssl_n)[z_1,\ldots,z_n]/(\sum z_i)$ generated by
  $\{\hbar E^{(r)}_{\alphavee}, \hbar F^{(r)}_{\alphavee}\}_{\alphavee\in \Delta^+}^{r\geq 0}\cup
   \{\hbar h_i^{(r)}\}_{1\leq i<n}^{r\geq 0}\cup \{\hbar z_i\}_{i=1}^n$.
Following our conventions of Remark~\ref{Warning}, we shall denote
$\hbar z_i$ simply by $z_i$.

\noindent
$\bullet$
  $\bY^\rtt_\hbar(\gl_n)[z_1,\ldots,z_n]/(\sum z_i):=\bY^\rtt_\hbar(\gl_n)
   \otimes_{\BC[\hbar]}\BC[\hbar][z_1,\ldots,z_n]/(\sum z_i)$, or alternatively
it can be viewed as a $\BC[\hbar]$-subalgebra of
$Y^\rtt_\hbar(\gl_n)[z_1,\ldots,z_n]/(\sum z_i)$ generated by
  $\{\hbar t^{(r)}_{ij}\}_{1\leq i,j\leq n}^{r\geq 1}\cup \{\hbar z_i\}_{i=1}^n$.
Here we denote $\hbar z_i$ simply by $z_i$ as above.

\begin{Rem}
Note that $\wt{\Upsilon}$ in~(\ref{diagram yangian full Gavarini}) is well-defined,
due to Proposition~\ref{yangian integral forms coincide} (see also our discussion in
Appendix~\ref{ssec RTT version of Drinfeld-Gavarini}).
\end{Rem}

\begin{Thm}\label{truncation as kernel yangian}
  $\unl{\CI}^{n\omega_{n-1}}_0=
  \Ker\left(\ol{\ev}\colon \bY_\hbar(\ssl_n)[z_1,\ldots,z_n]/(\sum z_i)\to \wt{\bU}(\ssl_n)\right)$.
\end{Thm}

\begin{proof}
In the particular case $\mu=0,\lambda=n\omega_{n-1}$, we note that
$Z_1(z)=\ldots=Z_{n-2}(z)=1, Z_{n-1}(z)=\prod_{s=1}^n (z-\hbar/2-z_s)$
and $a_k=k\ (1\leq k\leq n-1)$. Let us introduce extra currents $A_0(z),A_n(z)$ via
$A_0(z):=1, A_n(z):=\prod_{s=1}^n (1-z_s/z)$. Then, formula~(\ref{A-generators yangian})
relating the generating series $\{H_k(z)\}_{k=1}^{n-1}$ to $\{A_k(z)\}_{k=1}^{n-1}$ can
be uniformly written as
\begin{equation}\label{H first formula}
  H_k(z)=\frac{(z-\frac{\hbar}{2})^{2k}}{z^k(z-\hbar)^k}\cdot
  \frac{A_{k-1}(z-\frac{\hbar}{2})A_{k+1}(z-\frac{\hbar}{2})}{A_k(z)A_k(z-\hbar)}
  \ \ \mathrm{for\ any}\ \ 1\leq k\leq n-1.
\end{equation}

Let $\Delta_k(z)$ denote the $k$-th principal quantum minor $t^{1\ldots k}_{1\ldots k}(z)$
of $T(z)$, see Definition~\ref{quantum minor}. According to~\cite{m}, the following equality holds:
\begin{equation*}
  \Upsilon(H_k(z))=
  \frac{\Delta_{k-1}(z+\frac{k-1}{2}\hbar)\Delta_{k+1}(z+\frac{k+1}{2}\hbar)}
  {\Delta_{k}(z+\frac{k-1}{2}\hbar)\Delta_{k}(z+\frac{k+1}{2}\hbar)}.
\end{equation*}
This immediately implies
\begin{equation*}
  \wt{\Upsilon}(H_k(z))=
  \frac{\Delta_{k-1}(z+\frac{k-3}{2}\hbar)\Delta_{k+1}(z+\frac{k-1}{2}\hbar)}
  {\Delta_{k}(z+\frac{k-3}{2}\hbar)\Delta_{k}(z+\frac{k-1}{2}\hbar)}.
\end{equation*}
Generalizing $\wt{\Delta}_n(z)$, define $\wt{\Delta}_k(z)$ as the $k$-th
principal quantum minor of the matrix $zT(z)$. Explicitly, we have
$\wt{\Delta}_k(z)=z(z-\hbar)\cdots(z-(k-1)\hbar)\cdot \Delta_k(z)$.
Then, we get
\begin{equation*}
  \wt{\Upsilon}(H_k(z))=
  \frac{\wt{\Delta}_{k-1}(z+\frac{k-3}{2}\hbar)\wt{\Delta}_{k+1}(z+\frac{k-1}{2}\hbar)}
  {\wt{\Delta}_{k}(z+\frac{k-3}{2}\hbar)\wt{\Delta}_{k}(z+\frac{k-1}{2}\hbar)}.
\end{equation*}
Finally, define $\hat{\Delta}_k(z):=z^{-k}\wt{\Delta}_k(z+\frac{k-1}{2}\hbar)$.
Then, the above formula reads as
\begin{equation*}
  \wt{\Upsilon}(H_k(z))=
  \frac{(z-\frac{\hbar}{2})^{2k}}{z^k(z-\hbar)^k}\cdot
  \frac{\hat{\Delta}_{k-1}(z-\frac{\hbar}{2})\hat{\Delta}_{k+1}(z-\frac{\hbar}{2})}
  {\hat{\Delta}_{k}(z)\hat{\Delta}_{k}(z-\hbar)}.
\end{equation*}
By abuse of notation, let us denote the image $\pi(\hat{\Delta}_k(z))$ by
$\hat{\Delta}_k(z)$ again. Note that $\hat{\Delta}_n(z)=A_n(z)$, due to our
definition of $\pi$. Combining this with~(\ref{H first formula}), we obtain
the following result:

\begin{Cor}\label{image of A}
Under the isomorphism
\begin{equation*}
  \pi\circ\wt{\Upsilon}\colon
  \bY_\hbar(\ssl_n)[z_1,\ldots,z_n]/(z_1+\ldots+z_n) \iso
  \bY^\rtt_\hbar(\ssl_n)[z_1,\ldots,z_n]/(z_1+\ldots+z_n),
\end{equation*}
the generating series $A_k(z)$ are mapped into $\hat{\Delta}_k(z)$, that is,
$\pi\circ \wt{\Upsilon}(A_k(z))=\hat{\Delta}_k(z)$.
\end{Cor}


Define $\sT\in U(\ssl_n)\otimes \End(\BC^n)$ via $\sT:=(\wt{\gamma}\otimes 1)(T)$
with $T=\sum_{i,j} E_{ij}\otimes E_{ij}\in U(\gl_n)\otimes \End(\BC^n)$ as in
Remark~\ref{property T-matrix}(b). Set $\ol{\sT}(z):=zI_n+\hbar \sT$. Denote the $k$-th
principal quantum minor of $\ol{\sT}(z)$ by $\ol{\sT}^{1\ldots k}_{1\ldots k}(z)$.
The following is clear:
\begin{equation}\label{image of hat-delta}
   \ol{\ev}^\rtt(\hat{\Delta}_k(z))=
   z^{-k}\ol{\sT}^{1\ldots k}_{1\ldots k} \left(z+\frac{k-1}{2}\hbar\right).
\end{equation}

Combining Corollary~\ref{image of A} with~(\ref{image of hat-delta}) and the
commutativity of the diagram~(\ref{diagram yangian full Gavarini}), we get

\begin{Cor}
$\ol{\ev}(A^{(r)}_i)=0$ for any $1\leq i\leq n-1, r>i$.
In particular, $\unl{\CI}^{n\omega_{n-1}}_0\subseteq \Ker(\ol{\ev})$.
\end{Cor}

The opposite inclusion $\unl{\CI}^{n\omega_{n-1}}_0\supseteq \Ker(\ol{\ev})$ follows
from Theorem~\ref{alternative kernel Yangian integral} by noticing that
$\hat{\Delta}_1(z)=t_{11}(z)$ and so
  $(\pi\circ \wt{\Upsilon})^{-1}(t^{(r)}_{11})=A_1^{(r)}\in \unl{\CI}^{n\omega_{n-1}}_0$
for $r>1$.

\medskip
This completes our proof of Theorem~\ref{truncation as kernel yangian}.
\end{proof}

Now we are ready to present the proof of Theorem~\ref{Main Theorem 2}.

\begin{proof}[Proof of Theorem~\ref{Main Theorem 2}]
Consider a subtorus $T'_W=\{g\in T_W|\det(g)=1\}$ of $T_W$, and define
  $\unl{\CA}_\hbar:=H^{(\GL(V)\times T'_W)_\CO\rtimes \BC^\times}_{\bullet}(\CR_{\GL(V),\bN})$,
so that $\unl{\CA}_\hbar\simeq \CA_\hbar/(\sum z_i)$.
After imposing $\sum z_i=0$, the homomorphism
  $\unl{\Phi}{}^{n\omega_{n-1}}_0\colon
   \bY_\hbar(\ssl_n)[z_1,\ldots,z_n]/(\sum z_i)\to \wt{\CA}_\hbar/(\sum z_i)$
is a composition of the surjective homomorphism
  $\unl{\ol{\Phi}}{}^{n\omega_{n-1}}_0\colon
   \bY_\hbar(\ssl_n)[z_1,\ldots,z_n]/(\sum z_i)\to \unl{\CA}_\hbar$
(see Proposition~\ref{Surjectivity yangian}) and an embedding
  $\bz^*(\iota_*)^{-1}\colon \unl{\CA}_\hbar\hookrightarrow \wt{\CA}_\hbar/(\sum z_i)$,
so that $\Ker(\unl{\Phi}{}^{n\omega_{n-1}}_0)=\Ker(\unl{\ol{\Phi}}{}^{n\omega_{n-1}}_0)$.
The homomorphism $\unl{\ol{\Phi}}{}^{n\omega_{n-1}}_0$ factors through
$\ol{\phi}\colon \wt{\bU}(\ssl_n)\twoheadrightarrow \unl{\CA}_\hbar$
(due to Theorem~\ref{truncation as kernel yangian}), and it remains to prove the injectivity
of $\ol{\phi}$.
Note that $\ol{\phi}$ is compatible with the gradings, and it is known to be an
isomorphism modulo the ideal generated by $\hbar,z_1,\ldots,z_n$,
see e.g.~\cite[Theorem~4.12]{bfnc}: namely, both sides are isomorphic to the ring of functions
on the nilpotent cone $\CN\subset\ssl_n$. To prove the injectivity of $\ol{\phi}$ it suffices
to identify the graded characters of the algebras in question. But both graded characters
are equal to $\on{char}\BC[\CN]\cdot\on{char}\left(\BC[\hbar,z_1,\ldots,z_n]/(\sum z_i)\right)$.

This completes our proof of Theorem~\ref{Main Theorem 2}.
\end{proof}

\begin{Cor}\label{explicit Coulomb yangian}
The reduced quantized Coulomb branch $\unl{\CA}_\hbar$ is explicitly given
by $\unl{\CA}_\hbar\simeq \wt{\bU}(\ssl_n)$.
\end{Cor}


\section{Quantum Algebras}\label{sec quantum algebras}


\subsection{The RTT integral form of quantum $\gl_n$}\label{ssec RTT finite}
\

Let $\vv$ be a formal variable. Consider the $R$-matrix $R=R^\vv$ given by
\begin{equation}\label{finite R}
  R=\vv^{-1}\sum_{i=1}^n E_{ii}\otimes E_{ii}+\sum_{i\ne j} E_{ii}\otimes E_{jj}+
  (\vv^{-1}-\vv)\sum_{i>j}E_{ij}\otimes E_{ji}
\end{equation}
which is an element of $\BC[\vv,\vv^{-1}]\otimes_{\BC} (\End\ \BC^n)^{\otimes 2}$.
It satisfies the famous \emph{Yang-Baxter equation}
\begin{equation*}
  R_{12}R_{13}R_{23}=R_{23}R_{13}R_{12},
\end{equation*}
viewed as the equality in $\BC[\vv,\vv^{-1}]\otimes_{\BC} (\End\ \BC^n)^{\otimes 3}$.

Following~\cite{frt}, define the \emph{RTT integral form of quantum $\gl_n$}, denoted by
$\fU^\rtt_\vv(\gl_n)$, to be the associative $\BC[\vv,\vv^{-1}]$-algebra generated by
$\{t^+_{ij},t^-_{ij}\}_{i,j=1}^{n}$ with the following defining relations:
\begin{equation}\label{fRTT}
\begin{split}
  & t^\pm_{ii}t^\mp_{ii}=1\ \ \mathrm{for}\ 1\leq i\leq n,\\
  & t^{+}_{ij}=t^{-}_{ji}=0\ \ \mathrm{for}\ 1\leq j<i\leq n,\\
  & RT^+_1T^+_2=T^+_2T^+_1R,\ RT^-_1T^-_2=T^-_2T^-_1R,\ RT^-_1T^+_2=T^+_2T^-_1R.
\end{split}
\end{equation}
Here $T^\pm$ are the elements of the algebra $\fU^\rtt_\vv(\gl_n)\otimes \End\ \BC^n$,
defined by $T^\pm=\sum_{i,j} t^\pm_{ij}\otimes E_{ij}$. Thus, the last three defining
relations of~(\ref{fRTT}) should be viewed as equalities in
  $\fU^\rtt_\vv(\gl_n)\otimes (\End\ \BC^n)^{\otimes 2}$.

For completeness of the picture, define
  $\wt{R}\in \BC[\vv,\vv^{-1}]\otimes_{\BC} (\End\ \BC^n)^{\otimes 2}$
via\footnote{Let us note right away that this $\wt{R}$ is denoted by $R$
in~\cite[(2.2)]{df} and~\cite[$\S 1.15.1$]{m}.}
\begin{equation}\label{finite R-tilda}
  \wt{R}=\vv\sum_{i=1}^n E_{ii}\otimes E_{ii}+\sum_{i\ne j} E_{ii}\otimes E_{jj}+
  (\vv-\vv^{-1})\sum_{i<j}E_{ij}\otimes E_{ji}.
\end{equation}

\begin{Lem}\label{remaining fRTT}
The following equalities hold:
\begin{equation}\label{4th fRTT}
  \wt{R}T^+_1T^+_2=T^+_2T^+_1\wt{R},\
  \wt{R}T^-_1T^-_2=T^-_2T^-_1\wt{R},\
  \wt{R}T^+_1T^-_2=T^-_2T^+_1\wt{R}.
\end{equation}
\end{Lem}

\begin{proof}
Multiplying the last equality of~(\ref{fRTT}) by $R^{-1}$ on the left and
on the right, and conjugating further by the permutation operator
  $P=\sum_{i,j}E_{ij}\otimes E_{ji}\in (\End\ \BC^n)^{\otimes 2}$,
we get
\begin{equation*}
  (PR^{-1}P^{-1})T^+_1T^-_2=T^-_2T^+_1(PR^{-1}P^{-1}).
\end{equation*}
Since $\wt{R}=PR^{-1}P^{-1}$ (straightforward verification),
we obtain the last equality of~(\ref{4th fRTT}).

The other two equalities of~(\ref{4th fRTT}) are proved analogously.
\end{proof}

Note that specializing $\vv$ to $1$, i.e.\ taking
a quotient by $(\vv-1)$, $R^\vv$ specializes to the identity operator
  $\mathrm{I}=\sum_{i,j} E_{ii}\otimes E_{jj}\in (\End\ \BC^n)^{\otimes 2}$,
hence, the specializations of the generators $t^\pm_{ij}$ pairwise commute.
In other words, we get the following isomorphism:
\begin{equation}\label{integral_finite_gln}
  \fU^\rtt_\vv(\gl_n)/(\vv-1)\simeq
  \BC[t^+_{ij},t^-_{ji}]_{1\leq i\leq j\leq n}/(\langle t^\pm_{ii}t^\mp_{ii}-1 \rangle_{i=1}^n).
\end{equation}

We also define the $\BC(\vv)$-counterpart
  $U^\rtt_\vv(\gl_n):=\fU^\rtt_\vv(\gl_n)\otimes_{\BC[\vv,\vv^{-1}]} \BC(\vv)$.


\subsection{The RTT integral form of quantum affine $\gl_n$}\label{ssec RTT affine}
\

Consider the \emph{trigonometric} $R$-matrix $R_\trig(z,w)=R^\vv_\trig(z,w)$ given by
\begin{equation}\label{trigR}
\begin{split}
  R_\trig(z,w):=
  (\vv z-\vv^{-1}w)\sum_{i=1}^n E_{ii}\otimes E_{ii} + (z-w)\sum_{i\ne j} E_{ii}\otimes E_{jj}+\\
  (\vv-\vv^{-1})z\sum_{i<j} E_{ij}\otimes E_{ji}+(\vv-\vv^{-1})w \sum_{i>j}E_{ij}\otimes E_{ji}
\end{split}
\end{equation}
which is an element of $\BC[\vv,\vv^{-1}]\otimes_{\BC} (\End\ \BC^n)^{\otimes 2}$,
cf.~\cite[(3.7)]{df}.
It satisfies the famous \emph{Yang-Baxter equation with a spectral parameter}:
\begin{equation}\label{qYB}
  R_{\trig;12}(u,v)R_{\trig;13}(u,w)R_{\trig;23}(v,w)=
  R_{\trig;23}(v,w)R_{\trig;13}(u,w)R_{\trig;12}(u,v).
\end{equation}

Following~\cite{frt,df}, define the \emph{RTT integral form of quantum loop $\gl_n$},
denoted by $\fU^\rtt_\vv(L\gl_n)$, to be the associative $\BC[\vv,\vv^{-1}]$-algebra
generated by $\{t^\pm_{ij}[\pm r]\}_{1\leq i,j\leq n}^{r\in \BN}$ with
the following defining relations:
\begin{equation}\label{affRTT}
\begin{split}
  & t^\pm_{ii}[0]t^\mp_{ii}[0]=1\ \ \mathrm{for}\ 1\leq i\leq n,\\
  & t^+_{ij}[0]=t^-_{ji}[0]=0\ \ \mathrm{for}\ 1\leq j<i\leq n,\\
  & R_{\trig}(z,w)T^+_1(z)T^+_2(w)=T^+_2(w)T^+_1(z)R_\trig(z,w),\\
  & R_{\trig}(z,w)T^-_1(z)T^-_2(w)=T^-_2(w)T^-_1(z)R_\trig(z,w),\\
  & R_{\trig}(z,w)T^-_1(z)T^+_2(w)=T^+_2(w)T^-_1(z)R_\trig(z,w).
\end{split}
\end{equation}
Here $T^\pm(z)$ are the series in $z^{\mp 1}$ with coefficients in the
algebra $\fU^\rtt_\vv(L\gl_n)\otimes \End\ \BC^n$, defined by
$T^\pm(z)=\sum_{i,j} t^\pm_{ij}(z)\otimes E_{ij}$ with
$t^\pm_{ij}(z):=\sum_{r\geq 0} t^\pm_{ij}[\pm r]  z^{\mp r}$.
Thus, the last three relations should be viewed as equalities of
series in $z,w$ with coefficients in
$\fU^\rtt_\vv(L\gl_n)\otimes (\End\ \BC^n)^{\otimes 2}$.

In contrast to Lemma~\ref{remaining fRTT}, we have the following result
(cf.~\cite[(2.45)]{gm}):

\begin{Lem}\label{remaining affRTT}
For any $\epsilon,\epsilon'\in \{\pm\}$, the following holds:
\begin{equation}\label{all affRTT}
  R_{\trig}(z,w)T^\epsilon_1(z)T^{\epsilon'}_2(w)=
  T^{\epsilon'}_2(w)T^\epsilon_1(z)R_\trig(z,w).
\end{equation}
\end{Lem}

\begin{proof}
Multiplying the last equality of~(\ref{affRTT}) by $R^{-1}_{\trig}(z,w)$
on the left and on the right, and conjugating further by the permutation operator
  $P=\sum_{i,j}E_{ij}\otimes E_{ji}\in (\End\ \BC^n)^{\otimes 2}$,
we get
\begin{equation*}
  (PR^{-1}_\trig(z,w)P^{-1})T^+_1(w)T^-_2(z)=
  T^-_2(z)T^+_1(w)(PR^{-1}_\trig(z,w)P^{-1}).
\end{equation*}
Combining this with the equality
\begin{equation*}
  R_\trig(z,w)=(\vv z-\vv^{-1}w)(\vv w-\vv^{-1}z)\cdot PR^{-1}_\trig(w,z)P^{-1},
\end{equation*}
we derive the validity of~(\ref{all affRTT}) for the only remaining case
$\epsilon=+,\epsilon'=-$.
\end{proof}

Note that specializing $\vv$ to $1$, i.e.\
taking a quotient by $(\vv-1)$, $R^\vv_\trig(z,w)$ specializes to
$(z-w)\mathrm{I}=(z-w)\sum_{i,j} E_{ii}\otimes E_{jj}\in (\End\ \BC^n)^{\otimes 2}$,
hence, the specializations of the generators $t^\pm_{ij}[\pm r]$ pairwise commute.
In other words, we get the following isomorphism:
\begin{equation}\label{integral_affine_gln}
  \fU^\rtt_\vv(L\gl_n)/(\vv-1)\simeq
  \BC\left[t^\pm_{ji}[\pm r]\right]_{1\leq j,i\leq n}^{r\geq 0}/
  \left(\langle t^+_{ij}[0],t^-_{ji}[0],t^\pm_{kk}[0]t^\mp_{kk}[0]-1 \rangle_{1\leq j<i\leq n}^{1\leq k\leq n}\right).
\end{equation}

We also define the $\BC(\vv)$-counterpart
  $U^\rtt_\vv(L\gl_n):=\fU^\rtt_\vv(L\gl_n)\otimes_{\BC[\vv,\vv^{-1}]} \BC(\vv)$.


\subsection{The RTT evaluation homomorphism $\ev^\rtt$}\label{ssec RTT evaluation}
\

Recall the following two standard relations between $\fU^\rtt_\vv(L\gl_n)$ and
$\fU^\rtt_\vv(\gl_n)$, cf.~Lemma~\ref{embed+ev rtt yangian}.

\begin{Lem}\label{embedding_quantum}
The assignment $t^\pm_{ij}\mapsto t^\pm_{ij}[0]$ gives rise
to a $\BC[\vv,\vv^{-1}]$-algebra embedding
\begin{equation*}
  \iota\colon \fU^\rtt_\vv(\gl_n)\hookrightarrow \fU^\rtt_\vv(L\gl_n).
\end{equation*}
\end{Lem}

\begin{proof}
The above assignment is compatible with defining relations~(\ref{fRTT}),
hence, it gives rise to a $\BC[\vv,\vv^{-1}]$-algebra homomorphism
$\iota\colon \fU^\rtt_\vv(\gl_n)\to \fU^\rtt_\vv(L\gl_n)$. The injectivity
of $\iota$ follows from the PBW theorems for $\fU^\rtt_\vv(\gl_n)$ and
$\fU^\rtt_\vv(L\gl_n)$ of~\cite[Proposition 2.1, Theorem 2.11]{gm}.
\end{proof}

\begin{Lem}\label{ev_RTT}
For $a\in \BC^\times$, the assignment
  $T^+(z)\mapsto T^+-aT^-z^{-1},\ T^-(z)\mapsto T^--a^{-1}T^+z$
gives rise to a $\BC[\vv,\vv^{-1}]$-algebra epimorphism
\begin{equation*}
  \ev^\rtt_a\colon \fU^\rtt_\vv(L\gl_n)\twoheadrightarrow \fU^\rtt_\vv(\gl_n).
\end{equation*}
\end{Lem}

\begin{proof}
The above assignment is compatible with defining relations~(\ref{affRTT}),
due to~(\ref{fRTT}),~(\ref{4th fRTT}), and the equality
$R_\trig(z,w)=(z-w)R+(\vv-\vv^{-1})zP$ relating the two $R$-matrices,
cf.~\cite[Lemma 1.11]{h}. The resulting homomorphism
$\fU^\rtt_\vv(L\gl_n)\to \fU^\rtt_\vv(\gl_n)$ is clearly surjective.
\end{proof}

We will denote the \emph{RTT evaluation homomorphism}
$\ev^\rtt_1$ simply by $\ev^\rtt$.

\begin{Rem}
(a) For any $a\in \BC^\times$, the homomorphism $\ev^\rtt_a$ equals
the composition of $\ev^\rtt$ and the automorphism of $\fU^\rtt_\vv(L\gl_n)$
given by $T^\pm(z)\mapsto T^\pm(a^{-1}z)$.

\noindent
(b) The composition $\ev^\rtt_a\circ \iota$ is the identity endomorphism
of $\fU^\rtt_\vv(\gl_n)$ for any $a\in \BC^\times$.
\end{Rem}

The PBW theorems for $\fU^\rtt_\vv(\gl_n)$ and $\fU^\rtt_\vv(L\gl_n)$
of~\cite[Proposition 2.1, Theorem 2.11]{gm} imply the following simple result,
cf.~Lemma~\ref{simple kernel yangian}:

\begin{Lem}\label{simple kernel}
The kernel of $\ev^\rtt$ is the $2$-sided ideal generated by the following elements:
\begin{equation}\label{explicit kernel}
  \left\{t^+_{ij}[r],t^+_{ii}[s],t^+_{ji}[s],t^-_{ji}[-r],t^-_{ii}[-s],t^-_{ij}[-s]\right\}_{i<j}^{\substack{r\geq 1\\ s\geq 2}}\bigcup
  \left\{t^+_{ji}[1]+t^-_{ji}[0],t^-_{ij}[-1]+t^+_{ij}[0]\right\}_{i\leq j}.
\end{equation}
\end{Lem}

However, we will need an alternative description of this kernel $\Ker(\ev^\rtt)$,
cf.~Theorem~\ref{alternative kernel yangian}:

\begin{Thm}\label{alternative kernel quantum}
$\Ker(\ev^\rtt)=\fU^\rtt_\vv(L\gl_n)\cap I$, where $I$ is the $2$-sided ideal
of $U^\rtt_\vv(L\gl_n)$ generated by
  $\{t^+_{11}[s], t^-_{11}[-s]\}_{s\geq 2}\cup \{t^+_{11}[1]+t^-_{11}[0], t^-_{11}[-1]+t^+_{11}[0]\}$.
\end{Thm}

\begin{proof}
Note that the ideal $I$ is in the kernel of $\BC(\vv)$-extended
evaluation homomorphism $\ev^\rtt\colon U^\rtt_\vv(L\gl_n)\to U^\rtt_\vv(\gl_n)$,
hence, the inclusion $\fU^\rtt_\vv(L\gl_n)\cap I\subset \Ker(\ev^\rtt)$. To prove
the opposite inclusion $\Ker(\ev^\rtt)\subset \fU^\rtt_\vv(L\gl_n)\cap I$,
it suffices to verify that all elements of~(\ref{explicit kernel}) belong to $I$.
We write $x\underset{I}\equiv y$ if $x-y\in I$.

\medskip
\noindent
$\bullet$
\emph{Verification of $t^+_{1j}[r]\in I$ for all $j>1, r\geq 1$.}

Comparing the matrix coefficients $\langle v_1\otimes v_1|\cdots|v_1\otimes v_j\rangle$
of both sides of the equality~(\ref{all affRTT}) with $\epsilon=\epsilon'=+$, we get
  $(\vv z-\vv^{-1}w)t^+_{11}(z)t^+_{1j}(w)=
   (z-w)t^+_{1j}(w)t^+_{11}(z)+(\vv-\vv^{-1})wt^+_{11}(w)t^+_{1j}(z)$.
Evaluating the coefficients of $z^{-r}w^1$ in both sides of this equality, we find
\begin{equation*}
   -\vv^{-1}t^+_{11}[r]t^+_{1j}[0]=-t^+_{1j}[0]t^+_{11}[r]+(\vv-\vv^{-1})t^+_{11}[0]t^+_{1j}[r]
   \Longrightarrow t^+_{1j}[r]=\frac{t^-_{11}[0]\cdot [t^+_{1j}[0],t^+_{11}[r]]_{\vv^{-1}}}{\vv-\vv^{-1}}.
\end{equation*}
We claim that $[t^+_{1j}[0],t^+_{11}[r]]_{\vv^{-1}}\in I$. This is clear for
$r>1$ as $t^+_{11}[r]\in I$. For $r=1$, we note that
  $[t^+_{1j}[0],t^+_{11}[1]]_{\vv^{-1}}\underset{I}\equiv
   -[t^+_{1j}[0],t^-_{11}[0]]_{\vv^{-1}}=
   -(t^+_{11}[0])^{-1}\cdot [t^+_{11}[0],t^+_{1j}[0]]_{\vv^{-1}}\cdot (t^+_{11}[0])^{-1}$.
Finally, comparing the coefficients of $z^1w^0$ (instead of $z^{-r}w^1$) in
the above equality, we immediately find $[t^+_{11}[0],t^+_{1j}[0]]_{\vv^{-1}}=0$.
This completes our proof of the remaining inclusion $t^+_{1j}[1]\in I$.

\medskip
\noindent
$\bullet$
\emph{Verification of $t^+_{j1}[s]\in I$ for all $j>1, s\geq 2$.}

Comparing the matrix coefficients $\langle v_1\otimes v_j|\cdots|v_1\otimes v_1\rangle$
of both sides of the equality~(\ref{all affRTT}) with $\epsilon=\epsilon'=+$, we get
  $(z-w)t^+_{11}(z)t^+_{j1}(w)+(\vv-\vv^{-1})zt^+_{j1}(z)t^+_{11}(w)=
   (\vv z-\vv^{-1}w)t^+_{j1}(w)t^+_{11}(z)$.
Evaluating the coefficients of $z^{-r}w^0$ in both sides of this equality, we find
\begin{equation*}
  -t^+_{11}[r]t^+_{j1}[1]+(\vv-\vv^{-1})t^+_{j1}[r+1]t^+_{11}[0]=-\vv^{-1}t^+_{j1}[1]t^+_{11}[r]
  \Rightarrow t^+_{j1}[r+1]=\frac{[t^+_{11}[r],t^+_{j1}[1]]_{\vv^{-1}}\cdot t^-_{11}[0]}{\vv-\vv^{-1}}.
\end{equation*}
We claim that $[t^+_{11}[r],t^+_{j1}[1]]_{\vv^{-1}}\in I$ for $r=s-1\geq 1$. This is clear for
$r>1$ as $t^+_{11}[r]\in I$. For $r=1$, we note that
  $[t^+_{11}[1],t^+_{j1}[1]]_{\vv^{-1}}\underset{I}\equiv
   -[t^-_{11}[0],t^+_{j1}[1]]_{\vv^{-1}}=
   -(t^+_{11}[0])^{-1}\cdot [t^+_{j1}[1],t^+_{11}[0]]_{\vv^{-1}}\cdot (t^+_{11}[0])^{-1}$.
Finally, comparing the coefficients of $z^0w^0$ (instead of $z^{-r}w^0$) in
the above equality, we immediately find $[t^+_{j1}[1],t^+_{11}[0]]_{\vv^{-1}}=0$.
This implies the remaining inclusion $t^+_{j1}[2]\in I$.

\medskip
\noindent
$\bullet$
\emph{Verification of $t^+_{22}[s]\in I$ for all $s\geq 2$.}

Comparing the matrix coefficients $\langle v_2\otimes v_1|\cdots|v_1\otimes v_2\rangle$
of both sides of the equality~(\ref{all affRTT}) with $\epsilon=\epsilon'=+$, we get
\begin{equation*}
   (z-w)t^+_{21}(z)t^+_{12}(w)+(\vv-\vv^{-1})wt^+_{11}(z)t^+_{22}(w)=
   (z-w)t^+_{12}(w)t^+_{21}(z)+(\vv-\vv^{-1})wt^+_{11}(w)t^+_{22}(z).
\end{equation*}
Evaluating the coefficients of $z^{-s}w^1$ in both sides of this equality, we find
\begin{equation*}
   -t^+_{21}[s]t^+_{12}[0]+(\vv-\vv^{-1})t^+_{11}[s]t^+_{22}[0]=
   -t^+_{12}[0]t^+_{21}[s]+(\vv-\vv^{-1})t^+_{11}[0]t^+_{22}[s].
\end{equation*}
Since $t^+_{11}[s], t^+_{21}[s]\in I$ for $s\geq 2$ by above,
we immediately get the inclusion $t^+_{22}[s]\in I$.

\medskip
\noindent
$\bullet$
\emph{Verification of $t^+_{22}[1]+t^-_{22}[0]\in I$.}

Comparing the matrix coefficients $\langle v_2\otimes v_1|\cdots|v_1\otimes v_2\rangle$
of both sides of the equality~(\ref{all affRTT}) with $\epsilon=-,\epsilon'=+$, we get
\begin{equation*}
   (z-w)t^-_{21}(z)t^+_{12}(w)+(\vv-\vv^{-1})wt^-_{11}(z)t^+_{22}(w)=
   (z-w)t^+_{12}(w)t^-_{21}(z)+(\vv-\vv^{-1})wt^+_{11}(w)t^-_{22}(z).
\end{equation*}
Evaluating the coefficients of $z^0w^0$ in both sides of this equality, we find
\begin{equation*}
   -t^-_{21}[0]t^+_{12}[1]+(\vv-\vv^{-1})t^-_{11}[0]t^+_{22}[1]=
   -t^+_{12}[1]t^-_{21}[0]+(\vv-\vv^{-1})t^+_{11}[1]t^-_{22}[0].
\end{equation*}
Since $t^+_{12}[1], t^+_{11}[1]+t^-_{11}[0]\in I$, we immediately
get the inclusion $t^+_{22}[1]+t^-_{22}[0]\in I$.

\medskip
\noindent
$\bullet$
\emph{Verification of $t^+_{j1}[1]+t^-_{j1}[0]\in I$ for all $j>1$.}

Comparing the matrix coefficients $\langle v_1\otimes v_j|\cdots|v_1\otimes v_1\rangle$
of both sides of the equality~(\ref{all affRTT}) with $\epsilon=+,\epsilon'=-$, we get
\begin{equation*}
   (z-w)t^+_{11}(z)t^-_{j1}(w)+(\vv-\vv^{-1})zt^+_{j1}(z)t^-_{11}(w)=
   (\vv z-\vv^{-1}w)t^-_{j1}(w)t^+_{11}(z).
\end{equation*}
Evaluating the coefficients of $z^0w^0$ in both sides of this equality, we find
\begin{equation*}
   t^+_{11}[1]t^-_{j1}[0]+(\vv-\vv^{-1})t^+_{j1}[1]t^-_{11}[0]=
   \vv t^-_{j1}[0]t^+_{11}[1].
\end{equation*}
Since $t^+_{11}[1]+t^-_{11}[0]\in I$, we get
  $t^+_{j1}[1]\underset{I}\equiv
   \frac{[t^-_{11}[0],t^-_{j1}[0]]_\vv\cdot t^+_{11}[0]}{\vv-\vv^{-1}}$.
On the other hand, comparing the matrix coefficients
$\langle v_1\otimes v_j|\cdots|v_1\otimes v_1\rangle$ of both sides of
the equality~(\ref{all affRTT}) with $\epsilon=\epsilon'=-$, we get
  $(z-w)t^-_{11}(z)t^-_{j1}(w)+(\vv-\vv^{-1})zt^-_{j1}(z)t^-_{11}(w)=
   (\vv z-\vv^{-1}w)t^-_{j1}(w)t^-_{11}(z)$.
Evaluating the coefficients of $z^1w^0$ in both sides of this equality, we find
\begin{equation*}
   t^-_{11}[0]t^-_{j1}[0]+(\vv-\vv^{-1})t^-_{j1}[0]t^-_{11}[0]=\vv t^-_{j1}[0]t^-_{11}[0]
   \Longrightarrow t^-_{j1}[0]=-\frac{[t^-_{11}[0],t^-_{j1}[0]]_\vv\cdot t^+_{11}[0]}{\vv-\vv^{-1}}.
\end{equation*}
Hence, the inclusion $t^+_{j1}[1]+t^-_{j1}[0]\in I$.

\medskip
One can now apply the above five verifications with all lower indices
increased by $1$ to prove the inclusions
  $t^+_{2j}[r], t^+_{j2}[s], t^+_{33}[s],
   t^+_{33}[1]+t^-_{33}[0], t^+_{j2}[1]+t^-_{j2}[0]\in I$
for any $j>2, r\geq 1, s\geq 2$. Proceeding further step by step, we obtain
  $\{t^+_{ij}[r],t^+_{ii}[s],t^+_{ji}[s]\}_{i<j}^{r\geq 1,s\geq 2}\cup
   \{t^+_{ji}[1]+t^-_{ji}[0]\}_{i\leq j}\subset I$.
The proof of the remaining inclusion
  $\{t^-_{ji}[-r],t^-_{ii}[-s],t^-_{ij}[-s]\}_{i<j}^{r\geq 1,s\geq 2}\cup
   \{t^-_{ij}[-1]+t^+_{ij}[0]\}_{i\leq j}\subset I$
is analogous and we leave details to the interested reader.

\medskip
This completes our proof of Theorem~\ref{alternative kernel quantum}.
\end{proof}


\subsection{The Drinfeld-Jimbo quantum $\gl_n$ and $\ssl_n$}\label{ssec DJ finite}
\

Following~\cite{jim}, define the quantum $\gl_n$, denoted by $U_\vv(\gl_n)$,
to be the associative $\BC(\vv)$-algebra generated by
  $\{E_i,F_i,t_j,t^{-1}_j\}_{1\leq i<n}^{1\leq j\leq n}$
with the following defining relations:
\begin{equation}\label{DrJim finite}
\begin{split}
  & t_jt^{-1}_j=t^{-1}_jt_j=1,\ t_{j}t_{j'}=t_{j'}t_j,\\
  & t_jE_i=\vv^{-\delta_{ji}+\delta_{j,i+1}}E_it_j,\
    t_jF_i=\vv^{\delta_{ji}-\delta_{j,i+1}}F_it_j,\\
  & E_iF_{i'}-F_{i'}E_i=\delta_{ii'}\frac{K_i-K_i^{-1}}{\vv-\vv^{-1}},\\
  & E_iE_{i'}=E_{i'}E_i\ \mathrm{and}\ F_iF_{i'}=F_{i'}F_i\ \mathrm{if}\ c_{ii'}=0,\\
  & E_i^2E_{i'}-(\vv+\vv^{-1})E_iE_{i'}E_i+E_{i'}E_i^2=0\ \mathrm{if}\ c_{ii'}=-1,\\
  & F_i^2F_{i'}-(\vv+\vv^{-1})F_iF_{i'}F_i+F_{i'}F_i^2=0\ \mathrm{if}\ c_{ii'}=-1,
\end{split}
\end{equation}
where $K_i:=t_i^{-1}t_{i+1}$ and $(c_{ii'})_{i,i'=1}^{n-1}$ denotes the Cartan
matrix of $\ssl_n$.

\begin{Rem}
We note that our generators $E_i,F_i,t_j^{\pm 1}$ correspond to
the generators $f_i,e_i,\vv^{\pm H_j}$ of~\cite[Definition 2.3]{df}, respectively.
\end{Rem}

The $\BC(\vv)$-subalgebra of $U_\vv(\gl_n)$ generated by
$\{E_i,F_i,K^{\pm 1}_i\}_{i=1}^{n-1}$ is isomorphic to the Drinfeld-Jimbo
quantum $\ssl_n$, denoted by $U_\vv(\ssl_n)$, see~\cite{d1,jim}.

The following well-known result was conjectured in~\cite{frt}
and proved in~\cite[Theorem 2.1]{df}:

\begin{Thm}[\cite{df}]\label{Ding-Frenkel finite}
There is a unique $\BC(\vv)$-algebra isomorphism
\begin{equation*}
  \Upsilon\colon U_\vv(\gl_n)\iso U^\rtt_\vv(\gl_n)
\end{equation*}
defined by
\begin{equation}\label{Ding-Frenkel finite formulas}
   t^{\pm 1}_j\mapsto t^\pm_{jj},\
   E_i\mapsto \frac{t^-_{ii}t^+_{i,i+1}}{\vv-\vv^{-1}},\
   F_i\mapsto \frac{t^-_{i+1,i}t^+_{ii}}{\vv^{-1}-\vv}.
\end{equation}
\end{Thm}

As an immediate corollary, $\fU^\rtt_\vv(\gl_n)$ is realized
as a $\BC[\vv,\vv^{-1}]$-subalgebra of $U_\vv(\gl_n)$.
To describe this subalgebra explicitly, define the elements
$\{E_{j,i+1}, F_{i+1,j}\}_{1\leq j\leq i<n}$ of $U_\vv(\gl_n)$ via
\begin{equation}\label{PBW finite elements}
\begin{split}
  & E_{j,i+1}:=(\vv-\vv^{-1})[E_i,\cdots,[E_{j+1},E_j]_{\vv^{-1}}\cdots]_{\vv^{-1}},\\
  & F_{i+1,j}:=(\vv^{-1}-\vv)[\cdots[F_j,F_{j+1}]_{\vv},\cdots,F_i]_{\vv},
\end{split}
\end{equation}
where $[a,b]_x:=ab-x\cdot ba$. In particular,
$E_{i,i+1}=(\vv-\vv^{-1})E_i$ and $F_{i+1,i}=(\vv^{-1}-\vv)F_i$.

\begin{Def}\label{integral finite}
(a) Let $\fU_\vv(\gl_n)$ be the $\BC[\vv,\vv^{-1}]$-subalgebra of
$U_\vv(\gl_n)$ generated by
\begin{equation}\label{PBW basis finite}
   \{E_{j,i+1}, F_{i+1,j}\}_{1\leq j\leq i<n}\cup
   \{t_j^{\pm 1}\}_{1\leq j\leq n}.
\end{equation}

\noindent
(b) Let $\fU_\vv(\ssl_n)$ be the $\BC[\vv,\vv^{-1}]$-subalgebra of $U_\vv(\ssl_n)$
generated by
\begin{equation}\label{PBW basis finite sln}
   \{E_{j,i+1}, F_{i+1,j}\}_{1\leq j\leq i<n}\cup
   \{K_i^{\pm 1}\}_{1\leq i<n}.
\end{equation}
\end{Def}

\begin{Prop}\label{comparison of integral forms quantum finite}
$\fU_\vv(\gl_n)=\Upsilon^{-1}(\fU^\rtt_\vv(\gl_n))$.
\end{Prop}

This result follows immediately from Proposition~\ref{higher finite Gauss}
and Corollary~\ref{matching higher roots} below. To state those, define
the elements $\{\tilde{e}_{j,i+1},\tilde{f}_{i+1,j}\}_{1\leq j\leq i<n}$ of
$\fU_\vv^\rtt(\gl_n)$ via
\begin{equation}\label{Gauss ef-modes finite}
  \tilde{e}_{j,i+1}:=t^-_{jj}t^+_{j,i+1},\
  \tilde{f}_{i+1,j}:=t^-_{i+1,j}t^+_{jj}.
\end{equation}

\begin{Prop}\label{higher finite Gauss}
For any $1\leq j< i<n$, the following equalities hold in $U^\rtt_\vv(\gl_n)$:
\begin{equation*}
  \tilde{e}_{j,i+1}=\frac{[\tilde{e}_{i,i+1},\tilde{e}_{ji}]_{\vv^{-1}}}{\vv-\vv^{-1}},\
  \tilde{f}_{i+1,j}=\frac{[\tilde{f}_{ij},\tilde{f}_{i+1,i}]_{\vv}}{\vv^{-1}-\vv}.
\end{equation*}
\end{Prop}

The proof of this result is analogous to that of Proposition~\ref{higher ef-modes}
below (and actually it can be deduced from the latter by using the embedding
$\iota\colon \fU^\rtt_\vv(\gl_n)\hookrightarrow \fU^\rtt_\vv(L\gl_n)$ of
Lemma~\ref{embedding_quantum}).

\begin{Cor}\label{matching higher roots}
 $E_{j,i+1}=\Upsilon^{-1}(\tilde{e}_{j,i+1}),\
  F_{i+1,j}=\Upsilon^{-1}(\tilde{f}_{i+1,j})$
for any $1\leq j\leq i<n$.
\end{Cor}

\begin{proof}
For a fixed $1\leq i<n$, this follows by a decreasing induction in $j$.
The base of the induction $j=i$ is due to~(\ref{Ding-Frenkel finite formulas}),
while the induction step follows from Proposition~\ref{higher finite Gauss}.
\end{proof}

We order $\{E_{j,i+1}\}_{1\leq j\leq i<n}$ in the following way:
$E_{j,i+1}\leq E_{j',i'+1}$ if $j<j'$, or $j=j',i\leq i'$.
Likewise, we order $\{F_{i+1,j}\}_{1\leq j\leq i<n}$ so that
$F_{i+1,j}\geq F_{i'+1,j'}$ if $j<j'$, or $j=j',i\leq i'$.
Finally, we choose any total ordering of the Cartan generators
$\{t_j\}_{1\leq j\leq n}$ of $\fU_\vv(\gl_n)$
(or $\{K_i\}_{1\leq i<n}$ of $\fU_\vv(\ssl_n)$).
Having specified these three total orderings, elements $F\cdot H\cdot E$
with $F, E, H$ being ordered monomials in
$\{F_{i+1,j}\}_{1\leq j\leq i<n}$, $\{E_{j,i+1}\}_{1\leq j\leq i<n}$,
and the Cartan generators $\{t_j^{\pm 1}\}_{1\leq j\leq n}$ of $\fU_\vv(\gl_n)$
(or $\{K_i^{\pm 1}\}_{1\leq i<n}$ of $\fU_\vv(\ssl_n)$), respectively, are called
the \emph{ordered PBW monomials} (in the corresponding generators).
The proof of the following result is analogous to that of
Theorem~\ref{PBW basis coordinate affine} below and is based
on Proposition~\ref{comparison of integral forms quantum finite},
we leave details to the interested reader.

\begin{Thm}\label{PBW basis coordinate finite}
(a) The ordered PBW monomials in
  $\{F_{i+1,j}, t_k^{\pm 1}, E_{j,i+1}\}_{1\leq j\leq i<n}^{1\leq k\leq n}$
form a basis of a free $\BC[\vv,\vv^{-1}]$-module $\fU_\vv(\gl_n)$.

\noindent
(b) The ordered PBW monomials in
  $\{F_{i+1,j}, K_k^{\pm 1}, E_{j,i+1}\}_{1\leq j\leq i<n}^{1\leq k<n}$
form a basis of a free

\noindent
$\BC[\vv,\vv^{-1}]$-module $\fU_\vv(\ssl_n)$.
\end{Thm}

\begin{Rem}\label{limit of rtt quantum finite}
We note that $\fU_\vv(\gl_n)\simeq \fU^\rtt_\vv(\gl_n)$ quantizes
the algebra of functions on the big Bruhat cell in $\GL(n)$, that is
$\fU_\vv(\gl_n)/(\vv-1)\simeq \BC[N_{-}TN_{+}]$, due
to~(\ref{integral_finite_gln}) and the PBW theorem of~\cite[Proposition 2.1]{gm}.
Here $N_-$ (resp.\ $N_+$) denotes the subgroup of strictly lower (resp.\ strictly upper)
triangular matrices, and $T$ denotes the diagonal torus of $\GL(n)$.
\end{Rem}

\begin{Rem}\label{Luzstig form finite}
For a complete picture, let us recall in which sense $U_\vv(\gl_n)$ is
usually treated as a quantization of the universal enveloping algebra $U(\gl_n)$.
Let $\sU_\vv(\gl_n)$ be the $\BC[\vv,\vv^{-1}]$-subalgebra of $U_\vv(\gl_n)$ generated
by $\{t_j^{\pm 1}\}_{j=1}^n$ and the divided powers $\{E^{(m)}_i,F^{(m)}_i\}_{1\leq i<n}^{m\geq 1}$.
According to~\cite[Proposition 2.3(a)]{l} (cf.~\cite{jim}), the subalgebra
$\sU^<_\vv(\gl_n)$ (resp.\ $\sU^>_\vv(\gl_n)$) of $\sU_\vv(\gl_n)$ generated by
$\{F^{(m)}_i\}_{1\leq i<n}^{m\geq 1}$ (resp.\ $\{E^{(m)}_i\}_{1\leq i<n}^{m\geq 1}$)
is a free $\BC[\vv,\vv^{-1}]$-module with a basis consisting of the ordered products of the
divided powers of the root generators $F'_{i+1,j}:=[\cdots[F_j,F_{j+1}]_{\vv},\cdots,F_i]_{\vv}$
(resp.\ $E'_{j,i+1}:=[E_i,\cdots,[E_{j+1},E_j]_{\vv^{-1}}\cdots]_{\vv^{-1}}$).
Specializing $\vv$ to $1$, we have $t_j^2=1$ in a $\BC$-algebra
$\sU_1(\gl_n):=\sU_\vv(\gl_n)/(\vv-1)$.
Specializing further $t_j$ to $1$, we get a $\BC$-algebra isomorphism
$\sU_1(\gl_n)/(\langle t_j-1\rangle_{j=1}^n)\simeq U(\gl_n)$, under which
$E'_{j,i+1}\mapsto (-1)^{i-j}E_{j,i+1}, F'_{i+1,j}\mapsto (-1)^{i-j}E_{i+1,j}$.
\end{Rem}


\subsection{The Drinfeld quantum affine $\gl_n$ and $\ssl_n$}\label{ssec DJ affine}
\

Following~\cite{d2}, define the quantum loop $\gl_n$, denoted by $U_\vv(L\gl_n)$,
to be the associative $\BC(\vv)$-algebra generated by
  $\{e_{i,r},f_{i,r},\varphi^+_{j,s},\varphi^-_{j,-s}\}_{1\leq i<n,1\leq j\leq n}^{r\in \BZ, s\in \BN}$
with the following defining relations (cf.~\cite[Definition 3.1]{df}):
\begin{equation}\label{Dr affine}
\begin{split}
  & [\varphi_j^\epsilon(z),\varphi_{j'}^{\epsilon'}(w)]=0,\
    \varphi^\pm_{j,0}\cdot \varphi^\mp_{j,0}=1,\\
  & (z-\vv^{c_{ii'}}w)e_i(z)e_{i'}(w)=(\vv^{c_{ii'}}z-w)e_{i'}(w)e_i(z),\\
  & (\vv^{c_{ii'}}z-w)f_i(z)f_{i'}(w)=(z-\vv^{c_{ii'}}w)f_{i'}(w)f_i(z),\\
  & (\vv z-\vv^{-1}w)^{\delta_{ji}}(z-\vv w)^{\delta_{j,i+1}}\varphi^\epsilon_j(z)e_i(w)=
    (z-w)^{\delta_{ji}}(\vv z-w)^{\delta_{j,i+1}}e_i(w)\varphi^\epsilon_j(z),\\
  & (z-w)^{\delta_{ji}}(\vv z-w)^{\delta_{j,i+1}}\varphi^\epsilon_j(z)f_i(w)=
    (\vv z-\vv^{-1}w)^{\delta_{ji}}(z-\vv w)^{\delta_{j,i+1}}f_i(w)\varphi^\epsilon_j(z),\\
  & [e_i(z),f_{i'}(w)]=
    \frac{\delta_{ii'}}{\vv-\vv^{-1}}\delta\left(\frac{z}{w}\right)\left(\psi^+_i(z)-\psi^-_i(z)\right),\\
  & e_i(z)e_{i'}(w)=e_{i'}(w)e_i(z)\ \mathrm{and}\ f_i(z)f_{i'}(w)=f_{i'}(w)f_i(z)\ \mathrm{if}\ c_{ii'}=0,\\
  & [e_i(z_1),[e_i(z_2),e_{i'}(w)]_{\vv^{-1}}]_{\vv}+
    [e_i(z_2),[e_i(z_1),e_{i'}(w)]_{\vv^{-1}}]_{\vv}=0 \ \mathrm{if}\ c_{ii'}=-1,\\
  & [f_i(z_1),[f_i(z_2),f_{i'}(w)]_{\vv^{-1}}]_{\vv}+
    [f_i(z_2),[f_i(z_1),f_{i'}(w)]_{\vv^{-1}}]_{\vv}=0 \ \mathrm{if}\ c_{ii'}=-1,
\end{split}
\end{equation}
where the generating series are defined as follows:
\begin{equation*}
    e_i(z):=\sum_{r\in \BZ}{e_{i,r}z^{-r}},\
    f_i(z):=\sum_{r\in \BZ}{f_{i,r}z^{-r}},\
    \varphi_i^{\pm}(z):=\sum_{s\geq 0}{\varphi^\pm_{i,\pm s}z^{\mp s}},\
    \delta(z):=\sum_{r\in \BZ}{z^r},
\end{equation*}
and $\psi^\pm_i(z)=\sum_{s\geq 0}{\psi^\pm_{i,\pm s}z^{\mp s}}$ is determined via
  $\psi^\pm_i(z):=(\varphi^\pm_i(z))^{-1}\varphi^\pm_{i+1}(\vv^{-1}z)$.
We will also need \emph{Drinfeld half-currents} $e^\pm_i(z),f^\pm_i(z)$ defined via
\begin{equation*}\label{Drinfeld half-currents}
  e^+_i(z):=\sum_{r\geq 0} e_{i,r}z^{-r},\ e^-_i(z):=-\sum_{r<0} e_{i,r}z^{-r},\
  f^+_i(z):=\sum_{r>0} f_{i,r}z^{-r},\ f^-_i(z):=-\sum_{r\leq 0} f_{i,r}z^{-r},
\end{equation*}
so that $e_i(z)=e^+_i(z)-e^-_i(z),\ f_i(z)=f^+_i(z)-f^-_i(z)$.

The $\BC(\vv)$-subalgebra of $U_\vv(L\gl_n)$ generated by
  $\{e_{i,r},f_{i,r},\psi^\pm_{i,\pm s}\}_{1\leq i<n}^{r\in \BZ,s\in \BN}$
is isomorphic to the quantum loop $\ssl_n$, denoted by $U_\vv(L\ssl_n)$.
To be more precise,
this recovers the new Drinfeld realization of $U_\vv(L\ssl_n)$, see~\cite{d2}.
The latter also admits the original Drinfeld-Jimbo realization with
the generators $\{E_i,F_i,K^{\pm 1}_i\}_{i\in [n]}$ (here $[n]:=\{0,1,\ldots,n-1\}$
viewed as mod $n$ residues) and with the defining relations exactly as in~(\ref{DrJim finite}),
but with $(c_{ii'})_{i,i'\in [n]}$ denoting the Cartan matrix of $\widehat{\ssl}_n$.
We prefer to keep the same notation $U_\vv(L\ssl_n)$ for these two realizations.
However, we will need an explicit identification which expresses the Drinfeld-Jimbo
generators in terms of the ``loop'' generators
(featuring in the new Drinfeld realization), see~\cite{d2,jin}:
\begin{equation}\label{Jing's identification}
\begin{split}
  & E_i\mapsto e_{i,0},\ F_i\mapsto f_{i,0},\ K^{\pm 1}_i\mapsto \psi^\pm_{i,0}
    \ \ \mathrm{for}\ \ i\in [n]\backslash\{0\},\\
  & K_{0}^{\pm 1}\mapsto \psi^\mp_{1,0}\cdots \psi^\mp_{n-1,0},\\
  & E_{0}\mapsto (-\vv)^{-n+2}\cdot [\cdots[f_{1,1},f_{2,0}]_\vv,\cdots,f_{n-1,0}]_\vv
    \cdot \psi^-_{1,0}\cdots \psi^-_{n-1,0},\\
  & F_{0}\mapsto (-\vv)^n\cdot [e_{n-1,0},\cdots,[e_{2,0},e_{1,-1}]_{\vv^{-1}}\cdots]_{\vv^{-1}}
    \cdot \psi^+_{1,0}\cdots\psi^+_{n-1,0}.
\end{split}
\end{equation}

The relation between the algebras $U_\vv(L\gl_n)$ and $U^\rtt_\vv(L\gl_n)$
was conjectured in~\cite{frt} and proved in~\cite[Main Theorem]{df}.
To state the result, consider the Gauss decomposition of the matrices $T^\pm(z)$
of subsection~\ref{ssec RTT affine}:
\begin{equation*}
  T^\pm(z)=\wt{F}^\pm(z)\cdot \wt{G}^\pm(z)\cdot \wt{E}^\pm(z).
\end{equation*}
Here $\wt{F}^\pm(z),\wt{G}^\pm(z),\wt{E}^\pm(z)$ are the series in $z^{\mp 1}$
with coefficients in the algebra $\fU^\rtt_\vv(L\gl_n)\otimes \End\ \BC^n$
which are of the form
\begin{equation*}
    \wt{F}^\pm(z)=\sum_{i} E_{ii}+\sum_{i>j} \tilde{f}^\pm_{ij}(z)\cdot E_{ij},\
    \wt{G}^\pm(z)=\sum_{i} \tilde{g}^\pm_i(z)\cdot E_{ii},\
    \wt{E}^\pm(z)=\sum_{i} E_{ii}+\sum_{i<j} \tilde{e}^\pm_{ij}(z)\cdot E_{ij}.
\end{equation*}

\begin{Thm}[\cite{df}]\label{Ding-Frenkel affine}
There is a unique $\BC(\vv)$-algebra isomorphism
\begin{equation*}
  \Upsilon\colon U_\vv(L\gl_n)\iso U^\rtt_\vv(L\gl_n)
\end{equation*}
defined by
\begin{equation}\label{Ding-Frenkel formulas}
  e^\pm_i(z)\mapsto \frac{\tilde{e}^\pm_{i,i+1}(\vv^iz)}{\vv-\vv^{-1}},\
  f^\pm_i(z)\mapsto \frac{\tilde{f}^\pm_{i+1,i}(\vv^iz)}{\vv-\vv^{-1}},\
  \varphi^\pm_j(z)\mapsto \tilde{g}^\pm_j(\vv^jz).
\end{equation}
\end{Thm}

\begin{Rem}
To compare with the notations of~\cite{df}, we note that our generating series
$e_i(z), f_i(z), \varphi^\pm_j(z)$ of $U_\vv(L\gl_n)$ correspond to
  $\frac{X^-_i(\vv^i z)}{\vv^{-1}-\vv}, \frac{X^+_i(\vv^i z)}{\vv^{-1}-\vv}, k^\mp_j(\vv^j z)$
of~\cite[Definition 3.1]{df}, respectively. Likewise, our matrices $T^+(z)$ and $T^-(z)$
of subsection~\ref{ssec RTT affine} correspond to $L^-(z)$ and $L^+(z)$
of~\cite[Definition 3.2]{df}, respectively. After these identifications,
we see that Theorem~\ref{Ding-Frenkel affine} is just~\cite[Main Theorem]{df}
(for the trivial central charge).
\end{Rem}

As an immediate corollary, $\fU^\rtt_\vv(L\gl_n)$ is realized as a
$\BC[\vv,\vv^{-1}]$-subalgebra of $U_\vv(L\gl_n)$. To describe this
subalgebra explicitly, define the elements
$\{E^{(r)}_{j,i+1}, F^{(r)}_{i+1,j}\}_{1\leq j\leq i<n}^{r\in \BZ}$
of $U_\vv(L\gl_n)$ via
\begin{equation}\label{PBW basis elements}
\begin{split}
  & E^{(r)}_{j,i+1}:=
    (\vv-\vv^{-1})[e_{i,0},\cdots,[e_{j+1,0},e_{j,r}]_{\vv^{-1}}\cdots]_{\vv^{-1}},\\
  & F^{(r)}_{i+1,j}:=
    (\vv^{-1}-\vv)[\cdots[f_{j,r},f_{j+1,0}]_{\vv},\cdots,f_{i,0}]_{\vv}.
\end{split}
\end{equation}
These elements with $r=0,\pm 1$ played an important role
in~\cite[Section 10, Appendix G]{ft}. We also note that
$E^{(r)}_{i,i+1}=(\vv-\vv^{-1})e_{i,r}$ and $F^{(r)}_{i+1,i}=(\vv^{-1}-\vv)f_{i,r}$.

\begin{Def}\label{integral loop}
(a) Let $\fU_\vv(L\gl_n)$ be the $\BC[\vv,\vv^{-1}]$-subalgebra of $U_\vv(L\gl_n)$
generated by
\begin{equation}\label{PBW basis}
   \{E^{(r)}_{j,i+1},F^{(r)}_{i+1,j}\}_{1\leq j\leq i<n}^{r\in \BZ}\cup
   \{\varphi^\pm_{j,\pm s}\}_{1\leq j\leq n}^{s\in \BN}.
\end{equation}

\noindent
(b) Let $\fU_\vv(L\ssl_n)$ be the $\BC[\vv,\vv^{-1}]$-subalgebra of $U_\vv(L\ssl_n)$
generated by
\begin{equation}\label{PBW basis sln}
   \{E^{(r)}_{j,i+1},F^{(r)}_{i+1,j}\}_{1\leq j\leq i<n}^{r\in \BZ}\cup
   \{\psi^\pm_{i,\pm s}\}_{1\leq i<n}^{s\in \BN}.
\end{equation}
\end{Def}

The following result can be viewed as a trigonometric counterpart of
Proposition~\ref{yangian integral forms coincide}:

\begin{Prop}\label{comparison of integral forms quantum}
$\fU_\vv(L\gl_n)=\Upsilon^{-1}(\fU^\rtt_\vv(L\gl_n))$.
\end{Prop}

The proof of Proposition~\ref{comparison of integral forms quantum} follows
immediately from Proposition~\ref{higher ef-modes} and
Corollary~\ref{matching modes} below. To state those, let us express
the matrix coefficients of $\wt{F}^\pm(z),\wt{G}^\pm(z),\wt{E}^\pm(z)$
as series in $z^{\mp 1}$ with coefficients in $\fU^\rtt_\vv(L\gl_n)$:
\begin{equation*}
\begin{split}
  & \tilde{e}^+_{ij}(z)=\sum_{r\geq 0} \tilde{e}^{(r)}_{ij}z^{-r},\
    \tilde{e}^-_{ij}(z)=\sum_{r<0} \tilde{e}^{(r)}_{ij}z^{-r},\\
  & \tilde{f}^+_{ij}(z)=\sum_{r>0} \tilde{f}^{(r)}_{ij}z^{-r},\
    \tilde{f}^-_{ij}(z)=\sum_{r\leq 0} \tilde{f}^{(r)}_{ij}z^{-r},\\
  & \tilde{g}^\pm_i(z)=\tilde{g}^\pm_i+\sum_{r>0} \tilde{g}^{(\pm r)}_i z^{\mp r}.
\end{split}
\end{equation*}

The following result generalizes~\cite[Proposition G.9]{ft}:

\begin{Prop}\label{higher ef-modes}
For any $1\leq j<i<n$, the following equalities hold in $U^\rtt_\vv(L\gl_n)$:
\begin{equation}\label{Gauss e-modes}
  \tilde{e}^+_{j,i+1}(z)=\frac{[\tilde{e}^{(0)}_{i,i+1}, \tilde{e}^+_{ji}(z)]_{\vv^{-1}}}{\vv-\vv^{-1}},\
  \tilde{e}^-_{j,i+1}(z)=\frac{[\tilde{e}^{(0)}_{i,i+1}, \tilde{e}^-_{ji}(z)]_{\vv^{-1}}}{\vv-\vv^{-1}},
\end{equation}
\begin{equation}\label{Gauss f-modes}
  \tilde{f}^+_{i+1,j}(z)=\frac{[\tilde{f}^+_{ij}(z),\tilde{f}^{(0)}_{i+1,i}]_{\vv}}{\vv^{-1}-\vv},\
  \tilde{f}^-_{i+1,j}(z)=\frac{[\tilde{f}^-_{ij}(z),\tilde{f}^{(0)}_{i+1,i}]_{\vv}}{\vv^{-1}-\vv}.
\end{equation}
\end{Prop}

\begin{proof}
For any $1\leq i<n$, we proceed by an increasing induction in $j$.

\medskip
\noindent
$\bullet$
\emph{Verification of the first formula in~(\ref{Gauss e-modes})}.

Comparing the matrix coefficients $\langle v_j\otimes v_i|\cdots|v_i\otimes v_{i+1}\rangle$
of both sides of the equality~(\ref{all affRTT}) with $\epsilon=\epsilon'=+$, we get
\begin{equation*}
\begin{split}
   & (z-w)t^+_{ji}(z)t^+_{i,i+1}(w)+(\vv-\vv^{-1})zt^+_{ii}(z)t^+_{j,i+1}(w)=\\
   & (z-w)t^+_{i,i+1}(w)t^+_{ji}(z)+(\vv-\vv^{-1})wt^+_{ii}(w)t^+_{j,i+1}(z).
\end{split}
\end{equation*}
Evaluating the terms with $w^1$ in both sides of this equality, we find
\begin{equation*}
   -t^+_{ji}(z)t^+_{i,i+1}[0]=
   -t^+_{i,i+1}[0]t^+_{ji}(z)+(\vv-\vv^{-1})t^+_{ii}[0]t^+_{j,i+1}(z).
\end{equation*}
Note that $t^+_{ji}(z)t^+_{ii}[0]=\vv^{-1} t^+_{ii}[0]t^+_{ji}(z)$.
To see the latter, we compare the matrix coefficients
$\langle v_j\otimes v_i|\cdots|v_i\otimes v_i\rangle$
of both sides of the equality~(\ref{all affRTT}) with
$\epsilon=\epsilon'=+$, and then evaluate the terms with
$w^1$ as above. Combining this with
  $t^+_{i,i+1}[0]=\tilde{g}^+_i\tilde{e}^{(0)}_{i,i+1}=t^+_{ii}[0]\tilde{e}^{(0)}_{i,i+1}$,
we deduce
\begin{equation}\label{t+ generators upper}
  t^+_{j,i+1}(z)=\frac{[\tilde{e}^{(0)}_{i,i+1},t^+_{ji}(z)]_{\vv^{-1}}}{\vv-\vv^{-1}}.
\end{equation}
Recall that
\begin{equation*}
\begin{split}
   & t^+_{ji}(z)=\tilde{g}^+_j(z)\tilde{e}^+_{ji}(z)+
     \sum_{1\leq k\leq j-1}\tilde{f}^+_{jk}(z)\tilde{g}^+_k(z)\tilde{e}^+_{ki}(z),\\
   & t^+_{j,i+1}(z)=\tilde{g}^+_j(z)\tilde{e}^+_{j,i+1}(z)+
     \sum_{1\leq k\leq j-1}\tilde{f}^+_{jk}(z)\tilde{g}^+_k(z)\tilde{e}^+_{k,i+1}(z).
\end{split}
\end{equation*}
Let us further note that $\tilde{e}^{(0)}_{i,i+1}$ commutes with $\tilde{f}^+_{jk}(z)$
(since by the induction assumption the latter can be expressed via
$\tilde{f}^{(\bullet)}_{s,s-1}$ which clearly commute with $\tilde{e}^{(0)}_{i,i+1}$
for $s\leq j$) and with $\tilde{g}^+_k(z)$ for $k\leq j$. By the induction assumption
  $\frac{[\tilde{e}^{(0)}_{i,i+1},\tilde{e}^+_{ki}(z)]_{\vv^{-1}}}{\vv-\vv^{-1}}=
   \tilde{e}^+_{k,i+1}(z)$
for $k<j$. Hence, we get
\begin{equation*}
  \tilde{g}^+_j(z)\tilde{e}^+_{j,i+1}(z)+
  \sum_{k=1}^{j-1}\tilde{f}^+_{jk}(z)\tilde{g}^+_k(z)\tilde{e}^+_{k,i+1}(z)=
  \frac{\tilde{g}^+_j(z)[\tilde{e}^{(0)}_{i,i+1},\tilde{e}^+_{ji}(z)]_{\vv^{-1}}}{\vv-\vv^{-1}}+
  \sum_{k=1}^{j-1}\tilde{f}^+_{jk}(z)\tilde{g}^+_k(z)\tilde{e}^+_{k,i+1}(z),
\end{equation*}
which implies the first equality in~(\ref{Gauss e-modes}).

\medskip
\noindent
$\bullet$
\emph{Verification of the second formula in~(\ref{Gauss e-modes})}.

Comparing the matrix coefficients
$\langle v_j\otimes v_i|\cdots|v_i\otimes v_{i+1}\rangle$ of both sides
of the equality~(\ref{all affRTT}) with $\epsilon=-,\epsilon'=+$, we get
\begin{equation*}
\begin{split}
   & (z-w)t^-_{ji}(z)t^+_{i,i+1}(w)+(\vv-\vv^{-1})zt^-_{ii}(z)t^+_{j,i+1}(w)=\\
   & (z-w)t^+_{i,i+1}(w)t^-_{ji}(z)+(\vv-\vv^{-1})wt^+_{ii}(w)t^-_{j,i+1}(z).
\end{split}
\end{equation*}
Evaluating the terms with $w^1$ in both sides of this equality, we find
\begin{equation*}
   -t^-_{ji}(z)t^+_{i,i+1}[0]=
   -t^+_{i,i+1}[0]t^-_{ji}(z)+(\vv-\vv^{-1})t^+_{ii}[0]t^-_{j,i+1}(z).
\end{equation*}
Note that $t^-_{ji}(z)t^+_{ii}[0]=\vv^{-1} t^+_{ii}[0]t^-_{ji}(z)$
(which follows by comparing the matrix coefficients
$\langle v_j\otimes v_i|\cdots|v_i\otimes v_i\rangle$ of both sides
of the equality~(\ref{all affRTT}) with $\epsilon=-,\epsilon'=+$, and
then evaluating the terms with $w^1$ as above). Combining this with
  $t^+_{i,i+1}[0]=\tilde{g}^+_i\tilde{e}^{(0)}_{i,i+1}=t^+_{ii}[0]\tilde{e}^{(0)}_{i,i+1}$,
we obtain
\begin{equation}\label{t- generators upper}
  t^-_{j,i+1}(z)=\frac{[\tilde{e}^{(0)}_{i,i+1},t^-_{ji}(z)]_{\vv^{-1}}}{\vv-\vv^{-1}}.
\end{equation}
This implies the second equality in~(\ref{Gauss e-modes}) via
the same inductive arguments as above.

\medskip
\noindent
$\bullet$
\emph{Verification of the first formula in~(\ref{Gauss f-modes})}.

Comparing the matrix coefficients
$\langle v_{i+1}\otimes v_i|\cdots|v_i\otimes v_j\rangle$ of both sides
of the equality~(\ref{all affRTT}) with $\epsilon=-,\epsilon'=+$, we get
\begin{equation*}
\begin{split}
   & (z-w)t^-_{i+1,i}(z)t^+_{ij}(w)+(\vv-\vv^{-1})wt^-_{ii}(z)t^+_{i+1,j}(w)=\\
   & (z-w)t^+_{ij}(w)t^-_{i+1,i}(z)+(\vv-\vv^{-1})zt^+_{ii}(w)t^-_{i+1,j}(z).
\end{split}
\end{equation*}
Evaluating the terms with $z^0$ in both sides of this equality, we find
\begin{equation*}
   -t^-_{i+1,i}[0]t^+_{ij}(w)+(\vv-\vv^{-1})t^-_{ii}[0]t^+_{i+1,j}(w)=
   -t^+_{ij}(w)t^-_{i+1,i}[0].
\end{equation*}
Note that $t^-_{ii}[0]t^+_{i+1,j}(z)=t^+_{i+1,j}(z)t^-_{ii}[0]$ and
$t^-_{ii}[0]t^+_{ij}(z)=\vv t^+_{ij}(z)t^-_{ii}[0]$. To see these equalities,
we compare the matrix coefficients
$\langle v_{i+1}\otimes v_i|\cdots|v_j\otimes v_i\rangle$ and
$\langle v_i\otimes v_i|\cdots|v_j\otimes v_i\rangle$ of both sides
of the equality~(\ref{all affRTT}) with $\epsilon=+,\epsilon'=-$, and
then evaluate the terms with $w^0$ as above. Combining this with
  $t^-_{i+1,i}[0]=\tilde{f}^{(0)}_{i+1,i}\tilde{g}^-_i=
   \tilde{f}^{(0)}_{i+1,i}t^-_{ii}[0]$,
we deduce
\begin{equation}\label{t+ generators lower}
  t^+_{i+1,j}(w)=\frac{[t^+_{ij}(w),\tilde{f}^{(0)}_{i+1,i}]_{\vv}}{\vv^{-1}-\vv}.
\end{equation}
Recall that
\begin{equation*}
\begin{split}
   & t^+_{ij}(w)=\tilde{f}^+_{ij}(w)\tilde{g}^+_j(w)+
     \sum_{1\leq k\leq j-1}\tilde{f}^+_{ik}(w)\tilde{g}^+_k(w)\tilde{e}^+_{kj}(w),\\
   & t^+_{i+1,j}(w)=\tilde{f}^+_{i+1,j}(w)\tilde{g}^+_j(w)+
     \sum_{1\leq k\leq j-1}\tilde{f}^+_{i+1,k}(w)\tilde{g}^+_k(w)\tilde{e}^+_{kj}(w).
\end{split}
\end{equation*}
We further note that $\tilde{f}^{(0)}_{i+1,i}$ commutes with
$\tilde{e}^+_{kj}(z)$ (since by the induction assumption the latter can be
expressed via $\tilde{e}^{(\bullet)}_{s-1,s}$ which clearly commute with
$\tilde{f}^{(0)}_{i+1,i}$ for $s\leq j$) and with $\tilde{g}^+_k(w)$ for $k\leq j$.
By the induction assumption
  $\frac{[\tilde{f}^+_{ik}(w),\tilde{f}^{(0)}_{i+1,i}]_{\vv}}{\vv^{-1}-\vv}=\tilde{f}^+_{i+1,k}(w)$
for $k<j$. Hence, we finally get
\begin{equation*}
\begin{split}
  & \tilde{f}^+_{i+1,j}(w)\tilde{g}^+_j(w)+
    \sum_{1\leq k\leq j-1}\tilde{f}^+_{i+1,k}(w)\tilde{g}^+_k(w)\tilde{e}^+_{kj}(w)=\\
  & \frac{[\tilde{f}^+_{ij}(w),\tilde{f}^{(0)}_{i+1,i}]_{\vv}\cdot \tilde{g}^+_j(w)}{\vv^{-1}-\vv}+
    \sum_{1\leq k\leq j-1}\tilde{f}^+_{i+1,k}(w)\tilde{g}^+_k(w)\tilde{e}^+_{kj}(w),
\end{split}
\end{equation*}
which implies the first equality in~(\ref{Gauss f-modes}).

\medskip
\noindent
$\bullet$
\emph{Verification of the second formula in~(\ref{Gauss f-modes})}.

Comparing the matrix coefficients
$\langle v_{i+1}\otimes v_i|\cdots|v_i\otimes v_j\rangle$ of both sides of
the equality~(\ref{all affRTT}) with $\epsilon=\epsilon'=-$, we get
\begin{equation*}
\begin{split}
   & (z-w)t^-_{i+1,i}(z)t^-_{ij}(w)+(\vv-\vv^{-1})wt^-_{ii}(z)t^-_{i+1,j}(w)=\\
   & (z-w)t^-_{ij}(w)t^-_{i+1,i}(z)+(\vv-\vv^{-1})zt^-_{ii}(w)t^-_{i+1,j}(z).
\end{split}
\end{equation*}
Evaluating the terms with $z^0$ in both sides of this equality, we find
\begin{equation*}
   -t^-_{i+1,i}[0]t^-_{ij}(w)+(\vv-\vv^{-1})t^-_{ii}[0]t^-_{i+1,j}(w)=
   -t^-_{ij}(w)t^-_{i+1,i}[0].
\end{equation*}
Note that $t^-_{ii}[0]t^-_{i+1,j}(z)=t^-_{i+1,j}(z)t^-_{ii}[0]$ and
$t^-_{ii}[0]t^-_{ij}(z)=\vv t^-_{ij}(z)t^-_{ii}[0]$. To see these
equalities, we compare the matrix coefficients
$\langle v_{i+1}\otimes v_i|\cdots|v_j\otimes v_i\rangle$ and
$\langle v_i\otimes v_i|\cdots|v_j\otimes v_i\rangle$ of both sides
of the equality~(\ref{all affRTT}) with $\epsilon=\epsilon'=-$, and
then evaluate the terms with $w^0$ as above. Combining this with
$t^-_{i+1,i}[0]=\tilde{f}^{(0)}_{i+1,i}\tilde{g}^-_i=\tilde{f}^{(0)}_{i+1,i}t^-_{ii}[0]$,
we obtain
\begin{equation}\label{t- generators lower}
  t^-_{i+1,j}(w)=\frac{[t^-_{ij}(w),\tilde{f}^{(0)}_{i+1,i}]_{\vv}}{\vv^{-1}-\vv}.
\end{equation}
This implies the second equality in~(\ref{Gauss f-modes})
via the same inductive arguments as above.

\medskip
This completes our proof of Proposition~\ref{higher ef-modes}.
\end{proof}

\begin{Cor}
For any $1\leq j\leq i<n$ and $r\in \BZ$, the following equalities hold:
\begin{equation}\label{Gauss Matrix Entries}
\begin{split}
  & \tilde{e}^{(r)}_{j,i+1}:=
    (\vv-\vv^{-1})^{j-i}[\tilde{e}^{(0)}_{i,i+1},\cdots,[\tilde{e}^{(0)}_{j+1,j+2},\tilde{e}^{(r)}_{j,j+1}]_{\vv^{-1}}\cdots]_{\vv^{-1}},\\
  & \tilde{f}^{(r)}_{i+1,j}:=
    (\vv^{-1}-\vv)^{j-i}[\cdots[\tilde{f}^{(r)}_{j+1,j},\tilde{f}^{(0)}_{j+2,j+1}]_{\vv},\cdots,\tilde{f}^{(0)}_{i+1,i}]_{\vv}.
\end{split}
\end{equation}
\end{Cor}

Combining these explicit formulas with~(\ref{Ding-Frenkel formulas}), we obtain

\begin{Cor}\label{matching modes}
For any $1\leq j\leq i<n$ and $r\in \BZ$, we have the following  equalities:
\begin{equation}\label{matching modes formulas}
  E^{(r)}_{j,i+1}=(-1)^{\delta_{r<0}}\vv^{-jr}\cdot \Upsilon^{-1}(\tilde{e}^{(r)}_{j,i+1}),\
  F^{(r)}_{i+1,j}=(-1)^{\delta_{r>0}}\vv^{-jr}\cdot \Upsilon^{-1}(\tilde{f}^{(r)}_{i+1,j}).
\end{equation}
\end{Cor}

We now apply Proposition~\ref{comparison of integral forms quantum} to
construct bases of $\fU_\vv(L\gl_n)$ and $\fU_\vv(L\ssl_n)$.
It will be convenient to relabel the Cartan generators via
  $\varphi_{i,r}:=
   \begin{cases}
     \varphi^+_{i,r}, & \text{if } r\geq 0 \\
     \varphi^-_{i,r}, & \text{if } r<0
   \end{cases},$
  $\psi_{i,r}:=
   \begin{cases}
     \psi^+_{i,r}, & \text{if } r\geq 0 \\
     \psi^-_{i,r}, & \text{if } r<0
   \end{cases},$
so that $(\varphi_{i,0})^{-1}=\varphi^-_{i,0},(\psi_{i,0})^{-1}=\psi^-_{i,0}$.
We order the elements $\{E^{(r)}_{j,i+1}\}_{1\leq j\leq i<n}^{r\in \BZ}$ in
the following way: $E^{(r)}_{j,i+1}\leq E^{(r')}_{j',i'+1}$ if $j<j'$, or $j=j',i<i'$,
or $j=j',i=i',r\leq r'$. Likewise, we order $\{F^{(r)}_{i+1,j}\}_{1\leq j\leq i<n}^{r\in \BZ}$
so that $F^{(r)}_{i+1,j}\geq F^{(r')}_{i'+1,j'}$ if $j<j'$, or $j=j',i<i'$, or
$j=j',i=i',r\leq r'$. Finally, we choose any total ordering of the Cartan generators
$\{\varphi_{j,s}\}_{1\leq j\leq n}^{s\in \BZ}$ of $\fU_\vv(L\gl_n)$
(or $\{\psi_{i,s}\}_{1\leq i<n}^{s\in \BZ}$ of $\fU_\vv(L\ssl_n)$).
Having specified these three total orderings, elements $F\cdot H\cdot E$ with $F,E,H$
being ordered monomials in $\{F^{(r)}_{i+1,j}\}_{1\leq j\leq i<n}^{r\in \BZ}$,
$\{E^{(r)}_{j,i+1}\}_{1\leq j\leq i<n}^{r\in \BZ}$, and the Cartan generators
$\{\varphi_{j,s}\}_{1\leq j\leq n}^{s\in \BZ}$ of $\fU_\vv(L\gl_n)$
(or $\{\psi_{i,s}\}_{1\leq i<n}^{s\in \BZ}$ of $\fU_\vv(L\ssl_n)$), respectively,
are called the \emph{ordered PBWD monomials} (in the corresponding generators).

\begin{Thm}\label{PBW basis coordinate affine}
(a) The ordered PBWD monomials in
  $\{F^{(r)}_{i+1,j}, \varphi_{k,s}, E^{(r)}_{j,i+1}\}_{1\leq j\leq i<n, 1\leq k\leq n}^{r,s\in \BZ}$
form a basis of a free $\BC[\vv,\vv^{-1}]$-module $\fU_\vv(L\gl_n)$.

\noindent
(b) The ordered PBWD monomials in
  $\{F^{(r)}_{i+1,j}, \psi_{k,s}, E^{(r)}_{j,i+1}\}_{1\leq j\leq i<n, 1\leq k< n}^{r,s\in \BZ}$
form a basis of a free $\BC[\vv,\vv^{-1}]$-module $\fU_\vv(L\ssl_n)$.
\end{Thm}

This result generalizes (and its proof is actually based on)~\cite[Theorems 2.15, 2.17, 2.19]{t}.
To recall these theorems in the full generality (which is needed for the further use),
let us generalize the elements
$\{E^{(r)}_{j,i+1}, F^{(r)}_{i+1,j}\}_{1\leq j\leq i<n}^{r\in \BZ}$ first.
For every pair $1\leq j\leq i<n$ and any $r\in \BZ$, we choose a \emph{decomposition}
$\unl{r}=(r_j,\ldots,r_i)\in \BZ^{i-j+1}$ such that $r=r_j+r_{j+1}+\ldots+r_i$.
We define
\begin{equation}\label{generalized PBW basis elements}
\begin{split}
  & E_{j,i+1}(\unl{r}):=
    (\vv-\vv^{-1})[e_{i,r_i},\cdots,[e_{j+1,r_{j+1}},e_{j,r_j}]_{\vv^{-1}}\cdots]_{\vv^{-1}},\\
  & F_{i+1,j}(\unl{r}):=
    (\vv^{-1}-\vv)[\cdots[f_{j,r_j},f_{j+1,r_{j+1}}]_{\vv},\cdots,f_{i,r_i}]_{\vv}.
\end{split}
\end{equation}
In the particular case $r_j=r,r_{j+1}=\ldots=r_i=0$, we recover
$E^{(r)}_{j,i+1},F^{(r)}_{i+1,j}$ of~(\ref{PBW basis elements}).

Let $U^<_\vv(L\gl_n)$ and $U^>_\vv(L\gl_n)$ be the $\BC(\vv)$-subalgebras
of $U_\vv(L\gl_n)$ generated by $\{f_{i,r}\}_{1\leq i<n}^{r\in \BZ}$ and
$\{e_{i,r}\}_{1\leq i<n}^{r\in \BZ}$, respectively. Let $\fU^<_\vv(L\gl_n)$
and $\fU^>_\vv(L\gl_n)$ be the $\BC[\vv,\vv^{-1}]$-subalgebras of
$\fU_\vv(L\gl_n)$ generated by $\{F^{(r)}_{i+1,j}\}_{1\leq j\leq i<n}^{r\in \BZ}$
and $\{E^{(r)}_{j,i+1}\}_{1\leq j\leq i<n}^{r\in \BZ}$, respectively.

\begin{Thm}[\cite{t}]\label{PBW for half-integral}
For any $1\leq j\leq i<n$ and $r\in \BZ$, choose a decomposition $\unl{r}$ as above.

\noindent
(a) The ordered PBWD monomials in $\{E_{j,i+1}(\unl{r})\}_{1\leq j\leq i<n}^{r\in \BZ}$
form a basis of a free $\BC[\vv,\vv^{-1}]$-module $\fU^>_\vv(L\gl_n)$.

\noindent
(b) The ordered PBWD monomials in $\{E_{j,i+1}(\unl{r})\}_{1\leq j\leq i<n}^{r\in \BZ}$
form a basis of a $\BC(\vv)$-vector space $U^>_\vv(L\gl_n)$.

\noindent
(c) The ordered PBWD monomials in $\{F_{i+1,j}(\unl{r})\}_{1\leq j\leq i<n}^{r\in \BZ}$
form a basis of a free $\BC[\vv,\vv^{-1}]$-module $\fU^<_\vv(L\gl_n)$.

\noindent
(d) The ordered PBWD monomials in $\{F_{i+1,j}(\unl{r})\}_{1\leq j\leq i<n}^{r\in \BZ}$
form a basis of a $\BC(\vv)$-vector space $U^<_\vv(L\gl_n)$.

\noindent
(e) The ordered PBWD monomials in
  $\{F^{(r)}_{i+1,j}, \varphi_{k,s}, E^{(r)}_{j,i+1}\}_{1\leq j\leq i<n, 1\leq k\leq n}^{r,s\in \BZ}$
form a basis of the quantum loop algebra $U_\vv(L\gl_n)$.
\end{Thm}

\begin{proof}[Proof of Theorem~\ref{PBW basis coordinate affine}]
Due to Theorem~\ref{PBW for half-integral},
it suffices to verify that all unordered products
  $E^{(r)}_{j,i+1}\varphi^\pm_{j',\pm s},
   \varphi^\pm_{j',\pm s}F^{(r)}_{i+1,j},
   E^{(r)}_{j,i+1}F^{(s)}_{i'+1,j'}$
are equal to $\BC[\vv,\vv^{-1}]$-linear combinations of the ordered PBWD
monomials. The verification for the first two cases is simple. Indeed, we
can always move $\varphi^\pm_{j',0}$ to the left or to the right acquiring
an appropriate power of $\vv$. As for the other Cartan generators, it is more
convenient to work with another choice of Cartan generators $h_{j',\pm s}$ defined via
  $\varphi^\pm_{j'}(z)=\varphi^\pm_{j',0}\exp(\sum_{s>0}h_{j',s}z^{\mp s})$.
These generators satisfy simple commutation relations:
  $[h_{j',s},e_{i,r}]=c(i,j',r,s)e_{i,r+s},
   [h_{j',s},f_{i,r}]=-c(i,j',r,s)f_{i,r+s}$ for
certain $c(i,j',r,s)\in \BC[\vv,\vv^{-1}]$. Therefore,
  $E^{(r)}_{j,i+1}h_{j',s}-h_{j',s}E^{(r)}_{j,i+1}$
is a $\BC[\vv,\vv^{-1}]$-linear combination of the terms of the form
$E_{j,i+1}(\unl{r+s})$ for various decompositions of $r+s$ into the sum of
$i-j+1$ integers, hence, the claim for $E^{(r)}_{j,i+1}h_{j',s}$.
The case of $h_{j',s}F^{(r)}_{i+1,j}$ is analogous.

Thus, it remains to verify that $E^{(r)}_{j,i}F^{(s)}_{i',j'}$ is a
$\BC[\vv,\vv^{-1}]$-linear combination of the ordered PBWD monomials.
First, let us note that if $j\geq i'$ or $j'\geq i$, then
  $E^{(r)}_{j,i}F^{(s)}_{i',j'}=F^{(s)}_{i',j'}E^{(r)}_{j,i}$
and the latter is already an ordered PBWD monomial. Hence, from now on
we shall assume $i'>j,i>j'$. There are four cases to consider:
  (1) $r\geq 0,s>0$, (2) $r<0,s>0$, (3) $r\geq 0, s\leq 0$, (4) $r<0, s\leq 0$.
For simplicity of the current exposition, we shall treat only the first case,
while the proof is similar in the remaining three cases. Thus, we assume
$r\geq 0,s>0$ from now on. The proof will proceed by an increasing induction
in $r+s$, then by an increasing induction in $j'$, and finally by
an increasing induction in $r$.

Our proof is based on Proposition~\ref{comparison of integral forms quantum}.
In particular, applying Corollary~\ref{matching modes} to $E^{(r)}_{j,i}F^{(s)}_{i',j'}$,
the question is reduced to the proof of the fact that
$\tilde{e}^{(r)}_{ji}\tilde{f}^{(s)}_{i'j'}$ is a $\BC[\vv,\vv^{-1}]$-linear
combination of monomials in the generators
  $\{\tilde{e}^{(\bullet)}_{\bullet,\bullet},
   \tilde{f}^{(\bullet)}_{\bullet,\bullet},
   \tilde{g}^\pm_\bullet, \tilde{g}^{(\bullet)}_{\bullet}\}$
(ordered accordingly).

Recall that
  $t^+_{ji}(z)=\tilde{g}^+_j(z)\tilde{e}^+_{ji}(z)+
  \sum_{k=1}^{j-1}\tilde{f}^+_{jk}(z)\tilde{g}^+_k(z)\tilde{e}^+_{ki}(z)$,
which immediately implies
\begin{equation}\label{e vs t}
  t^+_{ji}[r]=
  \tilde{g}^+_j\tilde{e}^{(r)}_{ji}+
  \sum_{0\leq r'<r} \tilde{g}^{(r-r')}_j\tilde{e}^{(r')}_{ji}+
  \sum_{k=1}^{j-1}\sum_{r_1>0, r_2\geq 0,r_3\geq 0}^{r_1+r_2+r_3=r}
    \tilde{f}^{(r_1)}_{jk}\tilde{g}^{(r_2)}_k\tilde{e}^{(r_3)}_{ki},
\end{equation}
where $\tilde{g}^{(0)}_k$ denotes $\tilde{g}^+_k$.
Likewise,
  $t^+_{i'j'}(w)=\tilde{f}^+_{i'j'}(w)\tilde{g}^+_{j'}(w)+
   \sum_{k'=1}^{j'-1}\tilde{f}^+_{i'k'}(w)\tilde{g}^+_{k'}(w)\tilde{e}^+_{k'j'}(w)$
implies
\begin{equation}\label{f vs t}
  t^+_{i'j'}[s]=
  \tilde{f}^{(s)}_{i'j'}\tilde{g}^+_{j'}+
  \sum_{0<s'<s}\tilde{f}^{(s')}_{i'j'}\tilde{g}^{(s-s')}_{j'}+
  \sum_{k'=1}^{j'-1}\sum_{s_1>0, s_2\geq 0,s_3\geq 0}^{s_1+s_2+s_3=s}
    \tilde{f}^{(s_1)}_{i'k'}\tilde{g}^{(s_2)}_{k'}\tilde{e}^{(s_3)}_{k'j'}.
\end{equation}

Applying formulas~(\ref{e vs t}, \ref{f vs t}), let us now evaluate the product
$t^+_{ji}[r]t^+_{i'j'}[s]$ and consider the corresponding unordered terms
(we shall be ignoring the Cartan generators
$\tilde{g}^\pm_\bullet,\tilde{g}^{(\bullet)}_{\bullet}$ since they can be moved
to any side harmlessly as explained above). Besides for
$\tilde{e}^{(r)}_{ji}\tilde{f}^{(s)}_{i'j'}$, all other terms will be either
of the form $\tilde{e}^{(r)}_{ji}\tilde{f}^{(s)}_{i'k'}$ with $k'<j'$ or of
the form $\tilde{e}^{(r')}_{\bullet,\bullet}\tilde{f}^{(s')}_{\bullet,\bullet}$
with $r'+s'<r+s$. By the induction assumption, the latter terms are
$\BC[\vv,\vv^{-1}]$-linear combinations of the ordered monomials. Therefore,
it suffices to prove that so is $t^+_{ji}[r]t^+_{i'j'}[s]$.

To verify the latter, we start by comparing the matrix coefficients
$\langle v_j\otimes v_{i'}|\cdots|v_i\otimes v_{j'}\rangle$ of both sides
of the equality~(\ref{all affRTT}) with $\epsilon=\epsilon'=+$:
\begin{equation*}\label{ef vs t}
\begin{split}
   & (z-w)t^+_{ji}(z)t^+_{i'j'}(w)+
     (\vv-\vv^{-1})z t^+_{i'i}(z)t^+_{jj'}(w)=\\
   & (z-w)t^+_{i'j'}(w)t^+_{ji}(z)+(\vv-\vv^{-1})zt^+_{i'i}(w)t^+_{jj'}(z).
\end{split}
\end{equation*}
Evaluating the coefficients of $z^{1-r}w^{-s}$ in both sides of this equality,
we obtain
\begin{multline}\label{ef vs t explicit}
  t^+_{ji}[r]t^+_{i'j'}[s]=
    (\vv-\vv^{-1})t^+_{i'i}[s]t^+_{jj'}[r]-(\vv-\vv^{-1})t^+_{i'i}[r]t^+_{jj'}[s]+\\
  t^+_{ji}[r-1]t^+_{i'j'}[s+1]+t^+_{i'j'}[s]t^+_{ji}[r]-t^+_{i'j'}[s+1]t^+_{ji}[r-1].
\end{multline}
Let us now consider the unordered monomials appearing in each summand of the
right-hand side of~(\ref{ef vs t explicit}). First, we note that all the
unordered monomials appearing in the last three summands are of the form
$\tilde{e}^{(r')}_{\bullet,\bullet}\tilde{f}^{(s')}_{\bullet,\bullet}$ with
either $r'=r-1,s'=s+1$ or with $r'+s'<r+s$, hence, they are
$\BC[\vv,\vv^{-1}]$-linear combinations of the ordered monomials by the induction
assumption. Let us now consider the unordered terms appearing in
$t^+_{i'i}[r]t^+_{jj'}[s]$. If $i'\geq i$, then clearly all the unordered terms
are of the form $\tilde{e}^{(r')}_{\bullet,\bullet}\tilde{f}^{(s')}_{\bullet,\bullet}$
with $r'+s'<r+s$, to which the induction assumption applies. If $i'<i$, then all
the unordered terms in $t^+_{i'i}[r]t^+_{jj'}[s]$ are either as above
(to which the induction assumption applies) or of the form
$\tilde{e}^{(r)}_{i'i}\tilde{f}^{(s)}_{jk}$ with $k<j$. As $i\geq i'>j>k$, we have
  $\tilde{e}^{(r)}_{i'i}\tilde{f}^{(s)}_{jk}=\tilde{f}^{(s)}_{jk}\tilde{e}^{(r)}_{i'i}$
(for any $k<j$) which is an ordered monomial. Therefore, we have eventually proved that
$t^+_{i'i}[r]t^+_{jj'}[s]$ is a $\BC[\vv,\vv^{-1}]$-linear combination of the ordered
monomials. Swapping $r$ and $s$, we obtain the same result for $t^+_{i'i}[s]t^+_{jj'}[r]$.

Combining all the above, we see that $\tilde{e}^{(r)}_{ji}\tilde{f}^{(s)}_{i'j'}$
is a $\BC[\vv,\vv^{-1}]$-linear combination of the ordered monomials, hence,
$E^{(r)}_{j,i}F^{(s)}_{i',j'}$ is a $\BC[\vv,\vv^{-1}]$-linear combination of
the ordered PBWD monomials.

\medskip
This completes our proof of Theorem~\ref{PBW basis coordinate affine}.
\end{proof}

\begin{Rem}\label{thick slice}
We note that $\fU_\vv(L\gl_n)\simeq \fU^\rtt_\vv(L\gl_n)$ quantizes
the algebra of functions on the thick slice $^\dagger\CW_0$ of~\cite[4(viii)]{ft},
that is, $\fU_\vv(L\gl_n)/(\vv-1)\simeq \BC[^\dagger\CW_0]$.
\end{Rem}

\begin{Rem}\label{Luzstig form affine}
For a complete picture, recall that $U_\vv(L\ssl_n)$  is usually treated as a quantization
of the universal enveloping algebra $U(L\ssl_n)$, cf.~Remark~\ref{Luzstig form finite}.
Let $\sU_\vv(L\ssl_n)$ be the $\BC[\vv,\vv^{-1}]$-subalgebra of
$U_\vv(L\ssl_n)$ generated by $\{K_j^{\pm 1}\}_{j\in [n]}$ and the divided powers
$\{E^{(m)}_i,F^{(m)}_i\}_{i\in [n]}^{m\geq 1}$. Specializing $\vv$ to $1$, we have
$K_j^2=1$ in a $\BC$-algebra $\sU_1(L\ssl_n):=\sU_\vv(L\ssl_n)/(\vv-1)$.
Specializing further $K_j$ to $1$, we get an algebra isomorphism
$\sU_1(L\ssl_n)/(\langle K_j-1\rangle_{j\in [n]})\simeq U(L\ssl_n)$.
However, we are not aware of the description of $\sU_\vv(L\ssl_n)$ in the new Drinfeld
realization. In particular, it would be interesting to find an explicit basis of
$\sU_\vv(L\ssl_n)$ similar to that of~Theorem~\ref{PBW basis coordinate affine}.
\end{Rem}


\subsection{Shuffle algebra and its integral form}\label{ssec shuffle algebra}
\

In this section, we recall the shuffle realizations of
$U^>_\vv(L\gl_n), \fU^>_\vv(L\gl_n)$ established in~\cite{t}. Set
  $\Sigma_{(k_1,\ldots,k_{n-1})}:=\Sigma_{k_1}\times \cdots\times \Sigma_{k_{n-1}}$
for $k_1,\ldots,k_{n-1}\in \BN$. Consider an $\BN^{n-1}$-graded $\BC(\vv)$-vector space
  $\BS^{(n)}=\underset{\underline{k}=(k_1,\ldots,k_{n-1})\in \BN^{n-1}}\bigoplus\BS^{(n)}_{\underline{k}},$
where $\BS^{(n)}_{\unl{k}}$ consists of $\Sigma_{\unl{k}}$-symmetric rational functions
in the variables $\{x_{i,r}\}_{1\leq i<n}^{1\leq r\leq k_i}$. We also fix a matrix of
rational functions $(\zeta_{i,j}(z))_{i,j=1}^{n-1}$ by setting
$\zeta_{i,j}(z)=\frac{z-\vv^{-c_{ij}}}{z-1}$. Let us now introduce the bilinear
\emph{shuffle product} $\star$ on $\BS^{(n)}$:
given $F\in \BS^{(n)}_{\underline{k}}$ and $G\in \BS^{(n)}_{\underline{\ell}}$,
define $F\star G\in \BS^{(n)}_{\underline{k}+\underline{\ell}}$ via
\begin{equation}\label{shuffle product}
\begin{split}
  & (F\star G)(x_{1,1},\ldots,x_{1,k_1+\ell_1};\ldots;x_{n-1,1},\ldots, x_{n-1,k_{n-1}+\ell_{n-1}}):=
    \unl{k}!\cdot\unl{\ell}!\times\\
  & \Sym_{\Sigma_{\unl{k}+\unl{\ell}}}
    \left(F\left(\{x_{i,r}\}_{1\leq i<n}^{1\leq r\leq k_i}\right) G\left(\{x_{i',r'}\}_{1\leq i'<n}^{k_{i'}<r'\leq k_{i'}+\ell_{i'}}\right)\cdot
    \prod_{1\leq i<n}^{1\leq i'<n}\prod_{r\leq k_i}^{r'>k_{i'}}\zeta_{i,i'}(x_{i,r}/x_{i',r'})\right).
\end{split}
\end{equation}
Here $\unl{k}!:=\prod_{i=1}^{n-1} k_i!$, while for
$f\in \BC(\{x_{i,1},\ldots,x_{i,m_i}\}_{1\leq i<n})$ we define
its \emph{symmetrization} via
\begin{equation*}
  \Sym_{\Sigma_{\unl{m}}}(f):=\frac{1}{\unl{m}!}\cdot
  \sum_{(\sigma_1,\ldots,\sigma_{n-1})\in \Sigma_{\unl{m}}}
  f\left(\{x_{i,\sigma_i(1)},\ldots,x_{i,\sigma_i(m_i)}\}_{1\leq i<n}\right).
\end{equation*}
This endows $\BS^{(n)}$ with a structure of an associative $\BC(\vv)$-algebra
with the unit $\textbf{1}\in \BS^{(n)}_{(0,\ldots,0)}$.

We will be interested only in a certain $\BC(\vv)$-subspace of $\BS^{(n)}$,
defined by the \emph{pole} and \emph{wheel conditions}:

\medskip
\noindent
$\bullet$
We say that $F\in \BS^{(n)}_{\underline{k}}$ satisfies the \emph{pole conditions} if
\begin{equation}\label{pole conditions}
  F=\frac{f(x_{1,1},\ldots,x_{n-1,k_{n-1}})}{\prod_{i=1}^{n-2}\prod_{r\leq k_i}^{r'\leq k_{i+1}}(x_{i,r}-x_{i+1,r'})},\
  \mathrm{where}\ f\in (\BC(\vv)[\{x_{i,r}^{\pm 1}\}_{1\leq i<n}^{1\leq r\leq k_i}])^{\Sigma_{\unl{k}}}.
\end{equation}
\noindent
$\bullet$
We say that $F\in \BS^{(n)}_{\underline{k}}$ satisfies the \emph{wheel conditions} if
\begin{equation}\label{wheel conditions}
  F(\{x_{i,r}\})=0\ \mathrm{once}\ x_{i,r_1}=\vv x_{i+\epsilon,s}=\vv^2 x_{i,r_2}\
  \mathrm{for\ some}\ \epsilon, i, r_1, r_2, s,
\end{equation}
where
  $\epsilon\in \{\pm 1\},\ 1\leq i,i+\epsilon<n,\
   1\leq r_1,r_2\leq k_i,\ 1\leq s\leq k_{i+\epsilon}$.

Let $S^{(n)}_{\underline{k}}\subset \BS^{(n)}_{\underline{k}}$ denote the
$\BC(\vv)$-subspace of all elements $F$ satisfying these two conditions and set
  $S^{(n)}:=\underset{\underline{k}\in \BN^{n-1}}\bigoplus S^{(n)}_{\underline{k}}.$
It is straightforward to check that the $\BC(\vv)$-subspace $S^{(n)}\subset\BS^{(n)}$
is $\star$-closed. The resulting associative $\BC(\vv)$-algebra $\left(S^{(n)},\star\right)$
is called the \emph{shuffle algebra}. It is related to
$U^>_\vv(L\gl_n)\simeq U^>_\vv(L\ssl_n)$ via~\cite[Theorem 3.5]{t}
(cf.~\cite[Theorem 1.1]{n}):

\begin{Thm}[\cite{t}]\label{shuffle rational}
The assignment $e_{i,r}\mapsto x_{i,1}^r\ (1\leq i<n,r\in \BZ)$
gives rise to a $\BC(\vv)$-algebra isomorphism
$\Psi\colon U_\vv^{>}(L\gl_n)\iso S^{(n)}$.
\end{Thm}

For any $\unl{k}\in \BN^{n-1}$, consider a $\BC[\vv,\vv^{-1}]$-submodule
$\fS^{(n)}_{\unl{k}}\subset S^{(n)}_{\unl{k}}$ consisting of all \emph{integral}
elements, see~\cite[Definition 3.31]{t}. Set
  $\fS^{(n)}:=\underset{\underline{k}\in \BN^{n-1}}\bigoplus \fS^{(n)}_{\underline{k}}$
(it is a $\BC[\vv,\vv^{-1}]$-subalgebra of $S^{(n)}$ as follows from
Theorem~\ref{shuffle integral} below). While we skip an explicit definition of $\fS^{(n)}$
as it is quite involved, let us recall its relevant properties that were
established in~\cite[Proposition 3.36]{t}:

\begin{Prop}\label{key properties of integral shuffle}
(a) For any $1\leq \ell <n$, consider the linear map $\iota'_\ell\colon S^{(n)}\to S^{(n)}$
given by
\begin{equation}\label{shuffle shift homomorphism}
  \iota'_\ell(F)(\{x_{i,r}\}_{1\leq i< n}^{1\leq r\leq k_i}):=
  \prod_{r=1}^{k_\ell}(1-x_{\ell,r}^{-1})\cdot F(\{x_{i,r}\}_{1\leq i< n}^{1\leq r\leq k_i})
  \ \ \mathrm{for}\ \ F\in S^{(n)}_{\unl{k}}, \unl{k}\in \BN^{n-1}.
\end{equation}
Then
\begin{equation}\label{shift homomorphisms}
  F\in \fS^{(n)}\Longleftrightarrow \iota'_\ell(F)\in \fS^{(n)}.
\end{equation}

\noindent
(b) For any $\unl{k}\in \BN^{n-1}$ and a collection
  $g_i(\{x_{i,r}\}_{r=1}^{k_i})\in
   \BC[\vv,\vv^{-1}][\{x_{i,r}^{\pm 1}\}_{r=1}^{k_i}]^{\Sigma_{k_i}}\ (1\leq i< n)$,
set
\begin{equation}\label{for surjectivity in FT2}
  F:=(\vv-\vv^{-1})^{k_1+\ldots+k_{n-1}}\cdot
  \frac{\prod_{i=1}^{n-1}\prod_{1\leq r\ne r'\leq k_i} (x_{i,r}-\vv^{-2}x_{i,r'})\cdot
        \prod_{i=1}^{n-1}g_i(\{x_{i,r}\}_{r=1}^{k_i})}
       {\prod_{i=1}^{n-2}\prod_{1\leq r\leq k_i}^{1\leq r'\leq k_{i+1}}(x_{i,r}-x_{i+1,r'})}.
\end{equation}
Then $F\in \fS^{(n)}_{\unl{k}}$.
\end{Prop}

According to~\cite[Theorem 3.34]{t}, the isomorphism $\Psi$ of
Theorem~\ref{shuffle rational} identifies the integral forms
$\fU_\vv^{>}(L\gl_n)\subset U_\vv^{>}(L\gl_n)$ and $\fS^{(n)}\subset S^{(n)}$:

\begin{Thm}[\cite{t}]\label{shuffle integral}
The $\BC(\vv)$-algebra isomorphism $\Psi\colon U_\vv^{>}(L\gl_n)\iso S^{(n)}$
gives rise to a $\BC[\vv,\vv^{-1}]$-algebra isomorphism
$\Psi\colon \fU_\vv^{>}(L\gl_n)\iso \fS^{(n)}$.
\end{Thm}

We will crucially use this result in our proofs of
Theorems~\ref{PBW for integral shifted},~\ref{Surjectivity},~\ref{coproduct on integral form}.

\begin{Rem}\label{opposite shuffle}
For an algebra $A$, let $A^\op$ denote the opposite algebra. The assignment
$f_{i,r}\mapsto e_{i,r}\ (1\leq i<n,r\in \BZ)$ gives rise to a $\BC(\vv)$-algebra
isomorphism $U^<_\vv(L\gl_n)\iso U^{>}_\vv(L\gl_n)^{\op}$ and a $\BC[\vv,\vv^{-1}]$-algebra
isomorphism $\fU^<_\vv(L\gl_n)\iso \fU^{>}_\vv(L\gl_n)^{\op}$. Hence,
Theorems~\ref{shuffle rational} and~\ref{shuffle integral} give rise to a
$\BC(\vv)$-algebra isomorphism $\Psi\colon U_\vv^{<}(L\gl_n)\iso S^{(n),\op}$ and a
$\BC[\vv,\vv^{-1}]$-algebra isomorphism $\Psi\colon \fU_\vv^{<}(L\gl_n)\iso \fS^{(n),\op}$
(by abuse of notation, we still denote them by $\Psi$).
\end{Rem}


\subsection{The Jimbo evaluation homomorphism $\ev$}\label{ssec Jimbo evaluation}
\

While the quantum group $U_\vv(\fg)$ is always embedded into the quantum
loop algebra $U_\vv(L\fg)$, in type $A$ there also exist homomorphisms
$U_\vv(L\ssl_n)\to U_\vv(\gl_n)$, discovered in~\cite{jim}.
These homomorphisms are given in the Drinfeld-Jimbo realization of $U_\vv(L\ssl_n)$.

\begin{Thm}[\cite{jim}]\label{Jimbo's evaluation}
For any $a\in \BC^\times$, there is a unique $\BC(\vv)$-algebra homomorphism
\begin{equation*}
  \ev_a\colon U_\vv(L\ssl_n)\to U_\vv(\gl_n)
\end{equation*}
defined by
\begin{equation}\label{Jimbo's formulas}
\begin{split}
  & E_i\mapsto E_i,\ F_i\mapsto F_i,\
    K^{\pm 1}_i\mapsto K^{\pm 1}_i\ \mathrm{for}\ i\in [n]\backslash\{0\},\\
  & K^{\pm 1}_{0}\mapsto K^{\mp 1}_1\cdots K^{\mp 1}_{n-1},\\
  & E_{0}\mapsto (-1)^{n}\vv^{-n+1}a\cdot [\cdots[F_1,F_2]_\vv,\cdots,F_{n-1}]_\vv\cdot t_1^{-1}t_n^{-1},\\
  & F_{0}\mapsto (-1)^{n}\vv^{n-1}a^{-1}\cdot [E_{n-1},\cdots,[E_2,E_1]_{\vv^{-1}}\cdots]_{\vv^{-1}}\cdot t_1t_n.
\end{split}
\end{equation}
\end{Thm}

The key result of this subsection identifies the \emph{evaluation homomorphism}
$\ev_a$ with the restriction of $\BC(\vv)$-extended evaluation homomorphism
$\ev^\rtt_a$ of Lemma~\ref{ev_RTT}.

\begin{Thm}\label{compatibility of evaluations}
The following diagram is commutative:
\begin{equation}\label{diagram quantum}
  \begin{CD}
    U_\vv(L\ssl_n) @>\Upsilon>> U^\rtt_\vv(L\gl_n)\\
    @VV{\on{\ev_a}}V @V{\ev^\rtt_a}VV\\
    U_\vv(\gl_n) @>\sim>\Upsilon> U^\rtt_\vv(\gl_n)
    \end{CD}
\end{equation}
\end{Thm}

\begin{proof}
It suffices to verify $\Upsilon^{-1}(\ev^\rtt_a(\Upsilon(X)))=\ev_a(X)$
for all $X\in \{E_i,F_i,K_i\}_{i\in [n]}$. The only nontrivial cases are
$X=E_0\ \mathrm{or}\ F_0$, the verification for which is presented below.

\medskip
\noindent
$\bullet$
\emph{Verification of $\Upsilon^{-1}(\ev^\rtt_a(\Upsilon(E_0)))=\ev_a(E_0)$}.

According to~(\ref{matching modes formulas}), we have
\begin{equation*}
  \Upsilon([\cdots[f_{1,1},f_{2,0}]_\vv,\cdots,f_{n-1,0}]_\vv)=
  \frac{\Upsilon(F^{(1)}_{n1})}{\vv^{-1}-\vv}=
  \frac{\tilde{f}^{(1)}_{n1}}{\vv(\vv-\vv^{-1})}.
\end{equation*}
On the other hand, we have
  $t^+_{n1}[1]=\tilde{f}^{(1)}_{n1}\tilde{g}^+_1=
   \tilde{f}^{(1)}_{n1}\cdot t^+_{11}[0]$,
so that
  $\ev^\rtt_a(\tilde{f}^{(1)}_{n1})=-a\cdot t^-_{n1}(t^+_{11})^{-1}$.
Note that $\Upsilon^{-1}((t^+_{kk})^{-1})=t^{-1}_k$, while
  $\Upsilon^{-1}(t^-_{n1})=
   (\vv^{-1}-\vv)\cdot [\cdots[F_1,F_2]_\vv,\cdots,F_{n-1}]_\vv\cdot t^{-1}_1$,
due to Corollary~\ref{matching higher roots}.
Combining all the above with~(\ref{Jing's identification}), we finally obtain
\begin{equation*}
  \Upsilon^{-1}(\ev^\rtt_a(\Upsilon(E_0)))=
  (-1)^n\vv^{-n+1}a\cdot [\cdots[F_1,F_2]_\vv,\cdots,F_{n-1}]_\vv\cdot t_1^{-1}t_n^{-1}=
  \ev_a(E_0).
\end{equation*}

\medskip
\noindent
$\bullet$
\emph{Verification of $\Upsilon^{-1}(\ev^\rtt_a(\Upsilon(F_0)))=\ev_a(F_0)$}.

According to~(\ref{matching modes formulas}), we have
\begin{equation*}
  \Upsilon([e_{n-1,0},\cdots,[e_{2,0},e_{1,-1}]_{\vv^{-1}}\cdots]_{\vv^{-1}})=
  \frac{\Upsilon(E^{(-1)}_{1n})}{\vv-\vv^{-1}}=
  -\frac{\vv \tilde{e}^{(-1)}_{1n}}{\vv-\vv^{-1}}.
\end{equation*}
On the other hand,
  $t^-_{1n}[-1]=\tilde{g}^-_{1}\tilde{e}^{(-1)}_{1n}=
   t^-_{11}[0]\tilde{e}^{(-1)}_{1n}$,
so that
  $\ev^\rtt_a(\tilde{e}^{(-1)}_{1n})=-a^{-1}\cdot (t^-_{11})^{-1}t^+_{1n}$.
Note that $\Upsilon^{-1}((t^-_{kk})^{-1})=t_k$, while
  $\Upsilon^{-1}(t^+_{1n})=
   (\vv-\vv^{-1})t_1\cdot [E_{n-1},\cdots,[E_2,E_1]_{\vv^{-1}}\cdots]_{\vv^{-1}}$,
due to Corollary~\ref{matching higher roots}.
Combining all the above with~(\ref{Jing's identification}), we finally obtain
\begin{equation*}
  \Upsilon^{-1}(\ev^\rtt_a(\Upsilon(F_0)))=
  (-1)^{n}\vv^{n+1}a^{-1}\cdot t_1^2\cdot
  [E_{n-1},\cdots,[E_2,E_1]_{\vv^{-1}}\cdots]_{\vv^{-1}}\cdot t_1^{-1}t_n=
  \ev_a(F_0).
\end{equation*}

\medskip
This completes our proof of Theorem~\ref{compatibility of evaluations}.
\end{proof}

We will denote the evaluation homomorphism $\ev_1$ simply by $\ev$.


\subsection{Quantum minors of $T^\pm(z)$}\label{ssec quantum minors quantum}
\

We recall the notion of quantum minors of $T^\pm(z)$
following~\cite[$\S 1.15.6$]{m} and~\cite[Chapter~5]{h} (though a slight
change in our formulas is due to a different choice of the $R$-matrix).
For $1<r\leq n$, define $R(z_1,\ldots,z_r)\in (\End\ \BC^n)^{\otimes r}$ via
\begin{equation*}
  R(z_1,\ldots,z_r):=(R_{r-1,r})(R_{r-2,r}R_{r-2,r-1})\cdots (R_{1r}\cdots R_{12})\
  \mathrm{with}\ R_{ij}:=R_{\trig;ij}(z_i,z_j).
\end{equation*}
The following is implied by the Yang-Baxter equation~(\ref{qYB}) and~(\ref{all affRTT}):

\begin{Lem}\label{transfer}
  $R(z_1,\ldots,z_r)T^\pm_1(z_1)\cdots T^\pm_r(z_r)=
   T^\pm_r(z_r)\cdots T^\pm_1(z_1)R(z_1,\ldots,z_r)$.
\end{Lem}

Consider the $\vv$-permutation operator $P^\vv\in \End(\BC^n\otimes \BC^n)$ given by
\begin{equation*}
  P^\vv=\sum_i E_{ii}\otimes E_{ii}+\vv\sum_{i>j}E_{ij}\otimes E_{ji}+
  \vv^{-1}\sum_{i<j}E_{ij}\otimes E_{ji}.
\end{equation*}
It gives rise to the action of the symmetric group $\Sigma_r$ on $(\BC^n)^{\otimes r}$
with transpositions $(i,i+1)$ acting via $P^\vv_{i,i+1}$ (the operator $P^\vv$
acting on the $i$-th and $(i+1)$-st factors of $\BC^n$).
Define the $\vv$-antisymmetrizer $A^\vv_r\in(\End\ \BC^n)^{\otimes r}$ as
the image of the antisymmetrizer
$\sum_{\sigma\in \Sigma_r}(-1)^\sigma \cdot \sigma\in \BC[\Sigma_r]$
under this action of $\Sigma_r$ on $(\BC^n)^{\otimes r}$.
Recall the following classical observation (cf.~\cite[$\S 1.15.6$]{m} and~\cite[Lemma 5.5]{h}):

\begin{Thm}\label{key thm on qminors}
  $R(z,\vv^2z,\ldots,\vv^{2(r-1)}z)=
   \prod_{0\leq i< j\leq r-1} (\vv^{2i}-\vv^{2j})z^{\frac{r(r-1)}{2}} A^\vv_r$.
\end{Thm}

Combining Lemma~\ref{transfer} and Theorem~\ref{key thm on qminors},
we obtain the following

\begin{Cor}
We have
\begin{equation}\label{defining minors quantum}
  A^\vv_rT^\pm_1(z)T^\pm_2(\vv^{2}z)\cdots T^\pm_r(\vv^{2(r-1)}z)=
  T^\pm_r(\vv^{2(r-1)}z)\cdots T^\pm_2(\vv^{2}z)T^\pm_1(z)A^\vv_r.
\end{equation}
\end{Cor}

The operator of~(\ref{defining minors quantum}) can be written as
  $\sum t^{a_1\ldots a_r;\pm}_{b_1\ldots b_r}(z)
   \otimes E_{a_1,b_1}\otimes \cdots \otimes E_{a_r,b_r}$
with $t^{a_1\ldots a_r;\pm}_{b_1\ldots b_r}(z)\in \fU^\rtt_\vv(L\gl_n)[[z^{\mp 1}]]$
and the sum taken over all $a_1,\ldots,a_r,b_1,\ldots,b_r\in \{1,\ldots,n\}$.

\begin{Def}\label{quantum minor quantum}
The coefficients $t^{a_1\ldots a_r;\pm}_{b_1\ldots b_r}(z)$
are called the \emph{quantum minors} of $T^\pm(z)$.
\end{Def}

In the particular case $r=n$, the image of the operator $A^\vv_n$ acting on
$(\BC^n)^{\otimes n}$ is $1$-dimensional. Hence
  $A^\vv_nT^\pm_1(z)\cdots T^\pm_n(\vv^{2(n-1)}z)=A^\vv_n\cdot \qdet\ T^\pm(z)$
with $\qdet\ T^\pm(z)\in \fU^\rtt_\vv(L\gl_n)[[z^{\mp 1}]]$. We note that
$\qdet\ T^\pm(z)=t^{1\ldots n;\pm}_{1\ldots n}(z)$ in the above notations.

\begin{Def}\label{quantum determinant quantum}
$\qdet\ T^\pm(z)$ is called the \emph{quantum determinant} of $T^\pm(z)$.
\end{Def}

Define $d^\pm_{\pm r}\in \fU^\rtt_\vv(L\gl_n)$ via
$\qdet\ T^\pm(z)=\sum_{r\geq 0}d^\pm_{\pm r}z^{\mp r}$. The following result
is a trigonometric counterpart of Proposition~\ref{center of Yangian}:

\begin{Prop}\label{center of quantum affine}
The elements $\{d^\pm_{\pm r}\}_{r\geq 0}$ are central, subject to the only defining
relation $d^+_0d^-_0=1$, and generate the center $Z\fU^\rtt_\vv(L\gl_n)$ of
$\fU^\rtt_\vv(L\gl_n)$. In other words, we have a $\BC[\vv,\vv^{-1}]$-algebra isomorphism
$Z\fU^\rtt_\vv(L\gl_n)\simeq \BC[\vv,\vv^{-1}][\{d^\pm_{\pm r}\}_{r\geq 0}]/(d^+_0d^-_0-1)$.
\end{Prop}


\subsection{Enhanced algebras}\label{ssec enhanced algebras}
\

In this section, we slightly generalize the algebras of the previous subsections
as well as various relations between them.  This is needed mostly for our discussions
in subsection~\ref{ssec truncation ideal}.

\noindent
$\bullet$ Let $\fU^{\rtt,'}_\vv(\gl_n)$ be a $\BC[\vv,\vv^{-1}]$-algebra obtained
from $\fU^\rtt_\vv(\gl_n)$ by formally adjoining $n$-th roots of its central element
$t:=t^+_{11}\ldots t^+_{nn}=(t^-_{11}\ldots t^-_{nn})^{-1}$, that is,
$\fU^{\rtt,'}_\vv(\gl_n)=\fU^\rtt_\vv(\gl_n)[t^{\pm 1/n}]$. Its
$\BC(\vv)$-counterpart is denoted by $U^{\rtt,'}_\vv(\gl_n)$. Likewise, let
$U^{'}_\vv(\gl_n)$ be a $\BC(\vv)$-algebra obtained from $U_\vv(\gl_n)$ by formally
adjoining $n$-th roots of its central element $t:=t_1\ldots t_n$, that is,
$U^{'}_\vv(\gl_n)=U_\vv(\gl_n)[t^{\pm 1/n}]$. Then the isomorphism of
Theorem~\ref{Ding-Frenkel finite} gives rise to a $\BC(\vv)$-algebra
isomorphism $\Upsilon\colon U^{'}_\vv(\gl_n)\iso U^{\rtt,'}_\vv(\gl_n)$.

\noindent
$\bullet$ Let $\fU^{\rtt,'}_\vv(L\gl_n)$ be a $\BC[\vv,\vv^{-1}]$-algebra obtained
from $\fU^\rtt_\vv(L\gl_n)$ by formally adjoining $n$-th roots of its central element
$t[0]:=t^+_{11}[0]\ldots t^+_{nn}[0]=(t^-_{11}[0]\ldots t^-_{nn}[0])^{-1}$, that is,
$\fU^{\rtt,'}_\vv(L\gl_n)=\fU^\rtt_\vv(L\gl_n)[(t[0])^{\pm 1/n}]$. Its $\BC(\vv)$-counterpart
is denoted by $U^{\rtt,'}_\vv(L\gl_n)$. Likewise, let $U^{'}_\vv(L\gl_n)$ be a
$\BC(\vv)$-algebra obtained from $U_\vv(L\gl_n)$ by formally adjoining $n$-th roots
of its central element
  $\varphi:=\varphi^+_{1,0}\ldots \varphi^+_{n,0}=
   (\varphi^-_{1,0}\ldots \varphi^-_{n,0})^{-1}$,
that is, $U^{'}_\vv(L\gl_n)=U_\vv(L\gl_n)[\varphi^{\pm 1/n}]$. Then the
isomorphism of Theorem~\ref{Ding-Frenkel affine} gives rise to an algebra
isomorphism $\Upsilon\colon U^{'}_\vv(L\gl_n)\iso U^{\rtt,'}_\vv(L\gl_n)$.

\noindent
$\bullet$ Let $U^\ad_\vv(\ssl_n)$ be a $\BC(\vv)$-algebra obtained from $U_\vv(\ssl_n)$
by adding extra generators $\{\phi_i^{\pm 1}\}_{i=1}^{n-1}$ subject to
  $K_i=\prod_{j=1}^{n-1}\phi_j^{c_{ji}}, \phi_iE_j=\vv^{\delta_{ij}}E_j\phi_i,
   \phi_iF_j=\vv^{-\delta_{ij}}F_j\phi_i, \phi_i\phi_j=\phi_j\phi_i$.
Then, the natural embedding $U_\vv(\ssl_n)\hookrightarrow U_\vv(\gl_n)$ gives rise to a
$\BC(\vv)$-algebra embedding $U^\ad_\vv(\ssl_n)\hookrightarrow U^{'}_\vv(\gl_n)$
via $\phi_i\mapsto t_1^{-1}\ldots t_i^{-1}\cdot t^{i/n}$.

\noindent
$\bullet$ Likewise, let $U^\ad_\vv(L\ssl_n)$ be a $\BC(\vv)$-algebra obtained
from $U_\vv(L\ssl_n)$ by adding extra generators $\{\phi_i^{\pm 1}\}_{i=1}^{n-1}$
subject to
  $\psi^+_{i,0}=\prod_{j=1}^{n-1}\phi_j^{c_{ji}}, \phi_i \psi^\pm_j(z)=\psi^\pm_j(z)\phi_i,
  \phi_i e_j(z)=\vv^{\delta_{ij}}e_j(z)\phi_i, \phi_i f_j(z)=\vv^{-\delta_{ij}}f_j(z)\phi_i,
  \phi_i\phi_j=\phi_j\phi_i$.
Then, the natural embedding $U_\vv(L\ssl_n)\hookrightarrow U_\vv(L\gl_n)$ gives rise
to a $\BC(\vv)$-algebra embedding $U^\ad_\vv(L\ssl_n)\hookrightarrow U^{'}_\vv(L\gl_n)$
via $\phi_i\mapsto \varphi^-_{1,0}\ldots \varphi^-_{i,0}\cdot \varphi^{i/n}$.

\noindent
$\bullet$ The homomorphisms $\ev^\rtt,\ev$
of subsections~\ref{ssec RTT evaluation},~\ref{ssec Jimbo evaluation} extend to the
homomorphisms of the corresponding enhanced algebras, so that~(\ref{diagram quantum})
gives rise to the commutative diagram
\begin{equation}\label{diagram quantum adjoint}
  \begin{CD}
    U^\ad_\vv(L\ssl_n) @>\ev>> U^{'}_\vv(\gl_n)\\
    @VV{\Upsilon}V @V{\Upsilon}V\wr V\\
    U^{\rtt,'}_\vv(L\gl_n) @>\ev^\rtt>>  U^{\rtt,'}_\vv(\gl_n)
    \end{CD}
\end{equation}

\noindent
$\bullet$ Let $\fU^{\rtt}_\vv(L\ssl_n)$ (resp.\ $\fU^{\rtt,'}_\vv(L\ssl_n)$)
be the quotient of $\fU^\rtt_\vv(L\gl_n)$ (resp.\ $\fU^{\rtt,'}_\vv(L\gl_n)$)
by the relations $\qdet\ T^\pm(z)=1$ (resp.\ $\qdet\ T^\pm(z)=1, (t[0])^{1/n}=1$).
We denote its $\BC(\vv)$-counterpart by $U^{\rtt}_\vv(L\ssl_n)$
(resp.\ $U^{\rtt,'}_\vv(L\ssl_n)$). Clearly
  $\fU^{\rtt}_\vv(L\ssl_n)\simeq \fU^{\rtt,'}_\vv(L\ssl_n),\
   U^{\rtt}_\vv(L\ssl_n)\simeq U^{\rtt,'}_\vv(L\ssl_n)$.

\noindent
$\bullet$ The composition
\begin{equation}\label{isom of qsl and rtt-qsl}
  U^\ad_\vv(L\ssl_n)\hookrightarrow U^{'}_\vv(L\gl_n)
  \iso U^{\rtt,'}_\vv(L\gl_n)\twoheadrightarrow U^{\rtt,'}_\vv(L\ssl_n)
\end{equation}
is a $\BC(\vv)$-algebra isomorphism.

\noindent
$\bullet$
Analogously to Definition~\ref{integral loop}, let $\fU^\ad_\vv(L\ssl_n)$ be the
$\BC[\vv,\vv^{-1}]$-subalgebra of $U^\ad_\vv(L\ssl_n)$ generated by
  $\{E^{(r)}_{j,i+1},F^{(r)}_{i+1,j}\}_{1\leq j\leq i<n}^{r\in \BZ}\cup
   \{\psi^\pm_{i,\pm s}\}_{1\leq i<n}^{s>0}\cup \{\phi_i^{\pm 1}\}_{i=1}^{n-1}$.
Then the $\BC(\vv)$-algebra isomorphism~(\ref{isom of qsl and rtt-qsl}) gives
rise to a $\BC[\vv,\vv^{-1}]$-algebra isomorphism
\begin{equation}\label{isom integral of qsl and rtt-qsl}
  \fU^\ad_\vv(L\ssl_n)\iso \fU^{\rtt,'}_\vv(L\ssl_n).
\end{equation}

\noindent
$\bullet$ Define the generating series
  $\varphi^\pm(z)=\varphi^\pm+\sum_{r\geq 1}\varphi_{\pm r}z^{\mp r}$
with coefficients in the algebra $U_\vv(L\gl_n)$ (or $U^{'}_\vv(L\gl_n)$)
via $\varphi^\pm(z):=\prod_{i=1}^n \varphi^\pm_i(\vv^i z)$
(so that $\varphi^\pm=\varphi^{\pm 1}$). It is straightforward to check
that all $\varphi_{\pm r}$ are central elements of $U_\vv(L\gl_n)$
(or $U^{'}_\vv(L\gl_n)$). Moreover, it is known that the center
$ZU^{'}_\vv(L\gl_n)$ of $U^{'}_\vv(L\gl_n)$ is a polynomial algebra in
$\{\varphi_{\pm r}, \varphi^{\pm 1/n}\}_{r\geq 1}$ and
\begin{equation*}
  U^{'}_\vv(L\gl_n)\simeq U^\ad_\vv(L\ssl_n)\otimes_{\BC(\vv)} ZU^{'}_\vv(L\gl_n).
\end{equation*}
The latter in turn gives rise to a trigonometric counterpart of~(\ref{RTT gln vs sln}):
\begin{equation}\label{qaffine gln decomposition}
  U^{\rtt,'}_\vv(L\gl_n)\simeq U^{\rtt,'}_\vv(L\ssl_n)\otimes_{\BC(\vv)} ZU^{\rtt,'}_\vv(L\gl_n),
\end{equation}
where $U^{\rtt,'}_\vv(L\ssl_n)$ is viewed as a subalgebra of $U^{\rtt,'}_\vv(L\gl_n)$
(rather than a quotient) via~(\ref{isom of qsl and rtt-qsl}).


\section{$K$-theoretic Coulomb branch of type $A$ quiver gauge theory}\label{sec K-theoretic Coulomb branch}


\subsection{Homomorphism $\wt{\Phi}{}^\lambda_\mu$}\label{ssec hom to diff op-s}
\

Let us recall the construction of~\cite[\S7]{ft} for the type $A_{n-1}$
Dynkin diagram with arrows pointing $i\to i+1$ for $1\leq i\leq n-2$. We use
the same notations $\lambda,\mu,\unl{\lambda}, N, a_i$ as in
subsection~\ref{ssec hom to diff op-s yangian} (in particular, we set $a_0:=0, a_n:=0$).

Consider the associative $\BC[\vv,\vv^{-1}]$-algebra $\hat{\CA}^\vv$ generated by
  $\{D_{i,r}^{\pm 1}, \sw_{i,r}^{\pm 1/2}\}_{1\leq i\leq n-1}^{1\leq r\leq a_i}$
such that $D_{i,r}\sw^{1/2}_{i,r}=\vv\sw^{1/2}_{i,r}D_{i,r}$, while all other
generators pairwise commute. Let $\wt{\CA}^\vv$ be the localization of
$\hat{\CA}^\vv$ by the multiplicative set generated by
  $\{\sw_{i,r}-\vv^m\sw_{i,s}\}_{1\leq i<n, m\in \BZ}^{1\leq r\ne s\leq a_i}
   \cup\{1-\vv^m\}_{m\in\BZ\setminus\{0\}}$.
We define their $\BC(\vv)$-counterparts
  $\hat{\CA}^\vv_\fra:=\hat{\CA}^\vv\otimes_{\BC[\vv,\vv^{-1}]} \BC(\vv)$
and
  $\wt{\CA}^\vv_\fra:=\wt{\CA}^\vv\otimes_{\BC[\vv,\vv^{-1}]} \BC(\vv)$.
We also need the larger algebras
  $\wt{\CA}^\vv[\sz^{\pm 1}_1,\ldots,\sz^{\pm 1}_N]:=
   \wt{\CA}^\vv\otimes_{\BC[\vv,\vv^{-1}]} \BC[\vv,\vv^{-1}][\sz^{\pm 1}_1,\ldots,\sz^{\pm 1}_N]$
and
  $\wt{\CA}^\vv_\fra[\sz^{\pm 1}_1,\ldots,\sz^{\pm 1}_N]:=
   \wt{\CA}^\vv_\fra\otimes_{\BC(\vv)} \BC(\vv)[\sz^{\pm 1}_1,\ldots,\sz^{\pm 1}_N]$.
Define $\sW_0(z):=1, \sW_n(z):=1$, and
\begin{equation*}
\begin{split}
    &  \sZ_i(z):=\prod_{1\leq s\leq N}^{i_s=i} \left(1-\frac{\vv\sz_s}{z}\right),\
       \sW_i(z):=\prod_{r=1}^{a_i} \left(1-\frac{\sw_{i,r}}{z}\right),\
       \sW_{i,r}(z):=\prod_{1\leq s\leq a_i}^{s\ne r} \left(1-\frac{\sw_{i,s}}{z}\right).
\end{split}
\end{equation*}

To state~\cite[Theorem 7.1]{ft}, we need the following modifications
of $U_\vv(L\ssl_n)$. First, recall the \emph{simply-connected version
of shifted quantum affine algebra} $U^{\ssc,\mu}_{\vv}$ introduced
in~\cite[$\S$5(i)]{ft}, which is a $\BC(\vv)$-algebra generated by
  $\{e_{i,r},f_{i,r},\psi^+_{i,s^+_i}, \psi^-_{i,-s^-_i},
    (\psi^+_{i,0})^{-1}, (\psi^-_{i,b_i})^{-1}\}_{1\leq i\leq n-1}^{r\in \BZ, s^+_i\geq 0, s^-_i\geq -b_i}$,
where $b_i=\alphavee_i(\mu)$ as in subsection~\ref{ssec shifted Yangian} with
$\{\alphavee_i\}_{i=1}^{n-1}$ denoting the simple positive roots of $\ssl_n$.
Finally, we define $U^{\ad,\mu}_{\vv}[\sz^{\pm 1}_1,\ldots,\sz^{\pm 1}_N]$ as a
$\BC(\vv)[\sz^{\pm 1}_1,\ldots,\sz^{\pm 1}_N]$-algebra obtained from
  $U^{\ssc,\mu}_\vv[\sz^{\pm 1}_1,\ldots,\sz^{\pm 1}_N]:=
   U^{\ssc,\mu}_\vv\otimes_{\BC(\vv)} \BC(\vv)[\sz^{\pm 1}_1,\ldots,\sz^{\pm 1}_N]$
by adding generators
$\{(\phi^+_i)^{\pm 1},(\phi^-_i)^{\pm 1}\}_{i=1}^{n-1}$
subject to the following extra relations:
\begin{equation}\label{adjoint type}
\begin{split}
  & \psi^+_{i,0}=(\phi^+_i)^2\cdot \prod_{j - i}(\phi^+_j)^{-1},\
    (-\vv)^{-b_i}\prod_{1\leq s\leq N}^{i_s=i} \sz_s^{-1}\cdot \psi^-_{i,b_i}=
    (\phi^-_i)^2\cdot \prod_{j - i}(\phi^-_j)^{-1},\\
  & [\phi^\epsilon_i,\phi^{\epsilon'}_{i'}]=0,\
    \phi^\epsilon_i \psi^{\epsilon'}_{i'}(z)=\psi^{\epsilon'}_{i'}(z)\phi^\epsilon_i,\
    \phi^\epsilon_i e_{i'}(z)=\vv^{\epsilon\delta_{ii'}}e_{i'}(z)\phi^\epsilon_i,\
    \phi^\epsilon_i f_{i'}(z)=\vv^{-\epsilon\delta_{ii'}}f_{i'}(z)\phi^\epsilon_i
\end{split}
\end{equation}
for any $1\leq i,i'\leq n-1$ and $\epsilon,\epsilon'\in \{\pm\}$.

\begin{Thm}[\cite{ft}]\label{Homomorphism}
There exists a unique $\BC(\vv)[\sz^{\pm 1}_1,\ldots,\sz^{\pm 1}_N]$-algebra homomorphism
\begin{equation*}
    \wt{\Phi}^{\unl\lambda}_\mu\colon
    U^{\ad,\mu}_\vv[\sz^{\pm 1}_1,\ldots,\sz^{\pm 1}_N]\longrightarrow
    \wt{\CA}^\vv_\fra[\sz^{\pm 1}_1,\ldots,\sz^{\pm 1}_N],
\end{equation*}
such that
\begin{equation*}
\begin{split}
 & e_i(z)\mapsto \frac{1}{\vv-\vv^{-1}}
   \prod_{t=1}^{a_i}\sw_{i,t} \prod_{t=1}^{a_{i-1}} \sw_{i-1,t}^{-1/2}\cdot
   \sum_{r=1}^{a_i} \delta\left(\frac{\sw_{i,r}}{z}\right)\frac{\sZ_i(\sw_{i,r})}{\sW_{i,r}(\sw_{i,r})}
   \sW_{i-1}(\vv^{-1}\sw_{i,r})D_{i,r}^{-1},\\
 & f_i(z)\mapsto \frac{1}{1-\vv^2}\prod_{t=1}^{a_{i+1}} \sw_{i+1,t}^{-1/2}\cdot
   \sum_{r=1}^{a_i} \delta\left(\frac{\vv^2\sw_{i,r}}{z}\right)\frac{1}{\sW_{i,r}(\sw_{i,r})}
   \sW_{i+1}(\vv\sw_{i,r})D_{i,r},\\
 & \psi^\pm_i(z)\mapsto \prod_{t=1}^{a_i}\sw_{i,t}
   \prod_{t=1}^{a_{i-1}} \sw_{i-1,t}^{-1/2}\prod_{t=1}^{a_{i+1}} \sw_{i+1,t}^{-1/2}\cdot
   \left(\sZ_i(z)\frac{\sW_{i-1}(\vv^{-1}z)\sW_{i+1}(\vv^{-1}z)}{\sW_i(z)\sW_i(\vv^{-2}z)}\right)^\pm,\\
 & (\phi^+_i)^{\pm 1}\mapsto \prod_{t=1}^{a_i} \sw_{i,t}^{\pm 1/2},\
   (\phi^-_i)^{\pm 1}\mapsto \prod_{t=1}^{a_i} \sw_{i,t}^{\mp 1/2}.
\end{split}
\end{equation*}
We write $\gamma(z)^\pm$ for the expansion of a rational function
$\gamma(z)$ in $z^{\mp 1}$, respectively.
\end{Thm}

\begin{Rem}
We note that the algebras $U^{\ssc,\mu}_{\vv}$ and
$U^{\ad,\mu}_\vv[\sz^{\pm 1}_1,\ldots,\sz^{\pm 1}_N]$ were denoted by
$\CU^\ssc_{0,\mu}$ and $\CU^{\ad}_{0,\mu}[\sz^{\pm 1}_1,\ldots,\sz^{\pm 1}_N]$
in~\cite{ft}. Moreover, we used a slightly different renormalization of
$\phi^-_i$ in~\emph{loc.cit.}
\end{Rem}

In analogy with Definition~\ref{integral loop}, let us introduce
integral forms of the shifted quantum affine algebras $U^{\ssc,\mu}_{\vv}$
and $U^{\ad,\mu}_{\vv}[\sz^{\pm 1}_1,\ldots,\sz^{\pm 1}_N]$.

\begin{Def}\label{integral shifted}
(a) Let $\fU^{\ssc,\mu}_{\vv}$ be the $\BC[\vv,\vv^{-1}]$-subalgebra
of $U^{\ssc,\mu}_{\vv}$ generated by
\begin{equation}\label{PBW basis shifted sc}
   \{E^{(r)}_{j,i+1}, F^{(r)}_{i+1,j}\}_{1\leq j\leq i<n}^{r\in \BZ}\cup
   \{\psi^+_{i,s^+_i}, \psi^-_{i,-s^-_i},
    (\psi^+_{i,0})^{-1}, (\psi^-_{i,b_i})^{-1}\}_{1\leq i\leq n-1}^{r\in \BZ, s^+_i\geq 0, s^-_i\geq -b_i}.
\end{equation}

\noindent
(b) Let $\fU^{\ad,\mu}_\vv[\sz^{\pm 1}_1,\ldots,\sz^{\pm 1}_N]$ be the
$\BC[\vv,\vv^{-1}][\sz^{\pm 1}_1,\ldots,\sz^{\pm 1}_N]$-subalgebra of
$U^{\ad,\mu}_\vv[\sz^{\pm 1}_1,\ldots,\sz^{\pm 1}_N]$ generated by
\begin{equation}\label{PBW basis shifted ad}
  \{E^{(r)}_{j,i+1}, F^{(r)}_{i+1,j}\}_{1\leq j\leq i<n}^{r\in \BZ}\cup
  \{\psi^+_{i,s^+_i},\psi^-_{i,-s^-_i}\}_{1\leq i\leq n-1}^{s^+_i> 0, s^-_i> -b_i}\cup
  \{(\phi^+_i)^{\pm 1},(\phi^-_i)^{\pm 1}\}_{i=1}^{n-1}.
\end{equation}
\end{Def}

Here the elements $E^{(r)}_{j,i+1}, F^{(r)}_{i+1,j}$ are defined via~(\ref{PBW basis elements}).
Recall the total orderings on the collections $\{E^{(r)}_{j,i+1}\}_{1\leq j\leq i<n}^{r\in \BZ}$
and $\{F^{(r)}_{i+1,j}\}_{1\leq j\leq i<n}^{r\in \BZ}$ which were introduced right
before Theorem~\ref{PBW basis coordinate affine}, and choose any total ordering
on the corresponding Cartan generators. We introduce the \emph{ordered PBWD monomials}
(in the corresponding generators) accordingly. The following result generalizes
Theorem~\ref{PBW basis coordinate affine} to the shifted setting.

\begin{Thm}\label{PBW for integral shifted}
(a) The ordered PBWD monomials in the elements~(\ref{PBW basis shifted sc})
form a basis of a free $\BC[\vv,\vv^{-1}]$-module $\fU^{\ssc,\mu}_{\vv}$.

\noindent
(b) The ordered PBWD monomials in the elements~(\ref{PBW basis shifted ad})
form a basis of a free

\noindent
$\BC[\vv,\vv^{-1}][\sz^{\pm 1}_1,\ldots,\sz^{\pm 1}_N]$-module
$\fU^{\ad,\mu}_\vv[\sz^{\pm 1}_1,\ldots,\sz^{\pm 1}_N]$.
\end{Thm}

\begin{proof}
We will provide the proof only of part (a), since part (b) is proved analogously.

Following~\cite[$\S$5(i)]{ft}, consider the $\BC(\vv)$-subalgebras $U^{\ssc,\mu;>}_{\vv}$
and $U^{\ssc,\mu;<}_{\vv}$ of $U^{\ssc,\mu}_{\vv}$ generated by
$\{e_{i,r}\}_{1\leq i\leq n-1}^{r\in \BZ}$ and
$\{f_{i,r}\}_{1\leq i\leq n-1}^{r\in \BZ}$, respectively, and let
$U^{\ssc,\mu;0}_{\vv}$ be the $\BC(\vv)$-subalgebra of $U^{\ssc,\mu}_{\vv}$
generated by the Cartan generators. According
to~\cite[Proposition 5.1]{ft}, the multiplication map
  $m\colon U^{\ssc,\mu;<}_{\vv}\otimes U^{\ssc,\mu;0}_{\vv}\otimes U^{\ssc,\mu;>}_{\vv}
   \to U^{\ssc,\mu}_{\vv}$
is an isomorphism of $\BC(\vv)$-vector spaces, and the subalgebras
$U^{\ssc,\mu;<}_{\vv}, U^{\ssc,\mu;>}_{\vv}$ are isomorphic to
$U^<_\vv(L\ssl_n)\simeq U^<_\vv(L\gl_n), U^>_\vv(L\ssl_n)\simeq U^>_\vv(L\gl_n)$,
respectively. Combining this with Theorem~\ref{PBW for half-integral}(b,d),
we immediately see that the ordered PBWD monomials in the
elements~(\ref{PBW basis shifted sc}) form a basis of a
$\BC(\vv)$-vector space $U^{\ssc,\mu}_{\vv}$.

Therefore, as noted in the very beginning of our proof of
Theorem~\ref{PBW basis coordinate affine}, it suffices to verify
that all unordered products
  $E^{(r)}_{j,i+1}\psi^\pm_{j',\pm s},
   \psi^\pm_{j',\pm s}F^{(r)}_{i+1,j},
   E^{(r)}_{j,i+1}F^{(s)}_{i'+1,j'}$
are equal to $\BC[\vv,\vv^{-1}]$-linear combinations of the ordered PBWD
monomials. The first two cases are treated exactly as in our proof of
Theorem~\ref{PBW basis coordinate affine}. Hence, it remains to prove
the following result:

\begin{Prop}\label{normal reordering}
All unordered products $E^{(r)}_{j,i+1}F^{(s)}_{i'+1,j'}$
are equal to $\BC[\vv,\vv^{-1}]$-linear combinations of the ordered PBWD
monomials in the algebra $U^{\ssc,\mu}_{\vv}$.
\end{Prop}

The proof of Proposition~\ref{normal reordering} proceeds
in four steps and is reminiscent of~\cite[Appendix E]{ft}.

\medskip
\noindent
\emph{Step 1:} Case $\mu=0$.

The fact that $E^{(r)}_{j,i+1}F^{(s)}_{i'+1,j'}$ equals a $\BC[\vv,\vv^{-1}]$-linear
combination of the ordered PBWD monomials in $U^{\ssc,0}_\vv$ follows essentially
from Theorem~\ref{PBW basis coordinate affine}. To be more precise, recall the
``extended'' algebra $\fU^{\rtt,\ext}_\vv(L\gl_n)$ of~\cite[(2.15)]{gm}:
it is defined similarly to $\fU^{\rtt}_\vv(L\gl_n)$, but we add extra generators
$\{(t^\pm_{ii}[0])^{-1}\}_{i=1}^n$ and replace the first defining relation
of~(\ref{affRTT}) by
\begin{equation*}
  t^+_{ii}[0]t^-_{ii}[0]=t^-_{ii}[0]t^+_{ii}[0],\
  t^\pm_{ii}[0](t^\pm_{ii}[0])^{-1}=(t^\pm_{ii}[0])^{-1}t^\pm_{ii}[0]=1.
\end{equation*}
Set
  $U^{\rtt,\ext}_\vv(L\gl_n):=
   \fU^{\rtt,\ext}_\vv(L\gl_n)\otimes_{\BC[\vv,\vv^{-1}]} \BC(\vv)$.
Likewise, let $U^{\ssc,0}_\vv(L\gl_n)$ be a $\BC(\vv)$-algebra obtained
from $U_\vv(L\gl_n)$ by formally adding generators $(\varphi^\pm_{j,0})^{-1}$
and ignoring $\varphi^\pm_{j,0}\varphi^\mp_{j,0}=1$. Then, the isomorphism
$\Upsilon$ of Theorem~\ref{Ding-Frenkel affine} gives rise to the
$\BC(\vv)$-algebra isomorphism
\begin{equation*}
  \Upsilon^\ext\colon U^{\ssc,0}_\vv(L\gl_n)\iso U^{\rtt,\ext}_\vv(L\gl_n).
\end{equation*}
Hence, the arguments from our proof of Theorem~\ref{PBW basis coordinate affine}
can be applied without any changes to prove Proposition~\ref{normal reordering}
for $\mu=0$.

\medskip
\noindent
\emph{Step 2:} Reduction to $\breve{U}^{\ssc,\mu}_\vv$.

Consider the associative $\BC(\vv)$-algebra $\breve{U}^{\ssc,\mu}_\vv$
(resp.\ its $\BC[\vv,\vv^{-1}]$-subalgebra $\breve{\fU}^{\ssc,\mu}_\vv$),
defined in the same way as $U^{\ssc,\mu}_\vv$ (resp.\ as $\fU^{\ssc,\mu}_\vv$)
but without the generators $\{(\psi^+_{i,0})^{-1}, (\psi^-_{i,b_i})^{-1}\}_{i=1}^{n-1}$,
so that $U^{\ssc,\mu}_\vv$ is the localization of $\breve{U}^{\ssc,\mu}_\vv$ by
the multiplicative set $S$ generated by $\{\psi^+_{i,0}, \psi^-_{i,b_i}\}_{i=1}^{n-1}$.
Hence, Proposition~\ref{normal reordering} follows from its counterpart for $\breve{U}^{\ssc,\mu}_\vv$:

\begin{Prop}\label{normal reordering breve}
All unordered products $E^{(r)}_{j,i+1}F^{(s)}_{i'+1,j'}$
are equal to $\BC[\vv,\vv^{-1}]$-linear combinations of the ordered PBWD
monomials in the algebra $\breve{U}^{\ssc,\mu}_\vv$
\end{Prop}

We define the $\BC(\vv)$-subalgebras
  $\breve{U}^{\ssc,\mu;>}_{\vv}, \breve{U}^{\ssc,\mu;<}_{\vv}, \breve{U}^{\ssc,\mu;0}_{\vv}$
of $\breve{U}^{\ssc,\mu}_{\vv}$ accordingly. Analogously to~\cite[Proposition 5.1]{ft},
the multiplication map
  $m\colon \breve{U}^{\ssc,\mu;<}_{\vv}\otimes \breve{U}^{\ssc,\mu;0}_{\vv}\otimes \breve{U}^{\ssc,\mu;>}_{\vv}
   \to \breve{U}^{\ssc,\mu}_{\vv}$
is an isomorphism of $\BC(\vv)$-vector spaces, and the subalgebras
$\breve{U}^{\ssc,\mu;<}_{\vv}, \breve{U}^{\ssc,\mu;>}_{\vv}$ are isomorphic to
$U^<_\vv(L\ssl_n)\simeq U^<_\vv(L\gl_n), U^>_\vv(L\ssl_n)\simeq U^>_\vv(L\gl_n)$,
respectively. Combining this with Theorem~\ref{PBW for half-integral}(b,d),
we see that the ordered PBWD monomials form a basis of a $\BC(\vv)$-vector
space $\breve{U}^{\ssc,\mu}_{\vv}$. The following result generalizes the key
verification in our proof of Theorem~\ref{PBW basis coordinate affine}:

\begin{Lem}\label{auxiliary algebra}
Proposition~\ref{normal reordering breve} holds for $\mu=0$.
\end{Lem}

\begin{proof}
According to Step 1, $E^{(r)}_{j,i+1}F^{(s)}_{i'+1,j'}\in U^{\ssc,0}_\vv$ equals
a $\BC[\vv,\vv^{-1}]$-linear combination of the ordered PBWD monomials in $U^{\ssc,0}_\vv$.
Hence, it suffices to show that none of these ordered monomials contains negative powers
of either $\psi^+_{i,0}$ or $\psi^-_{i,0}$. Assume the contrary.
For $1\leq i<n$ and $\epsilon\in \{\pm\}$, choose $N^\epsilon_i\in \BN$ so that $-N^\epsilon_i$
is the minimal of the negative powers of $\psi^\epsilon_{i,0}$ among the corresponding summands.
Without loss of generality, we may assume that $N^-_1>0$. Set
$\psi:=\prod_{i=1}^{n-1} \left((\psi^+_{i,0})^{N^+_i} (\psi^-_{i,0})^{N^-_i}\right)\in S$.
Multiplying the equality in $U^{\ssc,0}_\vv$ expressing $E^{(r)}_{j,i+1}F^{(s)}_{i'+1,j'}$
as a $\BC[\vv,\vv^{-1}]$-linear combination of the ordered PBWD monomials by $\psi$, we obtain
an equality in $\breve{U}^{\ssc,0}_\vv$. Specializing further $\psi^-_{1,0}$ to $0$, gives rise to
an equality in $\breve{U}^{\ssc,-\omega_1}_\vv$ (as before, $\omega_1$ denotes the first fundamental coweight).
As $N^-_1>0$, the left-hand side specializes to zero. Meanwhile, every summand of the right-hand side
specializes either to zero or to an ordered PBWD monomial in $\breve{U}^{\ssc,-\omega_1}_\vv$.
Note that there is at least one summand which does not specialize to zero, and the images of
all those are pairwise distinct ordered PBWD monomials. This contradicts the fact
(pointed out right before Lemma~\ref{auxiliary algebra}) that the ordered
PBWD monomials form a basis of a $\BC(\vv)$-vector space $\breve{U}^{\ssc,-\omega_1}_{\vv}$.
Hence, the contradiction.

This completes our proof of Lemma~\ref{auxiliary algebra}.
\end{proof}

\medskip
\noindent
\emph{Step 3:} Case of antidominant $\mu$.

For an antidominant $\mu$, consider a $\BC(\vv)$-algebra epimorphism
$\pi_\mu\colon \breve{U}^{\ssc,0}_\vv\twoheadrightarrow \breve{U}^{\ssc,\mu}_\vv$
defined by
\begin{equation*}
  e_{i,r}\mapsto e_{i,r},\ f_{i,r}\mapsto f_{i,r},\
  \psi^+_{i,s}\mapsto \psi^+_{i,s},\
  \psi^-_{i,-s}\mapsto
  \begin{cases}
     \psi^-_{i,-s}, & \text{if } s\geq -b_i\\
     0, & \text{if } \mathrm{otherwise}
   \end{cases}
  \ \mathrm{for}\ 1\leq i<n, r\in \BZ, s\in \BN.
\end{equation*}
Using Lemma~\ref{auxiliary algebra}, let us express $E^{(r)}_{j,i+1}F^{(s)}_{i'+1,j'}$ as a
$\BC[\vv,\vv^{-1}]$-linear combination of the ordered PBWD monomials in
$\breve{U}^{\ssc,0}_\vv$, and apply $\pi_\mu$ to the resulting equality in $\breve{U}^{\ssc,0}_\vv$.
Since
  $\pi_\mu(E^{(r)}_{j,i+1}F^{(s)}_{i'+1,j'})=E^{(r)}_{j,i+1}F^{(s)}_{i'+1,j'}$
and $\pi_\mu$ maps ordered PBWD monomial in $\breve{U}^{\ssc,0}_\vv$
either to the ordered PBWD monomial in $\breve{U}^{\ssc,\mu}_\vv$ or to zero,
we see that Proposition~\ref{normal reordering breve} holds for
antidominant $\mu$.

\medskip
\noindent
\emph{Step 4:} General case.

Since Proposition~\ref{normal reordering breve} holds for antidominant $\mu$ (Step 3)
and any coweight can be written as a sum of an antidominant coweight and several
fundamental coweights $\omega_\ell$, it suffices to prove the following result:

\begin{Lem}\label{induction for pbw}
If Proposition~\ref{normal reordering breve} holds for a coweight $\mu$,
then it also holds for the coweights $\mu+\omega_\ell\ (1\leq \ell\leq n-1)$.
\end{Lem}

\begin{proof}
Recall the \emph{shift homomorphism}
  $\breve{\iota}_{\mu+\omega_\ell,-\omega_\ell,0}\colon
   \breve{U}^{\ssc,\mu+\omega_\ell}_\vv\to \breve{U}^{\ssc,\mu}_\vv$
(cf.~\cite[Lemma 10.24, Appendix E]{ft}) defined explicitly via
\begin{equation*}
    e_{i,r}\mapsto e_{i,r}-\delta_{i,\ell}e_{i,r-1},\ f_{i,r}\mapsto f_{i,r},\
    \psi^+_{i,s}\mapsto \psi^+_{i,s}-\delta_{i,\ell}\psi^+_{i,s-1},\
    \psi^-_{i,s}\mapsto \psi^-_{i,s}-\delta_{i,\ell}\psi^-_{i,s-1},
\end{equation*}
where we set $\psi^+_{\ell,-1}:=0$ and  $\psi^-_{\ell,b_\ell+1}:=0$ in the right-hand sides.

First, we note that
  $\breve{\iota}_{\mu+\omega_\ell,-\omega_\ell,0}(\breve{\fU}^{\ssc,\mu+\omega_\ell}_\vv)
   \subset \breve{\fU}^{\ssc,\mu}_\vv$.
Indeed, $F^{(r)}_{i+1,j}$ is clearly fixed by $\breve{\iota}_{\mu+\omega_\ell,-\omega_\ell,0}$,
while $E^{(r)}_{j,i+1}$ is either fixed (if $\ell<j$ or $\ell>i$) or is mapped to
$E^{(r)}_{j,i+1}-E_{j,i+1}(\unl{r-1})$ for a certain decomposition of $r-1$
(cf.~formula~(\ref{generalized PBW basis elements}) and the discussion preceding it),
and is therefore still an element of $\breve{\fU}^{\ssc,\mu}_\vv$, due to
Theorem~\ref{PBW for half-integral}(a). Hence, applying our assumption
to $\breve{U}^{\ssc,\mu}_\vv$, we see that
  $\breve{\iota}_{\mu+\omega_\ell,-\omega_\ell,0}(E^{(r)}_{j,i+1}F^{(s)}_{i'+1,j'})$
equals a $\BC[\vv,\vv^{-1}]$-linear combination of the ordered PBWD monomials
in $\breve{U}^{\ssc,\mu}_\vv$. On the other hand, let us write
$E^{(r)}_{j,i+1}F^{(s)}_{i'+1,j'}$ as a $\BC(\vv)$-linear combination of the
ordered PBWD monomials in $\breve{U}^{\ssc,\mu+\omega_\ell}_\vv$
(such a presentation exists and is unique as the ordered PBWD monomials form
a basis of a $\BC(\vv)$-vector space $\breve{U}^{\ssc,\mu+\omega_\ell}_\vv$):
\begin{equation}\label{general form}
  E^{(r)}_{j,i+1}F^{(s)}_{i'+1,j'}=
  \sum_{\unl{\alpha},\unl{\beta}^+,\unl{\beta}^-}
  F_{\unl{\alpha}} \psi^+_{\unl{\beta}^+} \psi^-_{\unl{\beta}^-} E(\unl{\alpha},\unl{\beta}^+,\unl{\beta}^-),
\end{equation}
where $F_{\unl{\alpha}}, \psi^+_{\unl{\beta}^+}, \psi^-_{\unl{\beta}^-}$ range
over all ordered monomials in
  $\{F^{(\bullet)}_{\bullet,\bullet}\},\
   \{\psi^+_{\bullet,\bullet}\},\
   \{\psi^-_{\bullet,\bullet}\}$,
respectively, while $E(\unl{\alpha},\unl{\beta}^+,\unl{\beta}^-)$ are elements
of $\breve{U}^{\ssc,\mu+\omega_\ell;>}_\vv$ and only finitely many of them are nonzero.
From now on, we identify
  $\breve{U}^{\ssc,\mu+\omega_\ell;>}_\vv\simeq
   U^>_\vv(L\ssl_n)\simeq \breve{U}^{\ssc,\mu;>}_\vv,\
   \breve{\fU}^{\ssc,\mu+\omega_\ell;>}_\vv\simeq
   \fU^>_\vv(L\ssl_n)\simeq \breve{\fU}^{\ssc,\mu;>}_\vv$.
Thus, it remains to verify the inclusions
\begin{equation}\label{key integral inclusion}
   E(\unl{\alpha},\unl{\beta}^+,\unl{\beta}^-)\in \fU^{>}_\vv(L\ssl_n)
   \ \mathrm{for\ all}\ \ \unl{\alpha},\unl{\beta}^+,\unl{\beta}^-.
\end{equation}

The proof of~(\ref{key integral inclusion}) utilizes the shuffle interpretations
of both the subalgebras $U^{>}_\vv(L\ssl_n)$, $\fU^{>}_\vv(L\ssl_n)$ and the restriction
of the shift homomorphism
  $\breve{\iota}_{\mu+\omega_\ell,-\omega_\ell,0}\colon
   \breve{U}^{\ssc,\mu+\omega_\ell;>}_\vv\to \breve{U}^{\ssc,\mu;>}_\vv$.
Recall the $\BC(\vv)$-algebra isomorphism $\Psi\colon U^>_\vv(L\ssl_n)\iso S^{(n)}$
of Theorem~\ref{shuffle rational}, which gives rise to a $\BC[\vv,\vv^{-1}]$-algebra isomorphism
$\Psi\colon \fU_\vv^{>}(L\gl_n)\iso \fS^{(n)}$, see Theorem~\ref{shuffle integral}.
By the above discussion, applying
$\breve{\iota}_{\mu+\omega_\ell,-\omega_\ell,0}$ to the right-hand side of~(\ref{general form}),
we get a $\BC[\vv,\vv^{-1}]$-linear combination of the ordered PBWD monomials.
Recall that $\breve{\iota}_{\mu+\omega_\ell,-\omega_\ell,0}$ fixes all $F_{\unl{\alpha}}$,
maps $\psi^+_{\unl{\beta}^+}$ to itself plus some smaller terms
(wrt the ordering) and maps $\psi^-_{\unl{\beta}^-}$ to itself
(with the indices of $\psi^-_{\ell,\bullet}$ shifted by $-1$) plus some smaller terms
(wrt the ordering). Furthermore, according to~\cite[Proposition I.4]{ft}, the homomorphism
  $\breve{\iota}_{\mu+\omega_\ell,-\omega_\ell,0}\colon
   \breve{U}^{\ssc,\mu+\omega_\ell;>}_\vv\to \breve{U}^{\ssc,\mu;>}_\vv$
is intertwined (under the above identifications of
$\breve{U}^{\ssc,\mu+\omega_\ell;>}_\vv, \breve{U}^{\ssc,\mu;>}_\vv$ with
$U^>_\vv(L\ssl_n)\simeq S^{(n)}$) with the graded $\BC(\vv)$-algebra homomorphism
$\iota'_\ell\colon S^{(n)}\to S^{(n)}$ of~(\ref{shuffle shift homomorphism}).
According to Proposition~\ref{key properties of integral shuffle}(a),
$f\in \fS^{(n)}$ if and only if $\iota'_\ell(f)\in \fS^{(n)}$.
Hence, a simple inductive argument (for every $\unl{\alpha}$, we use
a descending induction in $\unl{\beta}^+$, and then a descending induction
in $\unl{\beta}^-$) implies~(\ref{key integral inclusion}).

This implies the validity of Proposition~\ref{normal reordering breve} for
the coweight $\mu+\omega_\ell\ (1\leq \ell\leq n-1)$.
\end{proof}

This completes our proof of Theorem~\ref{PBW for integral shifted}.
\end{proof}


\subsection{$K$-theoretic Coulomb branch}\label{ssec Coulomb branch}
\

Following~\cite{bfna,bfn} and using our notations of subsection~\ref{ssec Coulomb branch cohomology},
consider the (extended) quantized $K$-theoretic Coulomb
branch
  $\CA^\vv=K^{(\wt{\GL}(V)\times T_W)_\CO\rtimes\wt\BC^\times}(\CR_{\GL(V),\bN})$.
Here $\wt{\GL}(V)$ is a certain $2^{n-1}$-cover of $\GL(V)$ and
$\wt\BC^\times$ is a two-fold cover of $\BC^\times$, as defined in~\cite[$\S$8(i)]{ft} .
We identify $K_{T_W}(\on{pt})=\BC[\sz_1^{\pm1},\ldots,\sz_N^{\pm1}]$
and $K_{\wt\BC^\times}(\on{pt})=\BC[\vv,\vv^{-1}]$.
Recall a $\BC[\vv,\vv^{-1}][\sz_1^{\pm1},\ldots,\sz_N^{\pm1}]$-algebra embedding
  $\bz^*(\iota_*)^{-1}\colon \CA^\vv\hookrightarrow
   \wt\CA{}^\vv[\sz_1^{\pm1},\ldots,\sz_N^{\pm1}]$
of~\cite[$\S$8(i)]{ft}.

Set $\CA^\vv_\fra:=\CA^\vv\otimes_{\BC[\vv,\vv^{-1}]} \BC(\vv)$.
According to~\cite[Theorem 8.5]{ft}, the homomorphism
  $\wt{\Phi}^{\unl\lambda}_\mu\colon U^{\ad,\mu}_\vv[\sz^{\pm 1}_1,\ldots,\sz^{\pm 1}_N]\to
   \wt{\CA}^\vv_\fra[\sz^{\pm 1}_1,\ldots,\sz^{\pm 1}_N]$
factors through $\CA^\vv_\fra$ (embedded via $\bz^*(\iota_*)^{-1}$).
In other words, there is a unique homomorphism
  $\ol{\Phi}^{\unl\lambda}_\mu\colon
   U^{\ad,\mu}_\vv[\sz^{\pm 1}_1,\ldots,\sz^{\pm 1}_N]\to \CA^\vv_\fra$,
such that the composition
  $U^{\ad,\mu}_\vv[\sz^{\pm 1}_1,\ldots,\sz^{\pm 1}_N]\xrightarrow{\ol{\Phi}^{\unl\lambda}_\mu}
   \CA^\vv_\fra \xrightarrow{\bz^*(\iota_*)^{-1}} \wt{\CA}^\vv_\fra[\sz^{\pm 1}_1,\ldots,\sz^{\pm 1}_N]$
coincides with $\wt{\Phi}^{\unl\lambda}_\mu$.

Our next result establishes a certain integrality property
of the homomorphism $\ol{\Phi}^{\unl\lambda}_\mu$:

\begin{Prop}\label{integrality of all t-modes}
  $\ol{\Phi}^{\unl\lambda}_\mu(\fU^{\ad,\mu}_\vv[\sz^{\pm 1}_1,\ldots,\sz^{\pm 1}_N]))
   \subset \CA^\vv$.
\end{Prop}

As the first ingredient of the proof, let us find explicit formulas for
  $\wt{\Phi}^{\unl\lambda}_\mu(E^{(r)}_{j,i+1}),
   \wt{\Phi}^{\unl\lambda}_\mu(F^{(r)}_{i+1,j})$.

\begin{Lem}\label{t-modes in diff op-s}
For any $1\leq j\leq i<n$ and $r\in \BZ$, the following equalities hold:
\begin{equation}\label{image of e-modes}
\begin{split}
   & \wt{\Phi}^{\unl\lambda}_\mu(E^{(r)}_{j,i+1})=(-1)^{i-j}\cdot
     \prod_{t=1}^{a_i}\sw_{i,t}\prod_{k=j}^{i-1}\prod_{t=1}^{a_k}\sw_{k,t}^{1/2}
     \prod_{t=1}^{a_{j-1}}\sw_{j-1,t}^{-1/2}\times\\
   & \sum_{\substack{1\leq r_j\leq a_j\\\cdots\\ 1\leq r_i\leq a_i}}
     \frac{\sW_{j-1}(\vv^{-1}\sw_{j,r_j})\prod_{k=j}^{i-1}\sW_{k,r_k}(\vv^{-1}\sw_{k+1,r_{k+1}})}
          {\prod_{k=j}^i \sW_{k,r_k}(\sw_{k,r_k})}\cdot
     \prod_{k=j}^i \sZ_k(\sw_{k,r_k})\cdot \frac{\sw_{j,r_j}^{1+r}}{\sw_{i,r_i}}
     \cdot \prod_{k=j}^i D_{k,r_k}^{-1},
\end{split}
\end{equation}
\begin{equation}\label{image of f-modes}
\begin{split}
   & \wt{\Phi}^{\unl\lambda}_\mu(F^{(r)}_{i+1,j})=(-1)^{i-j}\vv^{j-1-i+2r}\cdot
     \prod_{k=j+1}^{i+1}\prod_{t=1}^{a_k}\sw_{k,t}^{-1/2}\times\\
   & \sum_{\substack{1\leq r_j\leq a_j\\\cdots\\ 1\leq r_i\leq a_i}}
     \frac{\prod_{k=j+1}^{i}\sW_{k,r_k}(\vv\sw_{k-1,r_{k-1}})\sW_{i+1}(\vv\sw_{i,r_i})}
          {\prod_{k=j}^i \sW_{k,r_k}(\sw_{k,r_k})}\cdot
     \frac{\sw_{i,r_i}}{\sw_{j,r_j}^{1-r}}\cdot \prod_{k=j}^i D_{k,r_k}.
\end{split}
\end{equation}
\end{Lem}

\begin{proof}
Straightforward computation.
\end{proof}

This lemma may be viewed as a trigonometric counterpart
of Lemma~\ref{t-modes in diff op-s yangian}.

\begin{proof}[Proof of Proposition~\ref{integrality of all t-modes}]
By explicit formulas of Theorem~\ref{Homomorphism}, we clearly have
  $\ol{\Phi}^{\unl\lambda}_\mu((\phi^\epsilon_i)^{\pm 1})\in \CA^\vv$
for $\epsilon=\pm$. Since $\wt{\Phi}^{\unl\lambda}_\mu(\psi^\pm_{j,\pm s})$
are Laurent polynomials in $\{\sw^{1/2}_{i,t}\}_{1\leq i\leq n-1}^{1\leq t\leq a_i}$
with coefficients in $\BC[\vv,\vv^{-1}][\sz^{\pm 1}_1,\ldots,\sz^{\pm 1}_N]$
and are symmetric in each family $\{\sw^{1/2}_{i,t}\}_{t=1}^{a_i}\ (1\leq i<n)$,
we immediately get $\ol{\Phi}^{\unl\lambda}_\mu(\psi^\pm_{j,\pm s})\in \CA^\vv$.
Hence, it remains to verify the inclusions
  $\ol{\Phi}^{\unl\lambda}_\mu(E^{(r)}_{j,i+1}),
   \ol{\Phi}^{\unl\lambda}_\mu(F^{(r)}_{i+1,j})\in \CA^\vv$
for all $1\leq j\leq i<n$ and $r\in \BZ$.

Recall the setup of~\cite[$\S$8(i)]{ft}.
For $1\leq j\leq i<n$, we consider a coweight
  $\lambda_{ji}=(0,\ldots,0,\varpi_{j,1},\ldots,\varpi_{i,1},0,\ldots,0)$
(resp.\ $\lambda_{ji}^*=(0,\ldots,0,\varpi_{j,1}^*,\ldots,\varpi_{i,1}^*,0,\ldots,0)$)
of $\GL(V)=\GL(V_1)\times\cdots\times\GL(V_{n-1})$. The corresponding orbits
$\Gr_{\GL(V)}^{\lambda_{ji}},\Gr_{\GL(V)}^{\lambda_{ji}^*}\subset\Gr_{\GL(V)}$
are closed (they are products of the minuscule orbits, isomorphic to
$\BP^{a_j-1}\times\cdots\times\BP^{a_i-1}$). Their preimages in the variety
of triples $\CR_{\GL(V),\bN}$ are denoted by $\CR_{\lambda_{ji}},\CR_{\lambda_{ji}^*}$,
respectively.

Then the right-hand side of~(\ref{image of e-modes}) equals
\begin{equation}
  \label{E in K}
  \bz^*(\iota_*)^{-1}\left((-1)^{i-j}{\det}_{j-1}^{-1/2}\cdot{\det}_j^{1/2}\cdot\ldots
  \cdot{\det}_{i-1}^{1/2}\cdot{\det}_i\cdot\CO_{\varpi_{j,1}^*}(-r-1)\boxtimes\CO_{\varpi_{i,1}^*}(1)\right),
\end{equation}
while the right-hand side of~(\ref{image of f-modes}) equals
\begin{equation}
  \label{F in K}
  \bz^*(\iota_*)^{-1}\left((-1)^{i-j}\vv^{j-1-i+2r}{\det}_{j+1}^{-1/2}\cdot\ldots\cdot
  {\det}_{i+1}^{-1/2}\cdot \CO_{\varpi_{j,1}}(r-1)\boxtimes\CO_{\varpi_{i,1}}(1)\right).
\end{equation}
Here $\det_k$ stands for the determinant character of $\GL(V_k)$,
while $\CO_{\varpi_{k,1}}(s)$ stands for the class of the line bundle
$\CO(s)$ on $\Gr_{\varpi_{k,1}}\simeq\BP^{a_k-1}$, and everything is
pulled back to $\CR_{\lambda_{ji}}$ (similarly for $\CO_{\varpi_{k,1}^*}(s)$).
\end{proof}

To prove the main result of this subsection, let us obtain
shuffle descriptions of the restrictions
  $\wt{\Phi}^{\unl{\lambda}}_\mu\colon
   U^{\ad,\mu;>}_\vv[\sz^{\pm 1}_1,\ldots,\sz^{\pm 1}_N]
   \to \wt{\CA}^\vv_\fra[\sz^{\pm 1}_1,\ldots,\sz^{\pm 1}_N]$
and
  $\wt{\Phi}^{\unl{\lambda}}_\mu\colon
   U^{\ad,\mu;<}_\vv[\sz^{\pm 1}_1,\ldots,\sz^{\pm 1}_N]
   \to \wt{\CA}^\vv_\fra[\sz^{\pm 1}_1,\ldots,\sz^{\pm 1}_N]$.
In other words, evoking the isomorphism $\Psi\colon U^>_\vv(L\gl_n)\iso S^{(n)}$
of Theorem~\ref{shuffle rational} and the isomorphism
$\Psi\colon U^<_\vv(L\gl_n)\iso S^{(n),\op}$ of Remark~\ref{opposite shuffle},
we compute the resulting homomorphisms
\begin{equation}\label{homomorphism from shuffle}
  \wt{\Phi}^{\unl{\lambda}}_\mu\colon
   S^{(n)}[\sz^{\pm 1}_1,\ldots,\sz^{\pm 1}_N]\simeq U^{\ad,\mu;>}_\vv[\sz^{\pm 1}_1,\ldots,\sz^{\pm 1}_N]
   \longrightarrow \wt{\CA}^\vv_\fra[\sz^{\pm 1}_1,\ldots,\sz^{\pm 1}_N]
\end{equation}
and
\begin{equation}\label{opposite homomorphism from shuffle}
  \wt{\Phi}^{\unl{\lambda}}_\mu\colon
   S^{(n),\op}[\sz^{\pm 1}_1,\ldots,\sz^{\pm 1}_N]\simeq U^{\ad,\mu;<}_\vv[\sz^{\pm 1}_1,\ldots,\sz^{\pm 1}_N]
   \longrightarrow \wt{\CA}^\vv_\fra[\sz^{\pm 1}_1,\ldots,\sz^{\pm 1}_N].
\end{equation}

For any $1\leq i<n$ and $1\leq r\leq a_i$, we define
  $Y_{i,r}(z):=\frac{\sZ_i(z)\sW_{i-1}(\vv^{-1}z)}{\sW_{i,r}(z)},
   Y'_{i,r}(z):=\frac{\sW_{i+1}(\vv z)}{\sW_{i,r}(z)}$.
We also recall the functions $\zeta_{i,j}(z)=\frac{z-\vv^{-c_{ij}}}{z-1}$ of
subsection~\ref{ssec shuffle algebra}.

\begin{Thm}\label{shuffle homomorphism}
(a) For any $E\in S^{(n)}_{\unl{k}}[\sz^{\pm 1}_1,\ldots,\sz^{\pm 1}_N]$,
its image under the homomorphism $\wt{\Phi}^{\unl{\lambda}}_\mu$
of~(\ref{homomorphism from shuffle}) equals
\begin{equation}\label{Explicit shuffle homom}
\begin{split}
  & \wt{\Phi}^{\unl{\lambda}}_\mu(E)=
    \vv^{-\sum_{i=1}^{n-1} k_i(k_i-1)}(\vv-\vv^{-1})^{-\sum_{i=1}^{n-1} k_i}
    \prod_{i=1}^{n-1} \prod_{r=1}^{a_i}\sw_{i,r}^{k_i-\frac{1}{2}k_{i+1}}\times\\
  & \sum_{\substack{m^{(1)}_1+\ldots+m^{(1)}_{a_1}=k_1\\\cdots\\ m^{(n-1)}_1+\ldots+m^{(n-1)}_{a_{n-1}}=k_{n-1}}}^{m^{(i)}_r\in \BN}
    \left(\prod_{i=1}^{n-1}\prod_{r=1}^{a_i}\prod_{p=1}^{m^{(i)}_r}Y_{i,r}(\vv^{-2(p-1)}\sw_{i,r})\cdot
    E\left(\{\vv^{-2(p-1)}\sw_{i,r}\}_{\substack{1\leq i<n\\ 1\leq r\leq a_i\\ 1\leq p\leq m^{(i)}_r}}\right)\times\right.\\
  & \left.
    \prod_{i=1}^{n-1}\prod_{r=1}^{a_i}\prod_{1\leq p_1<p_2\leq m^{(i)}_r}
      \zeta^{-1}_{i,i}(\vv^{-2(p_1-1)}\sw_{i,r},\vv^{-2(p_2-1)}\sw_{i,r})\times\right.\\
  & \left.
    \prod_{i=1}^{n-1}\prod_{1\leq r_1\neq r_2\leq a_i}\prod_{1\leq p_1\leq m^{(i)}_{r_1}}^{1\leq p_2\leq m^{(i)}_{r_2}}
      \zeta^{-1}_{i,i}(\vv^{-2(p_1-1)}\sw_{i,r_1},\vv^{-2(p_2-1)}\sw_{i,r_2})\times\right.\\
  & \left.\prod_{i=1}^{n-2}\prod_{1\leq r_1\leq a_{i+1}}^{1\leq r_2\leq a_{i}}\prod_{1\leq p_1\leq m^{(i+1)}_{r_1}}^{1\leq p_2\leq m^{(i)}_{r_2}}
      \zeta^{-1}_{i+1,i}(\vv^{-2(p_1-1)}\sw_{i+1,r_1},\vv^{-2(p_2-1)}\sw_{i,r_2})\cdot
    \prod_{i=1}^{n-1}\prod_{r=1}^{a_i} D_{i,r}^{-m^{(i)}_r}\right).
\end{split}
\end{equation}

\noindent
(b) For any $F\in S^{(n),\op}_{\unl{k}}[\sz^{\pm 1}_1,\ldots,\sz^{\pm 1}_N]$,
its image under the homomorphism $\wt{\Phi}^{\unl{\lambda}}_\mu$
of~(\ref{opposite homomorphism from shuffle}) equals
\begin{equation}\label{Explicit opposite shuffle homom}
\begin{split}
  & \wt{\Phi}^{\unl{\lambda}}_\mu(F)=
    (1-\vv^2)^{-\sum_{i=1}^{n-1} k_i}\prod_{i=1}^{n-1} \prod_{r=1}^{a_i}\sw_{i,r}^{-\frac{1}{2}k_{i-1}}\times\\
  & \sum_{\substack{m^{(1)}_1+\ldots+m^{(1)}_{a_1}=k_1\\\cdots\\ m^{(n-1)}_1+\ldots+m^{(n-1)}_{a_{n-1}}=k_{n-1}}}^{m^{(i)}_r\in \BN}
    \left(\prod_{i=1}^{n-1}\prod_{r=1}^{a_i}\prod_{p=1}^{m^{(i)}_r}Y'_{i,r}(\vv^{2(p-1)}\sw_{i,r})\cdot
    F\left(\{\vv^{2p}\sw_{i,r}\}_{\substack{1\leq i<n\\ 1\leq r\leq a_i\\ 1\leq p\leq m^{(i)}_r}}\right)\times\right.\\
  & \left.
    \prod_{i=1}^{n-1}\prod_{r=1}^{a_i}\prod_{1\leq p_1<p_2\leq m^{(i)}_r}
    \zeta^{-1}_{i,i}(\vv^{2(p_2-1)}\sw_{i,r},\vv^{2(p_1-1)}\sw_{i,r})\times\right.\\
  & \left.
    \prod_{i=1}^{n-1}\prod_{1\leq r_1\neq r_2\leq a_i}\prod_{1\leq p_1\leq m^{(i)}_{r_1}}^{1\leq p_2\leq m^{(i)}_{r_2}}
    \left(\vv^{-1}\cdot \zeta^{-1}_{i,i}(\vv^{2(p_2-1)}\sw_{i,r_2},\vv^{2(p_1-1)}\sw_{i,r_1})\right)\times\right.\\
  & \left.
    \prod_{i=1}^{n-2}\prod_{1\leq r_1\leq a_{i+1}}^{1\leq r_2\leq a_{i}}\prod_{1\leq p_1\leq m^{(i+1)}_{r_1}}^{1\leq p_2\leq m^{(i)}_{r_2}}
    \zeta^{-1}_{i+1,i}(\vv^{2(p_1-1)}\sw_{i+1,r_1},\vv^{2(p_2-1)}\sw_{i,r_2})\cdot
    \prod_{i=1}^{n-1}\prod_{r=1}^{a_i} D_{i,r}^{m^{(i)}_r}\right).
\end{split}
\end{equation}
\end{Thm}

\begin{proof}
(a) Let us denote the right-hand side of~(\ref{Explicit shuffle homom}) by $\Theta(E)$.
A tedious straightforward verification proves $\Theta(E\star E')=\Theta(E)\Theta(E')$,
that is, $\Theta$ is a $\BC(\vv)[\sz^{\pm 1}_1,\ldots,\sz^{\pm 1}_N]$-algebra homomorphism.
On the other hand, $S^{(n)}[\sz^{\pm 1}_1,\ldots,\sz^{\pm 1}_N]$ is generated over
$\BC(\vv)[\sz^{\pm 1}_1,\ldots,\sz^{\pm 1}_N]$ by its components $\{S^{(n)}_{1_i}\}_{i=1}^{n-1}$
with $1_i=(0,\ldots,0,1,0,\ldots,0)\in \BN^{n-1}$ (here $1$ stays at the $i$-th coordinate),
due to the isomorphism $\Psi\colon U^>_\vv(L\gl_n)\iso S^{(n)}$.
Comparing~(\ref{Explicit shuffle homom}) with the formulas of
Theorem~\ref{Homomorphism}, we immediately get
$\Theta(E)=\wt{\Phi}^{\unl{\lambda}}_\mu(E)$ for $E\in S^{(n)}_{1_i}\ (1\leq i<n)$. Hence, we have
$\Theta(E)=\wt{\Phi}^{\unl{\lambda}}_\mu(E)$ for any
$E\in S^{(n)}[\sz^{\pm 1}_1,\ldots,\sz^{\pm 1}_N]$.
This completes our proof of Theorem~\ref{shuffle homomorphism}(a).

(b) The proof of Theorem~\ref{shuffle homomorphism}(b) is completely analogous.
\end{proof}

For any $1\leq j\leq i<n$, a vector
$\unl{k}=(0,\ldots,0,k_j,\ldots,k_i,0,\ldots,0)\in \BN^{n-1}$ with
$1\leq k_\ell\leq a_\ell\ (j\leq \ell\leq i)$, a collection
of integers $\gamma_\ell\in \BZ\ (j\leq \ell\leq i)$, and a collection
of symmetric Laurent polynomials
  $g_\ell(\{x_{\ell,r}\}_{r=1}^{k_\ell})\in
   \BC[\vv,\vv^{-1}][\{x^{\pm 1}_{\ell,r}\}_{r=1}^{k_\ell}]^{\Sigma_{k_\ell}}\ (j\leq \ell\leq i)$,
consider shuffle elements $\wt{E}\in S^{(n)}_{\unl{k}}$
and $\wt{F}\in S^{(n),\op}_{\unl{k}}$ given by:
\begin{equation}\label{important elements}
\begin{split}
  & \wt{E}:=(-1)^{\sum_{\ell=j}^{i-1} k_\ell k_{\ell+1}}
    \vv^{\sum_{\ell=j}^i k_\ell(k_\ell-1)} (\vv-\vv^{-1})^{\sum_{\ell=j}^i k_\ell}\times\\
  & \frac{\prod_{\ell=j}^i \prod_{1\leq r_1\ne r_2\leq k_\ell} (x_{\ell,r_1}-\vv^{-2}x_{\ell,r_2})\cdot
          \prod_{\ell=j}^i \prod_{r=1}^{k_\ell} x_{\ell,r}^{\gamma_\ell+1+k_{\ell-1}-k_\ell}\cdot
          \prod_{\ell=j}^j g_\ell(\{x_{\ell,r}\}_{r=1}^{k_\ell})}
         {\prod_{\ell=j}^{i-1}\prod_{1\leq r_1\leq k_\ell}^{1\leq r_2\leq k_{\ell+1}} (x_{\ell,r_1}-x_{\ell+1,r_2})}
\end{split}
\end{equation}
and
\begin{equation}\label{important opposite elements}
\begin{split}
  & \wt{F}:=(-1)^{\sum_{\ell=j}^{i} k_\ell}
    \vv^{\sum_{\ell=j}^i k_\ell(k_\ell+k_{\ell+1}-2\gamma_\ell)} (\vv-\vv^{-1})^{\sum_{\ell=j}^i k_\ell}\times\\
  & \frac{\prod_{\ell=j}^i \prod_{1\leq r_1\ne r_2\leq k_\ell} (x_{\ell,r_1}-\vv^{-2}x_{\ell,r_2})\cdot
          \prod_{\ell=j}^i \prod_{r=1}^{k_\ell} x_{\ell,r}^{\gamma_\ell+1+k_{\ell+1}-k_\ell}\cdot
          \prod_{\ell=j}^j g_\ell(\{\vv^{-2}x_{\ell,r}\}_{r=1}^{k_\ell})}
         {\prod_{\ell=j}^{i-1}\prod_{1\leq r_1\leq k_\ell}^{1\leq r_2\leq k_{\ell+1}} (x_{\ell,r_1}-x_{\ell+1,r_2})}.
\end{split}
\end{equation}

These elements obviously satisfy the pole conditions~(\ref{pole conditions}) as well
as the wheel conditions~(\ref{wheel conditions}), due to the presence of the factor
  $\prod_{\ell=j}^i \prod_{1\leq r_1\ne r_2\leq k_\ell} (x_{\ell,r_1}-\vv^{-2}x_{\ell,r_2})$
in the right-hand sides of~(\ref{important elements}, \ref{important opposite elements}).
Moreover,
  $\wt{E}\in \fS^{(n)}_{\unl{k}}$ and $\wt{F}\in \fS^{(n),\op}_{\unl{k}}$,
due to Proposition~\ref{key properties of integral shuffle}(b).
These elements are of crucial importance due to
Proposition~\ref{image of important elements} and Remark~\ref{geometric explanation},
which play the key role in our proof of Theorem~\ref{Surjectivity} below.

\begin{Prop}\label{image of important elements}
(a) For $\wt{E}\in \fS^{(n)}_{\unl{k}}$ given by~(\ref{important elements}), we have
\begin{equation}\label{Explicit image of important E}
\begin{split}
  & \wt{\Phi}^{\unl{\lambda}}_\mu(\wt{E})=
    \prod_{\ell=j-1}^{i} \prod_{r=1}^{a_\ell}\sw_{\ell,r}^{k_\ell-\frac{1}{2}k_{\ell+1}}\times\\
  & \sum_{\substack{J_j\subset\{1,\ldots,a_j\}:|J_j|=k_j\\\cdots\\ J_i\subset\{1,\ldots,a_i\}:|J_i|=k_i}}
    \left(\frac{\prod_{r\in J_j}^{1\leq s\leq a_{j-1}}(1-\frac{\sw_{j-1,s}}{\vv^{-1}\sw_{j,r}})\cdot
                \prod_{\ell=j}^{i-1} \prod_{r\in J_{\ell+1}}^{s\notin J_\ell} (1-\frac{\sw_{\ell,s}}{\vv^{-1}\sw_{\ell+1,r}})}
               {\prod_{\ell=j}^i \prod_{r\in J_\ell}^{s\notin J_\ell} (1-\frac{\sw_{\ell,s}}{\sw_{\ell,r}})}\times\right.\\
  & \left. \prod_{\ell=j}^i \prod_{r\in J_\ell} Z_\ell(\sw_{\ell,r})\cdot
           \prod_{\ell=j}^i \prod_{r\in J_\ell} \sw_{\ell,r}^{\gamma_\ell}\cdot
           \prod_{\ell=j}^i g_\ell(\{\sw_{\ell,r}\}_{r\in J_\ell})\cdot
           \prod_{\ell=j}^i \prod_{r\in J_\ell} D_{\ell,r}^{-1}\right).
\end{split}
\end{equation}

\noindent
(b) For $\wt{F}\in \fS^{(n),\op}_{\unl{k}}$ given by~(\ref{important opposite elements}),
we have
\begin{equation}\label{Explicit image of important F}
\begin{split}
  & \wt{\Phi}^{\unl{\lambda}}_\mu(\wt{F})=
    \prod_{\ell=j+1}^{i+1} \prod_{r=1}^{a_\ell}\sw_{\ell,r}^{-\frac{1}{2}k_{\ell-1}}\times\\
  & \sum_{\substack{J_j\subset\{1,\ldots,a_j\}:|J_j|=k_j\\\cdots\\ J_i\subset\{1,\ldots,a_i\}:|J_i|=k_i}}
    \left(\frac{\prod_{r\in J_i}^{1\leq s\leq a_{i+1}}(1-\frac{\sw_{i+1,s}}{\vv \sw_{i,r}})\cdot
                \prod_{\ell=j+1}^{i} \prod_{r\in J_{\ell-1}}^{s\notin J_\ell} (1-\frac{\sw_{\ell,s}}{\vv \sw_{\ell-1,r}})}
               {\prod_{\ell=j}^i \prod_{r\in J_\ell}^{s\notin J_\ell} (1-\frac{\sw_{\ell,s}}{\sw_{\ell,r}})}\times\right.\\
  & \left. \prod_{\ell=j}^i \prod_{r\in J_\ell} \sw_{\ell,r}^{\gamma_\ell}\cdot
           \prod_{\ell=j}^i g_\ell(\{\sw_{\ell,r}\}_{r\in J_\ell})\cdot
           \prod_{\ell=j}^i \prod_{r\in J_\ell} D_{\ell,r}\right).
\end{split}
\end{equation}
\end{Prop}

\begin{proof}
The proof is straightforward and is based
on~(\ref{Explicit shuffle homom},~\ref{Explicit opposite shuffle homom}).
Due to the presence of the factor
  $\prod_{\ell=j}^i \prod_{1\leq r_1\ne r_2\leq k_\ell} (x_{\ell,r_1}-\vv^{-2}x_{\ell,r_2})$
in~(\ref{important elements}, \ref{important opposite elements}),
all the summands of~(\ref{Explicit shuffle homom},~\ref{Explicit opposite shuffle homom})
with at least one index $m^{(\ell)}_r>1$ actually vanish. This explains why the
summations over all partitions of $k_\ell$ into the sum of $a_\ell$ nonnegative
integers in~(\ref{Explicit shuffle homom},~\ref{Explicit opposite shuffle homom}) are
replaced by the summations over all subsets of $\{1,\ldots,a_\ell\}$ of cardinality
$k_\ell$ in~(\ref{Explicit image of important E},~\ref{Explicit image of important F}).
\end{proof}

\begin{Rem}\label{identification}
In the particular case
  $k_\ell=1, \gamma_\ell=(r+1)\delta_{\ell,j}-\delta_{\ell,i}, g_\ell=1$
for $j\leq \ell\leq i$, the element $\wt{E}$ of~(\ref{important elements})
coincides with $\Psi((-1)^{i-j}E^{(r)}_{j,i+1})$. Likewise, in the particular case
  $k_\ell=1, \gamma_\ell=(r-1)\delta_{\ell,j}+\delta_{\ell,i}, g_\ell=1$
for $j\leq \ell\leq i$, the element $\wt{F}$ of~(\ref{important opposite elements})
coincides with $\Psi((-1)^{i-j}\vv^{i+1-j-2r}F^{(r)}_{i+1,j})$. Hence,
Proposition~\ref{image of important elements} generalizes
Lemma~\ref{t-modes in diff op-s}.
\end{Rem}

\begin{Rem}\label{geometric explanation}
For any $1\leq j\leq i<n$ and $\unl{k}\in \BN^{n-1}$ as above, we consider
a coweight
  $\kappa_{ji}=(0,\ldots,0,\varpi_{j,k_j},\ldots,\varpi_{i,k_i},0,\ldots,0)$
(resp.\ $\kappa^*_{ji}=(0,\ldots,0,\varpi^*_{j,k_j},\ldots,\varpi^*_{i,k_i},0,\ldots,0)$)
of $\GL(V)$, generalizing a coweight $\lambda_{ji}$ (resp.\ $\lambda^*_{ji}$) from
our proof of Proposition~\ref{integrality of all t-modes}.
The preimages of the corresponding orbits
$\Gr^{\kappa_{ji}}_{\GL(V)},\Gr^{\kappa_{ji}^*}_{\GL(V)}$ in the variety of
triples $\CR_{\GL(V),\bN}$ are denoted by $\CR_{\kappa_{ji}},\CR_{\kappa_{ji}^*}$,
respectively. Similarly to~(\ref{E in K},~\ref{F in K}), the right-hand sides
of~(\ref{Explicit image of important E},~\ref{Explicit image of important F})
equal $\bz^*(\iota_*)^{-1}$ of the appropriate classes in
  $K^{(\wt{\GL}(V)\times T_W)_\CO\rtimes\wt\BC^\times}(\CR_{\kappa_{ji}^*}),\
   K^{(\wt{\GL}(V)\times T_W)_\CO\rtimes\wt\BC^\times}(\CR_{\kappa_{ji}})$.
Moreover, any classes in these equivariant $K$-groups can be obtained this way
for an appropriate choice of symmetric Laurent polynomials $g_\ell$.
\end{Rem}

Our next result may be viewed as a trigonometric/$K$-theoretic counterpart
of Proposition~\ref{Surjectivity yangian} as well as a
generalization of~\cite[Theorem~9.2]{ft} and~\cite[Corollary~2.21]{cw}:

\begin{Thm}\label{Surjectivity}
  $\ol{\Phi}^{\unl\lambda}_\mu\colon
   \fU^{\ad,\mu}_\vv[\sz^{\pm 1}_1,\ldots,\sz^{\pm 1}_N]\to \CA^\vv$
is surjective.
\end{Thm}

\begin{proof}
We need to prove that $K_{(\wt{\GL}(V)\times T_W)_\CO\rtimes\wt\BC^\times}(\pt)$
together with RHS of~(\ref{E in K},~\ref{F in K}) generate
$K^{(\wt{\GL}(V)\times T_W)_\CO\rtimes\wt\BC^\times}(\CR_{\GL(V),\bN})$.
Recall the filtration by support on
  $K^{(\wt{\GL}(V)\times T_W)_\CO\rtimes\wt\BC^\times}(\CR_{\GL(V),\bN})$
defined in~\cite[$\S$6(i)]{bfna} (strictly speaking, it is defined on the
equivariant Borel-Moore homology
  $H^{(\wt{\GL}(V)\times T_W)_\CO\rtimes\wt\BC^\times}(\CR_{\GL(V),\bN})$,
but the definition works word-for-word in the case of $K$-theory).
It suffices to prove that the associated graded
  $\on{gr}K^{(\wt{\GL}(V)\times T_W)_\CO\rtimes\wt\BC^\times}(\CR_{\GL(V),\bN})=
   \bigoplus_\lambda K^{(\wt{\GL}(V)\times T_W)_\CO\rtimes\wt\BC^\times}(\CR_\lambda)$
is generated by the right-hand sides
of~(\ref{Explicit image of important E},~\ref{Explicit image of important F})
together with $K_{(\wt{\GL}(V)\times T_W)_\CO\rtimes\wt\BC^\times}(\pt)$. Now
the cone of dominant coweights of $\GL(V)$ is subdivided into chambers by the
generalized root hyperplanes~\cite[$\S$5(i)]{bfna}. Recall that the generalized
roots are either the roots $w_{i,r}-w_{i,s}\ (1\leq i<n,\ 1\leq r\ne s\leq a_i)$ of $\fgl(V)$
or the nonzero weights
$w_{i,r},\ w_{i,r}-w_{i+1,s}\ (1\leq i<n,\ 1\leq r\leq a_i,\ 1\leq s\leq a_{i+1})$
of its module $\bN$. Hence a chamber is cut out by the following conditions:

(a) For any pair of adjacent vertices $i,j$, we fix a {\em shuffle}, i.e.\
a permutation $\sigma$ of $\{1,\ldots,a_i,a_i+1,\ldots,a_i+a_j\}$ such that
$\sigma(b)<\sigma(c)$ if $1\leq b<c\leq a_i$ or $a_i<b<c\leq a_i+a_j$.
Then we require $\lambda^{(i)}_b\leq\lambda^{(j)}_{b'}$ if
$\sigma(b)>\sigma(a_i+b')$, and $\lambda^{(i)}_b\geq\lambda^{(j)}_{b'}$
if $\sigma(b)<\sigma(a_i+b')$.

(b) For any vertex $i$ we fix a number $0\leq d_i\leq a_i$ and require
$\lambda^{(i)}_b\geq0$ for $1\leq b\leq d_i$, and $\lambda^{(i)}_b\leq0$
for $d_i<b\leq a_i$.

So the chambers are numbered by the choices of shuffles for all the adjacent
pairs $(i,j=i\pm 1)$ of vertices and the choices of numbers $d_i$ for all the vertices.
The intersection of a chamber $C$ with the lattice of integral coweights is
generated by the collections of fundamental coweights $(\varpi^{(i)}_{b_i})$
and the collections of dual coweights $(\varpi^{(i)*}_{c_i})$
(we allow $0\leq b_i,c_i\leq a_i$) such that

(a) for any pair of adjacent vertices $(i,j)$ and the corresponding shuffle
$\sigma$, we have $\sigma(b)>\sigma(c)$ for any
$1\leq b\leq b_i,\ a_i+b_j<c\leq a_i+a_j$ as well as for any
$1\leq b\leq a_i-c_i,\ a_i+a_j-c_j<c\leq a_i+a_j$.

(b) For any vertex $i$ and the corresponding number $d_i$,
we have $b_i\leq d_i<a_i-c_i$.

\medskip

For any interval $[j,i]=\{j,j+1,\ldots,i\}\subset\{1,\ldots,n-1\}$,
we consider collections of coweights
  $\kappa_{ji}=(0,\ldots,0,\varpi_{j,k_j},\ldots,\varpi_{i,k_i},0,\ldots,0)$
and
  $\kappa^*_{ji}=(0,\ldots,0,\varpi^*_{j,k_j},\ldots,\varpi^*_{i,k_i},0,\ldots,0)$.
According to~Remark~\ref{geometric explanation}, any class in
  $K^{(\wt{\GL}(V)\times T_W)_\CO\rtimes\wt\BC^\times}(\CR_{\kappa_{ji}^*}),\
   K^{(\wt{\GL}(V)\times T_W)_\CO\rtimes\wt\BC^\times}(\CR_{\kappa_{ji}})$
lies in the image of $\fU^{\ad,\mu}_\vv[\sz^{\pm 1}_1,\ldots,\sz^{\pm 1}_N]$
under $\ol{\Phi}^{\unl\lambda}_\mu$. According to the previous paragraph,
for any chamber $C$, the equivariant $K$-groups
  $K^{(\wt{\GL}(V)\times T_W)_\CO\rtimes\wt\BC^\times}(\CR_{\kappa_{ji}^*})$
and
  $K^{(\wt{\GL}(V)\times T_W)_\CO\rtimes\wt\BC^\times}(\CR_{\kappa_{ji}})$
(we take all the collections $\kappa_{ji},\kappa_{ji}^*$ generating $C$)
generate the subring
  $\bigoplus_{\lambda\in C}K^{(\wt{\GL}(V)\times T_W)_\CO\rtimes\wt\BC^\times}(\CR_\lambda)$
of $\on{gr}K^{(\wt{\GL}(V)\times T_W)_\CO\rtimes\wt\BC^\times}(\CR_{\GL(V),\bN})$.
Indeed, if $\lambda,\mu$ lie in the same chamber $C$, then
$(\pi^*c_\lambda)*(\pi^*c_\mu)=\pi^*(c_\lambda*c_\mu)$ (as in~\cite[\S6(i)]{bfna},
$\pi$ stands for the projection $\CR_\lambda\to\Gr_{\GL(V)}^\lambda$) for any classes
  $c_\lambda\in K^{(\wt{\GL}(V)\times T_W)_\CO\rtimes\wt\BC^\times}(\Gr^\lambda_{\GL(V)}),\
   c_\mu\in K^{(\wt{\GL}(V)\times T_W)_\CO\rtimes\wt\BC^\times}(\Gr^\mu_{\GL(V)})$.
And the equivariant $K$-theory of $\Gr_{\GL(V)}$ is generated by the equivariant
$K$-groups of products of fundamental orbits by the argument in the beginning of
the proof of~\cite[Corollary~2.21]{cw}. Hence, the appropriate classes
in~(\ref{Explicit image of important E},~\ref{Explicit image of important F})
generate the entire associated graded ring
  $\on{gr}K^{(\wt{\GL}(V)\times T_W)_\CO\rtimes\wt\BC^\times}(\CR_{\GL(V),\bN})$
(cf.~\cite[\S6.3]{bdg}, especially the last paragraph).

\medskip
This completes our proof of Theorem~\ref{Surjectivity}.
\end{proof}

\begin{Rem}\label{Surjectivity for Yangian via Shuffle}
The above proof of Theorem~\ref{Surjectivity} follows the one of~\cite[Corollary~2.21]{cw},
but crucially relies on the construction of certain elements of the integral form
$\fU^{\ad,\mu}_\vv[\sz^{\pm 1}_1,\ldots,\sz^{\pm 1}_N]$ whose shuffle realization is given
by explicit formulas~(\ref{important elements},~\ref{important opposite elements})
(let us emphasize that the explicit formulas for $\Psi^{-1}(\wt{E}), \Psi^{-1}(\wt{F})$ are not known).
The same argument can be used to obtain a new proof of Proposition~\ref{Surjectivity yangian}.
To this end, let $W^{(n)}\supset \fW^{(n)}$ be the rational shuffle algebra and its integral form of~\cite[\S6]{t}.
Similar to Theorems~\ref{shuffle rational},~\ref{shuffle integral}, there is a $\BC[\hbar]$-algebra
isomorphism $\Psi\colon Y^>_\hbar(\ssl_n)\iso W^{(n)}$, which gives rise to a $\BC[\hbar]$-algebra
isomorphism $\Psi\colon \bY^>_\hbar(\ssl_n)\iso \fW^{(n)}$, see~\cite[Theorems 6.20, 6.27]{t}.
Then, for any $1\leq j\leq i<n$, a vector
$\unl{k}=(0,\ldots,0,k_j,\ldots,k_i,0,\ldots,0)\in \BN^{n-1}$ with
$1\leq k_\ell\leq a_\ell\ (j\leq \ell\leq i)$, and a collection
of symmetric polynomials
  $g_\ell(\{x_{\ell,r}\}_{r=1}^{k_\ell})\in
   \BC[\hbar][\{x_{\ell,r}\}_{r=1}^{k_\ell}]^{\Sigma_{k_\ell}}\ (j\leq \ell\leq i)$,
consider shuffle elements $\wt{E}\in \fW^{(n)}_{\unl{k}}$
and $\wt{F}\in \fW^{(n),\op}_{\unl{k}}$ given by:
\begin{equation}\label{important elements yangian}
  \wt{E}:=\hbar^{\sum_{\ell=j}^i k_\ell}\cdot
  \frac{\prod_{\ell=j}^i \prod_{1\leq r_1\ne r_2\leq k_\ell} (x_{\ell,r_1}-x_{\ell,r_2}+\hbar)\cdot
          \prod_{\ell=j}^j g_\ell(\{x_{\ell,r}\}_{r=1}^{k_\ell})}
         {\prod_{\ell=j}^{i-1}\prod_{1\leq r_1\leq k_\ell}^{1\leq r_2\leq k_{\ell+1}} (x_{\ell,r_1}-x_{\ell+1,r_2})}
\end{equation}
and
\begin{equation}\label{important opposite elements yangian}
  \wt{F}:=\hbar^{\sum_{\ell=j}^i k_\ell}\cdot
  \frac{\prod_{\ell=j}^i \prod_{1\leq r_1\ne r_2\leq k_\ell} (x_{\ell,r_1}-x_{\ell,r_2}+\hbar)\cdot
          \prod_{\ell=j}^j g_\ell(\{\vv^{-2}x_{\ell,r}\}_{r=1}^{k_\ell})}
         {\prod_{\ell=j}^{i-1}\prod_{1\leq r_1\leq k_\ell}^{1\leq r_2\leq k_{\ell+1}} (x_{\ell,r_1}-x_{\ell+1,r_2})}.
\end{equation}
These are the rational counterparts of the elements
in~(\ref{important elements},~\ref{important opposite elements}).
Similar to Proposition~\ref{image of important elements}, we have the following explicit formulas
(generalizing Lemma~\ref{t-modes in diff op-s yangian}, cf.~Remark~\ref{identification}):
\begin{equation}\label{Explicit image of important E yangian}
\begin{split}
  & \Phi^{\unl{\lambda}}_\mu(\Psi^{-1}(\wt{E}))=(-1)^{\sum_{\ell=j}^i k_\ell+\sum_{\ell=j}^{i-1}k_{\ell}k_{\ell+1}}\times\\
  & \sum_{\substack{J_j\subset\{1,\ldots,a_j\}:|J_j|=k_j\\\cdots\\ J_i\subset\{1,\ldots,a_i\}:|J_i|=k_i}}
    \left(\frac{\prod_{r\in J_j}^{1\leq s\leq a_{j-1}}(w_{j,r}-w_{j-1,s}-\frac{\hbar}{2})\cdot
                \prod_{\ell=j}^{i-1} \prod_{r\in J_{\ell+1}}^{s\notin J_\ell} (w_{\ell+1,r}-w_{\ell,s}-\frac{\hbar}{2})}
               {\prod_{\ell=j}^i \prod_{r\in J_\ell}^{s\notin J_\ell} (w_{\ell,r}-w_{\ell,s})}\times\right.\\
  & \left. \prod_{\ell=j}^i \prod_{r\in J_\ell} Z_\ell(w_{\ell,r})\cdot
           \prod_{\ell=j}^i g_\ell(\{w_{\ell,r}\}_{r\in J_\ell})\cdot
           \prod_{\ell=j}^i \prod_{r\in J_\ell} \sfu_{\ell,r}^{-1}\right)
\end{split}
\end{equation}
and
\begin{equation}\label{Explicit image of important F yangian}
\begin{split}
  & \Phi^{\unl{\lambda}}_\mu(\Psi^{-1}(\wt{F}))=\\
  & \sum_{\substack{J_j\subset\{1,\ldots,a_j\}:|J_j|=k_j\\\cdots\\ J_i\subset\{1,\ldots,a_i\}:|J_i|=k_i}}
    \left(\frac{\prod_{r\in J_i}^{1\leq s\leq a_{i+1}}(w_{i,r}-w_{i+1,s}+\frac{\hbar}{2})\cdot
                \prod_{\ell=j+1}^{i} \prod_{r\in J_{\ell-1}}^{s\notin J_\ell} (w_{\ell-1,r}-w_{\ell,s}+\frac{\hbar}{2})}
               {\prod_{\ell=j}^i \prod_{r\in J_\ell}^{s\notin J_\ell} (w_{\ell,r}-w_{\ell,s})}\times\right.\\
  & \left. \prod_{\ell=j}^i g_\ell(\{w_{\ell,r}\}_{r\in J_\ell})\cdot
           \prod_{\ell=j}^i \prod_{r\in J_\ell} \sfu_{\ell,r}\right).
\end{split}
\end{equation}
\end{Rem}


\subsection{Explicit description for $\mu=0,\lambda=n\omega_{n-1}$}\label{ssec truncation ideal}
\

Following~\cite[\S7(ii)]{ft}, consider new Cartan generators
$\{A^\pm_{i,\pm r}\}_{1\leq i<n}^{r\geq 0}$ of
$U^{\ad,\mu}_\vv[\sz^{\pm 1}_1,\ldots,\sz^{\pm 1}_N]$
which are uniquely characterized by $A^\pm_{i,0}:=(\phi^\pm_i)^{-1}$
and
\begin{equation}\label{A-generators}
  \psi^+_i(z)=
  \left(Z_i(z)\frac{\prod_{j-i}A^+_j(\vv^{-1}z)}{A^+_i(z)A^+_i(\vv^{-2}z)}\right)^+,\
  \psi^-_i(z)=
  \left(\frac{\hat{\sZ}_i(z)\prod_{1\leq s\leq N}^{i_s=i}\sz_s}{(-z/\vv)^{\alphavee_i(\mu)}}\cdot
        \frac{\prod_{j-i}A^-_j(\vv^{-1}z)}{A^-_i(z)A^-_i(\vv^{-2}z)}\right)^-,
\end{equation}
where we set $A^\pm_i(z):=\sum_{r\geq 0} A^\pm_{i,\pm r}z^{\mp r}$
and $\hat{\sZ}_i(z):=\prod_{1\leq s\leq N}^{i_s=i} (1-\frac{z}{\vv\sz_s})$.

Following~\cite[\S8(iii)]{ft}, define the \emph{truncation ideal}
$\CI^{\unl{\lambda}}_\mu$ as the $2$-sided ideal of
$U^{\ad,\mu}_\vv[\sz_1^{\pm 1},\ldots,\sz_N^{\pm 1}]$ generated over
$\BC(\vv)[\sz_1^{\pm 1},\ldots,\sz_N^{\pm 1}]$ by the following elements:
\begin{equation}\label{truncation definition}
  A^\pm_{i,0}A^\pm_{i,\pm a_i}-(-1)^{a_i},\
  A^\pm_{i,\pm s},\ A^+_{i,a_i-r}-(-1)^{a_i}A^-_{i,-r}\ (0\leq r\leq a_i<s).
\end{equation}
For any $\unl{\lambda},\mu$, we have
  $\wt{\Phi}^{\unl{\lambda}}_\mu\colon
   A^+_i(z)\mapsto \prod_{r=1}^{a_i} \sw^{-1/2}_{i,r}\cdot \sW_i(z),\
   A^-_i(z)\mapsto \prod_{r=1}^{a_i} \sw^{1/2}_{i,t}\cdot \prod_{r=1}^{a_i} (1-\frac{z}{\sw_{i,r}})$.
Hence $\CI^{\unl{\lambda}}_\mu\subset \Ker(\wt{\Phi}^{\unl{\lambda}}_\mu)$.
The opposite inclusion is the subject of ~\cite[Conjecture 8.14]{ft}.

Let us now formulate an integral version of this conjecture. Define the $2$-sided ideal
$\fI^{\unl{\lambda}}_\mu$ of $\fU^{\ad,\mu}_\vv[\sz^{\pm 1}_1,\ldots,\sz^{\pm 1}_N]$ as
the intersection
  $\fI^{\unl{\lambda}}_\mu:=
   \CI^{\unl{\lambda}}_\mu\cap \fU^{\ad,\mu}_\vv[\sz^{\pm 1}_1,\ldots,\sz^{\pm 1}_N]$.
We also note that
  $\wt{\Phi}^{\unl{\lambda}}_\mu(\fU^{\ad,\mu}_\vv[\sz^{\pm 1}_1,\ldots,\sz^{\pm 1}_N])
   \subset \wt{\CA}^\vv[\sz_1^{\pm 1},\ldots,\sz_N^{\pm 1}]$,
due to Proposition~\ref{integrality of all t-modes} and the inclusion
  $\bz^*(\iota_*)^{-1}(\CA^\vv)\subset \wt{\CA}^\vv[\sz_1^{\pm 1},\ldots,\sz_N^{\pm 1}]$.

\begin{Conj}\label{conjectural description of kernel}
  $\fI^{\unl{\lambda}}_\mu=\Ker\left(\wt{\Phi}^{\unl{\lambda}}_\mu\colon
   \fU^{\ad,\mu}_\vv[\sz^{\pm 1}_1,\ldots,\sz^{\pm 1}_N]\to
   \wt{\CA}^\vv[\sz^{\pm 1}_1,\ldots,\sz^{\pm 1}_N]\right)$
for all $\unl{\lambda},\mu$.
\end{Conj}

The goal of this subsection is to prove a \emph{reduced version} of this
equality in the particular case $\mu=0,\lambda=n\omega_{n-1}$  (so that
$N=n$ and $a_i=i$ for $1\leq i<n$; recall that $a_0=0, a_n=0$). Here,
a reduced version means that we impose an extra relation $\prod_{i=1}^n \sz_i=1$
in all our algebras. We use $\unl{\fI}^{n\omega_{n-1}}_0$
to denote the reduced version of the corresponding truncation ideal, while
$\unl{\wt{\Phi}}{}^{n\omega_{n-1}}_0$ denotes the resulting homomorphism
between the reduced algebras.

\begin{Thm}\label{Main Theorem 1}
  $\unl{\fI}^{n\omega_{n-1}}_0=\Ker(\unl{\wt{\Phi}}{}^{n\omega_{n-1}}_0)$.
\end{Thm}

Our proof of this result is based on the identification of the reduced truncation
ideal $\unl{\fI}^{n\omega_{n-1}}_0$ with the kernel of a certain version of
the evaluation homomorphism $\ev$, which is of independent interest.

Recall the commutative diagram~(\ref{diagram quantum adjoint}). Adjoining
extra variables $\{\sz^{\pm 1}_i\}_{i=1}^n$ subject to $\prod_{i=1}^n \sz_i=1$,
we obtain the following commutative diagram:
\begin{equation}\label{diagram quantum extended}
  \begin{CD}
    U^\ad_\vv(L\ssl_n)[\sz^{\pm 1}_1,\ldots, \sz^{\pm 1}_n]/(\prod \sz_i-1) @>\ev>>
    U^{'}_\vv(\gl_n)[\sz^{\pm 1}_1,\ldots, \sz^{\pm 1}_n]/(\prod \sz_i-1)\\
    @VV{\Upsilon}V @V{\Upsilon}V\wr
    V\\
    U^{\rtt,'}_\vv(L\gl_n)[\sz^{\pm 1}_1,\ldots, \sz^{\pm 1}_n]/(\prod \sz_i-1) @>\ev^\rtt>>
    U^{\rtt,'}_\vv(\gl_n)[\sz^{\pm 1}_1,\ldots, \sz^{\pm 1}_n]/(\prod \sz_i-1)
    \end{CD}
\end{equation}
where
  $U^{'}_\vv(\gl_n)[\sz^{\pm 1}_1,\ldots, \sz^{\pm 1}_n]/(\prod \sz_i-1):=
   U^{'}_\vv(\gl_n)\otimes_{\BC(\vv)}
   \BC(\vv)[\sz^{\pm 1}_1,\ldots, \sz^{\pm 1}_n]/(\prod_{i=1}^n \sz_i-1)$
and the other three algebras are defined likewise.

Recall the isomorphism
  $U^{\rtt,'}_\vv(L\gl_n)\simeq U^{\rtt,'}_\vv(L\ssl_n)\otimes_{\BC(\vv)} ZU^{\rtt,'}_\vv(L\gl_n)$
of~(\ref{qaffine gln decomposition}), which after adjoining
extra variables $\{\sz^{\pm 1}_i\}_{i=1}^n$ subject to $\prod_{i=1}^n \sz_i=1$
gives rise to an algebra isomorphism
\begin{equation*}
\begin{split}
  & U^{\rtt,'}_\vv(L\gl_n)[\sz^{\pm 1}_1,\ldots, \sz^{\pm 1}_n]/(\sz_1\ldots\sz_n-1)\simeq\\
  & U^{\rtt,'}_\vv(L\ssl_n)\otimes_{\BC(\vv)} ZU^{\rtt,'}_\vv(L\gl_n)
    \otimes_{\BC(\vv)}\BC(\vv)[\sz^{\pm 1}_1,\ldots, \sz^{\pm 1}_n]/(\sz_1\ldots\sz_n-1),
\end{split}
\end{equation*}
where $ZU^{\rtt,'}_\vv(L\gl_n)$ denotes the center of $U^{\rtt,'}_\vv(L\gl_n)$.

Let $\Delta^\pm_n(z)$ denote the quantum determinant $\qdet\ T^\pm(z)$ of
Definition~\ref{quantum determinant quantum}, and set
  $\hat{\Delta}^\pm_n(z):=\Delta(\vv^{1-n}z)$.
According to Proposition~\ref{center of quantum affine}, the center
$ZU^{\rtt,'}_\vv(L\gl_n)$ is isomorphic to the quotient of the polynomial algebra
in $\{(\hat{d}^\pm_{0})^{1/n},\hat{d}^\pm_{\pm r}\}_{r\geq 1}$ by the relation
$(\hat{d}^+_0)^{1/n}(\hat{d}^-_0)^{1/n}=1$, that is,
  $ZU^{\rtt,'}_\vv(L\gl_n)\simeq
   \BC(\vv)[\{(\hat{d}^\pm_{0})^{1/n},\hat{d}^\pm_{\pm r}\}_{r\geq 1}]/((\hat{d}^+_0)^{1/n}(\hat{d}^-_0)^{1/n}-1)$,
where $\hat{d}^\pm_{\pm r}$ are defined via
  $\hat{\Delta}^\pm_n(z)=\sum_{r\geq 0} \hat{d}^{\pm}_{\pm r} z^{\mp r}$
and $(\hat{d}^\pm_{0})^{1/n}=(t[0])^{\pm 1/n}$.
Let $\CJ$ be the $2$-sided ideal of
  $U^{\rtt,'}_\vv(L\gl_n)[\sz^{\pm 1}_1,\ldots, \sz^{\pm 1}_n]/(\prod\sz_i-1)$
generated by the following elements:
\begin{equation*}
\begin{split}
  & \hat{d}^\pm_{\pm s},\ (\hat{d}^\pm_{0})^{1/n}-1 \ \ (s>n),\\
  & \hat{d}^+_r-(-1)^re_r(\sz_1,\ldots,\sz_n),\ \hat{d}^-_{-r}-(-1)^r\sz_1\ldots\sz_ne_r(\sz^{-1}_1,\ldots,\sz^{-1}_n) \ \ (1\leq r\leq n),
\end{split}
\end{equation*}
where $e_r(\bullet)$ denotes the $r$-th elementary symmetric polynomial.
The ideal $\CJ$ is chosen so that
  $\hat{\Delta}^+_n(z)-\prod_{s=1}^n (1-\sz_s/z)\in \CJ[[z^{-1}]]$
and
  $\hat{\Delta}^-_n(z)-\prod_{s=1}^n (\sz_s-z) \in \CJ[[z]]$.
Let
\begin{equation*}
  \pi\colon U^{\rtt,'}_\vv(L\gl_n)[\sz^{\pm 1}_1,\ldots, \sz^{\pm 1}_n]/(\sz_1\ldots \sz_n-1)
  \twoheadrightarrow U^{\rtt,'}_\vv(L\ssl_n)[\sz^{\pm 1}_1,\ldots, \sz^{\pm 1}_n]/(\sz_1\ldots \sz_n-1)
\end{equation*}
be the natural projection along $\CJ$.
Set $X^{1/n}_0:=\ev^\rtt((\hat{d}^+_0)^{1/n}), X^{1/n}_n:=-\ev^\rtt((\hat{d}^-_{0})^{1/n})$,
and $X_r:=\ev^\rtt(\hat{d}^+_r)=(-1)^n\ev^\rtt(\hat{d}^-_{-n+r})$ for $0\leq r\leq n$,
where the last equality follows from the explicit
formulas for $\ev^\rtt$ (which also imply $\ev^\rtt(\hat{d}^\pm_{\pm s})=0$ for $s>n$).
Then, the center
  $ZU^{\rtt,'}_\vv(\gl_n)[\sz^{\pm 1}_1,\ldots, \sz^{\pm 1}_n]/(\sz_1\ldots\sz_n-1)$
of
  $U^{\rtt,'}_\vv(\gl_n)[\sz^{\pm 1}_1,\ldots, \sz^{\pm 1}_n]/(\sz_1\ldots\sz_n-1)$
is isomorphic to
  $\BC(\vv)[\sz^{\pm 1}_1,\ldots,\sz^{\pm 1}_n,X^{1/n}_0,X_1,\ldots,X_{n-1},X^{1/n}_n]/(X^{1/n}_0X^{1/n}_n+1,\sz_1\ldots\sz_n-1)$.

Define the \emph{extended quantized universal enveloping} $\wt{U}^\ad_\vv(\ssl_n)$
as the central reduction of
  $U^{'}_\vv(\gl_n)[\sz^{\pm 1}_1,\ldots, \sz^{\pm 1}_n]/(\prod \sz_i-1)$
by the $2$-sided ideal generated by
\begin{equation*}
  \Upsilon^{-1}(X^{1/n}_0)-1,\ \Upsilon^{-1}(X^{1/n}_n)+1,\
  \Upsilon^{-1}(X_r)-(-1)^re_r(\sz_1,\ldots,\sz_n) \ \ (0<r<n),
\end{equation*}
cf.~\cite{bg} (the appearance of $\ssl_n$ is due to the fact that $\Upsilon^{-1}(X_0)=1$).
By abuse of notation, we denote the corresponding projection
  $U^{'}_\vv(\gl_n)[\sz^{\pm 1}_1,\ldots, \sz^{\pm 1}_n]/(\prod \sz_i-1)
   \twoheadrightarrow \wt{U}^\ad_\vv(\ssl_n)$
by $\pi$ again. Likewise, define $\wt{U}^{\rtt,'}_\vv(\ssl_n)$ as the central reduction of
  $U^{\rtt,'}_\vv(\gl_n)[\sz^{\pm 1}_1,\ldots, \sz^{\pm 1}_n]/(\prod \sz_i-1)$
by the $2$-sided ideal generated by
  $\{X^{1/n}_0-1, X^{1/n}_n+1, X_r-(-1)^re_r(\sz_1,\ldots,\sz_n)\}_{r=1}^{n-1}$.
By abuse of notation, we denote the corresponding projection
  $U^{\rtt,'}_\vv(\gl_n)[\sz^{\pm 1}_1,\ldots, \sz^{\pm 1}_n]/(\prod \sz_i-1)
   \twoheadrightarrow \wt{U}^{\rtt,'}_\vv(\ssl_n)$
by $\pi$ again. We denote the composition
  $U^{\rtt,'}_\vv(L\gl_n)[\sz^{\pm 1}_1,\ldots, \sz^{\pm 1}_n]/(\prod \sz_i-1)
   \xrightarrow{\ev^\rtt} U^{\rtt,'}_\vv(\gl_n)[\sz^{\pm 1}_1,\ldots, \sz^{\pm 1}_n]/(\prod \sz_i-1)
   \xrightarrow{\pi} \wt{U}^{\rtt,'}_\vv(\ssl_n)$
by $\ol{\ev}^\rtt$. Note that by construction it factors through
  $\pi\colon U^{\rtt,'}_\vv(L\gl_n)[\sz^{\pm 1}_1,\ldots, \sz^{\pm 1}_n]/(\prod \sz_i-1)\to
   U^{\rtt,'}_\vv(L\ssl_n)[\sz^{\pm 1}_1,\ldots, \sz^{\pm 1}_n]/(\prod \sz_i-1)$,
and we denote the corresponding homomorphism
  $U^{\rtt,'}_\vv(L\ssl_n)[\sz^{\pm 1}_1,\ldots, \sz^{\pm 1}_n]/(\prod \sz_i-1)
   \to \wt{U}^{\rtt,'}_\vv(\ssl_n)$
by $\ol{\ev}^\rtt$ again. Likewise, we denote the composition
  $U^\ad_\vv(L\ssl_n)[\sz^{\pm 1}_1,\ldots, \sz^{\pm 1}_n]/(\prod \sz_i-1)\xrightarrow{\ev}
   U^{'}_\vv(\gl_n)[\sz^{\pm 1}_1,\ldots, \sz^{\pm 1}_n]/(\prod \sz_i-1) \xrightarrow{\pi}
   \wt{U}^\ad_\vv(\ssl_n)$
by $\ol{\ev}$.

Summarizing all the above, we obtain the following commutative diagram:
\begin{equation}\label{diagram quantum full}
  \begin{CD}
    U^{\ad}_\vv(L\ssl_n)[\sz^{\pm 1}_1,\ldots, \sz^{\pm 1}_n]/(\prod \sz_i-1) @>\ol{\ev}>> \wt{U}^{\ad}_\vv(\ssl_n)\\
      @VV{\Upsilon}V @V{\Upsilon}V\wr V\\
    U^{\rtt,'}_\vv(L\gl_n)[\sz^{\pm 1}_1,\ldots, \sz^{\pm 1}_n]/(\prod \sz_i-1) @>\ol{\ev}^\rtt>> \wt{U}^{\rtt,'}_\vv(\ssl_n)\\
    @VV{\pi}V @| \\
    U^{\rtt,'}_\vv(L\ssl_n)[\sz^{\pm 1}_1,\ldots, \sz^{\pm 1}_n]/(\prod \sz_i-1) @>\ol{\ev}^\rtt>> \wt{U}^{\rtt,'}_\vv(\ssl_n)\\
    \end{CD}
\end{equation}
Due to the isomorphism $U^\ad_\vv(L\ssl_n)\iso U^{\rtt,'}_\vv(L\ssl_n)$
of~(\ref{isom of qsl and rtt-qsl}), the composition of the left
vertical arrows of~(\ref{diagram quantum full}) is an isomorphism:
\begin{equation*}
  \pi\circ \Upsilon\colon U^\ad_\vv(L\ssl_n)[\sz^{\pm 1}_1,\ldots, \sz^{\pm 1}_n]/(\sz_1\ldots \sz_n-1)
  \iso U^{\rtt,'}_\vv(L\ssl_n)[\sz^{\pm 1}_1,\ldots, \sz^{\pm 1}_n]/(\sz_1\ldots\sz_n-1).
\end{equation*}

The commutative diagram~(\ref{diagram quantum full}) in turn gives rise to the following
commutative diagram:
\begin{equation}\label{diagram quantum full integral}
  \begin{CD}
    \fU^{\ad}_\vv(L\ssl_n)[\sz^{\pm 1}_1,\ldots, \sz^{\pm 1}_n]/(\prod \sz_i-1) @>\ol{\ev}>> \wt{\fU}^{\ad}_\vv(\ssl_n)\\
      @VV{\Upsilon}V @V{\Upsilon}V\wr V\\
    \fU^{\rtt,'}_\vv(L\gl_n)[\sz^{\pm 1}_1,\ldots, \sz^{\pm 1}_n]/(\prod \sz_i-1) @>\ol{\ev}^\rtt>> \wt{\fU}^{\rtt,'}_\vv(\ssl_n)\\
    @VV{\pi}V @| \\
    \fU^{\rtt,'}_\vv(L\ssl_n)[\sz^{\pm 1}_1,\ldots, \sz^{\pm 1}_n]/(\prod \sz_i-1) @>\ol{\ev}^\rtt>> \wt{\fU}^{\rtt,'}_\vv(\ssl_n)\\
    \end{CD}
\end{equation}
and the composition $\pi\circ \Upsilon$ on the left is again an algebra isomorphism.

Here we use the following notations:

\noindent
$\bullet$
  $\fU^{\ad}_\vv(L\ssl_n)[\sz^{\pm 1}_1,\ldots, \sz^{\pm 1}_n]/(\prod \sz_i-1):=
    \fU^\ad_\vv(L\ssl_n)\otimes_{\BC[\vv,\vv^{-1}]}
    \BC[\vv,\vv^{-1}][\sz^{\pm 1}_1,\ldots, \sz^{\pm 1}_n]/(\prod \sz_i-1)$,
or alternatively it can be defined as a $\BC[\vv,\vv^{-1}]$-subalgebra of
$U^\ad_\vv(L\ssl_n)[\sz^{\pm 1}_1,\ldots, \sz^{\pm 1}_n]/(\prod \sz_i-1)$
generated by
  $\{E^{(r)}_{j,i+1},F^{(r)}_{i+1,j}\}_{1\leq j\leq i<n}^{r\in \BZ}
   \cup\{\psi^\pm_{i,\pm s}\}_{1\leq i<n}^{s\geq 1}\cup\{\phi_i^{\pm 1}\}_{i=1}^{n-1}
   \cup\{\sz_i^{\pm 1}\}_{i=1}^n$.

\noindent
$\bullet$
  $\fU^{\rtt,'}_\vv(L\gl_n)[\sz^{\pm 1}_1,\ldots, \sz^{\pm 1}_n]/(\prod \sz_i-1):=
   \fU^{\rtt,'}_\vv(L\gl_n)\otimes_{\BC[\vv,\vv^{-1}]}
   \BC[\vv,\vv^{-1}][\sz^{\pm 1}_1,\ldots, \sz^{\pm 1}_n]/(\prod \sz_i-1)$
or alternatively it can be viewed as a $\BC[\vv,\vv^{-1}]$-subalgebra of
$U^{\rtt,'}_\vv(L\gl_n)[\sz^{\pm 1}_1,\ldots, \sz^{\pm 1}_n]/(\prod \sz_i-1)$ generated by
  $\{t^\pm_{ij}[\pm r]\}_{1\leq i,j\leq n}^{r\in \BN}
   \cup\{(t[0])^{\pm 1/n}\} \cup\{\sz_i^{\pm 1}\}_{i=1}^n$.

\noindent
$\bullet$
  $\fU^{\rtt,'}_\vv(L\ssl_n)[\sz^{\pm 1}_1,\ldots, \sz^{\pm 1}_n]/(\prod \sz_i-1)$
is defined similarly.

\noindent
$\bullet$
$\wt{\fU}^\ad_\vv(\ssl_n)$ denotes the reduced extended version of $\fU^{'}_\vv(\gl_n)$,
or alternatively it can be viewed as a $\BC[\vv,\vv^{-1}]$-subalgebra of
$\wt{U}^\ad_\vv(\ssl_n)$ generated by
  $\{E_{j,i+1},F_{i+1,j}\}_{1\leq j\leq i<n}
   \cup\{\phi_i^{\pm 1}\}_{i=1}^{n-1}\cup\{\sz_i^{\pm 1}\}_{i=1}^n$.

\noindent
$\bullet$
$\wt{\fU}^{\rtt,'}_\vv(\ssl_n)$ denotes the reduced extended version of
$\fU^{\rtt,'}_\vv(\gl_n)$, or alternatively it can be viewed as a
$\BC[\vv,\vv^{-1}]$-subalgebra of $\wt{U}^{\rtt,'}_\vv(\ssl_n)$ generated by
$\{t^\pm_{ij}\}_{i,j=1}^{n}\cup \{t^{\pm 1/n}\}\cup\{\sz_i^{\pm 1}\}_{i=1}^n$.

\medskip
Consider a natural projection
\begin{equation}\label{shifted nonshifted}
  \varkappa\colon
  \fU^{\ad,0}_\vv[\sz^{\pm 1}_1,\ldots,\sz^{\pm 1}_n]/(\sz_1\ldots \sz_n-1)\twoheadrightarrow
  \fU^\ad_\vv(L\ssl_n)[\sz^{\pm 1}_1,\ldots, \sz^{\pm 1}_n]/(\sz_1\ldots \sz_n-1)
\end{equation}
whose kernel is a $2$-sided ideal generated by $\{\phi^+_i\phi^-_i-1\}_{i=1}^{n-1}$.
Let $\wt{\ev}$ denote the composition $\ol{\ev}\circ \varkappa$.
The following result can be viewed as a trigonometric counterpart
of Theorem~\ref{truncation as kernel yangian}:

\begin{Thm}\label{truncation as kernel quantum}
  $\unl{\fI}^{n\omega_{n-1}}_0=
  \Ker\left(\wt{\ev}\colon \fU^{\ad,0}_\vv[\sz^{\pm 1}_1,\ldots, \sz^{\pm 1}_n]/(\prod \sz_i-1)
  \to \wt{\fU}^\ad_\vv(\ssl_n)\right)$.
\end{Thm}

\begin{proof}
In the particular case $\mu=0,\lambda=n\omega_{n-1}$, we note that
  $Z_1(z)=\ldots=Z_{n-2}(z)=1, Z_{n-1}(z)=\prod_{s=1}^n (1-\frac{\sz_s}{\vv^{-1}z}),\
   \hat{Z}_1(z)=\ldots=\hat{Z}_{n-2}(z)=1,
   \hat{Z}_{n-1}(z)=\prod_{s=1}^n (1-\frac{\vv^{-1}z}{\sz_s})$.
Let us introduce extra currents $A^\pm_0(z),A^\pm_n(z)$ via
  $A^\pm_0(z):=1, A^+_n(z):=\prod_{s=1}^n (1-\sz_s/z), A^-_n(z)=\prod_{s=1}^n (\sz_s-z)$.
Then, formula~(\ref{A-generators}) relating the generating
series $\{\psi^\pm_k(z)\}_{k=1}^{n-1}$ to $\{A^\pm_k(z)\}_{k=1}^{n-1}$
can be uniformly written as follows:
\begin{equation}\label{psi first formula}
  \psi^\pm_k(z)=\frac{A^\pm_{k-1}(\vv^{-1}z)A^\pm_{k+1}(\vv^{-1}z)}{A^\pm_k(z)A^\pm_k(\vv^{-2}z)}
  \ \ \mathrm{for\ any}\ 1\leq k\leq n-1.
\end{equation}
Denoting the $\varkappa$-images of $\psi^\pm_k(z),A^\pm_k(z)$ again by
$\psi^\pm_k(z),A^\pm_k(z)$, we will view~(\ref{psi first formula}) from now on as
an equality of the series with coefficients in the algebra
$U^\ad_\vv(L\ssl_n)[\sz^{\pm 1}_1,\ldots, \sz^{\pm 1}_n]/(\prod \sz_i-1)$.

Let $\Delta^\pm_k(z)$ denote the $k$-th principal quantum minor
$t^{1\ldots k;\pm}_{1\ldots k}(z)$ of $T^\pm(z)$, see Definition~\ref{quantum minor quantum}.
According to~\cite{m}, the following equality holds:
\begin{equation*}
  \Upsilon(\psi^\pm_k(z))=
  \frac{\Delta^\pm_{k-1}(\vv^{1-k}z)\Delta^\pm_{k+1}(\vv^{-1-k}z)}
  {\Delta^\pm_{k}(\vv^{1-k}z)\Delta^\pm_{k}(\vv^{-1-k}z)}.
\end{equation*}
Generalizing $\hat{\Delta}^\pm_n(z)$, define
  $\hat{\Delta}^\pm_k(z):=\Delta^\pm_k(\vv^{1-k}z)$.
Then, the above formula reads as
\begin{equation*}
  \Upsilon(\psi^\pm_k(z))=
  \frac{\hat{\Delta}^\pm_{k-1}(\vv^{-1}z)\hat{\Delta}^\pm_{k+1}(\vv^{-1}z)}
  {\hat{\Delta}^\pm_{k}(z)\hat{\Delta}^\pm_{k}(\vv^{-2}z)}.
\end{equation*}
By abuse of notation, let us denote the image $\pi(\hat{\Delta}^\pm_k(z))$ by
$\hat{\Delta}^\pm_k(z)$ again. Note that $\hat{\Delta}^\pm_n(z)=A^\pm_n(z)$,
due to our definition of $\pi$. Combining this with~(\ref{psi first formula}),
we obtain the following result:

\begin{Cor}\label{image of A quantum}
Under the isomorphism
\begin{equation*}
  \pi\circ \Upsilon\colon
  \fU^\ad_\vv(L\ssl_n)[\sz^{\pm 1}_1,\ldots, \sz^{\pm 1}_n]/(\sz_1\ldots\sz_n-1) \iso
  \fU^{\rtt,'}_\vv(L\ssl_n)[\sz^{\pm 1}_1,\ldots, \sz^{\pm 1}_n]/(\sz_1\ldots\sz_n-1),
\end{equation*}
the generating series $A^\pm_k(z)$ are mapped into $\hat{\Delta}^\pm_k(z)$,
that is, $\pi\circ \Upsilon(A^\pm_k(z))=\hat{\Delta}^\pm_k(z)$.
\end{Cor}

Combining this result with the commutativity of the
diagram~(\ref{diagram quantum full integral}) and the explicit formulas
$\ev^\rtt(T^+(z))=T^+-T^-z^{-1}, \ev^\rtt(T^-(z))=T^--T^+z$, we get

\begin{Cor}
  $\unl{\fI}^{n\omega_{n-1}}_0\subseteq \Ker(\wt{\ev})$.
\end{Cor}

The opposite inclusion $\unl{\fI}^{n\omega_{n-1}}_0\supseteq \Ker(\wt{\ev})$
follows from the equality $\wt{\ev}=\ol{\ev}\circ \varkappa$, the obvious inclusion
$\Ker(\varkappa)\subset \unl{\fI}^{n\omega_{n-1}}_0$, the commutativity of the
diagrams~(\ref{diagram quantum full},~\ref{diagram quantum full integral}),
and Theorem~\ref{alternative kernel quantum} by noticing that
$\hat{\Delta}^\pm_1(z)=t^\pm_{11}(z)$ and so
\begin{equation*}
\begin{split}
  & (\pi\circ \Upsilon)^{-1}(t^\pm_{11}[\pm r])=
    A^\pm_{1,\pm r}\in \varkappa(\unl{\fI}^{n\omega_{n-1}}_0)\ \mathrm{for}\ r>1,\\
  & (\pi\circ \Upsilon)^{-1}(t^\pm_{11}[\pm 1]+t^\mp_{11}[0])=
    A^\pm_{1,\pm 1}+A^\mp_{1,0}\in \varkappa(\unl{\fI}^{n\omega_{n-1}}_0).
\end{split}
\end{equation*}

This completes our proof of Theorem~\ref{truncation as kernel quantum}.
\end{proof}

Now we are ready to present the proof of Theorem~\ref{Main Theorem 1}.

\begin{proof}[Proof of Theorem~\ref{Main Theorem 1}]
Recall the subtorus $T'_W=\{g\in T_W|\det(g)=1\}$ of $T_W$, and define
  $\unl{\CA}^\vv:=K^{(\wt{\GL}(V)\times T'_W)_\CO\rtimes\wt\BC^\times}(\CR_{\GL(V),\bN})$,
so that $\unl{\CA}^\vv\simeq \CA^\vv/(\prod \sz_i-1)$.
After imposing $\prod\sz_i=1$, the homomorphism
  $\unl{\wt{\Phi}}{}^{n\omega_{n-1}}_0\colon
   \fU^{\ad,0}_\vv[\sz^{\pm 1}_1,\ldots, \sz^{\pm 1}_n]/(\prod \sz_i-1)\to
   \wt{\CA}^\vv[\sz^{\pm 1}_1,\ldots,\sz^{\pm 1}_n]/(\prod \sz_i-1)$
is a composition of the surjective homomorphism
  $\unl{\ol{\Phi}}{}^{n\omega_{n-1}}_0\colon
   \fU^{\ad,0}_\vv[\sz^{\pm 1}_1,\ldots, \sz^{\pm 1}_n]/(\prod \sz_i-1)
   \twoheadrightarrow \unl{\CA}^\vv$
(see Theorem~\ref{Surjectivity}) and an embedding
  $\bz^*(\iota_*)^{-1}\colon \unl{\CA}^\vv\hookrightarrow
   \wt{\CA}^\vv[\sz_1^{\pm 1},\ldots,\sz_n^{\pm 1}]/(\prod \sz_i-1)$,
so that
  $\Ker(\unl{\wt{\Phi}}{}^{n\omega_{n-1}}_0)=
   \Ker(\unl{\ol{\Phi}}{}^{n\omega_{n-1}}_0)$.
The homomorphism $\unl{\ol{\Phi}}{}^{n\omega_{n-1}}_0$ factors through
  $\ol{\phi}\colon \wt{\fU}^\ad_\vv(\ssl_n)\twoheadrightarrow \unl{\CA}^\vv$
(due to Theorem~\ref{truncation as kernel quantum}), and it remains
to prove the injectivity of $\ol{\phi}$. Since both $\wt{\fU}^\ad_\vv(\ssl_n)$
and $\unl{\CA}^\vv$ are free $\BC[\vv,\vv^{-1}]$-modules, $\Ker(\ol{\phi})$
is a flat $\BC[\vv,\vv^{-1}]$-module. Hence, to prove the vanishing
of $\Ker(\ol{\phi})$, it suffices to prove the vanishing of
  $\Ker(\ol{\phi}_\fra\colon \wt{U}^\ad_\vv(\ssl_n)\twoheadrightarrow \unl{\CA}^\vv_\fra)$.

To this end we will need the action of $U^\ad_\vv(L\ssl_n)$ on the localized
$T_W$-equivariant $K$-theory of the Laumon based complete quasiflags' moduli
spaces $\fQ$, see e.g.~\cite[\S12(v)]{ft}. This action factors through the
evaluation homomorphism and the action of $U^{'}_\vv(\gl_n)$ on the $T_W$-equivariant
$K$-theory in question, see~\cite[Remark~12.8(c)]{ft}. According to~\cite[\S2.26]{bf},
the resulting $U^{'}_\vv(\gl_n)$-module is nothing but the universal Verma module.
It is known that the action of $U^{'}_\vv(\gl_n)$ on the universal Verma module extends
uniquely to the action of the \emph{extended quantized universal enveloping} $\wt{U}^{'}_\vv(\gl_n)$,
and the latter action is effective.
This implies that the resulting action of $\wt{U}^\ad_\vv(\ssl_n)$ on the localized
$T'_W$-equivariant $K$-theory in question is also effective.
According to~\cite{bdghk}, the $K$-theoretic Coulomb branch $\unl{\CA}^\vv_\fra$ acts
naturally on the $T'_W$-equivariant $K$-theory in question, and the action of
$\wt{U}^\ad_\vv(\ssl_n)$ factors through the homomorphism
  $\ol{\phi}_\fra\colon \wt{U}^\ad_\vv(\ssl_n)\twoheadrightarrow \unl{\CA}^\vv_\fra$
(see~\cite[Remark~12.8(c)]{ft}). Hence, $\ol{\phi}_\fra$ is injective.

This completes our proof of Theorem~\ref{Main Theorem 1}.
\end{proof}

\begin{Cor}\label{explicit Coulomb quantum}
The reduced quantized $K$-theoretic Coulomb branch $\unl{\CA}^\vv$ is explicitly given
by $\unl{\CA}^\vv\simeq \wt{\fU}^\ad_\vv(\ssl_n)$.
\end{Cor}


\subsection{Coproduct on $\fU^{\ssc,\mu}_\vv$}\label{ssec coproduct on integral shifted}
\

In this subsection, we verify that the $\BC(\vv)$-algebra homomorphisms
  $\Delta_{\mu_1,\mu_2}\colon U^{\ssc,\mu_1+\mu_2}_\vv\to
   U^{\ssc,\mu_1}_\vv\otimes U^{\ssc,\mu_2}_\vv$
constructed in~\cite[Theorem 10.26]{ft} give rise to the same named
$\BC[\vv,\vv^{-1}]$-algebra homomorphisms
  $\Delta_{\mu_1,\mu_2}\colon \fU^{\ssc,\mu_1+\mu_2}_\vv\to
   \fU^{\ssc,\mu_1}_\vv\otimes \fU^{\ssc,\mu_2}_\vv$.
In other words, we have

\begin{Thm}\label{coproduct on integral form}
For any coweights $\mu_1,\mu_2$, the image of the $\BC[\vv,\vv^{-1}]$-subalgebra
  $\fU^{\ssc,\mu_1+\mu_2}_\vv\subset U^{\ssc,\mu_1+\mu_2}_\vv$
under the homomorphism $\Delta_{\mu_1,\mu_2}$
belongs to the $\BC[\vv,\vv^{-1}]$-subalgebra
  $\fU^{\ssc,\mu_1}_\vv\otimes \fU^{\ssc,\mu_2}_\vv\subset
   U^{\ssc,\mu_1}_\vv\otimes U^{\ssc,\mu_2}_\vv$.
This gives rise to the $\BC[\vv,\vv^{-1}]$-algebra homomorphism
\begin{equation*}
  \Delta_{\mu_1,\mu_2}\colon \fU^{\ssc,\mu_1+\mu_2}_\vv\to
  \fU^{\ssc,\mu_1}_\vv\otimes \fU^{\ssc,\mu_2}_\vv.
\end{equation*}
\end{Thm}

Before proving this result, let us recall the
key properties of $\Delta_{\mu_1,\mu_2}$. Define
integers $b_{1,i}:=\alphavee_i(\mu_1), b_{2,i}:=\alphavee_i(\mu_2)$
for $1\leq i< n$. The homomorphism $\Delta_{0,0}$ essentially
coincides with the Drinfeld-Jimbo coproduct $\Delta$ on $U_\vv(L\ssl_n)$.

If $\mu_1$ and $\mu_2$ are antidominant (that is, $b_{1,i},b_{2,i}\leq 0$
for all $i$), then our construction of $\Delta_{\mu_1,\mu_2}$
in~\cite[Theorem 10.22]{ft} is explicit and is based on the Levendorskii type
presentation of antidominantly shifted quantum affine algebras,
see~\cite[Theorem 5.5]{ft}. To state the key property of $\Delta_{\mu_1,\mu_2}$
(for antidominant $\mu_1$ and $\mu_2$) of~\cite[Propositions H.1, H.22]{ft},
we introduce the following notations:

$\bullet$
Let $U^{+}_\vv$ and $U^{-}_\vv$ be the \emph{positive} and the \emph{negative}
Borel subalgebras in the Drinfeld-Jimbo realization of $U_\vv(L\ssl_n)$, respectively.
Explicitly, they are generated over $\BC(\vv)$ by
$\{e_{i,0},(\psi^+_{i,0})^{\pm 1},F^{(1)}_{n1}\}_{i=1}^{n-1}$ and
$\{f_{i,0},(\psi^-_{i,0})^{\pm 1},E^{(-1)}_{1n}\}_{i=1}^{n-1}$, respectively.

$\bullet$
Likewise, let $U^{\ssc,\mu_1,\mu_2;+}_\vv$ and $U^{\ssc,\mu_1,\mu_2;-}_\vv$ be
the $\BC(\vv)$-subalgebras of $U^{\ssc,\mu_1+\mu_2}_\vv$ generated by
$\{e_{i,0},(\psi^+_{i,0})^{\pm 1}, F^{(1)}_{n1}\}_{i=1}^{n-1}$ and
  $\{f_{i,b_{1,i}},(\psi^-_{i,b_{1,i}+b_{2,i}})^{\pm 1}, \hat{E}^{(-1)}_{1n}\}_{i=1}^{n-1}$,
respectively. Here the element $\hat{E}^{(-1)}_{1n}$ is defined via
  $\hat{E}^{(-1)}_{1n}:=(\vv-\vv^{-1})[e_{n-1,b_{2,n-1}},\cdots,[e_{2,b_{2,2}},
   e_{1,b_{2,1}-1}]_{\vv^{-1}}\cdots]_{\vv^{-1}}$.

\begin{Prop}[\cite{ft}]\label{compatibility with jmath}
(a) There are unique $\BC(\vv)$-algebra homomorphisms
\begin{equation}\label{jmath maps}
  \jmath^+_{\mu_1,\mu_2}\colon  U^{+}_\vv\longrightarrow U^{\ssc,\mu_1,\mu_2;+}_\vv,\
  \jmath^-_{\mu_1,\mu_2}\colon  U^{-}_\vv\longrightarrow U^{\ssc,\mu_1,\mu_2;-}_\vv,
\end{equation}
such that
\begin{equation*}
\begin{split}
  & \jmath^+_{\mu_1,\mu_2}\colon e_{i,r}\mapsto e_{i,r},
    \psi^+_{i,0}\mapsto \psi^+_{i,0}, F^{(1)}_{n1}\mapsto F^{(1)}_{n1}
    \ \ \mathrm{for}\ \  1\leq i\leq n-1, r\geq 0,\\
  & \jmath^-_{\mu_1,\mu_2}\colon f_{i,s}\mapsto f_{i,s+b_{1,i}},
     \psi^-_{i,0}\mapsto \psi^-_{i,b_{1,i}+b_{2,i}}, E^{(-1)}_{1n}\mapsto \hat{E}^{(-1)}_{1n}
    \ \ \mathrm{for}\ \  1\leq i\leq n-1, s\leq 0.
\end{split}
\end{equation*}

\noindent
(b) The following diagram is commutative:
\begin{equation}\label{jmath commutativity}
  \begin{CD}
    U^{\pm}_\vv @>\Delta>> U^{\pm}_\vv\otimes U^{\pm}_\vv\\
    @VV{\jmath^\pm_{\mu_1,\mu_2}}V @VV{\jmath^\pm_{\mu_1,0}\otimes \jmath^\pm_{0,\mu_2}}V\\
    U^{\ssc,\mu_1,\mu_2;\pm}_\vv @>\Delta_{\mu_1,\mu_2}>> U^{\ssc,\mu_1,0;\pm}_\vv\otimes U^{\ssc,0,\mu_2;\pm}_\vv
    \end{CD}
\end{equation}
\end{Prop}

We shall crucially need the so-called \emph{shift homomorphisms} $\iota_{\mu,\nu_1,\nu_2}$
of~\cite[Lemma 10.24]{ft} (which are injective due to~\cite[Theorem 10.25, Appendix I]{ft}):

\begin{Prop}[\cite{ft}]\label{shift homomorphism}
For any coweight $\mu$ and antidominant coweights $\nu_1,\nu_2$, there is a
unique $\BC(\vv)$-algebra embedding
\begin{equation}\label{iota maps}
  \iota_{\mu,\nu_1,\nu_2}\colon
  U^{\ssc,\mu}_\vv\hookrightarrow U^{\ssc,\mu+\nu_1+\nu_2}_\vv
\end{equation}
defined by
\begin{equation*}
  e_i(z)\mapsto (1-z^{-1})^{-\alphavee_i(\nu_1)}e_i(z),
  f_i(z)\mapsto (1-z^{-1})^{-\alphavee_i(\nu_2)}f_i(z),
  \psi^\pm_i(z)\mapsto (1-z^{-1})^{-\alphavee_i(\nu_1+\nu_2)}\psi^\pm_i(z).
\end{equation*}
\end{Prop}

In~\cite{ft}, we used these shift homomorphisms to reduce the construction
of $\Delta_{\mu_1,\mu_2}$ for general $\mu_1,\mu_2$ to the aforementioned
case of antidominant $\mu_1,\mu_2$ by proving the following result:

\begin{Prop}[\cite{ft}]\label{compatibility with iota}
The homomorphisms $\{\Delta_{\mu_1,\mu_2}\}_{\mu_1,\mu_2}$ exist and are uniquely
determined by the condition that they coincide with those constructed before for
antidominant $\mu_1,\mu_2$ and that the following diagram is commutative for any
antidominant $\nu_1,\nu_2$:
\begin{equation}\label{iota commutativity}
  \begin{CD}
    U^{\ssc,\mu_1+\mu_2}_\vv @>\Delta_{\mu_1,\mu_2}>> U^{\ssc,\mu_1}_\vv\otimes U^{\ssc,\mu_2}_\vv\\
    @VV{\iota_{\mu,\nu_2,\nu_1}}V @VV{\iota_{\mu_1,0,\nu_1}\otimes \iota_{\mu_2,\nu_2,0}}V\\
    U^{\ssc,\mu_1+\mu_2+\nu_1+\nu_2}_\vv @>\Delta_{\mu_1+\nu_1,\mu_2+\nu_2}>> U^{\ssc,\mu_1+\nu_1}_\vv\otimes U^{\ssc,\mu_2+\nu_2}_\vv
    \end{CD}
\end{equation}
\end{Prop}

Having summarized the key properties of the coproduct homomorphisms
$\Delta_{\mu_1,\mu_2}$ of~\cite{ft}, let us now proceed to the proof of
Theorem~\ref{coproduct on integral form}.

\begin{proof}[Proof of Theorem~\ref{coproduct on integral form}]
The proof proceeds in three steps (cf.~our proof of Theorem~\ref{PBW for integral shifted}).

\medskip
\noindent
\emph{Step 1: Case $\mu_1=\mu_2=0$}.

Under the embedding
  $\Upsilon\colon U_\vv(L\ssl_n)\hookrightarrow U^\rtt_\vv(L\gl_n)$,
the Drinfeld-Jimbo coproduct $\Delta$ on $U_\vv(L\ssl_n)$ is
intertwined with the $\BC(\vv)$-extension of the RTT-coproduct
  $\Delta^\rtt\colon  \fU^\rtt_\vv(L\gl_n)\to
   \fU^\rtt_\vv(L\gl_n)\otimes \fU^\rtt_\vv(L\gl_n)$
defined via $\Delta^\rtt(T^\pm(z))=T^\pm(z)\otimes T^\pm(z)$,
see~\cite{df}. As the $\Upsilon$-preimage of $\fU^\rtt_\vv(L\gl_n)$
coincides with $\fU_\vv(L\ssl_n)$ (due to
Proposition~\ref{comparison of integral forms quantum} and the
equality\footnote{The equality $\fU_\vv(L\ssl_n)=U_\vv(L\ssl_n)\cap \fU_\vv(L\gl_n)$
immediately follows from Theorem~\ref{PBW basis coordinate affine}.}
$\fU_\vv(L\ssl_n)=U_\vv(L\ssl_n)\cap \fU_\vv(L\gl_n)$), we obtain
  $\Delta(\fU_\vv(L\ssl_n))\subset \fU_\vv(L\ssl_n)\otimes \fU_\vv(L\ssl_n)$.
This immediately implies the result of the theorem for $\mu_1=\mu_2=0$,
since $\Delta_{0,0}$ essentially coincides with $\Delta$.\footnote{To be
more precise, one needs to replace $\fU^\rtt_\vv(L\gl_n)\subset U^\rtt_\vv(L\gl_n)$
by $\fU^{\rtt,\ext}_\vv(L\gl_n)\subset U^{\rtt,\ext}_\vv(L\gl_n)$
introduced in Step 1 of our proof of Theorem~\ref{PBW for integral shifted},
while $\fU_\vv(L\ssl_n)\subset U_\vv(L\ssl_n)$ should be replaced by
$\fU^{\ssc,0}_\vv\subset U^{\ssc,0}_\vv$.}

\medskip
\noindent
\emph{Step 2: Case of antidominant $\mu_1,\mu_2$}.

For any $1\leq j\leq i<n$ and $\unl{r}=(r_j,\ldots,r_i)\in \BZ^{i-j+1}$,
recall the elements
  $E_{j,i+1}(\unl{r})\in \fU^>_\vv(L\gl_n)\simeq
   \fU^>_\vv(L\ssl_n)\simeq \fU^{\ssc,\mu_1+\mu_2;>}_\vv$
and
  $F_{i+1,j}(\unl{r})\in \fU^<_\vv(L\gl_n)\simeq
   \fU^<_\vv(L\ssl_n)\simeq \fU^{\ssc,\mu_1+\mu_2;<}_\vv$
defined in~(\ref{generalized PBW basis elements}).
We start with the following result:

\begin{Lem}\label{via Borel}
(a) If $r_j,r_{j+1},\ldots,r_i\geq 0$, then
  $\Delta_{\mu_1,\mu_2}(E_{j,i+1}(\unl{r}))\in
   \fU^{\ssc,\mu_1}_\vv\otimes \fU^{\ssc,\mu_2}_\vv$.

\noindent
(b) If $r_j\leq b_{1,j},r_{j+1}\leq b_{1,j+1},\ldots,r_i\leq b_{1,i}$, then
  $\Delta_{\mu_1,\mu_2}(F_{i+1,j}(\unl{r}))\in
   \fU^{\ssc,\mu_1}_\vv\otimes \fU^{\ssc,\mu_2}_\vv$.
\end{Lem}

\begin{proof}
(a) If $\unl{r}\in \BN^{i-j+1}$, then clearly
  $E_{j,i+1}(\unl{r})\in U^+_\vv\cap \fU_\vv(L\ssl_n)$.
As $\Delta(U^+_\vv)\subset U^+_\vv\otimes U^+_\vv$ and
  $\Delta(\fU_\vv(L\ssl_n))\subset \fU_\vv(L\ssl_n)\otimes \fU_\vv(L\ssl_n)$
(see Step 1), we get
  $\Delta(E_{j,i+1}(\unl{r}))\in
   (U^+_\vv\cap \fU_\vv(L\ssl_n))\otimes (U^+_\vv\cap \fU_\vv(L\ssl_n))$.
Combining the commutativity of the diagram~(\ref{jmath commutativity}) with
the equality $\jmath^+_{\mu_1,\mu_2}(E_{j,i+1}(\unl{r}))=E_{j,i+1}(\unl{r})$,
it remains to prove the inclusion
  $\jmath^+_{\nu_1,\nu_2}(U^+_\vv\cap \fU_\vv(L\ssl_n))\subset \fU^{\ssc,\nu_1+\nu_2}_\vv$
for antidominant $\nu_1,\nu_2$.
The latter follows from~\cite[Lemma H.9]{ft}\footnote{Here we refer to the equalities
  $\jmath^+_{\nu_1,\nu_2}(E^{(r)}_{j,i+1})=E^{(r)}_{j,i+1},
   \jmath^+_{\nu_1,\nu_2}(F^{(s)}_{i+1,j})=F^{(s)}_{i+1,j},
   \jmath^+_{\nu_1,\nu_2}(\psi^+_{i,s})=\psi^+_{i,s}$
for any $1\leq j\leq i<n$ and $r\geq 0,s\geq 1$. Actually, in~\cite[Lemma H.9]{ft}
we proved those only for $r=0,s=1$. However, since the matrix
$([c_{ij}]_\vv)_{i,j=1}^{n-1}$ is non-degenerate, for every $1\leq i<n$
there is a unique $\BC(\vv)$-linear combination of
$\{(\psi^+_{j,0})^{-1}\psi^+_{j,1}\}_{j=1}^{n-1}$, denoted by $h^\perp_{i,1}$, such that
  $[h^\perp_{i,1},e_{j,r}]=\delta_{ij}e_{j,r+1},
   [h^\perp_{i,1},f_{j,r}]=-\delta_{ij}f_{j,r+1}$.
As $\jmath^+_{\nu_1,\nu_2}(h^\perp_{i,1})=h^\perp_{i,1}$ and
the elements $E^{(r)}_{j,i+1},F^{(s)}_{i+1,j}$ can be obtained
by iteratively commuting $E^{(0)}_{j,i+1},F^{(1)}_{i+1,j}$ with
$h^\perp_{j,1}$, we immediately obtain the claimed equalities
  $\jmath^+_{\nu_1,\nu_2}(E^{(r)}_{j,i+1})=E^{(r)}_{j,i+1},
   \jmath^+_{\nu_1,\nu_2}(F^{(s)}_{i+1,j})=F^{(s)}_{i+1,j}$
for any $r\geq 0, s\geq 1$. The remaining equality
$\jmath^+_{\nu_1,\nu_2}(\psi^+_{i,s})=\psi^+_{i,s}$ follows from
$\psi^+_{i,s}=(\vv-\vv^{-1})[e_{i,0},f_{i,s}]$ for $s\geq 1$.}  and the following result:

\begin{Lem}
The $\BC[\vv,\vv^{-1}]$-subalgebra $U^+_\vv\cap \fU_\vv(L\ssl_n)$ of
$U_\vv(L\ssl_n)$ is generated by
\begin{equation*}
  \{E^{(r)}_{j,i+1}\}_{1\leq j\leq i<n}^{r\in \BN}\cup
  \{F^{(r)}_{i+1,j}\}_{1\leq j\leq i<n}^{r>0}\cup
  \{\psi^+_{i,s},(\psi^+_{i,0})^{\pm 1}\}_{1\leq i<n}^{s>0}.
\end{equation*}
\end{Lem}

\begin{proof}
Recall the embedding
  $\Upsilon\colon U_\vv(L\ssl_n)\hookrightarrow U^\rtt_\vv(L\gl_n)$.
Note that the Borel subalgebra $U^+_\vv$ of $U_\vv(L\ssl_n)$ coincides with
the $\Upsilon$-preimage of the $\BC(\vv)$-subalgebra $U^{\rtt,+}_\vv(L\gl_n)$
of $U^\rtt_\vv(L\gl_n)$ generated by
  $\{t^+_{ij}[r]\}_{1\leq i,j\leq n}^{r\in \BN}\cup \{(t^+_{ii}[0])^{-1}\}_{i=1}^n$.
Evoking the Gauss decomposition of $T^+(z)$, we see that the
$\BC[\vv,\vv^{-1}]$-subalgebra
  $\fU^{\rtt,+}_\vv(L\gl_n)=U^{\rtt,+}_\vv(L\gl_n)\cap \fU^{\rtt}_\vv(L\gl_n)$
is generated by
  $\{\tilde{e}^{(r)}_{j,i+1}\}_{1\leq j\leq i<n}^{r\in \BN}\cup
   \{\tilde{f}^{(r)}_{i+1,j}\}_{1\leq j\leq i<n}^{r>0}\cup
   \{\tilde{g}^{(r)}_{i},(\tilde{g}^+_i)^{\pm 1}\}_{1\leq i<n}^{r>0}$.
Combining this with Corollary~\ref{matching modes} and the above
equality $\fU_\vv(L\ssl_n)=U_\vv(L\ssl_n)\cap \fU_\vv(L\gl_n)$
yields the claim.
\end{proof}

This completes our proof of part (a).

\medskip
\noindent
(b) The proof of part (b) is completely analogous and utilizes
homomorphisms $\jmath^-_{\bullet,\bullet}$ instead.
\end{proof}

Let us now prove
\begin{equation}\label{coproduct restriction}
  \Delta_{\mu_1,\mu_2}(\fU^{\ssc,\mu_1+\mu_2}_\vv)\subset
  \fU^{\ssc,\mu_1}_\vv\otimes \fU^{\ssc,\mu_2}_\vv
\end{equation}
for antidominant $\mu_1,\mu_2$ by induction in $-\mu_1-\mu_2$.
The base of induction, $\mu_1=\mu_2=0$, is established in Step 1.
The following result establishes the induction step:

\begin{Prop}\label{induction coproduct down}
If~(\ref{coproduct restriction}) holds for a pair of antidominant coweights
$(\mu_1,\mu_2)$, then it also holds both for $(\mu_1,\mu_2-\omega_\ell)$ and
$(\mu_1-\omega_\ell, \mu_2)$ for any $1\leq \ell\leq n-1$.
\end{Prop}

\begin{proof}
We will prove this only for $(\mu_1,\mu_2-\omega_\ell)$, since the verification for the
second pair is completely analogous. For any $1\leq j\leq i<n$ and $r\in \BZ$,
we pick a particular decomposition $\unl{r}=(r_j,\ldots,r_i)\in \BZ^{i-j+1}$ with
$r_j+\ldots+r_i=r$ as follows: we set $r_j=r,r_{j+1}=\ldots=r_i=0$ if $\ell<j$ or $\ell>i$,
and we set $r_\ell=r,r_j=\ldots=r_{\ell-1}=r_{\ell+1}=\ldots=r_i=0$ if $j\leq \ell\leq i$.

Identifying
  $\fU^{\ssc,\mu_1+\mu_2;>}_\vv\simeq \fU^>_\vv(L\ssl_n)\simeq \fU^>_\vv(L\gl_n)$,
Theorem~\ref{PBW for half-integral}(a) guarantees that the ordered PBWD monomials
in $E_{j,i+1}(\unl{r})$ form a basis of a free $\BC[\vv,\vv^{-1}]$-module
$\fU^{\ssc,\mu_1+\mu_2;>}_\vv$. Let us now apply the morphisms of the commutative
diagram~(\ref{iota commutativity}) with $\nu_1=0,\nu_2=-\omega_\ell$ to the element
$E_{j,i+1}(\unl{r})$. As $E_{j,i+1}(\unl{r})\in \fU^{\ssc,\mu_1+\mu_2}_\vv$,
our assumption guarantees that
  $\Delta_{\mu_1,\mu_2}(E_{j,i+1}(\unl{r}))\in
   \fU^{\ssc,\mu_1}_\vv\otimes \fU^{\ssc,\mu_2}_\vv$.
Meanwhile, for any coweight $\mu$ and antidominant coweights $\nu'_1,\nu'_2$, we have
\begin{equation}\label{iota integral}
  \iota_{\mu,\nu'_1,\nu'_2}(\fU^{\ssc,\mu}_\vv)\subset \fU^{\ssc,\mu+\nu'_1+\nu'_2}_\vv,
\end{equation}
since every generator $E_{j,i+1}(\unl{r})$ (resp.\ $F_{i+1,j}(\unl{r})$ or $\psi^\pm_{i,s}$)
is mapped to a $\BC$-linear combination of elements of the form $E_{j,i+1}(\unl{r}')$
(resp.\ $F_{i+1,j}(\unl{r}')$ or $\psi^\pm_{i,s'}$) for various $\unl{r}',s'$.
Thus, we obtain
\begin{equation}\label{raz}
\begin{split}
  & \Delta_{\mu_1,\mu_2-\omega_\ell}(\iota_{\mu_1+\mu_2,-\omega_\ell,0}(E_{j,i+1}(\unl{r})))=\\
  & (\mathrm{Id}\otimes \iota_{\mu_2,-\omega_\ell,0})(\Delta_{\mu_1,\mu_2}(E_{j,i+1}(\unl{r})))
    \in \fU^{\ssc,\mu_1}_\vv\otimes \fU^{\ssc,\mu_2-\omega_\ell}_\vv.
\end{split}
\end{equation}

If $\ell<j$ or $\ell>i$, then
  $\iota_{\mu_1+\mu_2,-\omega_\ell,0}(E_{j,i+1}(\unl{r}))=E_{j,i+1}(\unl{r})$,
and so
  $\Delta_{\mu_1,\mu_2-\omega_\ell}(E_{j,i+1}(\unl{r}))\in
   \fU^{\ssc,\mu_1}_\vv\otimes \fU^{\ssc,\mu_2-\omega_\ell}_\vv$,
due to~(\ref{raz}). If $j\leq \ell\leq i$, then
  $\iota_{\mu_1+\mu_2,-\omega_\ell,0}(E_{j,i+1}(\unl{r}))=E_{j,i+1}(\unl{r})-E_{j,i+1}(\unl{r-1})$,
hence,
  $\Delta_{\mu_1,\mu_2-\omega_\ell}(E_{j,i+1}(\unl{r})-E_{j,i+1}(\unl{r-1}))\in
   \fU^{\ssc,\mu_1}_\vv\otimes \fU^{\ssc,\mu_2-\omega_\ell}_\vv$,
due to~(\ref{raz}). Combining this with
  $\Delta_{\mu_1,\mu_2-\omega_\ell}(E_{j,i+1}(\unl{r}))\in
   \fU^{\ssc,\mu_1}_\vv\otimes \fU^{\ssc,\mu_2-\omega_\ell}_\vv$
for $r\geq 0$, due to Lemma~\ref{via Borel}(a), we get
  $\Delta_{\mu_1,\mu_2-\omega_\ell}(E_{j,i+1}(\unl{r}))\in
   \fU^{\ssc,\mu_1}_\vv\otimes \fU^{\ssc,\mu_2-\omega_\ell}_\vv$
for any $r\in \BZ$. This completes the proof of the inclusion
\begin{equation*}
   \Delta_{\mu_1,\mu_2-\omega_\ell}(E_{j,i+1}(\unl{r}))\in
   \fU^{\ssc,\mu_1}_\vv\otimes \fU^{\ssc,\mu_2-\omega_\ell}_\vv
   \ \ \mathrm{for\ any}\ \ 1\leq j\leq i<n, r\in \BZ.
\end{equation*}

\medskip
The proof of the inclusion
\begin{equation*}
   \Delta_{\mu_1,\mu_2-\omega_\ell}(F_{i+1,j}(\unl{r}))\in
   \fU^{\ssc,\mu_1}_\vv\otimes \fU^{\ssc,\mu_2-\omega_\ell}_\vv
   \ \ \mathrm{for\ any}\ \ 1\leq j\leq i<n, r\in \BZ
\end{equation*}
is analogous. However, to apply Lemma~\ref{via Borel}(b),
we need another choice of decompositions $\unl{r}$. For any $1\leq j\leq i<n$ and
$r\in \BZ$, we pick a decomposition $\unl{r}=(r_j,\ldots,r_i)\in \BZ^{i-j+1}$
with $r_j+\ldots+r_i=r$ as follows: we set
  $r_j=r-b_{1,j+1}-\ldots-b_{1,i},r_{j+1}=b_{1,j+1},\ldots,r_i=b_{1,i}$
if $\ell<j$ or $\ell>i$, and we set
  $r_\ell=r-b_{1,j}-\ldots-b_{1,\ell-1}-b_{1,\ell+1}-\ldots-b_{1,i},
   r_j=b_{1,j},\ldots,r_{\ell-1}=b_{1,\ell-1},r_{\ell+1}=b_{1,\ell+1},\ldots,r_i=b_{1,i}$
if $j\leq \ell\leq i$.


\medskip
Finally, we note that
\begin{equation}\label{dva}
  \Delta_{\mu_1,\mu_2-\omega_\ell}(\iota_{\mu_1+\mu_2,-\omega_\ell,0}(\psi^\pm_{i,s}))=
  (\mathrm{Id}\otimes \iota_{\mu_2,-\omega_\ell,0})(\Delta_{\mu_1,\mu_2}(\psi^\pm_{i,s}))
  \in \fU^{\ssc,\mu_1}_\vv\otimes \fU^{\ssc,\mu_2-\omega_\ell}_\vv.
\end{equation}
Therefore,
  $\Delta_{\mu_1,\mu_2-\omega_\ell}(\psi^\pm_{i,s}-\delta_{i,\ell}\psi^\pm_{i,s-1})\in
   \fU^{\ssc,\mu_1}_\vv\otimes \fU^{\ssc,\mu_2-\omega_\ell}_\vv$.
This implies (after a simple induction in $s$ for $i=\ell$) that
  $\Delta_{\mu_1,\mu_2-\omega_\ell}(\psi^\pm_{i,s})\in
   \fU^{\ssc,\mu_1}_\vv\otimes \fU^{\ssc,\mu_2-\omega_\ell}_\vv
   \ \ \mathrm{for\ any}\ \ i,s$.

Thus, the images of all generators of $\fU^{\ssc,\mu_1+\mu_2-\omega_\ell}_\vv$
under $\Delta_{\mu_1,\mu_2-\omega_\ell}$ belong to
$\fU^{\ssc,\mu_1}_\vv\otimes \fU^{\ssc,\mu_2-\omega_\ell}_\vv$. This implies
the validity of~(\ref{coproduct restriction}) for the pair $(\mu_1,\mu_2-\omega_\ell)$.
\end{proof}

This completes our proof of Theorem~\ref{coproduct on integral form}
for antidominant $\mu_1,\mu_2$.

\medskip
\noindent
\emph{Step 3: General case}.

Having established~(\ref{coproduct restriction}) for all antidominant $\mu_1,\mu_2$
(Step 2), the validity of~(\ref{coproduct restriction}) for arbitrary $\mu_1,\mu_2$
is implied by the following result:

\begin{Lem}\label{induction coproduct up}
If~(\ref{coproduct restriction}) holds for a pair of coweights $(\mu_1,\mu_2)$,
then it also holds both for $(\mu_1,\mu_2+\omega_\ell)$ and $(\mu_1+\omega_\ell, \mu_2)$
for any $1\leq \ell\leq n-1$.
\end{Lem}

\begin{proof}
We will prove this only for $(\mu_1,\mu_2+\omega_\ell)$, since the verification for the
second pair is completely analogous. The commutativity of the
diagram~(\ref{iota commutativity}) implies the following equality:
  $\Delta_{\mu_1,\mu_2}(\iota_{\mu_1+\mu_2+\omega_\ell,-\omega_\ell,0}(E_{j,i+1}(\unl{r})))=
   (\mathrm{Id}\otimes \iota_{\mu_2+\omega_\ell,-\omega_\ell,0})
   (\Delta_{\mu_1,\mu_2+\omega_\ell}(E_{j,i+1}(\unl{r})))$.
Its left-hand side belongs to $\fU^{\ssc,\mu_1}_\vv\otimes \fU^{\ssc,\mu_2}_\vv$,
due to~(\ref{iota integral}) and our assumption. However, the argument
identical to the one used in Step 4 of our proof of
Theorem~\ref{PBW for integral shifted} yields the following implication:
\begin{equation*}
  (\mathrm{Id}\otimes \iota_{\mu_2+\omega_\ell,-\omega_\ell,0})(X)\in
  \fU^{\ssc,\mu_1}_\vv\otimes \fU^{\ssc,\mu_2}_\vv \Longrightarrow
  X\in \fU^{\ssc,\mu_1}_\vv\otimes \fU^{\ssc,\mu_2+\omega_\ell}_\vv.
\end{equation*}
This completes our proof of the inclusion
  $\Delta_{\mu_1,\mu_2+\omega_\ell}(E_{j,i+1}(\unl{r}))\in
   \fU^{\ssc,\mu_1}_\vv\otimes \fU^{\ssc,\mu_2+\omega_\ell}_\vv$.

The verification of inclusions
  $\Delta_{\mu_1,\mu_2+\omega_\ell}(F_{i+1,j}(\unl{r})),
   \Delta_{\mu_1,\mu_2+\omega_\ell}(\psi^\pm_{i,s})\in
   \fU^{\ssc,\mu_1}_\vv\otimes \fU^{\ssc,\mu_2+\omega_\ell}_\vv$
is analogous. Hence, the images of all generators of
$\fU^{\ssc,\mu_1+\mu_2+\omega_\ell}_\vv$ under $\Delta_{\mu_1,\mu_2+\omega_\ell}$
belong to $\fU^{\ssc,\mu_1}_\vv\otimes \fU^{\ssc,\mu_2+\omega_\ell}_\vv$. This implies
the validity of~(\ref{coproduct restriction}) for the pair $(\mu_1,\mu_2+\omega_\ell)$.
\end{proof}

This completes our proof of Theorem~\ref{coproduct on integral form}.
\end{proof}

We conclude this subsection with the following result:

\begin{Lem}\label{integral iota}
For any $\mu,\nu_1,\nu_2$, we have
$\iota_{\mu,\nu_1,\nu_2}^{-1}(\fU^{\ssc,\mu+\nu_1+\nu_2}_\vv)=\fU^{\ssc,\mu}_\vv$.
\end{Lem}

\begin{proof}
Since
  $\iota_{\mu+\nu_1+\nu_2,\nu'_1,\nu'_2}\circ\iota_{\mu,\nu_1,\nu_2}=
   \iota_{\mu,\nu_1+\nu'_1,\nu_2+\nu'_2}$
for any coweight $\mu$ and antidominant coweights $\nu_1,\nu_2,\nu'_1,\nu'_2$,
it suffices to verify the claim for the simplest pairs
$(\nu_1=-\omega_\ell,\nu_2=0)$ and $(\nu_1=0,\nu_2=-\omega_\ell)$, $1\leq \ell\leq n-1$.
In both cases, the inclusion
\begin{equation*}
  \{x\in U^{\ssc,\mu}_\vv: \iota_{\mu,\nu_1,\nu_2}(x)\in \fU^{\ssc,\mu+\nu_1+\nu_2}_\vv\}
  \subset \fU^{\ssc,\mu}_\vv
\end{equation*}
has been already used in Step 3 above and follows from the argument used in
Step 4 of our proof of Theorem~\ref{PBW for integral shifted}.
The opposite inclusion is just~(\ref{iota integral}).

This completes our proof of Lemma~\ref{integral iota}.
\end{proof}


\appendix
\centerline{\Large \textsf{Appendices}}
\centerline{\textsf{By Alexander Tsymbaliuk and Alex Weekes}}


\section{PBW Theorem and Rees algebra realization for the Drinfeld-Gavarini dual, and the shifted Yangian}\label{appendix with Alex Weekes 1}

In~\cite{kwwy}, dominantly shifted Yangians were defined for any semisimple Lie algebra $\fg$,
generalizing Brundan-Kleshchev's definition \cite{bk} for $\mathfrak{gl}_n$. Two issues with
the definition given in~\cite{kwwy} are now clear:
\begin{enumerate}
\item[(a)] \cite[\S 3C]{kwwy} recalled Drinfeld-Gavarini duality, and an explicit
  description of the Drinfeld-Gavarini dual based on the discussion in~\cite[\S3.5]{g}.
  However, additional assumptions seem necessary in order for this description to be correct.

\item[(b)] Applying the explicit description of (a), a presentation of the dominantly shifted
  Yangians was given in~\cite[Theorem 3.5, Definition 3.10]{kwwy}. But it was incomplete,
  as it does not include a full set of relations. In fact, writing down a complete (explicit)
  set of relations seems very difficult (at least, in terms of new Drinfeld generators).
\end{enumerate}
We rectify (a) in Proposition~\ref{appendix 1: main prop}, which is of independent interest.
We then verify that this result applies to the Yangian, yielding Theorem~\ref{appendix: main thm 1}.
This proves that the set of generators given just before~\cite[Theorem 3.5]{kwwy} is indeed correct.
We also verify that Proposition~\ref{appendix 1: main prop} applies to the RTT Yangian
$Y^\rtt_\hbar(\gl_n)$, which implies an identification of its Drinfeld-Gavarini dual with the subalgebra
$\bY^\rtt_\hbar(\gl_n)$ of Definition~\ref{RTT integral Yangian}, establishes the PBW theorem for the latter
(that we referred to in Section~\ref{sec additive counterpart}), and provides a conceptual proof of
Proposition~\ref{yangian integral forms coincide}.

Another definition of the shifted Yangian (for an arbitrary, not necessarily dominant, shift)
as a Rees algebra was given in~\cite[\S5.4]{fkprw} and~\cite[Appendix B(i)]{bfn}, precisely
to avoid the issues mentioned above. This raises

\medskip
\ \emph{Question:} Are these two definitions of the shifted Yangians equivalent for dominant shifts?

\medskip
We answer this question in the affirmative in Theorem~\ref{identification of two definitions},
which generalizes Theorem~\ref{Gavarini=Rees} for any semisimple Lie algebra $\fg$.

We conclude this appendix with one more equivalent definition of the shifted Yangians, see
Appendix~\ref{ssec shifted Yangian 3}, in particular, Theorem~\ref{I=III}.


\subsection{Drinfeld's functor}\label{ssec Drinfeld functor}
\

Let $\mathfrak{a}$ be a Lie algebra over $\BC$.  Assume that $A$ is a
{\em deformation quantization} of the Hopf algebra $U(\mathfrak{a})$,\footnote{Here
$U(\mathfrak{a})$ denotes the universal enveloping algebra of $\mathfrak{a}$ over $\BC$,
in contrast to Definition~\ref{universal enveloping}.} over $\BC[\hbar]$. In other words,
$A$ is a Hopf algebra over $\BC[\hbar]$, and there is an isomorphism of Hopf algebras
$A/\hbar A \simeq U(\mathfrak{a})$.

Denote the coproduct and the counit of $A$ by $\Delta$ and $\epsilon$, respectively.
For any $n\geq 0$, let $\Delta^n\colon A\rightarrow A^{\otimes n}$ be the $n$-th
iterated coproduct (tensor product over $\BC[\hbar]$). It is defined inductively by
$\Delta^0=\epsilon$, $\Delta^1=\operatorname{id}$, and
$\Delta^n=(\Delta\otimes \operatorname{id}^{\otimes(n-2)} )\circ \Delta^{n-1}$.
Define $\delta_n\colon A\rightarrow A^{\otimes n}$ via
\begin{equation}\label{delta for Drinfeld dual}
  \delta_n:=(\operatorname{id} - \epsilon)^{\otimes n}\circ \Delta^n.
\end{equation}

Drinfeld~\cite{d3} introduced functors on Hopf algebras, which have been studied extensively
in work of Gavarini, see e.g.~\cite{g}. In particular, the {\em Drinfeld-Gavarini dual} of $A$
is the sub-Hopf algebra $A'\subset A$ defined by
\begin{equation}\label{definition of Drinfeld dual}
  A'=\left\{a\in A : \delta_n(a)\in \hbar^n A^{\otimes n} \ \mathrm{for\ all}\ n\in \BN \right\}.
\end{equation}

As in the second part of the proof of~\cite[Proposition 2.6]{g}:

\begin{Lem}\label{Gavarini's Lemma}
For any $a,b\in A'$, we have $[a,b] \in \hbar A'$.
\end{Lem}

The dual $A'$ can be defined for any Hopf algebra $A$ over $\BC[\hbar]$.
However, the case of the greatest interest is precisely when $A/\hbar A\simeq U(\mathfrak{a})$.
In this case, one can prove that $A'$ is a deformation quantization of the coordinate ring of
a `dual' algebraic group. This is a part of the {\em quantum duality principle},
see~\cite[Theorem 1.6]{g} (called {\em Drinfeld-Gavarini duality} in~\cite[\S 3C]{kwwy}).


\subsection{PBW basis for $A'$}\label{ssec PBW for Drinfeld-Gavarini}
\

We will now make some additional assumptions on $A$. Suppose that there exists
a totally ordered set $(\cI, \leq)$, and elements $\{ x_i \}_{i\in \cI} \subset A$.
By an (ordered) {\em PBW monomial}, we mean any ordered monomial
\begin{equation}\label{ordered monomial}
  x_{i_1}\cdots x_{i_\ell} \in A
\end{equation}
with $\ell\in \BN$ and $i_1\leq \ldots \leq i_\ell$. Assume that:
\begin{equation}\tag{As1} \label{assumption 1}
  \{x_i\}_{i\in \cI}\ \mathrm{lifts\ a\ basis}\
  \{\overline{x}_i\}_{i\in \cI}\ \mathrm{for}\ \mathfrak{a},
\end{equation}
\begin{equation}\tag{As2} \label{assumption 2}
  A\ \mathrm{is\ free\ over}\ \BC[\hbar],\
  \mathrm{with\ a\ basis\ given\ by\ the\ PBW\ monomials\ in}\ \{x_i\}_{i\in \cI},
\end{equation}
\begin{equation}\tag{As3} \label{assumption 3}
  \mathrm{for\ all}\ i\in \cI,\ \mathrm{we\ have}\ \hbar x_i \in A'.
\end{equation}
%
%

We will use the multi-index notation $x^\alpha$ to denote a PBW monomial
$\prod_{i\in \cI} x_i^{\alpha_i}$ in the PBW generators $x_i$. We write
$|\alpha|=\sum_{i\in \cI} \alpha_i$ for the total degree of $x^\alpha$.
Finally, for $\overline{a}\in U(\mathfrak{a})$ we denote by $\partial(\overline{a})$
its degree with respect to the usual filtration, i.e.\ the maximal value of $|\alpha|$
over all summands $\overline{x}^\alpha$ that appear in $\overline{a}$.

In \cite[\S3.5]{g}, an explicit description of $A'$ is given in the {\em formal} case,
i.e.\ when working with complete algebras over $\BC[[\hbar]]$.  The next result is inspired
by this description, but with the aim of working instead over $\BC[\hbar]$.

\begin{Prop}\label{appendix 1: main prop}
Suppose that $A$ satisfies assumptions~(\ref{assumption 1})--(\ref{assumption 3}).
Then $A'$ is free over $\BC[\hbar]$, with a basis given by the PBW monomials in the elements
$\{\hbar x_i\}_{i\in \cI}$. In particular, $A'\subset A$ is the $\BC[\hbar]$-subalgebra
generated by $\{ \hbar x_i \}_{i\in \cI}$.
\end{Prop}

In the proof, we will make use of~\cite[Lemma 4.12]{ek} (cf.~\cite[Lemma 3.3]{g}):

\begin{Lem}\label{Etingof-Kazhdan lemma}
Let $a\in A'$ be non-zero, and write $a=\hbar^n b$ where $b \in A\setminus \hbar A$.
Then $\partial(\overline{b} )\leq n$.
\end{Lem}

\prf[Proof of Proposition~\ref{appendix 1: main prop}]
Let $c\in A'$. By assumption~(\ref{assumption 2}), we can write
$c=\sum_{k,\alpha} c_{k,\alpha}\hbar^k x^\alpha$ for some $c_{k,\alpha}\in \BC$ which
are almost all zero. Since $A'$ is an algebra, by assumption~(\ref{assumption 3}) we know
that $\hbar^k x^\alpha\in A'$ whenever $k\geq |\alpha|$.
Subtracting all such elements from $c$, we conclude that the element
\begin{equation}
  a=\sum_{k,\alpha:k < |\alpha|} c_{k,\alpha} \hbar^k x^\alpha \ \in \ A'.
\end{equation}
Assume that $a\neq 0$. Choosing
\begin{equation}
  n=\min\{k : \exists \alpha\ \mathrm{such\ that}\ k< |\alpha| \text{ and } c_{k,\alpha} \neq 0\},
\end{equation}
we can write $a = \hbar^{n} b$, where
\begin{equation}
  b\in \sum_{\alpha:n<|\alpha|} c_{n, \alpha}x^\alpha + \hbar A.
\end{equation}
From assumption~(\ref{assumption 1}) it follows that
\begin{equation}
  \overline{b}=\sum_{\alpha:n<|\alpha|} c_{n,\alpha} \overline{x}^\alpha \ \in \ U(\mathfrak{a}),
\end{equation}
and so $\partial(\overline{b})>n$. But by Lemma~\ref{Etingof-Kazhdan lemma} we should
have $\partial(\overline{b})\leq n$. So we conclude that $a=0$.

This shows that the PBW monomials in $\{\hbar x_i\}_{i\in \cI}$ span $A'$ over $\BC[\hbar]$.
But they are also linearly independent, because of assumption~(\ref{assumption 2}).
Thus, they form a basis.
\epr


\subsection{Rees algebra description of $A'$}\label{sec: Rees algebras}
\

In this subsection, we make a further assumption on $A$:
\begin{equation}\tag{As4} \label{assumption 4}
  A\ \mathrm{is\ a\ graded\ Hopf\ algebra,\ with}\ \deg(\hbar)=1\
  \mathrm{and}\ \{x_i\}_{i\in \cI}\ \mathrm{being\ homogeneous}.
\end{equation}
Note that $A'\subset A$ is then a graded sub-Hopf algebra, and that the specializations of $A, A'$
at $\hbar=0$ inherit gradings. By Proposition \ref{appendix 1: main prop}, we see that the inclusion
$A'\subset A$ induces an isomorphism of their specializations at $\hbar=1$:
\begin{equation}
\label{eq: h to 1 specialization}
  A'/(\hbar-1)A'\simeq A/(\hbar-1)A.
\end{equation}
Moreover, the images of their respective PBW generators and bases agree as
$\hbar x_i + (\hbar-1)A =x_i + (\hbar-1) A$.

Denote the algebra in~(\ref{eq: h to 1 specialization}) by $A_{\hbar=1}$.
If assumption~(\ref{assumption 4}) holds, it follows that $A_{\hbar =1}$ inherits two filtrations
$F'_\bullet A_{\hbar=1}, F_\bullet A_{\hbar=1}$, coming from $A'$ and $A$, respectively.
Denoting $\mathrm{d}_i=\deg x_i$, these filtrations may be defined explicitly in terms of the PBW monomials:
\begin{align}
  F_k A_{\hbar=1} & = \operatorname{span}_{\BC} \Big\{ x^\alpha + (\hbar-1) A\ :\ \sum_i \mathrm{d}_i \alpha_i \leq k \Big\}, \\
  F'_k A_{\hbar=1} & = \operatorname{span}_{\BC} \Big\{ x^\alpha +(\hbar-1) A\ :\ \sum_i (\mathrm{d}_i+1) \alpha_i \leq k \Big\}. \label{eq: Kazhdan filtration}
\end{align}
In particular, $F'_k A_{\hbar =1}\subset F_k A_{\hbar=1}$ for all $k\in \BZ$.

By the above discussion, we obtain another description of $A'$, as a Rees algebra:

\begin{Prop}
\label{appendix 1: main cor}
Suppose that $A$ satisfies assumptions~(\ref{assumption 1})--(\ref{assumption 4}).
Then there is a canonical isomorphism of graded Hopf algebras
\begin{equation*}
  A'\simeq \operatorname{Rees}^{F'_\bullet} (A_{\hbar =1}).
\end{equation*}
It is compatible with the canonical isomorphism $A\simeq \operatorname{Rees}^{F_\bullet}(A_{\hbar=1})$,
under the natural inclusions $A'\subset A$ and
$\operatorname{Rees}^{F'_\bullet}(A_{\hbar=1})\subset \operatorname{Rees}^{F_\bullet}(A_{\hbar=1})$.
\end{Prop}


\subsection{The Yangian of $\fg$}\label{appendix: yangian section}
\

Consider the {\em Yangian} $\Yangian=\Yangian(\fg)$ associated to a semisimple Lie algebra $\fg$.
It is the associative $\BC[\hbar]$-algebra with generators
$\{e_i^{(r)}, h_i^{(r)}, f_i^{(r)}\}_{i\in I}^{r\in \BN}$
(here $I$ denotes the set of vertices of the Dynkin diagram of $\fg$),
and relations as in~\cite[\S 3A]{kwwy} and~\cite[Definition 2.1]{gnw}, cf.~(\ref{Dr Yangian}).
For each $i\in I$, define the element $s_i:=h_i^{(1)}-\tfrac{\hbar}{2} (h_i^{(0)})^2\in \Yangian$,
cf.~(\ref{shift elements}). Then $\Yangian$ is also generated by
$\{e_i^{(0)}, h_i^{(0)}, f_i^{(0)}, s_i\}_{i\in I}$, cf.~subsection~\ref{ssec Drinfeld evaluation}.

For each positive root $\alphavee$ and $r\geq 0$, define the elements
$e_{\alphavee}^{(r)},f_{\alphavee}^{(r)}$ of $\Yangian$ via
\begin{equation}\label{eq: Y PBW}
\begin{split}
  & e_{\alphavee}^{(r)}:=[[\cdots[e_{\alphavee_{i_{\ell}}}^{(r)},e_{\alphavee_{i_{\ell-1}}}^{(0)}],\cdots, e_{\alphavee_{i_2}}^{(0)}],e_{\alphavee_{i_1}}^{(0)}],\\
  & f_{\alphavee}^{(r)}:=[f_{\alphavee_{i_1}}^{(0)},[f_{\alphavee_{i_2}}^{(0)},\cdots,[f_{\alphavee_{i_{\ell-1}}}^{(0)},f_{\alphavee_{i_{\ell}}}^{(r)}]\cdots]],
\end{split}
\end{equation}
where $\alphavee=\alphavee_{i_1}+\alphavee_{i_2}+\ldots+\alphavee_{i_{\ell-1}}+\alphavee_{i_\ell}$
is an (ordered) decomposition into simple roots such that the element
  $[f_{\alphavee_{i_1}},[f_{\alphavee_{i_2}},\cdots,[f_{\alphavee_{i_{\ell-1}}},f_{\alphavee_{i_{\ell}}}]\cdots]]$
is a non-zero element of $\fg$ (here $\{f_{\alphavee_i}\}_{i\in I}$ denote the standard
Chevalley generators of $\fg$). We will refer to these elements, together with $\{h_i^{(r)}\}_{i\in I}^{r\in \BN}$,
as the Yangian {\em PBW generators}. Throughout this appendix (and the next one), we fix some total ordering
on the set of all PBW generators. It is well-known that $\Yangian$ is free over $\BC[\hbar]$  with a basis given by
the PBW monomials, as was proven in~\cite{le}. Since the original proof of~\cite{le} contains a significant gap,
we give an alternative short proof in Appendix~\ref{appendix with Alex Weekes 2}.

$\Yangian$ is a graded Hopf algebra, with $\deg (\hbar)=1$
and $\deg (x^{(r)})=r$ for $x=e_{\alphavee}, h_i, f_{\alphavee}$.
Its coproduct is uniquely determined by
\begin{equation}\label{appendix: coprod}
\begin{split}
  & \Delta(x^{(0)})= x^{(0)}\otimes 1+1\otimes x^{(0)}\ \mathrm{for}\ x=e_{\alphavee}, h_i, f_{\alphavee},\\
  & \Delta(s_i)= s_i\otimes 1+1\otimes s_i-\hbar \sum_{\gammavee>0} \langle \alpha_i,\gammavee\rangle f_{\gammavee}^{(0)}\otimes e_{\gammavee}^{(0)}.
\end{split}
\end{equation}
A proof of these formulas appears in~\cite{gnw}. Meanwhile, the counit of $\Yangian$ is given simply by
\begin{equation}\label{appendix: counit}
  \epsilon(e_{\alphavee}^{(0)})=\epsilon(f_{\alphavee}^{(0)})=\epsilon(h_i^{(0)})=\epsilon(s_i)=0.
\end{equation}

Finally, we note that in the classical limit there is an isomorphism of graded Hopf algebras
\begin{equation}\label{eq: classical limt of Yangian}
  \Yangian/\hbar \Yangian\simeq U(\fg[t]),
\end{equation}
where $U(\fg[t])$ carries the loop grading.


\subsection{The Drinfeld-Gavarini dual of the Yangian}\label{ssec Drinfeld-Gavarini of Yangian}
\

In this subsection, we describe the Drinfeld-Gavarini dual $\Yangian'$, by applying
the results of the previous subsections. To this end, we will verify that
assumptions~(\ref{assumption 1})--(\ref{assumption 4}) hold for $\Yangian$ and its PBW generators.
Note that only assumption~(\ref{assumption 3}) remains; the others hold as discussed
in subsection~\ref{appendix: yangian section}.

Using equations (\ref{appendix: coprod}, \ref{appendix: counit}),
a straightforward calculation shows that:
\begin{Lem}
\mbox{}

\begin{enumerate}
\item[(1)]
  $\delta_n(x^{(0)})=\left\{
   \begin{array}{cl} x, & \text{if }n=1 \\
     0, & \text{otherwise} \end{array} \right.$,
  for any $x = e_{\alphavee}, h_i, f_{\alphavee}$.
\item[(2)]
  $\delta_n(s_i)=\left\{
    \begin{array}{ll} s_i, & \text{if } n=1 \\
    - \hbar \sum_{\gammavee>0} \langle \alpha_i,\gammavee\rangle f_{\gammavee}^{(0)}\otimes e_{\gammavee}^{(0)}, & \text{if } n=2 \\
    0, & \text{otherwise} \end{array}. \right.$
\end{enumerate}
\end{Lem}

Using this lemma, we can now verify assumption~(\ref{assumption 3}):

\begin{Lem}
For any PBW generator $x^{(r)}$ of $\Yangian$, the element $X^{(r+1)}:=\hbar x^{(r)}$
belongs to the Drinfeld-Gavarini dual $\Yangian'$.
\end{Lem}

\prf
By the previous lemma, we have
  $\hbar e_{\alphavee}^{(0)}, \hbar h_i^{(0)}, \hbar f_{\alphavee}^{(0)}, \hbar s_i\in \Yangian'$.
All the PBW generators $x^{(r)}$ can be obtained by taking repeated commutators of
these elements, and $X^{(r+1)}$ by repeated application of the operation
$a, b\mapsto \tfrac{1}{\hbar} [a,b]$ to the elements
$\hbar e_{\alphavee}^{(0)}, \hbar h_i^{(0)}, \hbar f_{\alphavee}^{(0)}$, and $\hbar s_i$.
Since $\Yangian'$ is closed under this operation, due to Lemma~\ref{Gavarini's Lemma},
the claim follows.
\epr

Thus Proposition~\ref{appendix 1: main prop} applies providing a complete proof of
the description of $\Yangian'$ given just before~\cite[Theorem 3.5]{kwwy}.
Note that, as mentioned above, the relations given
in~\cite[Theorem 3.5]{kwwy} are incomplete (with the exception of $\fg=\ssl_2$).
We do not address this issue here, as our methods do not provide a complete set of relations.

\begin{Thm}\label{appendix: main thm 1}
The Drinfeld-Gavarini dual $\Yangian'$ is free over $\BC[\hbar]$, with a basis given by
the PBW monomials in the elements $X^{(r+1)}$. In particular, $\Yangian'\subset \Yangian$ is
the $\BC[\hbar]$-subalgebra generated by the elements $X^{(r+1)}$.
\end{Thm}

Applying Proposition~\ref{appendix 1: main cor}, we also obtain the Rees algebra description
of $\Yangian'$ of~\cite{fkprw}. In the case of the Yangian, the filtration $F'_\bullet Y_{\hbar=1}$
from~(\ref{eq: Kazhdan filtration}) is known as the {\em Kazhdan filtration}.
\begin{Cor}\label{cor: Gavarini = Rees}
There is a canonical $\BC[\hbar]$-algebra isomorphism
\begin{equation*}
  \Yangian'\simeq \operatorname{Rees}^{F'_\bullet}(Y_{\hbar=1})
\end{equation*}
with the Rees algebra of $Y_{\hbar =1}$ with respect to the Kazhdan filtration.
\end{Cor}


\subsection{The RTT version}\label{ssec RTT version of Drinfeld-Gavarini}
\

Recall the RTT Yangian $Y^\rtt_\hbar(\gl_n)$ of subsection~\ref{ssec RTT Yangian}.
We refer to the elements $\{t^{(r)}_{ij}\}_{1\leq i,j\leq n}^{r\geq 1}$ as the
\emph{PBW generators} of $Y^\rtt_\hbar(\gl_n)$. Fix some total ordering on the set
of all PBW generators; this gives rise to the notion of the PBW monomials in
$\{t^{(r)}_{ij}\}_{1\leq i,j\leq n}^{r\geq 1}$.

$Y^\rtt_\hbar(\gl_n)$ is an $\BN$-graded Hopf algebra with $\deg(\hbar)=1,\ \deg(t^{(r)}_{ij})=r-1$.
Its coproduct $\Delta^\rtt$ and counit $\epsilon^\rtt$ are determined explicitly by
\begin{equation}\label{coproduct,counit RTT yangian}
  \Delta^\rtt(T(z))=T(z)\otimes T(z),\
  \epsilon^\rtt(T(z))=I_n.
\end{equation}

Moreover, according to Remark~\ref{limit of rtt yangian}, we have an isomorphism of graded Hopf algebras
\begin{equation}\label{rtt classical limit}
  Y^\rtt_{\hbar}(\gl_n)/\hbar Y^\rtt_{\hbar}(\gl_n)\simeq U(\gl_n[t]).
\end{equation}

\begin{Prop}\label{PBW for RTT yangian}
The PBW monomials in $\{t^{(r)}_{ij}\}_{1\leq i,j\leq n}^{r\geq 1}$ form a basis
of a free $\BC[\hbar]$-module $Y^\rtt_\hbar(\gl_n)$.
\end{Prop}

\begin{proof}
The proof is similar to that of Theorem~\ref{PBW Theorem Yangian} below.
First, combining the isomorphism~(\ref{rtt classical limit}) with the PBW theorem
for $U(\gl_n[t])$, we immediately see that the PBW monomials span $Y^\rtt_\hbar(\gl_n)$
over $\BC[\hbar]$. To prove the linear independence of the PBW monomials over $\BC[\hbar]$,
it suffices to verify that their images are linearly independent over $\BC$ when we specialize
$\hbar$ to any nonzero complex number. The latter holds for $\hbar=1$
(and thus for any $\hbar\neq 0$, since all such specializations are isomorphic),
due to~\cite[Theorem 1.4.1]{m}.

This completes our proof of Proposition~\ref{PBW for RTT yangian}.
\end{proof}

The following result provides a new viewpoint towards $\bY^\rtt_\hbar(\gl_n)$ of
Definition~\ref{RTT integral Yangian}:

\begin{Thm}\label{appendix: main thm 2}
The Drinfeld-Gavarini dual ${Y^\rtt_\hbar(\gl_n)}'$ is free over $\BC[\hbar]$,
with a basis given by the PBW monomials in the elements $\hbar t^{(r)}_{ij}$.
In particular, ${Y^\rtt_\hbar(\gl_n)}'=\bY^\rtt_\hbar(\gl_n)$.
\end{Thm}

\begin{proof}
This follows from Proposition~\ref{appendix 1: main prop} once we verify
that assumptions~(\ref{assumption 1})--(\ref{assumption 4}) hold for $Y^\rtt_\hbar(\gl_n)$.
Note that only assumption~(\ref{assumption 3}) remains; the others hold as discussed above.
The desired inclusion $\hbar t^{(r)}_{ij}\in {Y^\rtt_\hbar(\gl_n)}'$ follows immediately
from~(\ref{coproduct,counit RTT yangian}), due to
  $(\operatorname{id}-\epsilon^\rtt)(\hbar t^{(r)}_{ij})=\hbar t^{(r)}_{ij}\in \hbar Y^\rtt_\hbar(\gl_n)$.
\end{proof}

Since the $\BC[\hbar]$-algebra isomorphism $\Upsilon\colon Y_\hbar(\gl_n)\iso Y^\rtt_\hbar(\gl_n)$
of Theorem~\ref{Yangian Gauss decomposition} is actually an isomorphism of Hopf algebras,
we conclude that it gives rise to an isomorphism of the corresponding Drinfeld-Gavarini duals
$\Upsilon\colon \bY_\hbar(\gl_n)\iso \bY^\rtt_\hbar(\gl_n)$. This provides an alternative
computation-free proof of Proposition~\ref{yangian integral forms coincide}.

\begin{Rem}
Let us compare the above exposition with that of~\cite{m}, where an opposite order of
reasoning is used. In~\emph{loc.cit.}, the author works with the $\BC$-algebra $Y^\rtt(\gl_n)$
defined as a common specialization $Y^\rtt(\gl_n)=Y^\rtt_{\hbar=1}(\gl_n)=\bY^\rtt_{\hbar=1}(\gl_n)$,
endowed with two different filtrations $F_\bullet Y^\rtt(\gl_n), F'_\bullet Y^\rtt(\gl_n)$ determined by
$\deg^{F_\bullet} (t^{(r)}_{ij})=r, \deg^{F'_\bullet} (t^{(r)}_{ij})=r-1$
(these notations are opposite to those we used in subsection~\ref{sec: Rees algebras}).
First, in~\cite[Corollary 1.4.2]{m}, an algebra isomorphism
  $\operatorname{gr}^{F_\bullet} Y^\rtt(\gl_n)\simeq \BC[\{t^{(r)}_{ij}\}_{1\leq i,j\leq n}^{r\geq 1}]$
is proven, and only then an algebra isomorphism $\operatorname{gr}^{F'_\bullet} Y^\rtt(\gl_n)\simeq U(\gl_n[t])$
is deduced in~\cite[Proposition 1.5.2]{m}.
\end{Rem}


\subsection{The shifted Yangian}\label{ssec appendix shifted yangians}
\

In subsections~\ref{ssec shifted Yangian} and~\ref{ssec shifted Yangian 2}, respectively, the algebras
$\bY_{\mu}$ (for any coweight $\mu$) and $\bY'_{\mu}$ (only for a dominant coweight $\mu$)
are defined. In this section, we show that these two definitions are equivalent when $\mu$ is dominant.
Note that although these definitions were only given in the case of $\fg=\ssl_n$, they can be easily
extended to any semisimple Lie algebra $\fg$ (cf.~\cite{kwwy,fkprw}).
Till the end of this subsection, we assume that $\mu$ is dominant.

We first recall two auxiliary algebras. The first is the $\BC$-algebra $Y_\mu$ defined
in subsection~\ref{ssec shifted Yangian}, and the second is the $\BC[\hbar]$-algebra
$Y_{\mu, \hbar}$ introduced in subsection~\ref{ssec shifted Yangian 2}. Both have PBW bases
in the corresponding generators over their respective ground rings,
by Theorems~\ref{PBW fkprw} and~\ref{PBW fkprw II}, respectively.

Fixing a splitting $\mu=\mu_1+\mu_2$, recall that $Y_\mu$ has a corresponding filtration
$F^\bullet_{\mu_1, \mu_2} Y_\mu$, see~(\ref{yangian filtration}). Similarly $Y_{\mu,\hbar}$
has a corresponding grading, defined by setting $\deg (\hbar) =1$  and
\begin{equation*}
  \deg (e_{\alphavee}^{(r)})=\alphavee(\mu_1)+r,\
  \deg (f_{\alphavee}^{(r)})=\alphavee(\mu_2)+r,\
  \deg (h_i^{(r)})=\alphavee_i(\mu)+r.
\end{equation*}
Thus for $x=e_{\alphavee}, f_{\alphavee}, h_i$, we have $\deg (x^{(r)})=\deg (x)+r$, where
the internal grading $\deg(x)$ is defined via
  $\deg (e_{\alphavee}) = \alphavee(\mu_1),
   \deg (f_{\alphavee}) = \alphavee(\mu_2),
   \deg (h_i) = \alphavee_i(\mu)$.

By comparing their defining relations, it is clear that there is a $\BC$-algebra isomorphism
\begin{equation}\label{eq: iso of shifted Yangians}
  Y_{\mu,\hbar}/(\hbar-1) Y_{\mu,\hbar} \iso  Y_\mu.
\end{equation}
On the PBW generators, this isomorphism involves a shift of labels: $x^{(r)} \mapsto X^{(r+1)}$ for
$x=e_{\alphavee}, f_{\alphavee}, h_i$. It follows that $Y_\mu$ inherits a second filtration
$G_{\mu_1,\mu_2}^\bullet Y_\mu$, coming from the above grading on $Y_{\mu,\hbar}$. This is analogous
to the situation in subsection~\ref{sec: Rees algebras}: if by abuse of notation we denote
the PBW generators of $Y_\mu$ by $x^{(r)}$, then $G^k_{\mu_1, \mu_2} Y_\mu$ is the span of all PBW monomials
\begin{equation}\label{eq: PBW in Ymu}
 x_1^{(r_1)}\cdots x_\ell^{(r_\ell)}
\end{equation}
with $(\deg(x_1)+r_1)+\ldots+(\deg(x_\ell)+r_\ell)\leq k$.
Meanwhile, $F^k_{\mu_1,\mu_2} Y_\mu$ is the span of those monomials~(\ref{eq: PBW in Ymu})
with $(\deg(x_1)+r_1+1)+\ldots+(\deg(x_\ell)+r_\ell+1)\leq k$.

In particular, there is an inclusion $F^k_{\mu_1,\mu_2} Y_\mu\subset G^k_{\mu_1,\mu_2} Y_\mu$,
hence, an embedding of the Rees algebras
\begin{equation}\label{eq: Rees inclusion}
  \Rees^{F^\bullet_{\mu_1,\mu_2}} (Y_\mu) \ \subset \ \Rees^{G^\bullet_{\mu_1,\mu_2}} (Y_\mu).
\end{equation}
Now on the one hand, $\bY_\mu=\Rees^{F^\bullet_{\mu_1,\mu_2}} (Y_\mu)$ by Definition~\ref{defn 1 of shifted y}.
On the other hand, since $Y_{\mu, \hbar}$ is free over $\BC[\hbar]$, we have
$Y_{\mu,\hbar}\simeq \Rees^{G^\bullet_{\mu_1,\mu_2}} (Y_\mu)$.  Explicitly,
on PBW generators this isomorphism is defined by
\begin{equation}\label{eq: Ymuhbar as Rees}
  Y_{\mu, \hbar}\ni x^{(k)} \mapsto \hbar^{\deg(x)+k} x^{(k)}\in \hbar^{\deg(x)+k} G^{\deg(x)+k}_{\mu_1,\mu_2} Y_\mu
  \subset \Rees^{G^\bullet_{\mu_1,\mu_2}} (Y_\mu).
\end{equation}
Altogether, we obtain an injective homomorphism of graded $\BC[\hbar]$-algebras
$\bY_\mu \hookrightarrow Y_{\mu,\hbar}$.

\medskip
We can now prove a generalization of Theorem \ref{Gavarini=Rees} for an arbitrary $\fg$:

\begin{Thm}\label{identification of two definitions}
For any dominant coweight $\mu$, there is a canonical $\BC[\hbar]$-algebra isomorphism $\bY_\mu\simeq \bY'_\mu$.
For any splitting $\mu=\mu_1+\mu_2$, this isomorphism is compatible with the associated gradings on $\bY_\mu$
(from the filtration $F^\bullet_{\mu_1, \mu_2} Y_\mu$) and $\bY'_\mu$ (as a graded subalgebra of $Y_{\mu,\hbar}$).
\end{Thm}

\prf
All that remains is to check that the image of $\bY_\mu \hookrightarrow Y_{\mu,\hbar}$
is precisely $\bY'_\mu$. This follows from the ``shift'' that distinguishes the filtrations
$F^k_{\mu_1,\mu_2} Y_\mu$ and $G^k_{\mu_1, \mu_2} Y_\mu$. Indeed, for a monomial
$x_1^{(r_1)}\cdots x_\ell^{(r_\ell)}\in F^k_{\mu_1,\mu_2} Y_\mu$, the
corresponding element in the Rees algebra is
\begin{equation*}
  \hbar^{k} x_1^{(r_1)}\cdots x_\ell^{(r_\ell)}\in \hbar^k F^k_{\mu_1,\mu_2} Y_\mu
  \subset \Rees^{F^\bullet_{\mu_1,\mu_2}} (Y_\mu).
\end{equation*}
But in the Rees algebra $\Rees^{G^\bullet_{\mu_1,\mu_2}} (Y_\mu) \simeq Y_{\mu,\hbar}$,
by inverting (\ref{eq: Ymuhbar as Rees}) this element gets sent to
\begin{equation*}
  \hbar^{k-(\deg(x_1)+r_1)-\ldots-(\deg(x_\ell)+r_\ell)} x_1^{(r_1)}\cdots x_\ell^{(r_\ell)}\in Y_{\mu,\hbar}.
\end{equation*}
Since $(\deg(x_1)+r_1+1)+\ldots+(\deg(x_\ell)+r_\ell+1)\leq k$, we can rewrite this as
\begin{equation*}
  \hbar^{k-(\deg(x_1)+r_1+1)-\ldots-(\deg(x_\ell)+r_\ell+1)} (\hbar x_1^{(r_1)})\cdots (\hbar x_\ell^{(r_\ell)}),
\end{equation*}
which lies in $\bY'_\mu$. Taking spans of such monomials, we see that
$\bY_\mu=\Rees^{F^\bullet_{\mu_1,\mu_2}} (Y_\mu)\subset \bY'_\mu$. But it is easy to see that
the generators of $\bY'_\mu$ lie in $\Rees^{F^\bullet_{\mu_1,\mu_2}} (Y_\mu$), so actually $\bY_\mu=\bY'_\mu$.
\epr


\subsection{The shifted Yangian, construction III}\label{ssec shifted Yangian 3}
\

Motivated by the discussion of the previous subsection, we provide one more alternative definition of
the shifted Yangians. Fix a coweight $\mu$ of $\fg$ and set $b_i:=\alphavee_i(\mu)$, where
$\{\alphavee_i\}_{i\in I}$ are the simple roots of $\fg$. Let $\CY_{\mu,\hbar}$ be the
associative $\BC[\hbar,\hbar^{-1}]$-algebra generated by
  $\{e_i^{(r)},f_i^{(r)},h_i^{(s_i)}\}_{i\in I}^{r\geq 0, s_i\geq -b_i}$ with the defining relations
similar to those of~(\ref{kwwy yangian}) (but generalized to any $\fg$) with the only exception:
\begin{equation}\label{kwwy yangian nondominant}
  [e_i^{(r)},f_j^{(r')}]=
    \begin{cases}
      h_i^{(r+r')}, & \mbox{if } i=j\ \mathrm{and}\ r+r'\geq -b_i  \\
      \hbar^{-1}, & \mbox{if } i=j\ \mathrm{and}\ r+r'=-b_i-1  \\
      0, & \mbox{otherwise}
    \end{cases}.
\end{equation}

\begin{Rem}
If $\mu$ is dominant, the equality $r+r'=-b_i-1$ never occurs for $r,r'\geq 0$.
Thus, $\CY_{\mu,\hbar}$ is the $\BC[\hbar,\hbar^{-1}]$-extension of scalars
of $Y_{\mu,\hbar}$ of Appendix~\ref{ssec appendix shifted yangians} for dominant $\mu$.
\end{Rem}

Define the elements $\{e_{\alphavee}^{(r)},f_{\alphavee}^{(r)}\}_{\alphavee\in \Delta^+}^{r\geq 0}$
(here $\Delta^+$ denotes the set of positive roots of $\fg$) of $\CY_{\mu,\hbar}$
following~(\ref{PBW basis elements yangian non-h II}) (but generalized from type $A$
to any $\fg$, cf.~\cite[(3.1)]{fkprw}). Choose any total ordering on the set of \emph{PBW generators}
as in~(\ref{pbw bases II}) (but generalized from type $A$ to any $\fg$, cf.~\cite[(3.4)]{fkprw}).
The following is analogous to~\cite[Corollary 3.15]{fkprw} (cf.~Theorem~\ref{PBW fkprw} in type $A$):

\begin{Thm}\label{PBW fkprw III}
For any coweight $\mu$, the PBW monomials form a basis of a free $\BC[\hbar,\hbar^{-1}]$-module $\CY_{\mu,\hbar}$.
\end{Thm}

This follows immediately from~\cite[Corollary 3.15]{fkprw} and the following simple result:

\begin{Lem}\label{trivial extension}
Fix a pair of coweights $\mu_1,\mu_2$ such that $\mu_1+\mu_2=\mu$.

\noindent
(a) There is an isomorphism of $\BC[\hbar,\hbar^{-1}]$-algebras
$\CY_{\mu, \hbar} \iso Y_\mu[\hbar, \hbar^{-1}]$
defined by
\begin{equation*}
  h_i^{(r)} \mapsto \hbar^{\alphavee_i(\mu)+r} H_i^{(r+1)},\
  e_i^{(r)} \mapsto \hbar^{\alphavee_i(\mu_1)+r} E_i^{(r+1)},\
  f_i^{(r)} \mapsto \hbar^{\alphavee_i(\mu_2)+r} F_i^{(r+1)}.
\end{equation*}

\noindent
(b) The above isomorphism sends the PBW generators as follows:
\begin{equation*}
  e_{\alpha^\vee}^{(r)} \mapsto \hbar^{\alphavee(\mu_1)+r} E_{\alpha^\vee}^{(r+1)},\
  f_{\alpha^\vee}^{(r)} \mapsto \hbar^{\alphavee(\mu_2) + r} F_{\alpha^\vee}^{(r+1)}.
\end{equation*}
\end{Lem}

\begin{proof}
Part (a) is straightforward.
Part (b) follows by comparing~(\ref{PBW basis elements yangian non-h})
and~(\ref{PBW basis elements yangian non-h II}).
\end{proof}

Following Definition~\ref{dominant shifted Yangian}, we introduce:

\begin{Def}
Let $\CY_{\mu}$ be the $\BC[\hbar]$-subalgebra of $\CY_{\mu,\hbar}$ generated by
\begin{equation*}
  \{\hbar e_{\alphavee}^{(r)}\}_{\alphavee\in \Delta^+}^{r\geq 0}\cup
  \{\hbar f_{\alphavee}^{(r)}\}_{\alphavee\in \Delta^+}^{r\geq 0}\cup
  \{\hbar h_i^{(s_i)}\}_{i\in I}^{s_i\geq -b_i}.
\end{equation*}
\end{Def}

The following is the main result of this subsection:

\begin{Thm}\label{I=III}
For any coweight $\mu$, there is a canonical $\BC[\hbar]$-algebra isomorphism
$\bY_\mu\simeq \CY_\mu$.
\end{Thm}

\begin{Rem}
We note that Theorem~\ref{I=III} does not imply Theorem~\ref{identification of two definitions},
since the algebra $\bY'_\mu$ could a priori have an $\hbar$-torsion.
\end{Rem}

As $\bY_\mu=\Rees^{F_{\mu_1,\mu_2}^\bullet} (Y_\mu)$, by the very definition of the Rees algebra
we have a natural embedding $\bY_\mu \subset Y_\mu [\hbar, \hbar^{-1}]$. Applying
Lemma~\ref{trivial extension} with the same decomposition $\mu=\mu_1+\mu_2$, we also
obtain an embedding $\CY_\mu \subset \CY_{\mu,\hbar} \iso Y_\mu[\hbar, \hbar^{-1}]$.
Therefore, Theorem~\ref{I=III} follows from:

\begin{Lem}\label{auxiliary lemma}
The images of $\bY_\mu$ and $\CY_\mu$ in $Y_\mu[\hbar, \hbar^{-1}]$ are equal.
\end{Lem}

\begin{proof}
The filtration $F_{\mu_1,\mu_2}^\bullet Y_\mu$ is defined by the degrees of PBW monomials
as in~\cite[(5.1)]{fkprw} (cf.~(\ref{yangian filtration}) in type $A$). In particular,
$\bY_\mu$ is the $\BC[\hbar]$-subalgebra of $Y_\mu[\hbar, \hbar^{-1}]$ generated by the elements
\begin{equation*}
  \hbar^{\alphavee(\mu_1) + r} E_{\alpha^\vee}^{(r)},\
  \hbar^{\alphavee_i(\mu) +r} H_i^{(r)},\
  \hbar^{\alphavee(\mu_2) + r} F_{\alpha^\vee}^{(r)}.
\end{equation*}
Note that these are precisely the images of the generators
$\hbar e_{\alpha^\vee}^{(r-1)}, \hbar h_i^{(r-1)}, \hbar f_{\alpha^\vee}^{(r-1)}$ of $\CY_\mu$
under the isomorphism of Lemma~\ref{trivial extension}. The claim follows.
\end{proof}

\begin{Rem}
We note that Lemma~\ref{auxiliary lemma} provides another proof of the fact that
the Rees algebras $\Rees^{F^\bullet_{\mu_1,\mu_2}} (Y_\mu)$ are canonically isomorphic
for any choice of a splitting $\mu=\mu_1+\mu_2$.
\end{Rem}


\section{A short proof of the PBW theorem for the Yangians}\label{appendix with Alex Weekes 2}

%
%
%

The PBW theorem for the Yangians is well-known and was first proven by Levendorskii in~\cite{le}.
However, we feel that the proof of~\cite{le} contains a gap: in~\cite[p.~40]{le} it is stated that certain
exponents $m^\pm(i,j), m^0(r,j)$ are independent of $j$ without any hint (actually, this seems to be wrong),
and this fact plays a crucial role in the proof.  For this reason, we present here a short proof of the PBW
theorem for the Yangians, which is inspired by Levendorskii's, but which avoids the aforementioned
gap.\footnote{A similar proof appears in~\cite{w}, while a completely different proof of the PBW theorem for
the Yangian defined in its $J$-realization was recently presented in~\cite[Proposition 2.2]{grw}.}


\subsection{Useful Lemma}\label{ssec useful lemma}
\

Let $A=\bigoplus_{k\in \BZ} A_k$ be a graded algebra over $\BC[\hbar]$, with $1\in A_0$ and $\hbar\in A_1$.
Consider its two specializations $A_{\hbar=0}=A/\hbar A$ and $A_{\hbar=1}=A/(\hbar-1) A$.
The former is naturally graded via $A_{\hbar=0}=\bigoplus_{k\in \BZ} A_k/\hbar A_{k-1}$,
while the latter inherits a natural filtration $F_{\bullet}A_{\hbar=1}$ with
$F_kA_{\hbar=1}$ denoting the image of $\bigoplus_{\ell\leq k} A_\ell\subset A$, giving rise
to a graded algebra $\operatorname{gr} A_{\hbar=1}=\operatorname{gr}^{F_\bullet} A_{\hbar=1}$.

An explicit relation between the resulting graded $\BC$-algebras
$A_{\hbar=0}$ and $\operatorname{gr} A_{\hbar=1}$ is presented in the following result:

\begin{Lem}\label{useful lemma}
(a) There is a canonical epimorphism of graded $\BC$-algebras
  $\vartheta\colon A_{\hbar=0}\twoheadrightarrow \operatorname{gr}  A_{\hbar=1}$.

\noindent
(b) The kernel of $\vartheta$ is the image of the $\hbar$--torsion\footnote{Explicitly,
the $\hbar$--torsion of $A$ is given by
$T_\hbar(A)=\left\{a\in A:\hbar^r a=0 \text{ for some } r\geq 0\right\}$.} of $A$ in $A_{\hbar=0}$.
\end{Lem}

\begin{proof}
The proof is straightforward.
\end{proof}

\subsection{Setup}\label{ssec Setup for PBW yangian}
\

We follow subsection~\ref{appendix: yangian section} for the conventions regarding the Yangian $\Yangian$.
However, throughout this section we will work with its specialization $Y_{\hbar=1}$, which we denote simply by $Y$.
Below, we prove the PBW theorem for $Y$ over $\BC$.  We then give a simple argument extending the PBW theorem to
the one for $\Yangian$ over $\BC[\hbar]$. By abuse of notation, we denote the images of the Yangian's PBW generators
in $Y$ by $e_{\alphavee}^{(r)}, h_i^{(r)}, f_{\alphavee}^{(r)}$.

Let us recall a few basic facts about $Y$. First of all, there is a natural linear map $\fg \rightarrow Y$,
defined on the Chevalley generators by $e_i \mapsto e_i^{(0)}, h_i \mapsto h_i^{(0)}, f_i\mapsto f_i^{(0)}$,
cf.~Lemma~\ref{embed+ev rtt yangian}(a). This map is injective. Indeed,
according to~\cite[Theorem 8]{d1}\footnote{The proof of this result is presented in~\cite[Section 6]{cp3}.},
the faithful action of $\fg$ on $\fg\oplus \BC$ (the direct sum of the adjoint representation and the trivial
one-dimensional) can be extended to an action of $Y$, hence, any element in the kernel of the above map
$\fg \rightarrow Y$ is zero.

Second, the grading on $\Yangian$ of subsection~\ref{appendix: yangian section} gives rise to a filtration
$F_\bullet Y$ as in subsection~\ref{ssec useful lemma}. In particular, every PBW generator $x^{(r)}$ belongs to $F_rY$.
We note that the coproduct
$\Delta\colon Y \rightarrow Y\otimes Y$ satisfies
\begin{equation}
  \text{Total filtered degree}\Big(\Delta(x^{(r)})-x^{(r)}\otimes 1-1\otimes x^{(r)}\Big)<r
\end{equation}
for any PBW generator $x^{(r)}$, which follows from~(\ref{appendix: coprod}).
Note that the aforementioned embedding $\fg\hookrightarrow Y$ yields a surjection
$U(\fg)\twoheadrightarrow F_0 Y$. Moreover, combining the isomorphism~(\ref{eq: classical limt of Yangian})
with Lemma~\ref{useful lemma}, we obtain a graded algebra epimorphism
  $\vartheta\colon U(\fg[t])\twoheadrightarrow \bigoplus_{k\geq 0}F_k Y / F_{k-1} Y$,
in particular, we get a surjective linear map from the degree $k$ part of $U(\fg[t])$ to $F_k Y / F_{k-1} Y$.

Finally, we recall that there is a {\em translation homomorphism} $\tau_a\colon Y\rightarrow Y[a]$
(here $a$ is a formal parameter) defined on the PBW generators by
\begin{equation}
  \tau_a(x^{(r)})=\sum_{s=0}^r {r \choose s}a^{r-s}x^{(s)}
\end{equation}
for any PBW generator $x=e_{\alphavee}, h_i, f_{\alphavee}$ (note that this formula is valid for
$e_{\alphavee}, f_{\alphavee}$ with $\alphavee$ a non-simple root because of our choices (\ref{eq: Y PBW})).
In particular, it follows that the filtered degree of any PBW monomial $y\in Y$ is precisely the degree
in $a$ of $\tau_a(y)$.

Define a homomorphism $\Delta_n\colon Y\rightarrow Y[a_1]\otimes\cdots\otimes Y[a_n]$ as the composition
\begin{equation}
  Y \stackrel{\Delta^n}{\relbar\joinrel\longrightarrow} Y^{\otimes n}
  \stackrel{\tau_{a_1}\otimes \cdots \otimes \tau_{a_n}}{\relbar\joinrel\relbar\joinrel\relbar\joinrel\longrightarrow}
  Y[a_1]\otimes \cdots \otimes Y[a_n].
\end{equation}
Here $\Delta^n$ is the $n$-th iterated coproduct as in subsection~\ref{ssec Drinfeld functor}, and
$\tau_{a_i}\colon Y\rightarrow Y[a_i]$ is the translation homomorphism.
In particular, it follows from the above discussion that for any PBW generator $x^{(r)}$, we have
\begin{equation}\label{eq: monomial total degree}
  \Delta_n(x^{(r)})=a_1^r x\otimes 1\otimes\cdots\otimes 1 +1\otimes a_2^r x\otimes 1\otimes\cdots\otimes 1
  +\ldots +1\otimes\cdots\otimes 1\otimes a_n^r x
\end{equation}
modulo terms of total degree $<r$ in $a_1,\ldots,a_n$.


\subsection{The PBW Theorem for the Yangians}\label{ssec PBW yangian}
\

In this subsection, we prove the PBW theorems for $Y$ and $\Yangian$.

\begin{Thm}\label{PBW for h=1 yangian}
The PBW monomials in the generators $e_{\alphavee}^{(r)}, h_i^{(r)}, f_{\alphavee}^{(r)}$
form a $\BC$-basis of $Y$.
\end{Thm}

\prf
First, we claim that the PBW monomials span $Y$. The proof is by induction in the filtered degree.
For degree $0$, we recall that there is an algebra epimorphism $U(\fg)\twoheadrightarrow F_0 Y$,
so the usual PBW theorem for $U(\fg)$ applies; in particular, the PBW monomials in
$e_{\alphavee}^{(0)}, h_i^{(0)}, f_{\alphavee}^{(0)}$ span $F_0 Y$. For any $k>0$, recall
that the degree $k$ part of $U(\fg[t])$ surjects onto $F_k Y / F_{k-1} Y$. Combining this with the
PBW theorem for $U(\fg[t])$, we see that $F_k Y$ is spanned by the PBW monomials modulo terms
of the lower filtered degree. By induction, the claim follows.

Next, suppose that we could find a relation $R$ between some PBW monomials.
Consider the set of the PBW monomials of the maximal filtered degree $\sd$ that appear
non-trivially in this relation. Since this is a finite set, we may find a list of PBW generators
$x_1^{(\sd_1)}\leq\ldots\leq x_n^{(\sd_n)}$ (possibly with multiplicities) such that
each of these maximal degree monomials has the form
\begin{equation}
  (x_1^{(\sd_1)})^{\eps_1}\cdots (x_n^{(\sd_n)})^{\eps_n},
\end{equation}
with all $\eps_i\in \{0,1\}$ and $\sum_i \eps_i \sd_i=\sd$.
When multiplicities do occur, we take the convention
that the $\eps_i=1$ appear to the left of $\eps_i=0$.
With this convention, each tuple $(\eps_1,\ldots,\eps_n)$ corresponds uniquely to a PBW monomial.

By (\ref{eq: monomial total degree}), we find that $\Delta_n(R)$ is a sum of expressions of the form
\begin{equation}\label{eq: image of monomial under modified coproduct}
  \left(\sum_{i=1}^n 1^{\otimes (i-1)}\otimes a_i^{\sd_1} x^{(0)}_1\otimes 1^{\otimes(n-i)}\right)^{\eps_1}\cdots
  \left(\sum_{i=1}^n 1^{\otimes (i-1)}\otimes a_i^{\sd_n} x^{(0)}_n\otimes 1^{\otimes(n-i)}\right)^{\eps_n},
\end{equation}
modulo terms of total degree $<\sd$ in $a_1,\ldots,a_n$. In particular, in the
expression~(\ref{eq: image of monomial under modified coproduct}) there is a summand
\begin{equation}\label{eq: image of monomial under modified coproduct 2}
  (a_1^{\sd_1} x^{(0)}_1)^{\eps_1}\otimes (a_2^{\sd_2} x^{(0)}_2)^{\eps_2}\otimes\cdots\otimes (a_n^{\sd_n} x^{(0)}_n)^{\eps_n}
\end{equation}
which appears with coefficient $1$. Moreover, there is a unique PBW monomial for
which~(\ref{eq: image of monomial under modified coproduct 2}) appears as a summand.

Since $x^{(0)}_r$ are in the image of the embedding $\fg \hookrightarrow Y$, the
elements~(\ref{eq: image of monomial under modified coproduct 2}) are linearly independent in
$Y[a_1]\otimes\cdots\otimes Y[a_n]$.
Thus the expressions~(\ref{eq: image of monomial under modified coproduct}) are also linearly independent.
This implies that the top total degree term in $\Delta_n(R)$ must be zero, a contradiction.

Hence no linear relations exist, proving the PBW theorem for $Y$.
\epr

This PBW theorem can be easily generalized to $\Yangian$ over $\BC[\hbar]$:

\begin{Thm}\label{PBW Theorem Yangian}
$\Yangian$ is free over $\BC[\hbar]$, with a basis of the PBW monomials in the generators
$e_{\alphavee}^{(r)}, h_i^{(r)}, f_{\alphavee}^{(r)}$.
\end{Thm}

\prf
Similarly to the proof of Theorem~\ref{PBW for h=1 yangian}, we see that the PBW monomials span
$\Yangian$ over $\BC[\hbar]$. Moreover, if we specialize $\hbar$ to any complex number,
the images of the PBW monomials form a basis. Indeed, the previous theorem proves this for
$\hbar=1$ (and thus for any $\hbar\neq 0$, since all such specializations are isomorphic),
while the case $\hbar=0$ follows from~(\ref{eq: classical limt of Yangian}) and the PBW theorem for $U(\fg[t])$.

Suppose that there is some linear relation among the PBW monomials.
Its coefficients are elements of $\BC[\hbar]$. But they must vanish
wherever $\hbar$ is specialized in $\BC$, since the PBW monomials become a basis.
Therefore, all the coefficients are zero. So there are no relations, and the theorem is proved.
\epr


\end{document}